\documentclass{memo-l}
\usepackage{latexsym,amsmath,amsthm,amssymb,amsfonts,amscd,enumerate,epsfig}

\theoremstyle{definition}

\theoremstyle{remark}

\numberwithin{section}{chapter}
\numberwithin{equation}{chapter}

\swapnumbers

\newtheorem{thm}{Theorem}
\newtheorem{prop}[thm]{Proposition}
\newtheorem{lemma}[thm]{Lemma}
\newtheorem{claim}{Claim}
\newtheorem*{Claim}{Claim}
\newtheorem{defi}[thm]{Definition}
\newtheorem{rmk}[thm]{Remark}
\newtheorem*{Rmk}{Remark}

\newtheorem*{Rmks}{Remarks}
\newtheorem{cor}[thm]{Corollary}
\newtheorem{xpl}[thm]{Example}
\newtheorem*{Xpl}{Example}

\newtheorem{conj}[thm]{Conjecture}

\newenvironment{pf}[1][Proof.]{\noindent \emph{#1.}}{}
\newenvironment{enui}{\begin{enumerate}[(i)]}{\end{enumerate}}
\newenvironment{enua}{\begin{enumerate}[(a)]}{\end{enumerate}}

\newcommand\const{\equiv}
\newcommand\op{{\operatorname{op}}}
\newcommand\disj{\coprod}
\newcommand\Then{\,\Longrightarrow\,}
\newcommand\then{\Longrightarrow}
\newcommand\im{{\operatorname{im}}}
\newcommand\Aut{{\operatorname{Aut}}}
\newcommand\id{{\operatorname{id}}}
\newcommand\nn{{\nonumber}}
\newcommand\wt[1]{{\widetilde{#1}}}
\newcommand{\BAR}[1]{{\overline{#1}}}
\newcommand\al{{\alpha}}
\newcommand\be{\beta}
\newcommand\Ga{\Gamma}
\newcommand\ga{\gamma}
\newcommand\ka{\kappa}
\newcommand\de{\delta}
\newcommand\eps{\varepsilon}
\newcommand\Lam{\Lambda}
\newcommand\lam{\lambda}
\newcommand\Om{\Omega}
\newcommand\om{\omega}
\newcommand\Si{\Sigma} 
\newcommand\si{\sigma}
\newcommand\ze{\zeta}
\renewcommand\phi{\varphi}
\newcommand\na{\nabla}
\newcommand\U{\mathcal{U}}
\newcommand\V{\mathcal{V}}
\newcommand\B{\mathcal{B}}
\newcommand\BB{\widetilde{\mathcal{B}}}
\newcommand\E{E}
\newcommand\EE{\mathcal{E}}
\newcommand\EEE{\widetilde{\mathcal{E}}}
\newcommand\X{\mathcal{X}}
\newcommand\XX{\widetilde{\mathcal{X}}}
\newcommand\XXX{\hhat{\mathcal{X}}}
\newcommand\Y{\mathcal{Y}}
\newcommand\YY{\widetilde{\mathcal{Y}}}
\newcommand\YYY{\hhat{\mathcal{Y}}}
\newcommand{\N}{\mathbb{N}}
\newcommand\NN{\mathcal{N}}
\newcommand{\Z}{\mathbb{Z}}
\newcommand{\Q}{\mathbb{Q}}
\newcommand{\R}{\mathbb{R}}
\newcommand\C{\mathbb C}
\newcommand\D{{\mathcal{D}}}
\newcommand\DD{{\hhat{\mathcal{D}}}}
\newcommand\DDD{\mathbb{D}}
\newcommand\G{\mathcal{G}}
\newcommand\g{\mathfrak g}
\newcommand\A{\mathcal A}
\newcommand\AAA{\mathbb A}
\newcommand\Lie{\operatorname{Lie}}
\newcommand\EG{{\operatorname{EG}}}
\newcommand\BG{{\operatorname{BG}}}
\newcommand\sub{\subseteq}
\newcommand\x{\times}

\newcommand\wo{\setminus}
\newcommand\one{\mathbf{1}}
\newcommand\iso{\cong}
\newcommand\dd{\partial}
\newcommand\lan{\langle}
\newcommand\ran{\rangle}
\newcommand\loc{{\operatorname{loc}}}
\newcommand\M{\mathcal{M}}
\newcommand\MM{\wt{\mathcal{M}}}

\newcommand\ME{\mathcal{M}_{<\infty}}

\newcommand\barev{\BAR{\operatorname{ev}}}
\newcommand\ev{\operatorname{ev}}

\newcommand\W{{\bf W}}
\newcommand\WW{\mathcal{W}}
\newcommand\WWW{\wt{\mathcal{W}}}

\newcommand\z{{\bf z}}

\newcommand\Sym{\operatorname{Sym}}
\newcommand\Isom{\operatorname{Isom}}
\newcommand\wrt{w.r.t.~}
\newcommand\Emin{\operatorname{E}_{\min}}
\newcommand\hhat[1]{\widehat{#1}}
\newcommand\noi{\noindent}
\newcommand\new{{\operatorname{new}}}

\newcommand\Qk{{\operatorname{Q}\!\kappa}}
\newcommand\QH{{\operatorname{QH}}}

\newcommand\PSL{\operatorname{PSL}(2,\C)}
\newcommand\TR{{\mathcal{T}_{\C}}}
\newcommand\CP{\mathbb{C}\!\operatorname{P}}
\newcommand\UU{\mathcal{U}}
\newcommand\VV{\mathcal{V}}
\newcommand\La{\Delta}
\newcommand\inj{\hookrightarrow}

\newcommand\Ad{{\operatorname{Ad}}}
\newcommand\pr{{\operatorname{pr}}}
\newcommand\GW{\operatorname{GW}}

\newcommand\wi{{A^i}}
\newcommand\wzeroone{{A^{0,1}}}
\newcommand\wone{{A^1}}
\newcommand\wtwo{{A^2}}
\newcommand\ind{{\operatorname{ind}}}
\newcommand\SSS{\mathcal{S}}
\newcommand\SSSS{\widetilde{\mathcal{S}}}
\newcommand\Hom{\operatorname{Hom}}
\newcommand\End{\operatorname{End}}
\newcommand\PD{\operatorname{PD}}
\newcommand\Co{\mathcal{C}}
\newcommand\Int{{\operatorname{int}}}
\newcommand\HAT{{\widehat{\phantom{.}}}}
\newcommand\unhat[1]{\overset{\vee}{#1}}
\newcommand\Unhat{\unhat{\phantom{.}}}
\newcommand\cont{\supseteq}
\newcommand\coker{\operatorname{coker}}
\newcommand\PR{\operatorname{Pr}}

\newcommand\SO{\operatorname{SO}}
\newcommand\st{{\operatorname{st}}}
\newcommand\Slash{/\!}

\makeindex

\begin{document}

\frontmatter

\title{A Quantum Kirwan Map: Bubbling and Fredholm Theory for Symplectic Vortices over the Plane}

\author{Fabian Ziltener}
\address{Korea Institute for Advanced Study, 87 Hoegiro, Dongdaemun-gu, Seoul 130-722, Republic of Korea}
\curraddr{}
\email{fabian@kias.re.kr}
\thanks{The author acknowledges financial support by the Swiss National Science Foundation (fellowship 200020-101611).}

\author{}
\address{}
\curraddr{}
\email{}
\thanks{}

\date{}



\begin{abstract} 
Consider a Hamiltonian action of a compact connected Lie group on a symplectic manifold $(M,\om)$. Conjecturally, under suitable assumptions there exists a morphism of cohomological field theories from the equivariant Gromov-Witten theory of $(M,\om)$ to the Gromov-Witten theory of the symplectic quotient. The morphism should be a deformation of the Kirwan map. The idea, due to D.~A.~Salamon, is to define such a deformation by counting gauge equivalence classes of symplectic vortices over the complex plane $\C$.

The present memoir is part of a project whose goal is to make this definition rigorous. Its main results deal with the symplectically aspherical case. The first one states that every sequence of equivalence classes of vortices over the plane has a subsequence that converges to a new type of genus zero stable map, provided that the energies of the vortices are uniformly bounded. Such a stable map consists of equivalence classes of vortices over the plane and holomorphic spheres in the symplectic quotient. The second main result is that the vertical differential of the vortex equations over the plane (at the level of gauge equivalence) is a Fredholm operator of a specified index.

Potentially the quantum Kirwan map can be used to compute the quantum cohomology of symplectic quotients.
\end{abstract}

\maketitle

\tableofcontents

\mainmatter

\chapter{Motivation and main results}\label{chap:main}
\section{Quantum deformations of the Kirwan map}\label{sec:defo}
Let $(M,\om)$ be a symplectic manifold without boundary, and $G$ a compact connected Lie group with Lie algebra $\g$. We fix a Hamiltonian action of $G$ on $M$ and an (equivariant) momentum map%
\footnote{Momentum maps are often called \emph{moment maps} by symplectic geometers. However, the first term seems more appropriate, since the notion generalizes the linear and angular momenta appearing in classical mechanics.}
 $\mu:M\to\g^*$. Throughout this memoir, we make the following standing assumption:\\[2ex]
\noindent {\bf Hypothesis (H):} \textit{$G$ acts freely on $\mu^{-1}(0)$ and the momentum map $\mu$ is proper.}\\[2ex]
Then the symplectic quotient $\big(\BAR M:=\mu^{-1}(0)/G,\BAR\om\big)$ is well-defined, smooth and closed (i.e., compact and without boundary). The \emph{Kirwan map} is a canonical ring homomorphism 
\[\ka:H_G^*(M)\to H^*(\BAR M).\]
Here $H^*$ and $H_G^*$ denote cohomology and equivariant cohomology with rational coefficients, and the product structures are the cup products. F.~Kirwan proved \cite{Kir} that this map is surjective. Based on this result, the cohomology ring $H^*(\BAR M)$ was described in different ways by L.~C.~Jeffrey and F.~Kirwan \cite[Theorem 8.1]{JK}, S.~Tolman and J.~Weitsman \cite[Theorem 1]{TW}, and many others.

The present memoir is concerned with the problem of ``quantizing'' the Kirwan map, which was first investigated by R.~Gaio and D.~A.~Salamon. Assuming symplectic asphericity and some other restrictive conditions, in \cite[Corollary A']{GS} these authors constructed a ring homomorphism from $H_G^*(M)$ to the (small) quantum cohomology of $(\BAR M,\BAR\om)$, which intertwines the Gromov-Witten invariants of the symplectic quotient with the symplectic vortex invariants. Their result is based on an adiabatic limit in the symplectic vortex equations. It was used by K.~Cieliebak and D.~A.~Salamon in \cite[Theorem 1.3]{CS} to prove that given a monotone linear symplectic torus action on $\R^{2n}$ with minimal first equivariant Chern number at least 2, the quantum cohomology of $(\BAR M,\BAR\om)$ is isomorphic to the Batyrev ring. 

The result by Gaio and Salamon motivates the following conjecture. We denote by $\QH^*_G(M,\om)$ the equivariant quantum cohomology. By this we mean the $\Q$-vector space of all maps $\al:H_2^G(M,\Z)\to H_G^*(M)$ satisfying an equivariant version of the Novikov condition, together with a product counting holomorphic maps from $S^2$ to the fibers of the Borel construction for the action of $G$ on $M$.%
\footnote{For the definition of this product see \cite{Gi,GiK,KimEq,Lu,Ru}.}
 The Novikov condition states that for every number $C\in\R$ there are only finitely many classes $B\in H^G_2(M,\Z)$, such that $\al(B)\neq0$ and $\big\lan[\om-\mu],B\big\ran\leq C$. Here $[\om-\mu]\in H_G^2(M)$ denotes the cohomology class of the two-cocycle $\om-\mu$ in the Cartan model. 

The space $\QH^*_G(M,\om)$ is naturally a module over the equivariant Novikov ring $\Lam_\om^\mu$.%
\footnote{This ring consists of all maps $\lam:H_2^G(M,\Z)\to\Q$ satisfying an equivariant version of the Novikov condition, analogous to the one above. The product is given by convolution.}
 We denote by $\QH^*(\BAR M,\BAR\om)$ the quantum cohomology of $(\BAR M,\BAR\om)$ with coefficients in this ring. A map
\begin{equation}\label{eq:phi}\phi:H_2^G(M,\Z)\to\Hom_\Q\big(H_G^*(M),H^*(\BAR M)\big)\end{equation}
satisfying the equivariant Novikov condition%
\footnote{This condition is analogous to the one above.}%
, induces a $\Lam_\om^\mu$-module homomorphism
\begin{equation}\label{eq:phi * QH}\phi_*:\QH^*_G(M,\om)\to\QH^*(\BAR M,\BAR\om),\quad(\phi_*\al)(B):=\sum\phi(B_1)\al(B_2),\end{equation}
where the sum is over all pairs $B_1,B_2\in H_2^G(M,\Z)$ satisfying $B_1+B_2=B$. We denote by $c_1^G(M,\om)\in H_G^2(M,\Z)$ the first $G$-equivariant Chern class of $(TM,\om)$, and by 
\begin{equation}\label{eq:N inf}N:=\inf\left(\big\{\big\lan c_1^G(M,\om),B\big\ran\,\big|\,B\in H_2^G(M,\Z):\,\textrm{spherical}\big\}\cap\N\right)\in\N\cup\{\infty\}
\footnote{ In this memoir $\N:=\{1,2,\ldots\}$ does not include $0$.}
\end{equation}
the minimal equivariant Chern number. We call $(M,\om,\mu)$ \emph{semipositive} iff there exists a constant $c\in\R$ such that 
\[\big\lan[\om-\mu],B\big\ran=c\big\lan c_1^G(M,\om),B\big\ran,\]
for every spherical class $B\in H_2^G(M,\Z)$, and if $c<0$ then $N\geq\frac12\dim\BAR M$.
\begin{conj}[Quantum Kirwan map, semipositive case]\label{conj:QK} Assume that (H) holds and that $(M,\om,\mu)$ is convex at $\infty$
\footnote{This means that there exists an $\om$-compatible and $G$-invariant almost complex structure $J$ on $M$, such that the quadruple $\big(M,\om,\mu,J\big)$ is convex at $\infty$ in the sense explained before Theorem \ref{thm:bubb} below.}
 and semipositive. Then there exists a map $\phi$ as in (\ref{eq:phi}), satisfying the equivariant Novikov condition, such that the induced map $\phi_*$ as in (\ref{eq:phi * QH}) is a surjective ring homomorphism, and
\begin{eqnarray}\label{eq:phi 0}&\phi(0)=\ka,&\\
\label{eq:om mu B}&\big\lan[\om-\mu],B\big\ran\leq0,\,B\neq0\Then\phi(B)=0.& 
\end{eqnarray}
\end{conj}
Once proven, this conjecture will give rise to a recursion formula for $\QH^*(\BAR M,\BAR\om)$ in terms of $\QH^*_G(M,\om)$ and $\phi$.%
\footnote{The recursion is over the set 
\[\left\{\left\lan[\om-\mu],\sum_{i=1}^kB_i\right\ran\,\Bigg|\,k\in\N\cup\{0\},\,B_i\in H_2^G(M,\Z):\,\phi(B_i)\neq0\textrm{ or }\GW^G_{B_i}\neq0,\,i=1,\ldots,k\right\},\] 
where $\GW^G_{B_i}$ denotes the 3-point genus 0 equivariant Gromov-Witten invariant of $(M,\om)$ in the class $B_i$.}
 As noticed in \cite{NWZ}, without the semipositivity condition, the conjecture likely needs to be modified as follows:
\begin{conj}[Quantum Kirwan map, general situation]\label{conj:QK } Assume that (H) holds, and that $(M,\om,\mu)$ is convex at $\infty$. Then there exists a morphism of cohomological field theories (CohFT's) from the equivariant Gromov-Witten theory of $(M,\om)$ to the Gromov-Witten theory of $(\BAR M,\BAR\om)$. 
\end{conj}
For the notion of a morphism between two CohFT's $V$ and $W$ see \cite{NWZ}. Such a morphism consists of a sequence of $S_n$-invariant multilinear maps
\[\psi^n:V^{\x n}\x H^*(\BAR M_{n,1}(\AAA))\to W,\quad n\in\N_0:=\N\cup\{0\},\]
satisfying relations involving the composition maps of $V$ and $W$. Here $\BAR M_{n,1}(\AAA)$ denotes the moduli space of stable $n$-marked scaled lines. Furthermore, the action of the symmetric group $S_n$ is by permutations of the first $n$ arguments of $\psi^n$. The map $\psi^1$ plays the role of $\phi_*$ as in Conjecture \ref{conj:QK}. The map $\psi^0$ measures how much $\psi^1$ fails to be a ring homomorphism. Once proven, Conjecture \ref{conj:QK } will give rise to a recursion formula for the Gromov-Witten invariants of $(\BAR M,\BAR\om)$ in terms of the equivariant Gromov-Witten invariants of $(M,\om)$ and the morphism $(\psi^n)_n$.

The present memoir is part of a project whose goal is to prove Conjectures \ref{conj:QK} and \ref{conj:QK }.%
\footnote{Further relevant results will appear elsewhere, including \cite{ZiConsEv,ZiTrans}.}
 The approach pursued here was suggested by D.~A.~Salamon.%
\footnote{Private communication.}
 The idea is to construct the maps of the conjectures by counting symplectic vortices over the complex plane $\C$. In a first step, we will consider the (symplectically) aspherical case. This means that 
\begin{equation}\label{eq:int S 2 u om}\int_{S^2}u^*\om=0,\quad\forall u\in C^\infty(S^2,M).\end{equation}
In this case the equivariant quantum cup product is induced by the ordinary cup product on $H^*_G(M)$.
\section[Symplectic vortices]{Symplectic vortices, idea of the proof of existence of a quantum Kirwan map}
To explain the idea of the proofs of Conjectures \ref{conj:QK} and \ref{conj:QK }, we recall the symplectic vortex equations: Let $J$ be an $\om$-compatible and $G$-invariant almost complex structure on $M$, $\lan\cdot,\cdot\ran_\g$ an invariant inner product on $\g$, $(\Si,j)$ a Riemann surface, and $\om_\Si$ a compatible area form on $\Si$.%
\footnote{This means that $j$ and $\om_\Si$ determine the same orientation of $\Si$.}
 For every smooth (principal) $G$-bundle $P$ over $\Si$ we denote by $\A(P)$ the affine space of smooth connection one-forms on $P$, and by $C^\infty_G(P,M)$ the set of smooth equivariant maps from $P$ to $M$. Consider the class
\begin{align}\label{eq:BB}\BB:=\BB_\Si:=\big\{w:=(P,A,u)\,\big|\,&P\textrm{ smooth }G\textrm{-bundle over }\Si,\\
\nn&A\in\A(P),\,u\in C^\infty_G(P,M)\big\}.\end{align}
The \emph{symplectic vortex equations} are the equations 
\begin{eqnarray}\label{eq:BAR dd J A u}\bar\dd_{J,A}(u)&=&0,\\
\label{eq:F A mu}F_A+(\mu\circ u)\om_\Si&=&0
\end{eqnarray}
for a triple $(P,A,u)\in\BB$. To explain these conditions, note that the pullback bundle $u^*TM\to P$ descends to a complex vector bundle $(u^*TM)/G\to\Si$.%
\footnote{The complex structure on this bundle is induced by the almost complex structure $J$.}
 For every $x\in M$ we denote by $L_x:\g\to T_xM$ the infinitesimal action, corresponding to the action of $G$ on $M$. With this notation, $\bar\dd_{J,A}(u)$ means the complex antilinear part of $d_Au:=du+L_uA$, which we think of as a one-form on $\Si$ with values in $(u^*TM)/G\to\Si$. In (\ref{eq:F A mu}) we view the curvature $F_A$ of $A$ as a two-form on $\Si$ with values in the adjoint bundle $\g_P:=(P\x\g)/G\to\Si$
\footnote{Here $G$ acts on $\g$ in the adjoint way.}.
 Furthermore, identifying $\g^*$ with $\g$ via $\lan\cdot,\cdot\ran_\g$, we view $\mu\circ u$ as a section of $\g_P$. The vortex equations (\ref{eq:BAR dd J A u},\ref{eq:F A mu}) were discovered by K.~Cieliebak, A.~R.~Gaio, and D.~A.~Salamon \cite{CGS}, and independently by I.~Mundet i Riera \cite{MuPhD,MuHam}.%
\footnote{In the case $G:=S^1\sub\C$ acting on $M:=\C$ by multiplication, and $\Si:=\C$, the corresponding energy functional was introduced previously by V.~L.~Ginzburg and L.~D.~Landau \cite{GL}, in order to model superconductivity. More generally, in the case $M:=\C^n$ and $G$ a closed subgroup of $\U(n)$, the functional appeared in physics in the context of \emph{gauged linear sigma models}, starting with the work of E.~Witten \cite{Wi}.}
 A solution of these equations is called a \emph{(symplectic) vortex}.

Two elements $w,w'\in\BB$ are called \emph{(gauge) equivalent} iff there exists an isomorphism $\Phi:P'\to P$ of smooth $G$-bundles (which descends to the identity on $\Si$), such that
\[\Phi^*(A,u):=(A\circ d\Phi,u\circ\Phi)=(A',u').\]
In this case we write $w\sim w'$. We define
\begin{equation}\label{eq:B}\B:=\B_\Si:=\BB\Slash\sim.\end{equation}
The equations (\ref{eq:BAR dd J A u},\ref{eq:F A mu}) are invariant under equivalence. We call an element $W\in\B$ a \emph{vortex class} iff it consists of vortices. We define the \emph{energy density} of a class $W\in\B$ to be 
\begin{equation}\label{eq:e W}e_W:=e_w:=\frac12\big(|d_Au|^2+|F_A|^2+|\mu\circ u|^2\big),\end{equation}
where $w:=(P,A,u)$ is any representative of $W$. Here the norms are induced by the Riemannian metrics $\om_\Si(\cdot,j\cdot)$ on $\Si$ and $\om(\cdot,J\cdot)$ on $M$, and by $\lan\cdot,\cdot\ran_\g$. This definition does not depend on the choice of $w$. Vortex classes are absolute minimizers of the \emph{(Yang-Mills-Higgs) energy functional}
\begin{equation}\label{eq:E B}E:\B\to[0,\infty],\quad E(W):=\int_\Si e_W\om_\Si\end{equation}
in a given second equivariant homology class.%
\footnote{See \cite[Proposition 3.1]{CGS}. Here we assume that $\Si$ is closed, and vortices in the given homology class exist.}
 We define the \emph{image} of a class $W\in\B$ to be the set of orbits of $u(P)$, where $(P,A,u)$ is any representative of $W$. This is a subset of the orbit space $M/G$. 

Consider now the complex plane $\Si:=\C$, equipped with the standard area form $\om_\C:=\om_0$.%
\footnote{In this case a vortex may be viewed as a map $(\Phi,\Psi,u):\C\to\g\x\g\x M$, see Remark \ref{rmk:trivial} below.}
 Let $W\in\B_\C$ be a vortex class of finite energy, such that the image of $W$ has compact closure%
\footnote{with respect to the quotient topology on $M/G$}%
. Then $W$ naturally carries an equivariant homology class $[W]\in H^G_2(M,\Z)$. (See Section \ref{SEC:HOMOLOGY CHERN}.) Let $B\in H^G_2(M,\Z)$. We denote by $\M_B$ the set of vortex classes $W$ representing the class $B$, and by $\EG\to\BG$ a universal $G$-bundle. There are natural evaluation maps 
\[\ev_z:\M_B\to(M\x\EG)/G,\quad\barev_\infty:\M_B\to \BAR M\]
at $z\in \C$ and $\infty\in \C\cup\{\infty\}$.%
\footnote{See \cite{ZiPhD,ZiConsEv} and Section \ref{sec:stable}.}
 We denote by $\BAR\PD:H_*(\BAR M)\to H^*(\BAR M)$ the Poincar\'e duality map. To prove Conjecture \ref{conj:QK}, heuristically, we define
\begin{eqnarray}\label{eq:phi Qk}&\phi:H_2^G(M,\Z)\to\Hom_\Q\big(H^*_G(M),H^*(\BAR M)\big),&\\
\label{eq:phi B al bar b}&\displaystyle\big\lan\phi(B)\al,\BAR b\big\ran:=\int_{\M_B}\ev_0^*\al\smile\barev_\infty^*\BAR\PD(\BAR b),&
\end{eqnarray}
for $B\in H_2^G(M,\Z)$, $\al\in H^*_G(M)$, and $\BAR b\in H_*(\BAR M)$. Under the hypotheses of Conjecture \ref{conj:QK}, this map is ``morally'' well-defined and satisfies the conditions of the conjecture: If $J$ is chosen as in the definition of convexity below, then there exists a compact subset of $M/G$ containing the image of every finite energy vortex class $W\in\B_\C$ whose image has compact closure. This ensures that for every $B\in H_2^G(M,\Z)$, the space $\M_B$ can be compactified by including holomorphic spheres in $\BAR M$ and in the fibers of the Borel construction $(M\x\EG)/G$. In the transverse case, it follows that the ``boundary'' of $\M_B$ has codimension at least 2. This makes the map $\phi$ ``well-defined''. It satisfies the equivariant Novikov condition as a consequence of the compactification argument, conservation of the equivariant homology class in the limit (see \cite{ZiPhD,ZiConsEv}), and the identity
\[E(W)=\big\lan[\om-\mu],[W]\big\ran.\]
This holds for every vortex class $W\in\B_\C$ of finite energy, such that the image of $W$ has compact closure.%
\footnote{The equality follows from \cite[Proposition 3.1]{CGS} with $\Si:=S^2\iso\C\cup\{\infty\}$ and a smoothening argument at $\infty$.}
The identity also implies conditions (\ref{eq:phi 0},\ref{eq:om mu B}). 

The ring homomorphism property for the induced map $\Qk:=\phi_*$ follows from an argument involving two marked points on the plane $\C$ that either move together or infinitely apart. The semipositivity assumption ensures that in the limit there is no bubbling of vortex classes over $\C$ without marked points. In contrast with holomorphic planes, such vortex classes may occur in stable maps in top dimensional strata, even in the transverse case. This is due to the fact that vortices over $\C$ ``should not be rotated'', which is explained below.

Surjectivity of $\Qk$ will be a consequence of surjectivity of the Kirwan map $\ka$, and the equivariant Novikov property. The idea of the proof of Conjecture \ref{conj:QK } is to define $\Qk^1:=\Qk$ as above, and for general $n\in\N_0$, $\Qk^n$ in a similar way, using $n$ marked points. The map $\Qk^0$ counts vortex classes over $\C$ without marked points.

The ``quantum Kirwan morphism'' $(\Qk^n)_{n\in\N_0}$ will intertwine the genus 0 symplectic vortex invariants with the Gromov-Witten invariants of $(\BAR M,\BAR\om)$. This will follow from a bubbling argument for a sequence of vortex classes over the sphere $S^2$, equipped with an area form that converges to $\infty$.%
\footnote{This corresponds to the adiabatic limit studied by Gaio and Salamon in \cite{GS}. The new feature here is that in the limit, vortex classes over $\C$ may bubble off.}

The goal of the present memoir is to establish bubbling (i.e., ``compactification'') and Fredholm results for vortices over $\C$ in the aspherical case. Together with a transversality result (see \cite{ZiTrans}), the Fredholm result will provide a natural structure of an oriented manifold on the set $\M_B$. Furthermore, the bubbling result will imply that the map 
\[(\ev_0,\barev_\infty):\M_B\to(M\x\EG)/G\x\BAR M\]
is a pseudocycle%
\footnote{as defined in \cite[Definition 6.5.1]{MS04}}.%
 This will give a rigorous meaning to the integral (\ref{eq:phi B al bar b}). The ring homomorphism property and the relations defining a morphism of CohFT's will be a consequence of the bubbling result and a suitable gluing result.
\section{Bubbling for vortices over the plane}
To explain the first main result of this memoir, we assume that $(M,\om)$ is aspherical, i.e., condition (\ref{eq:int S 2 u om}) is satisfied.%
\footnote{The general situation is discussed in Remark \ref{rmk:general} in Section \ref{sec:stable}.}
 We denote by 
\begin{equation}\label{eq:M}\MM:=\big\{(P,A,u)\in\BB_\C\,\big|\,(\ref{eq:BAR dd J A u},\ref{eq:F A mu})\big\},\quad\M:=\MM\Slash\sim
\end{equation}
the class of all vortices over $\C$ and the set of equivalence classes of such vortices. The latter is equipped with a natural topology.%
\footnote{It is induced by the $C^\infty$-topology on compact subsets of $\C$.}
 Consider the subspace of all classes in $\M$ with fixed finite energy $E>0$. There are three sources of non-compactness of this space: Consider a sequence $W_\nu\in\M$, $\nu\in\N$, of classes of energy $E$. In the limit $\nu\to\infty$, the following scenarios (and combinations) may occur:\\

\noindent{\bf Case 1.}~The energy density of $W_\nu$ blows up at some point in $\C$.\\

\noindent{\bf Case 2.}~There exists a number $r>0$ and a sequence of points $z_\nu\in\C$ that converges to $\infty$, such that the energy density of $W_\nu$ on the ball $B_r(z_\nu)$ is bounded above and below by some fixed positive constants.\\

\noindent{\bf Case 3.}~The energy densities converge to 0, i.e., the energy is spread out more and more.\\

In case 1, by rescaling $W_\nu$ around the bubbling point, in the limit $\nu\to\infty$, we obtain a non-constant $J$-holomorphic map from $\C$ to $M$. Using removal of singularity, this is excluded by the asphericity condition. In case 2, we pull $W_\nu$ back by the translation $z\mapsto z+z_\nu$, and in the limit $\nu\to\infty$, obtain a vortex class over $\C$. Finally, in case 3, we ``zoom out'' more and more. In the limit $\nu\to\infty$ and after removing the singularity at $\infty$, we obtain a pseudo-holomorphic map from $S^2$ to the symplectic quotient $\BAR M=\mu^{-1}(0)/G$.

Hence in the limit, passing to a subsequence, we expect $W_\nu$ to converge to a new sort of stable map, which consists of vortex classes over $\C$ and pseudo-holomorphic spheres in $\BAR M$. Here an important difference to Gromov-convergence for pseudo-holomorphic maps is the following: Although the vortex equations are invariant under all orientation preserving isometries of $\Si$, only translations on $\C$ should be allowed as reparametrizations used to obtain a vortex class over $\C$ in the limit. Hence we should disregard some symmetries of the equations. The reasons are that otherwise the reparametrization group does not always act with finite isotropy on the set of simple stable maps, and that there is no suitable evaluation map on the set of vortex classes, which is invariant under rotation.%
\footnote{See Remarks \ref{rmk:conv rot} and \ref{rmk:Isom +} in Sections \ref{sec:conv} and \ref{sec:repara}.}

We are now ready to formulate the first main result. Here we say that $(M,\om,\mu,J)$ is \emph{convex at $\infty$}\label{convex} iff there exists a proper $G$-invariant function $f\in C^\infty(M,[0,\infty))$ and a constant $C\in[0,\infty)$, such that 
\[\om(\na_v\na f(x),Jv)-\om(\na_{Jv}\na f(x),v)\geq0,\quad df(x)JL_x\mu(x)\geq0,\]
for every $x\in f^{-1}([C,\infty))$ and $0\neq v\in T_xM$. Here $\na$ denotes the Levi-Civita connection of the metric $\om(\cdot,J\cdot)$. This condition reduces to the existence of a plurisubharmonic function in the case in which $G$ is trivial. It is satisfied e.g.~if $M$ is closed, and for linear actions on symplectic vector spaces.%
\footnote{See \cite[Example 2.8]{CGMS}. Here the standing assumption that $\mu$ is proper is used.}
\begin{thm}[Bubbling]\label{thm:bubb}  Assume that hypothesis (H) is satisfied, $(M,\om)$ is aspherical, and $(M,\om,\mu,J)$ is convex at $\infty$. Let $k\in\N_0$, and for $\nu\in\N$ let $W_\nu\in\M$ be a vortex class and $z_1^\nu,\ldots,z_k^\nu\in\C$ be points. Suppose that the closure of the image of each $W_\nu$ is compact, and
\[\E(W_\nu)>0,\,\forall\nu\in\N,\quad\sup_{\nu\in\N}\E(W_\nu)<\infty.\]
Then there exists a subsequence of $\big(W_\nu,z_0^\nu:=\infty,z_1^\nu,\ldots,z_k^\nu\big)$ that converges to some genus 0 stable map $(\W,\z)$ consisting of vortex classes over $\C$ and pseudo-holomorphic spheres in $\BAR M$, with $k+1$ marked points. (See Definitions \ref{defi:st} and \ref{defi:conv} in Chapter \ref{chap:bubbling}.%
\footnote{The reasons for introducing the additional marked points $z_0^\nu=\infty$ are explained in Remark \ref{rmk:z 0} in Section \ref{sec:conv}.}%
)
\end{thm}
The proof of this result combines Gromov compactness for pseudo-holomorphic maps with Uhlenbeck compactness. It relies on work \cite{CGMS,GS} by K.~Cieliebak, R.~Gaio, I.~Mundet i Riera, and D.~A.~Salamon. The idea is the following. In order to capture all the energy, we ``zoom out rapidly'', i.e., rescale the vortices so much that the energies of the rescaled vortices are concentrated near the origin in $\C$. Now we ``zoom back in'' in such a way that we capture the first bubble, which may either be a vortex class over $\C$ or a $J$-holomorphic sphere in $\BAR M$. In the first case we are done. In the second case we ``zoom in'' further, to obtain a finite number of vortices and spheres that are attached to the first bubble. Iterating this procedure, we construct the limit stable map. 

The proof involves generalizations of results for pseudo-holomorphic maps to vortices: a bound on the energy density of a vortex, quantization of energy, compactness with bounded derivatives, and hard and soft rescaling. The proof that the bubbles connect and no energy is lost between them, uses an isoperimetric inequality for the invariant symplectic action functional, proved in \cite{ZiA}, based on a version of the inequality by R.~Gaio and D.~A.~Salamon \cite{GS}. 

Another crucial point is that when ``zooming out'', no energy is lost locally in $\C$ in the limit. This relies on an upper bound on the ``momentum map component'' of a vortex, due to R.~Gaio and D.~A.~Salamon. 
\section{Fredholm theory for vortices over the plane}\label{sec:Fredholm}
The space of gauge equivalence classes of symplectic vortices can be viewed as the zero set of a section of an infinite dimensional vector bundle. Formally, the second main result of this memoir states that in the case $\Si=\C$ the vertical differential of this section is Fredholm when seen as an operator between suitable weighted Sobolev spaces. We will first state the result and then interpret it in this way.
\subsection*{Statement of the Fredholm result}
Consider the case $\Si:=\C$ and $\om_{\C}:=\om_0$. Let $p>2$ and $\lam$ be real numbers.%
\footnote{In this memoir, $p$ and $\lam$ always refer to finite values, unless otherwise stated.}
 We define the set $\B^p_\lam$ as follows. For a measurable function $f:\R^n\to\R$ we denote $\Vert f\Vert_p:=\Vert f\Vert_{L^p(\C)}\in[0,\infty]$ and define the \emph{$\lam$-weighted $L^p$-norm of $f$} to be
\[\Vert f\Vert_{p,\lam}:=\big\Vert(1+|\cdot|^2)^\frac\lam2f\big\Vert_p\in[0,\infty].\] 
We define $\BB^p_\loc$ to be the class consisting of all triples $(P,A,u)$, where $P\to\C$ is a $G$-bundle of class $W^{2,p}_\loc$
\footnote{By definition, $P$ is a topological $G$-bundle over $\C$, equipped with an atlas of local trivializations whose transition functions lie in $W^{2,p}_\loc$. Every such bundle is trivializable, but we do not fix a trivialization here.}%
, $A$ a connection (one-form) of class $W^{1,p}_\loc$, and $u:P\to M$ a $G$-equivariant map of class $W^{1,p}_\loc$. We call two elements $w,w'\in\BB^p_\loc$ \emph{$p$-equivalent} iff there exists an isomorphism $\Phi:P'\to P$ of $G$-bundles of class $W^{2,p}_\loc$ (descending to the identity on $\C$), such that $\Phi^*(A,u)=(A',u')$. In this case we write $w\sim_pw'$. We define 
\begin{eqnarray}\label{eq:BB p lam}&\BB^p_\lam:=\big\{w:=(P,A,u)\in\BB^p_\loc\,\big|\,\BAR{u(P)}\textrm{ compact, }\Vert\sqrt{e_w}\Vert_{p,\lam}<\infty\big\},&\\
\label{eq:B p lam}&\B^p_\lam:=\BB^p_\lam\Slash\sim_p,&
\end{eqnarray}
where the energy density $e_w$ is defined as in (\ref{eq:e W}).

Let $W\in\B^p_\lam$. We define normed vector spaces $\X_W^{p,\lam}$ and $\Y_W^{p,\lam}$ as follows. Let $E$ be a real vector bundle over $\C$. We denote by $\wi(E)$ the bundle of alternating $i$-forms on $\C$ with values in $E$. If $E$ is a complex vector bundle, then we denote by $\wzeroone(E)$ the bundle of anti-linear one-forms on $\C$ with values in $E$. We denote by $\Ga^p_\loc(E)$ and $\Ga^{1,p}_\loc(E)$ the spaces of its $L^p_\loc$- and $W^{1,p}_\loc$-sections, respectively. We fix $w:=(P,A,u)\in\BB^p_\lam$, and denote by 
\[\g_P:=(P\x\g)/G\to\Si,\quad TM^u:=(u^*TM)/G\to\C\]
the adjoint bundle and the quotient of the pullback bundle $u^*TM\to P$. Let $\ze:=(\al,v)\in\Ga^{1,p}_\loc\big(\wone(\g_P)\oplus TM^u\big)$. We denote $d_A\al:=d\al+[A\wedge\al]$ and by $\na^A$ the connection on $TM^u$ induced by the Levi-Civita connection $\na$ of $\om(\cdot,J\cdot)$ and $A$
\footnote{See definition (\ref{eq:na A}) in Appendix \ref{sec:add}.}%
, and we abbreviate $\na^A\ze:=(d_A\al,\na^Av)$. We define the weighted norm
\begin{equation}\label{eq:ze w p lam}\Vert\ze\Vert_{w,p,\lam}:=\Vert\ze\Vert_\infty+\big\Vert|\na^A\ze|+|d\mu(u)v|+|\al|\big\Vert_{p,\lam}\in[0,\infty].
\end{equation}
Here the norms are taken with respect to $\om(\cdot,J\cdot)$, the standard metric on $\C$, and $\lan\cdot,\cdot\ran_\g$. We denote by $*$ the Hodge star operator with respect to the standard metric on $\C$, and by 
\[d_A^*=-*d_A*:\Ga^{1,p}_\loc(\wone(\g_P))\to\Ga^p_\loc(\g_P)\]
the formal adjoint of the twisted differential $d_A$. For $x\in M$ we denote by $L_x^*:T_xM\to\g$ the adjoint map of the infinitesimal action of $G$ on $M$ at $x$, with respect to the Riemannian metric $\om(\cdot,J\cdot)$ and the inner product $\lan\cdot,\cdot\ran_\g$. The collection $(L_x^*)_{x\in M}$ induces a map $L_u^*$ from the space of sections of $TM^u$ to the space of sections of $\g_P$. We define%
\footnote{As explained in the next subsection, the map $L_w^*$ is the formal adjoint for the infinitesimal action of the gauge group on the product of the spaces of connections and equivariant maps from $P$ to $M$.}%
\begin{equation}
\label{eq:L w *}L_w^*:\Ga^{1,p}_\loc\big(\wone(\g_P)\oplus TM^u\big)\to\Ga^p_\loc(\g_P),\quad L_w^*(\al,v):=-d_A^*\al+L_u^*v,
\end{equation}
\begin{align}
\label{eq:XX w}\XX_w^{p,\lam}&:=\big\{\ze\in\Ga^{1,p}_\loc\big(\wone(\g_P)\oplus TM^u\big)\,\big|\,L_w^*\ze=0,\,\Vert\ze\Vert_{w,p,\lam}<\infty\big\},\\ 
\label{eq:YY w}\YY_w^{p,\lam}&:=\big\{\ze'\in\Ga^p_\loc\big(\wzeroone(TM^u)\oplus\wtwo(\g_P)\big)\,\big|\,\Vert\ze'\Vert_{p,\lam}<\infty\big\},\\
\label{eq:X W}\X_W^{p,\lam}&:=\disj_{w\in W}\XX_w^{p,\lam}\Slash\sim_p,\\
\label{eq:Y W}\Y_W^{p,\lam}&:=\disj_{w\in W}\YY_w^{p,\lam}\Slash\sim_p.
\end{align}
Here the equivalence relations in (\ref{eq:X W},\ref{eq:Y W}) are defined similarly to the $p$-equivalence relation on $\BB^{p,\lam}$. Since the energy-density is invariant under the gauge transformations, the gauge group of $P$
\footnote{i.e., the group of transformations on $P$}
 of class $W^{2,p}_\loc$ naturally acts on the set
\begin{equation}\label{eq:BB p lam P}\BB^p_\lam(P):=\big\{(A,u)\,\big|\,(P,A,u)\in\BB^p_\lam\big\}.
\end{equation}
Assume that $\lam>1-2/p$. Then this action is free.%
\footnote{See Lemma \ref{le:free} in Appendix \ref{sec:add}.}
 Therefore, $\X_W^{p,\lam}$ is naturally a normed vector space, which can canonically be identified with $\XX_w^{p,\lam}$, for any representative $w$ of $W$. Similarly, $\Y_W^{p,\lam}$ is naturally a normed vector space, which may be identified with $\YY_w^{p,\lam}$, for any representative $w$ of $W$. 

Consider the operator
\begin{eqnarray}
\label{eq:D p lam W}&\D^{p,\lam}_W:\X^{p,\lam}_W\to \Y^{p,\lam}_W&\\
\label{eq:D w v al}&\D_W^{p,\lam}[w;\al,v]:=\left[\begin{array}{c}\big(\na^Av+L_u\al\big)^{0,1}-\frac12J(\na_vJ)(d_Au)^{1,0} \\
d_A\al+\om_0\, d\mu(u)v
\end{array}\right].&
\end{eqnarray}
Here the brackets $[\cdots]$ denote equivalence classes. Formally, this is the vertical differential of a section of a Banach space bundle over a Banach manifold, whose zeros are the gauge equivalence classes of vortices. (For explanations see the next subsection.) Recall that $c_1^G(M,\om)\in H_2^G(M,\Z)$ denotes the equivariant first Chern class of $(M,\om)$. The second main result of this memoir is the following.
\begin{thm}[Fredholm property]\label{thm:Fredholm} Let $p>2$, $\lam\in\R$, and $W\in\B^p_\lam$ (defined as in (\ref{eq:B p lam})). Assume that $\dim M>2\dim G$. Then the following statements hold.
\begin{enui}\item \label{thm:Fredholm:X Y} If $\lam>1-2/p$ then the normed vector spaces $\X^{p,\lam}_W$ and $\Y^{p,\lam}_W$ (defined as in (\ref{eq:X W},\ref{eq:Y W})) are complete. 
\item\label{thm:Fredholm:DDD} If $1-2/p<\lam<2-2/p$ then the operator $\D^{p,\lam}_W$ (defined as in (\ref{eq:D p lam W},\ref{eq:D w v al})) is well-defined and Fredholm of real index 
\begin{equation}\label{eq:ind}\ind\D^{p,\lam}_W=\dim M-2\dim G+2\big\lan c_1^G(M,\om),[W]\big\ran,
\end{equation}
where $[W]$ denotes the equivariant homology class of $W$ (see Section \ref{SEC:HOMOLOGY CHERN}).
\end{enui}
\end{thm}
The contraction appearing in formula (\ref{eq:ind}) can be interpreted as a certain Maslov index. (See Proposition \ref{prop:Chern Maslov} in Section \ref{SEC:HOMOLOGY CHERN}.) The condition $1-2/p<\lam<2-2/p$ in part (\ref{thm:Fredholm:DDD}) of this result captures the geometry of finite energy vortices over $\C$. (See Remark \ref{rmk:geometry} below.) The condition $\lam<2-2/p$ is also needed for the map $\D^{p,\lam}_W$ to have the right Fredholm index. (See Remark \ref{rmk:index}.) The definition of the space $\X^{p,\lam}_W$ naturally parallels the definition of $\B^p_\lam$. (See Remark \ref{rmk:X}.) Note that some naive choices for the domain and target of the operator $\D_W^{p,\lam}$ do not work. (See Remark \ref{rmk:naive}.)

The proof of Theorem \ref{thm:Fredholm} is based on a Fredholm result for the augmented vertical differential and the existence of a bounded right inverse for $L_w^*$. (See Theorems \ref{thm:Fredholm aug} and \ref{thm:L w * R} in Section \ref{subsec:reform}.) The proof of Theorem \ref{thm:Fredholm aug} has two main ingredients. The first one is the existence of a suitable complex trivialization of the bundle $\wone(\g_P)\oplus TM^u$. For $R$ large, $z\in\C\wo B_R$ and $p\in\pi^{-1}(z)\sub P$ such a trivialization respects the splitting 
\begin{equation}\label{eq:im L C}T_{u(p)}M=(\im L^\C_{u(p)})^\perp\oplus\im L^\C_{u(p)},\end{equation}
where $L_x^\C:\g\otimes\C\to T_xM$ denotes the complexified infinitesimal action, for $x\in M$. The second ingredient are two propositions stating that the standard Cauchy-Riemann operator $\dd_{\bar z}$ and a related matrix differential operator are Fredholm maps between suitable weighted Sobolev spaces. These results are based on the analysis of weighted Sobolev spaces carried out by R.~B.~Lockhart and R.~C.~McOwen \cite{Lockhart PhD,Lockhart Fred,Lockhart Hodge,LM ell R n,LM ell mf,McOwen Lap,McOwen Fred,McOwen ell}. A crucial analytical ingredient is a Hardy-type inequality (Proposition \ref{prop:Hardy} in Appendix \ref{sec:weighted}). 
\subsection*{Motivation, a formal setting}
To put the Fredholm result into context, let $(\Si,j)$ be a connected smooth Riemann surface, equipped with a compatible area form $\om_\Si$. Recall the definitions (\ref{eq:BB},\ref{eq:B}) of $\BB,\B$. Consider the subclass $\BB^*\sub\BB$ of triples $(P,A,u)$ for which there exists a point $p\in P$, such that the action of $G$ at the point $u(p)\in M$ is free. We define
\[\B^*:=\BB^*\Slash\sim,\]
where $\sim$ is defined as before the definition (\ref{eq:B}). Formally, $\B^*$ may be viewed as an infinite dimensional manifold, since for every smooth $G$-bundle $P$ over $\Si$, the natural action of the gauge group $\G_P=C^\infty_G(P,G)$ on the ``infinite dimensional manifold'' 
\begin{equation}\label{eq:BB P}\BB^*_P:=\big\{(A,u)\,\big|\,(P,A,u)\in\BB^*\big\}
\end{equation} 
is free. Furthermore, the set of vortex classes may be viewed as the zero set of a section of a vector bundle $\EE$ over $\B^*$, with infinite dimensional fiber, as follows. Consider the ``vector bundle'' $\EEE:=\EEE_\Si$ over $\BB^*$, whose fiber over a point $w=(P,A,u)\in\BB^*$ is given by
\begin{equation}\label{eq:EEE Si w}\EEE_w:=\Ga\left(\wzeroone(TM^u)\oplus\wtwo(\g_P)\right).%
\footnote{Here $\Ga(E)$ denotes the space of smooth sections of a vector bundle $E\to\Si$.}%
\end{equation}
The bundle $\EE:=\EE_\Si$ over $\B^*$ is now defined to be the quotient of the bundle $\EEE\to\BB^*$ under the natural equivalence relation lifting the relation $\sim$ on $\BB^*$. Finally,
\[\SSS:\B^*\to\EE\]
is defined to be the section induced by
\[\SSSS:\BB^*\to\EEE,\quad \SSSS(A,u):=\big(\bar \dd_{J,A}(u),F_A+(\mu\circ u)\om_\Si\big).\]
Heuristically, $\EE$ is an infinite dimensional vector bundle over $\B^*$, and $\SSS$ is a smooth section of $\EE$. The zero set $\SSS^{-1}(0)\sub\B^*$ consists of all vortex classes over $\Si$. Assume that $W\in\SSS^{-1}(0)$. Then formally, there is a canonical map $T:T_{(W,0)}\EE\to\EE_W$, where $\EE_W\sub\EE$ denotes the fiber over $W$. We define the vertical differential of $\SSS$ at $W$ to be the map
\begin{equation}\label{eq:d V S W}d^V\SSS(W)=T\,d\SSS(W):T_W\B^*\to\EE_W.\end{equation}
Heuristically, if this map is Fredholm and surjective, for every $W\in\SSS^{-1}(0)$, then the zero set $\SSS^{-1}(0)$ is a smooth submanifold of $\B^*$. The dimension of a connected component of this submanifold equals the Fredholm index of $d^V\SSS(W)$, where $W$ is any point in the connected component. 

At a formal level, in the case $\Si=\C$, equipped with $\om_\Si=\om_0$, the vertical differential (\ref{eq:d V S W}) coincides with the operator $\D^{p,\lam}_W$, which was defined in (\ref{eq:D p lam W},\ref{eq:D w v al}) and occurred in the Fredholm result, Theorem \ref{thm:Fredholm}. To see this, let $W\in\B^*$. We interpret $T_W\B^*$ as a quotient, as follows. Let $P$ be a smooth $G$-bundle over $\Si$, and $(A,u)\in\BB^*_P$. Denoting $w:=(P,A,u)$, the infinitesimal action at the point $(A,u)$, corresponding to the action of $\G_P$ on $\BB^*_P$, is given by 
\[L_w:\Lie(\G_P)=\Ga(\g_P)\to T_{(A,u)}\BB^*_P=\Ga\left(\wone(\g_P)\oplus TM^u\right),\quad L_w\xi=(-d_A\xi,L_u\xi),\]
where $d_A\xi:=d\xi+[A,\xi]$. Defining
\begin{equation}\label{eq:XX P A u}\XX_w:=T_{(A,u)}\BB^*_P/\im L_w,
\end{equation}
we may identify
\begin{equation}\label{eq:X W disj}T_W\B^*=\X_W:=\left(\disj_{w\in W}\XX_w\right)\Slash\sim,
\end{equation}
where $\sim$ denotes the natural lift of the equivalence relation on $\BB^*$. Assume formally that $\BB^*_P$ and $\Lie(\G_P)$ are equipped with a $\G_P$-invariant Riemannian metric and a $\G_P$-invariant inner product, respectively. For $(A,u)\in\BB^*_P$ we denote by $L_w^*:T_{(A,u)}\BB^*_P\to\Lie(\G_P)$ the adjoint map of $L_w$. Then by (\ref{eq:XX P A u}), we may identify 
\[\XX_w=\ker L_w^*\sub\Ga\left(\wone(\g_P)\oplus TM^u\right).\]
Using this and (\ref{eq:X W disj},\ref{eq:EEE Si w}), the vertical differential (\ref{eq:d V S W}) at $W\in\SSS^{-1}(0)$ agrees with the map 
\[\displaystyle\left(\disj_{w\in W}\ker L_w^*\right)\Slash\sim\,\,\to\,\,\left(\disj_{w\in W}\Ga\left(\wzeroone(TM^u)\oplus\wtwo(\g_P)\right)\right)\Slash\sim,\]
given by (\ref{eq:D w v al}), in the case $\Si=\C$ and $\om_\Si=\om_0$. Here on either side, $\sim$ denotes a natural lift of the equivalence relation on $\BB^*$.
\section{Remarks, related work, organization}\label{sec:remarks}
\subsection*{Remarks}
\begin{rmk}[Vortices as triples]\label{rmk:bundle} In some earlier work (e.g.~\cite{CGS} and \cite{ZiPhD}), the $G$-bundle $P$ was fixed and the vortex equations were seen as equations for a pair $(A,u)$ rather than a triple $(P,A,u)$.%
\footnote{However, in \cite{MT} I.~Mundet i Riera and G.~Tian took the viewpoint of the present memoir.}
 The motivation for making $P$ part of the data is twofold: 

When formulating convergence for a sequence of vortex classes over $\C$ to a stable map, one has to pull back the vortices by translations of $\C$. (See Section \ref{sec:conv}.) If the principal bundle is fixed and vortices are defined as pairs $(A,u)$ solving (\ref{eq:BAR dd J A u},\ref{eq:F A mu}), then there is no natural such pullback. However, there \emph{is} a natural pullback if the bundle is made part of the data for a vortex.%
\footnote{Given a $G$-bundle $P$ over $\C$, we may of course choose a trivialization of $P$, and then define a pull back for pairs $(A,u)$, using the trivialization. However, this approach is unnatural, since it depends on the choice of a trivialization.}
 More generally, it is possible to pull back vortex triples $(P,A,u)$ by a K\"ahler transformation of a Riemann surface equipped with a compatible area form.

Another motivation is the following: If the area form or the complex structure on the surface $\Si$ vary, then in the limit we may obtain a surface $\Si'$ with singularities. It does not make sense to consider $P$ as a bundle over $\Si'$. One way of solving this problem is by decomposing $\Si'$ into smooth surfaces, and constructing smooth $G$-bundles over these surfaces. Hence the $G$-bundle should be viewed as a varying object. 

Once $P$ is made part of the data, it is natural to consider \emph{equivalence classes} of triples $(P,A,u)\in\MM$ (as defined in (\ref{eq:M})), rather than the triples themselves. One reason is that all important quantities, like energy density and energy, are invariant (or equivariant) under equivalence. Note also that the bubbling and Fredholm results are more naturally stated for equivalence classes of vortices. Viewing the equivalence classes as the fundamental objects also matches the physical viewpoint that the ``gauge field'', i.e., the connection $A$, is physically relevant only ``up to gauge''. $\Box$
\end{rmk}
\begin{rmk}\label{rmk:trivial} Let $\Si$ be the plane $\C$, equipped with the standard area form $\om_0$, and consider the trivial $G$-bundle $P_0:=\C\x G$. Then the solutions $(A,u)$ of the vortex equations (\ref{eq:BAR dd J A u},\ref{eq:F A mu}) on $P_0$ bijectively correspond to solutions $(\Phi,\Psi,f)\in C^\infty\big(\C,\g\x\g\x M\big)$ of the equations
\begin{eqnarray}
\label{eq:dd s f L f Phi}&\dd_sf+L_f\Phi+J(f)(\dd_tf+L_f\Psi)=0,&\\
\label{eq:dd s Psi}&\dd_s\Psi-\dd_t\Phi+[\Phi,\Psi]+\mu(f)=0.&
\end{eqnarray}
Here we denote by $s$ and $t$ the standard coordinates in $\C=\R^2$, and in the second equation we identify the Lie algebra $\g$ with its dual via the inner product $\lan\cdot,\cdot\ran_\g$. The correspondence maps such a triple $(\Phi,\Psi,f)$ to $\big(A,u\big)$, where $A$ denotes the connection on $P_0$ defined by 
\[A_{(z,g)}(\ze,g\xi):=\big(\Phi(z)ds+\Psi(z)dt\big)\ze+\xi,\quad\forall \ze\in T_z\C,\,\xi\in\g,\,z\in\C,\,g\in G,\]
and the map $u:P_0\to M$ is given by $u(z,g):=g^{-1}f(z)$. The group $C^\infty(\C,G)$ acts on the set $\MM_0$ of solutions of (\ref{eq:dd s f L f Phi},\ref{eq:dd s Psi}) by 
\[h^*(\Phi,\Psi,f):=\Big(h^{-1}\dd_sh+\Ad_{h^{-1}}\Phi,h^{-1}\dd_th+\Ad_{h^{-1}}\Psi,h^{-1}f\Big),\]
where we denote the adjoint representation of an element $g\in G$ by $\Ad_g:\g\to\g$. This group naturally corresponds to the gauge group $C^\infty_G(P_0,G)$, and its action to the action 
\[g^*(A,u):=\big(g^{-1}dg+\Ad_{g^{-1}}A,g^{-1}u\big).\]
Since every $G$-bundle over $\C$ is trivializable, it follows that the quotient of $\MM_0$ by the action of $C^\infty(\C,G)$ bijectively corresponds to the quotient $\M=\MM\Slash\sim$, consisting of gauge equivalence classes of triples $(P,A,u)$ of solutions of (\ref{eq:BAR dd J A u},\ref{eq:F A mu}). Hence the results of the present memoir can alternatively be formulated in terms of solutions of the equations (\ref{eq:dd s f L f Phi},\ref{eq:dd s Psi}). However, the intrinsic approach using equations (\ref{eq:BAR dd J A u},\ref{eq:F A mu}) seems more natural. $\Box$
\end{rmk}
\begin{rmk}[Asphericity]\label{rmk:asphericity} Without the asphericity condition one needs to include holomorphic spheres in the fibers of the Borel construction in the definition of a stable map. In this situation, to compactify the space of vortices over $\C$ with an upper energy bound, one needs to combine the proof of Theorem \ref{thm:bubb} with the analysis carried out by I.~Mundet i Riera and G.~Tian in \cite{MuPhD,MT}, or by A.~Ott in \cite{Ott}. $\Box$
\end{rmk}
\begin{rmk}[Quotient spaces]\label{rmk:quotient} The space $\X_W^{p,\lam}$ occurring in Theorem \ref{thm:Fredholm} is a quotient of a disjoint union of normed vector spaces. It is canonically isomorphic to the space $\XX_w^{p,\lam}$, for every representative $w$ of $W$. Similar statements hold for $\Y_W^{p,\lam}$. The description of the spaces $\X_W^{p,\lam}$ and $\Y_W^{p,\lam}$ as such quotients may look unconventional, however, it seems natural, since it does not involve any choice of a representative of $W$. 

Alternatively, one could phrase the Fredholm result in terms of the spaces $\XX_w^{p,\lam}$ and $\YY_w^{p,\lam}$. However, in view of the last part of Remark \ref{rmk:bundle}, this seems less natural than the present formulation. $\Box$
\end{rmk}

\begin{rmk}[Decay condition and vortices]\label{rmk:geometry} The condition $\Vert\sqrt{e_w}\Vert_{p,\lam}<\infty$ in the definition (\ref{eq:BB p lam}) of $\BB^p_\lam$ and the requirement $1-2/p<\lam<2-2/p$ in Theorem \ref{thm:Fredholm}(\ref{thm:Fredholm:DDD}) capture the geometry of finite energy vortices over $\C$, in the following sense. Let $w=(P,A,u)\in\BB^p_\loc$ be a finite energy vortex such that $u(P)$ has compact closure. (Here $\BB^p_\loc$ is defined as at the beginning of Section \ref{sec:Fredholm}.) Then for every $\eps>0$ there exists a constant $C$ such that $e_w(z)\leq C|z|^{-4+\eps}$, for every $z\in\C\wo B_1$. This follows from Theorem \ref{thm:reg bundle} in Appendix \ref{sec:smooth} and \cite[Corollary 1.4]{ZiA}. 

It follows that $w\in\BB^p_\lam$ if $\lam<2-2/p$. This bound is sharp. To see this, let $\lam\geq2-2/p$ and $M:=S^2$, equipped with the standard symplectic form $\om_\st$, complex structure $J:=i$, and the action of the trivial group $G:=\{\one\}$. Consider the inclusion $u:\C\x\{\one\}\iso\C\to S^2\iso\C\cup\{\infty\}$. Since this map is holomorphic, the triple $\big(\C\x\{\one\},0,u\big)$ is a finite energy vortex whose image has compact closure. It does not lie in $\BB^p_\lam$. 

On the other hand, every $w\in \BB^p_\lam$ has finite energy whenever $p>2$ and $\lam>1-2/p$.%
\footnote{This follows from the estimates
\[\Vert\sqrt{e_w}\Vert_2\leq\big\Vert\sqrt{e_w}\lan\cdot\ran^\lam\big\Vert_p\big\Vert\lan\cdot\ran^{-\lam}\big\Vert_q,\quad\big\Vert\lan\cdot\ran^{-\lam}\big\Vert_q<\infty,\]
where $q:=2p/(p-2)$. The first estimate is H\"older's inequality and the second one follows from a calculation in radial coordinates.}
 The latter condition is sharp. To see this, consider $M:=\R^2,\om:=\om_0,G:=\{\one\},J:=i$. We choose a smooth map $u:\C\x\{\one\}\iso\C\to\R^2$, such that
\[u(z)=\left(\cos\left(\sqrt{\log|z|}\right),\sin\left(\sqrt{\log|z|}\right)\right),\quad\forall z\in\C\wo B_2.\]
Then the triple $\big(\C\x\{\one\},0,u\big)$ lies in $\BB^p_\lam$ for every $p>2$ and $\lam\leq1-2/p$. However, it has infinite energy.%
\footnote{In the present setting, a simpler example of an infinite energy triple $w=(P,A,u)$ satisfying $\sqrt{e_w}\in L^p_\lam$ for every $p>2$ and $\lam\leq1-2/p$, is $w:=\big(\C\x\{\one\},0,u\big)$, where $u(z):=\left(\sqrt{\log|z|},0\right)$, for every $z\in\C\wo B_2$. However, the closure of the image of such a map $u$ is non-compact, since it contains the set $[1,\infty)\x\{0\}$. Therefore, $w$ does not lie in $\BB^p_\lam$ for any $p$ and $\lam$.}%
 $\Box$
\end{rmk}
\begin{rmk}[Index]\label{rmk:index} 
The condition $\lam<2-2/p$ in part (\ref{thm:Fredholm:DDD}) of Theorem \ref{thm:Fredholm} is needed for the map $\D^{p,\lam}_W$ to have the right Fredholm index. Namely, let $\lam>1-2/p$ be such that $\lam+2/p\not\in\Z$, and $W\in\B^p_\lam$. Then the proof of Theorem \ref{thm:Fredholm} shows that $\D^{p,\lam}_W$ is Fredholm with index equal to 
\[(2-k)(\dim M-2\dim G)+2\big\lan c_1^G(M,\om),[W]\big\ran,\]
where $k$ is the largest integer less than $\lam+2/p$. In particular, the index changes when $\lam$ passes the value $2-2/p$. 

Note also that the condition $\lam>1-2/p$ is needed, in order for the homology class $[W]$ to be well-defined. (See Remark \ref{rmk:W well-defined} in Section \ref{SEC:HOMOLOGY CHERN}.) $\Box$
\end{rmk}
\begin{rmk}[Weighted Sobolev spaces and energy density]\label{rmk:X} The definition of the space $\X^{p,\lam}_W$ naturally parallels the definition (\ref{eq:B p lam}) of $\B^p_\lam$. Namely, by linearizing with respect to $A$ and $u$ the terms $d_Au,F_A$ and $\mu\circ u$ occurring in the energy density $e_w$, we obtain the terms $\na^A\ze,d\mu(u)v$ and $L_u\al$. These expressions occur in $\Vert\ze\Vert_{w,p,\lam}$ (defined as in (\ref{eq:ze w p lam})), except for the factor $L_u$ in $L_u\al$.%
\footnote{It follows from hypothesis (H) and Lemma \ref{le:si} in Appendix \ref{sec:proofs homology} that this factor is irrelevant.}
 The expression $\Vert\ze\Vert_\infty$ is needed in order to make $\Vert\cdot\Vert_{w,p,\lam}$ non-degenerate. $\Box$
\end{rmk}
\begin{rmk}[Sobolev spaces and 0-th order terms]\label{rmk:naive} Consider the situation in which the norm (\ref{eq:ze w p lam}) is replaced by the usual $W^{1,p}$-norm, and the norm $\Vert\cdot\Vert_{p,\lam}$ (defining $\YY_w^{p,\lam}$) is replaced by the usual $L^p$-norm. Then in general, the map defined by (\ref{eq:D w v al}) does not have closed image, and hence it is not Fredholm. Note also that the 0-th order terms $\al\mapsto(L_u\al)^{0,1}$ and $v\mapsto\om_0\,d\mu(u)v$ in (\ref{eq:D w v al}) are not compact (neither with respect to $\X_W^{p,\lam}$ and $\Y_W^{p,\lam}$, nor with respect to the usual $W^{1,p}$- and $L^p$-norms). The reason is that the embedding of $W^{1,p}(\C)$ (for $p>2$) into the space of bounded continuous functions on $\C$ is not compact. Because of these terms, the map (\ref{eq:D w v al}) is not well-defined between spaces that look like the standard weighted Sobolev spaces in ``logarithmic'' coordinates $\tau+i\phi$ (with $e^{\tau+i\phi}=z\in\C\wo\{0\}$). 

This is in contrast with the situation in which $\Si$ is the infinite cylinder $\R\x S^1$, equipped with the standard complex structure and area form. In that situation the splitting (\ref{eq:im L C}) is unnecessary, and standard weighted Sobolev spaces in ``logarithmic'' coordinates can be used. In the relative setting, with the cylinder replaced by the infinite strip $\R\x[0,1]$, this was worked out by U.~Frauenfelder \cite[Proposition 4.7]{FrPhD}. The proof of the Fredholm result then relies on results \cite{RoSa,Sa} by J.~Robbin and D.~A.~Salamon.%
\footnote{In that setting, the index of the operator equals a certain spectral flow.}%
 $\Box$
\end{rmk}
\subsection*{Related work}\label{subsec:related work}
\subsubsection*{Quantum Kirwan maps} 
The history of Conjectures \ref{conj:QK} and \ref{conj:QK } is as follows. Assume that (H) holds, $(M,\om)$ is aspherical, $(M,\om,\mu)$ is convex at $\infty$ and monotone, and $H^*_G(M)$ is generated by classes of degrees less than $2N$, where $N$ is the minimal equivariant Chern number (defined as in (\ref{eq:N inf})). In this case R.~Gaio and D.~A.~Salamon \cite{GS} proved that there exists a ring homomorphism from $H^*_G(M)$ to the quantum cohomology of $(\BAR M,\BAR\om)$ that agrees with the Kirwan map on classes of degrees less than $2N$, see \cite[Corollary A']{GS}. The idea of the proof of this result is to fix an area form $\om_{S^2}$ on $S^2$ and to relate symplectic vortices for the area form $C\om_{S^2}$ with pseudo-holomorphic spheres in $\BAR M$, for sufficiently large $C>0$. The authors noticed that in general, this correspondence does not work, since in the limit $C\to\infty$, vortices over $\C$ may bubble off. Accordingly, the calculation of the quantum cohomology of monotone toric manifolds in \cite{CS}, which is based on \cite{GS}, does not extend to the situation of a general toric manifold. This follows from examples by H.~Spielberg \cite{SpGW,SpHirzebruch}.

Based on these observations, Salamon suggested to construct a ring homomorphism from $H_G^*(M)$ to the quantum cohomology of $(\BAR M,\BAR\om)$, by counting symplectic vortices over the plane $\C$, provided that $(M,\om)$ is aspherical and $(M,\om,\mu)$ is convex at $\infty$. This homomorphism should intertwine the symplectic vortex invariants and the Gromov-Witten invariants of $(\BAR M,\BAR\om)$. This gave rise to the Ph.D.-thesis \cite{ZiPhD}, which served as a basis for the present memoir. There it is observed that in the definition of convergence for vortices over $\C$ to a stable map, only translations should be allowed as reparametrizations used to obtain a vortex component in the limit.%
\footnote{As explained in Remark \ref{rmk:conv rot} in Section \ref{sec:conv}, the reason for this is that the evaluation of a vortex class at a point in $\C$ is not invariant under rotations.}
 C.~Woodward realized that with this restriction, vortices over $\C$ without marked points may appear in the bubbling argument used in the proof of the ring homomorphism property of the quantum Kirwan map. (This may happen even in the transverse case.) As a solution, he suggested to interpret the quantum Kirwan map as a morphism of cohomological field theories. (See \cite{NWZ} and Conjecture \ref{conj:QK } above.) On the other hand, under the semipositivity introduced above, the vortices without marked points can be excluded in the transverse case. This gave rise to Conjecture \ref{conj:QK}.

In his recent article \cite{Wo} C.~Woodward developed these ideas in an algebraic geometric setting. He defined a quantum Kirwan map in the case of a smooth projective variety with an action of a complex reductive group. (See \cite[Theorem 1.3]{Wo}.) Theorem \ref{thm:bubb} of the present memoir is used in the proof of that result to show properness of the Deligne-Mumford stack of stable scaled gauged maps in $(M,\om)$ of genus 0. (See \cite[Theorem 5.25]{Wo}.) 

In \cite{GWToric} E.~Gonzalez and C.~Woodward used Woodward's definition to calculate the quantum cohomology of a compact toric orbifold with projective coarse moduli space. Furthermore, in \cite{GWWall} they used it to formulate a quantum version of Kalkman's wall-crossing formula.
\subsubsection*{Bubbling and Fredholm results for vortices}
Assume that $\Si$ is closed, (H) holds, and $M$ is symplectically aspherical and equivariantly convex at $\infty$. In this case, in \cite[Theorem 3.4]{CGMS}, K.~Cieliebak et al.~proved compactness of the space of vortex classes with energy bounded above by a fixed constant. In the case in which $M$ and $\Si$ are closed, in \cite[Theorem 4.4.2]{MuPhD} I.~Mundet i Riera compactified the space of bounded energy vortex classes with fixed complex structure on $\Si$. Assuming also that $G:=S^1$, this was extended by I.~Mundet i Riera and G.~Tian in \cite[Theorem 1.4]{MT} to the situation of varying complex structure. That work is based on a version of Gromov-compactness for continuous almost complex structures, proved by S.~Ivashkovich and V.~Shevchishin in \cite{IS}. 

In \cite[Theorem 1.8]{Ott} A.~Ott compactified the space of bounded energy vortex classes in a different way, for a general Lie group and closed $M$ and $\Si$, the last with fixed complex structure. He used the approach to Gromov-compactness by D.~McDuff and D.~A.~Salamon in \cite{MS04}. In the case in which $\Si$ is an infinite strip, equipped with the standard area form and complex structure, the compactification was carried out in a relative setting by U.~Frauenfelder in \cite[Theorem 4.12]{FrPhD}. (See also \cite{FrArnold}.)

In \cite{GS} R.~Gaio and D.~A.~Salamon investigated the vortex equations with area form $C\om_\Si$ in the limit $C\to\infty$. Here $\Si$ is a closed surface equipped with a fixed area form $\om_\Si$. They proved that three types of objects may bubble off: a holomorphic sphere in $\BAR M$, a vortex over $\C$, and a holomorphic sphere in $M$. (See the proof of \cite[Theorem A]{GS}.)

For a \emph{closed} Riemann surface $\Si$, in \cite{CGMS} K.~Cieliebak et al.~proved that the augmented vertical differential of the vortex equations is Fredholm.
\subsubsection*{Other related work} In \cite{VW} S.~Venugopalan and C.~Woodward establish a Hitchin-Kobayashi correspondence for symplectic vortices over the plane $\C$.
\subsection*{Organization of this memoir}
This memoir is organized as follows. Chapter \ref{chap:bubbling} contains the bubbling analysis for vortices over the plane $\C$. In Section \ref{sec:stable} we define the notion of a stable map consisting of vortex classes over $\C$, pseudo-holomorphic spheres in $\BAR M$, and marked points. In Section \ref{sec:conv} we formulate convergence of a sequence of vortex classes over $\C$ to such a stable map. Stable maps and convergence are explicitly described in the Ginzburg-Landau setting in Section \ref{SEC:EXAMPLE}. Section \ref{sec:repara} covers an additional topic, which will be relevant in a future definition of the quantum Kirwan map. The main result of Section \ref{sec:comp} is that given a sequence of rescaled vortices with uniformly bounded energies, there exists a subsequence that converges up to gauge, modulo bubbling at finitely many points. Section \ref{sec:soft} contains a result that tells how to find the next bubble in the bubbling tree, at a bubbling point of a given sequence of rescaled vortices. Based on these results, the bubbling result, Theorem \ref{thm:bubb}, is proven in Section \ref{sec:proof:thm:bubb}. Section \ref{sec:proof:prop:conv S 1 C} contains the proof of the result characterizing convergence in the Ginzburg-Landau setting. 

Chapter \ref{chap:Fredholm} contains the Fredholm theory for vortices over the plane $\C$. In Section \ref{SEC:HOMOLOGY CHERN} the equivariant homology class of an equivalence class of triples $(P,A,u)$ is defined, and the contraction appearing in formula (\ref{eq:ind}) is interpreted as a certain Maslov index. Section \ref{sec:proof:Fredholm} contains the proof of the Fredholm result, Theorem \ref{thm:Fredholm}. 

In the appendix we recollect some auxiliary results, which are used in the proofs of the main results.
\subsection*{Acknowledgments}
This memoir arose from my Ph.D.-thesis. I would like to thank my adviser, Dietmar A.~Salamon, for the inspiring topic. I highly profited from his mathematical insight. I am very much indebted to Chris Woodward for his interest in my work, for sharing his ideas with me, and for his continuous encouragement. It was he who coined the term ``quantum Kirwan map''. I would also like to thank Urs Frauenfelder, Kai Cieliebak, Eduardo Gonzalez, Andreas Ott, and Bumsig Kim for stimulating discussions. Finally, I am grateful to the referees for their valuable suggestions.
\chapter{Bubbling for vortices over the plane}\label{chap:bubbling} 
In this chapter, stable maps consisting of vortex classes over the plane $\C$, holomorphic spheres in the symplectic quotient, and marked points, are defined, and the first main result of this memoir, Theorem \ref{thm:bubb}, is proven. This result states that given a sequence of vortex classes over $\C$, with uniformly bounded energies, and sequences of marked points, there exists a subsequence that converges to some stable map. We also describe stable maps and convergence in the simplest interesting example, the Ginzburg-Landau setting.
\section{Stable maps}\label{sec:stable}
Let $M,\om,G,\g,\lan\cdot,\cdot\ran_\g,\mu,J$ be as in Chapter \ref{chap:main}. Our standing hypothesis (H) implies that the symplectic quotient 
\[\big(\BAR M=\mu^{-1}(0)/G,\BAR\om\big)\]
is well-defined and closed. The structure $J$ induces an $\BAR\om$-compatible almost complex structure on $\BAR M$ as follows. For every $x\in M$ we denote by $L_x:\g\to T_xM$ the infinitesimal action at $x$. We define the \emph{horizontal distribution} $H\sub T(\mu^{-1}(0))$ by
\[H_x:=\ker d\mu(x)\cap\im L_x^\perp,\quad\forall x\in\mu^{-1}(0).\]
Here $\perp$ denotes the orthogonal complement with respect to the metric $\om(\cdot,J\cdot)$ on $M$. We denote by $\pi:\mu^{-1}(0)\to \BAR M:=\mu^{-1}(0)/G$ the canonical projection. We define $\bar J$ to be the unique isomorphism of $T\BAR M$ such that
\begin{equation}\label{eq:bar J}\bar J\,d\pi= d\pi J\textrm{ on }H.%
\footnote{Such a $\bar J$ exists and is unique, since the map $d\pi$ is an isomorphism from $H$ to $T\BAR M$, and $J$ preserves $H$.}%
\end{equation}
We identify $\C\cup\{\infty\}$ with $S^2$. The (Connectedness) condition in the definition of a stable map below will involve the evaluation of a vortex class at the point $\infty\in S^2$. In order to make sense of this, we need the following. We denote by $Gx$ the orbit of a point $x\in M$. Let $P$ be a smooth (principal) $G$-bundle over $\C$
\footnote{Such a bundle is trivializable, but we do not fix a trivialization here.}
 and $u\in C^\infty_G(P,M)$ a map. We denote by $M/G$ the orbit space of the action of $G$ on $M$, and define
\begin{equation}\nn\bar u:\C\to M/G,\quad\bar u(z):=Gu(p),\end{equation}
where $p\in P$ is an arbitrary point in the fiber over $z$. For $W\in\B$ we define 
\begin{equation}\label{eq:BAR u W}\bar u_W:=\bar u,\end{equation}
where $w=(P,A,u)$ is any representative of $W$. This is well-defined, i.e., does not depend on the choice of $w$. Recall the definition (\ref{eq:M}) of $\M$, and that by the image of a class $W\in\M$ we mean the set of orbits of $u(P)$, where $(P,A,u)$ is any vortex representing $W$. We define the set
\begin{equation}\label{eq:ME}\ME:=\big\{W\in\M\,\big|\,\BAR{\textrm{image}(W)}\textrm{ compact, }E(W)<\infty\big\}.
\end{equation}
\begin{prop}[Continuity at $\infty$]\label{prop:bar u} If $W\in\ME$ then the map $\bar u_W:\C\to M/G$ extends continuously to a map $f:S^2\to M/G$, such that $f(\infty)\in\BAR M=\mu^{-1}(0)/G$. 
\end{prop}
\begin{proof}[Proof of Proposition \ref{prop:bar u}]\setcounter{claim}{0} This follows from the estimate (\ref{eq:sup z z' bar d}) with $R=\infty$ in Proposition \ref{prop:en conc} (Section \ref{sec:soft} below).%
\footnote{Alternatively, it is a consequence of \cite[Proposition 11.1]{GS}.}%
\end{proof}
\begin{defi}\label{defi:barev} We define the \emph{evaluation map} $\barev$ to be the map from the disjoint union of $C^0(S^2,M/G)\x S^2$ and $\ME\x\{\infty\}$ to $M/G$, given as follows. For $(\bar u,z)\in C^0(S^2,M/G)\x S^2$ we define
\begin{equation}\label{eq:barev}
\barev_z(\bar u):=\barev(\bar u,z):=\bar u(z). 
\end{equation}
Furthermore, for $W\in\ME$ we define
\begin{equation}\label{eq:barev w}\barev_\infty(W):=f(\infty),
\end{equation}
where $f$ is as in Proposition \ref{prop:bar u}.
\end{defi}
By a tree relation on a set $T$ we mean a symmetric, anti-reflexive relation on $T$, such that each two points in $T$ are connected by a unique simple path of edges. A central definition of this memoir is the following.
\begin{defi}[Stable map]\label{defi:st} For every $k\in\N_0=\N\cup\{0\}$ a \emph{(genus 0) stable map consisting of vortex classes over $\C$ and pseudo-holomorphic spheres in $\BAR M$, with $k+1$ marked points}, is a tuple
\begin{equation}\label{eq:W z}(\W,\z):=\Big(T_0,T_1,T_\infty,E,(W_\alpha)_{\alpha\in T_1},(\bar u_\alpha)_{\al\in T_\infty},(z_{\alpha\beta})_{\alpha E\beta},(\alpha_i,z_i)_{i=0,\ldots,k}\Big),\end{equation}
where $T_i$ is a finite set for $i=0,1,\infty$, $E$ is a tree relation on $T:=T_0\disj T_1\disj T_\infty$, $W_\al\in\ME$ (for $\al\in T_1$), $\bar u_\al:S^2\to\BAR M=\mu^{-1}(0)/G$ is a $\bar J$-holomorphic map (for $\al\in T_\infty$), $z_{\al\beta}\in S^2$ is a point for each edge $\al E\beta$, $\al_i\in T$ is a vertex, and $z_i\in S^2$ is a point, for $i=0,\ldots,k$, such that the following conditions hold.

\begin{enui} 
\item\label{defi:st comb}{\bf (Combinatorics)} 
\begin{itemize}
 \item We have $\al_0\in T_1\cup T_\infty$. 
 \item For every $\al\in T$ there exist an integer $k\in\N$ and vertices $\be_1,\ldots,\be_k\in T$ such that 
\[\be_1=\al,\quad\be_k=\al_0,\]
and for every $i=1,\ldots,k-1$ we have
\begin{equation}\label{eq:be i be i 1}\be_iE\be_{i+1},\quad\be_i\in T_0\then\be_{i+1}\in T_0\cup T_1,\quad\be_i\in T_1\cup T_\infty\then\be_{i+1}\in T_\infty.
\end{equation}
\end{itemize}
\item \label{defi:st dist}{\bf (Special points)} 
\begin{itemize}
 \item If $\al_0\in T_1$ then $z_0=\infty$.
 \item If $\al\in T_1$ and $\be\in T_\infty$ are such that $\al E\be$ then $z_{\al\be}=\infty$.
\item Fix $\al\in T$. Then the points $z_{\al\beta}$ with $\beta\in T$ such that $\al E\beta$ and the points $z_i$ with $i=0,\ldots,k$ such that $\al_i=\al$, are all distinct. 
\end{itemize}
\item\label{defi:st conn}{\bf (Connectedness)} Let $\al,\beta\in T_1\cup T_\infty$ be such that $\al E\beta$. Then 
\[\barev_{z_{\al\beta}}(W_\al)=\barev_{z_{\beta\al}}(W_\beta).\]
Here $\barev$ is defined as in (\ref{eq:barev}) and (\ref{eq:barev w}) and we set $W_\al:=\bar u_\al$ if $\al\in T_\infty$. 
\item\label{defi:st st}{\bf(Stability)} Let $\al\in T$. 
\begin{itemize}
 \item If $\al\in T_1$ and $\E(W_\al)=0$ then the set
\begin{equation}\label{eq:special}\big\{\beta\in T\,|\,\al E\beta\big\}\cup\big\{i\in\{0,\ldots,k\}\,|\,\al_i=\al\big\}
\end{equation}
contains at least two points. 
 \item If either $\al\in T_\infty$ and $E(\bar u_\al)=0$, or $\al\in T_0$, then the set (\ref{eq:special}) consists of at least three points.
\end{itemize}
\end{enui}
\end{defi}
\begin{figure}
  \centering
\leavevmode\epsfbox{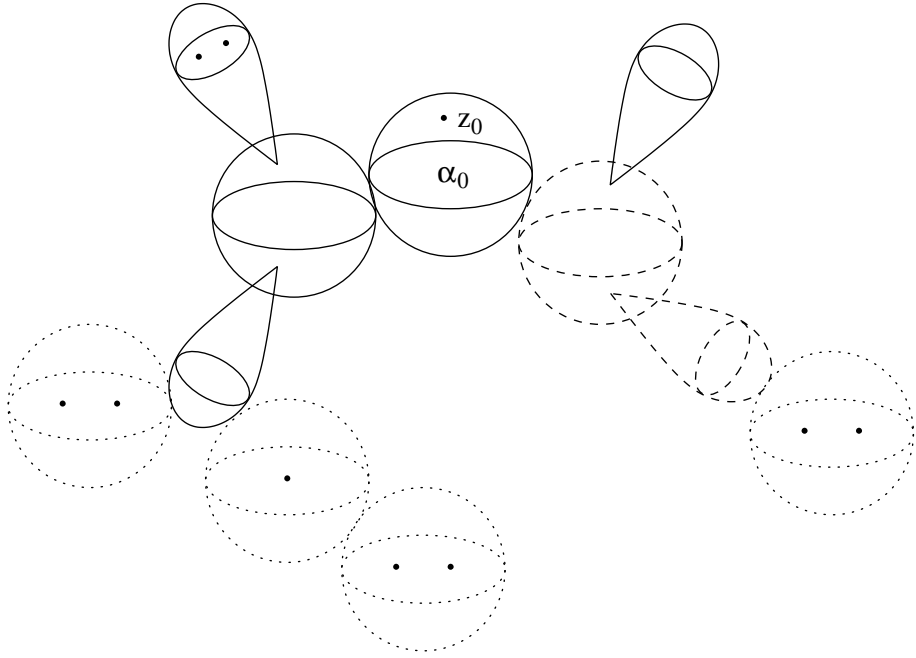}
\caption{Stable map consisting of vortex classes over $\C$ and pseudo-holomorphic spheres in $\BAR M$.}
  \label{fig:stable map} % magnification factor = 57 %
\end{figure}
This definition is modeled on the notion of a genus 0 pseudo-holomorphic stable map, as introduced by Kontsevich in \cite{Kontsevich}.%
\footnote{For an exhaustive exposition of those stable maps see the book by D.~McDuff and D.~A.~Salamon \cite{MS04}.}
 Roughly speaking, a stable map in the sense of Definition \ref{defi:st} can be thought of as a collection of vortex classes over $\C$, pseudo-holomorphic spheres in the symplectic quotient $\BAR M$, ``ghost spheres of type 0'' corresponding to the vertices of $T_0$, and marked and nodal points. A vortex class may be connected to a sphere in $\BAR M$ at the nodal point $\infty$, and to ``ghost spheres of type 0'' at points in $\C$. Furthermore, spheres of the same type may be connected at nodal points. The ``ghost spheres of type 0'' should be thought of as constant spheres in the Borel construction $(M\x\EG)/G$. They are needed for the bubbling result (Theorem \ref{thm:bubb}) to capture colliding marked points in $\C$.%
\footnote{As explained in \cite{ZiPhD,ZiConsEv} vortex classes over $\C$ evaluate to points in $(M\x\EG)/G$ at points in $\C$. Therefore, identifying each ``ghost sphere of type 0'' with a point in $(M\x\EG)/G$, it makes sense to ask that the sphere is connected to a vortex class over $\C$ at a given nodal point.}

Figure \ref{fig:stable map} shows an example of a stable map. Here the ``teardrops'' correspond to vortex classes over $\C$, the solid and dashed spheres to pseudo-holomorphic spheres in $\BAR M$, and the dotted spheres to ``ghost spheres of type 0''. The solid objects have positive energy, and the dashed and dotted spheres are ``ghosts'', i.e., their energy vanishes. Each ``teardrop'' is connected to a sphere (in $\BAR M$) via a nodal point at its vertex, which corresponds to the point $\infty\in\C\cup\{\infty\}$. 

To explain the stability condition (\ref{defi:st st}), we fix $\al\in T$. We define the \emph{set of nodal points on $\al$} to be 
\begin{equation}\label{eq:Z al}Z_\al:=\{z_{\al\beta}\,\big|\,\beta\in T,\,\al E\beta\}\sub S^2,\end{equation}
the set of \emph{marked points on $\al$} to be 
\[\big\{z_i\,\big|\,\al_i=\al,\,i\in\{0,\ldots,k\}\big\},\] 
and the \emph{set $Y_\al$ of special points on $\al$} to be the union of these two sets. The stable map of Figure \ref{fig:stable map} carries ten marked points, which are drawn as dots. The stability condition says the following. Assume that $\al\in T$ is a ``ghost component'', i.e., $\al\in T_0$ or $W_\al$ carries zero energy (in the case $\al\in T_1\cup T_\infty$). Then the following holds: If $\al\in T_1$ then it carries at least one special point in $\C$.%
\footnote{It then also carries a special point at $\infty$.}
 Furthermore, if $\al\in T_0\cup T_\infty$ then $\al$ carries at least three special points. 

This condition ensures that the action of a natural reparametrization group on the set of \emph{simple} stable maps of a given type is free.%
\footnote{See Proposition \ref{prop:simple} in Section \ref{sec:repara} below.}
 In a future definition of the quantum Kirwan map this will be needed in order to show that the evaluation map on the set of non-trivial vortex classes (with marked points) is a pseudo-cycle. 
\begin{Rmks} Condition (\ref{defi:st comb}) implies that if $T_1$ is empty then so is $T_0$, and hence a stable map in the sense of Definition \ref{defi:st} is a genus 0 stable map of $\bar J$-holomorphic spheres in $\BAR M$. 

If $\al_0\in T_1$ then $T_1=\{\al_0\}$ and $T_\infty=\emptyset$. This follows from the second part of condition (\ref{defi:st comb}) for $\al\in T_\infty$, using the last condition in (\ref{eq:be i be i 1}). Hence in this case a stable map consists of a single vortex class and special points.

If $\al_0\in T_\infty$ then the sets $T_\infty$ and $T_1\cup T_\infty$ are subtrees of $T$, and every element of $T_1$ is adjacent to a unique vertex in $T_\infty$ and to no vertex in $T_1$. In particular, each element of $T_1$ is a leaf of the tree $T_1\cup T_\infty$. These statements follow from condition (\ref{defi:st comb}) and the fact that $T$ does not contain any simple cycle.

The vertices in $T_0$ are not adjacent to those in $T_\infty$. Furthermore, for each connected component of $T_0$ there exists a unique vertex in $T_1$ that is adjacent to some element of the connected component. These assertions follow from condition (\ref{defi:st comb}) and the fact that $T$ does not contain any simple cycle. $\Box$
\end{Rmks}
\begin{rmk}\label{rmk:z i}\rm If $1\leq i\leq k$ is such that $\al_i\in T_1$ then $z_i\neq\infty$. This follows from condition (\ref{defi:st dist}) and the fact that either $\al_0\in T_1$ or every vertex in $T_1$ is adjacent to some vertex in $T_\infty$. $\Box$
\end{rmk}
\begin{rmk}\label{rmk:general} Without the asphericity assumption, a stable map should also include holomorphic maps from the sphere $S^2$ to the fibers of the Borel construction $(M\x\EG)/G$. These occur if in a sequence of vortices over $\C$ energy is concentrated around some point in $\C$. The necessary analysis was worked out by I.~Mundet i Riera and G.~Tian in \cite{MuPhD,MT}, and by A.~Ott in \cite{Ott}. $\Box$
\end{rmk}

\begin{Xpl} The simplest example of a stable map consists of the tree with one vertex $T=T_1=\{\al_0\}$, a vortex class $W\in \ME$, the marked point $z_0:=\infty$ and a finite number of distinct points $z_i\in \C$, $i=1,\ldots,k$, where $k\geq1$ if $\E(W)=0$. $\Box$
\end{Xpl}
\begin{xpl}\rm\label{xpl:stable} We set $k:=0$, choose an integer $\ell\in\N_0$, and define 
\[T_0:=\emptyset,\quad T_1:=\{1,\ldots,\ell\},\quad T_\infty:=\{0\},\quad E:=\big\{(0,1),\ldots,(0,\ell),(1,0),\ldots,(\ell,0)\big\},\]
\[\al_0:=0,\quad z_{i0}:=\infty,\,\forall i=1,\ldots,\ell.\] 
Let $z_0,z_{0i}\in S^2$, $i=1,\ldots,\ell$ be distinct points, $\bar u_0$ a $\bar J$-holomorphic sphere, and $W_i\in\ME$ be a vortex class, for $i=1,\ldots,\ell$. Assume that $E(W_i)>0$ for every $i$, and if $\ell\leq 1$ then $\bar u_0$ is nonconstant. Then the tuple
\[(\W,\z):=\Big(T_0:=\emptyset,T_1,T_\infty,E,(W_i)_{i\in\{1,\ldots,\ell\}},\bar u_0,(z_{ij})_{i Ej},(0,z_0)\Big)\]
is a stable map. (See Figure \ref{fig:example stable}.) $\Box$
\end{xpl}
\begin{figure}
  \centering
\leavevmode\epsfbox{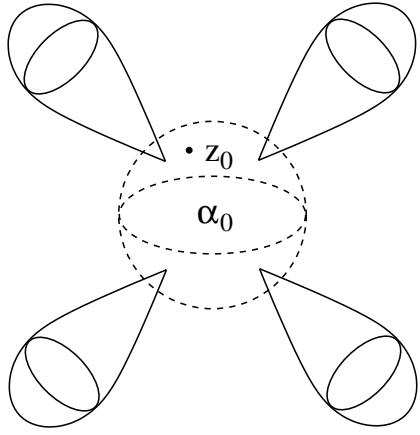}
\caption{This is the stable map described in Example \ref{xpl:stable} with $\ell:=4$ and a constant sphere $\bar u_0$.}
  \label{fig:example stable} % magnification factor = 57 %
\end{figure}
\section{Convergence to a stable map}\label{sec:conv}
Let $k\geq0$, for $\nu\in\N$ let $W_\nu\in\ME$ be a vortex class and $z^\nu_1,\ldots,z_k^\nu\in \C$ be points, and let 
\[(\W,\z):=\Big(T_0,T_1,T_\infty,E,(W_\alpha)_{\alpha\in T_1},(\bar u_\al)_{\al\in T_\infty},(z_{\alpha\beta})_{\alpha E\beta},(\alpha_i,z_i)_{i=0,\ldots,k}\Big)\]
be a stable map. In order to define convergence of $\big(W_\nu,z^\nu_0:=\infty,z^\nu_1,\ldots,z_k^\nu\big)$ to $(\W,\z)$, we need the following notations. For a $\bar J$-holomorphic map $f:S^2\to\BAR M$ we denote its energy by 
\[E(f)=\int_{S^2}f^*\BAR\om.\]
Let $\al\in T:=T_0\disj T_1\disj T_\infty$ and $i=0,\ldots,k$. We define $z_{\al,i}\in S^2$ as follows. If $\al=\al_i$ then we set 
\begin{equation}
  \label{eq:z al i}z_{\al,i}:=z_i.
\end{equation}
Otherwise we define $z_{\al,i}$ to be the first special point encountered on a path of edges from $\al$ to $\al_i$. To explain this, we denote by $\be\in T_\infty$ the unique vertex such that the chain of vertices of $T$ running from $\al$ to $\al_i$ is given by $(\al,\be,\ldots,\al_i)$. ($\be=\al_i$ is also allowed.) We define 
\begin{equation}
\label{eq:z al i }z_{\al,i}:=z_{\al\be}.
\end{equation}
Let $\Si$ be a compact connected smooth surface with non-empty boundary. Recall the definition (\ref{eq:B}) of the set $\B_\Si$ of equivalence classes of triples $(P,A,u)$. We define the \emph{$C^\infty$-topology $\tau_\Si$} on this set as follows: We fix a (smooth) $G$-bundle $P$ over $\Si$, and denote by $\G_P$ its gauge group. Since by hypothesis, $G$ is connected, every $G$-bundle over $\Si$ is trivializable. It follows that the map 
\begin{equation}\label{eq:A C G}\big(\A(P)\x C^\infty_G(P,M)\big)/\G_P\ni\G_P^*(A,u)\mapsto[P,A,u]\in\B_\Si\end{equation}
is a bijection.%
\footnote{Recall here that $\A(P)$ denotes the affine space of smooth connection one-forms on $P$. We use the simplified notation $[P,A,u]$ for the equivalence class $[(P,A,u)]$.}
\begin{defi}\label{defi:tau Si} We define $\tau_\Si$ to be the pushforward of the quotient topology of the $C^\infty$-topology on $\A(P)\x C^\infty_G(P,M)$ under the map (\ref{eq:A C G}).
\end{defi}
Let $\Si$ be a smooth surface, $W=[P,A,u]\in\B_\Si$, and $\Om\sub\Si$ an open subset with compact closure and smooth boundary. We define the \emph{restriction $W|_{\BAR\Om}$} to be the equivalence class of the pullback of $(P,A,u)$ under the inclusion map $\BAR\Om\to\Si$. 

For $W=[P,A,u]\in\B_\Si$ and $\phi$ a translation on $\C$, we define the pullback of $W$ by $\phi$ to be 
\begin{equation}\label{eq:phi []}\phi^*[P,A,u]:=\big[\phi^*P,\Phi^*(A,u)\big],\end{equation}
where $\Phi:\phi^*P\to P$ is defined by $\Phi(z,p):=p$.%
\footnote{Recall here that a point in the pullback bundle $\phi^*P$ has the form $(z,p)$, where $z\in\C$ and $p$ lies in the fiber of $P$ over $\phi(z)\in\C$.}
 We define 
\begin{equation}\label{eq:M *}M^*:=\big\{x\in M\,|\,\textrm{if }g\in G:\,gx=x\Rightarrow g=\one\big\}.\end{equation}
Note that $\mu^{-1}(0)\sub M^*$ by our standing hypothesis (H). 
\begin{defi}[Convergence]\label{defi:conv} The sequence $(W_\nu,z_0^\nu:=\infty,z_1^\nu,\ldots,z_k^\nu)$ is said to converge to $(\W,\z)$, as $\nu\to\infty$, iff the limit $E:=\lim_{\nu\to\infty} E(W_\nu)$ exists,
\begin{equation}
  \label{eq:E V bar T}E=\sum_{\al\in T_1}E(W_\al)+\sum_{\al\in T_\infty}E(\bar u_\al),
\end{equation}
and there exist M\"obius transformations $\phi_\al^\nu:S^2\to S^2$, for $\al\in T$, $\nu\in\N$, such that the following conditions hold.

\begin{enui}\item \label{defi:conv phi z} 
\begin{itemize}
\item If $\al\in T_1$ then $\phi_\al^\nu$ is a translation on $\C$. 
\item For every $\al\in T_\infty$ we have $\phi_\al^\nu(z_{\al,0})=\infty$, where $z_{\al,0}$ is defined as in (\ref{eq:z al i}), (\ref{eq:z al i }). 
\item Let $\al\in T_\infty$ and $\psi_\al$ be a M\"obius transformation such that $\psi_\al(\infty)=z_{\al,0}$. Then the derivatives $(\phi_\al^\nu\circ\psi_\al)'(z)$ converge to $\infty$, for every $z\in\C$. 
\end{itemize}
\item \label{defi:conv al be} If $\al,\be\in T$ are such that $\al E\beta$ then $(\phi_\al^\nu)^{-1}\circ\phi_\beta^\nu\to z_{\al\beta}$, uniformly on compact subsets of $S^2\wo\{z_{\beta\al}\}$.
\item \label{defi:conv w} 
\begin{itemize}
\item Let $\al\in T_1$ and $\Om\sub R^2$ be an open connected subset with compact closure and smooth boundary. Then the restriction $(\phi_\al^\nu)^*W_\nu|_{\BAR\Om}$ converges to $W_\al|_{\BAR\Om}$, with respect to the topology $\tau_{\BAR\Om}$ (as in Definition \ref{defi:tau Si}).
\item Fix $\al\in T_\infty$. Let $Q$ be a compact subset of $S^2\wo (Z_\al\cup\{z_{\al,0}\})$. For $\nu$ large enough, we have
\[\bar u_{W_\nu}\circ\phi_\al^\nu(Q)\sub M^*/G,\]
and $\bar u_{W_\nu}\circ\phi_\al^\nu$ converges to  $\bar u_\al$ in $C^1$ on $Q$. (Here $Z_\al$ and $\bar u_{W_\nu}$ are defined as in (\ref{eq:Z al},\ref{eq:BAR u W}).) 
\end{itemize}
\item\label{defi:conv z} We have $(\phi_{\al_i}^\nu)^{-1}(z_i^\nu)\to z_i$ for every $i=1,\ldots,k$. 
\end{enui}
\end{defi}
The meaning of this definition is illustrated by Figure \ref{fig:convergence}.
\begin{figure} 
  \centering
\leavevmode\epsfbox{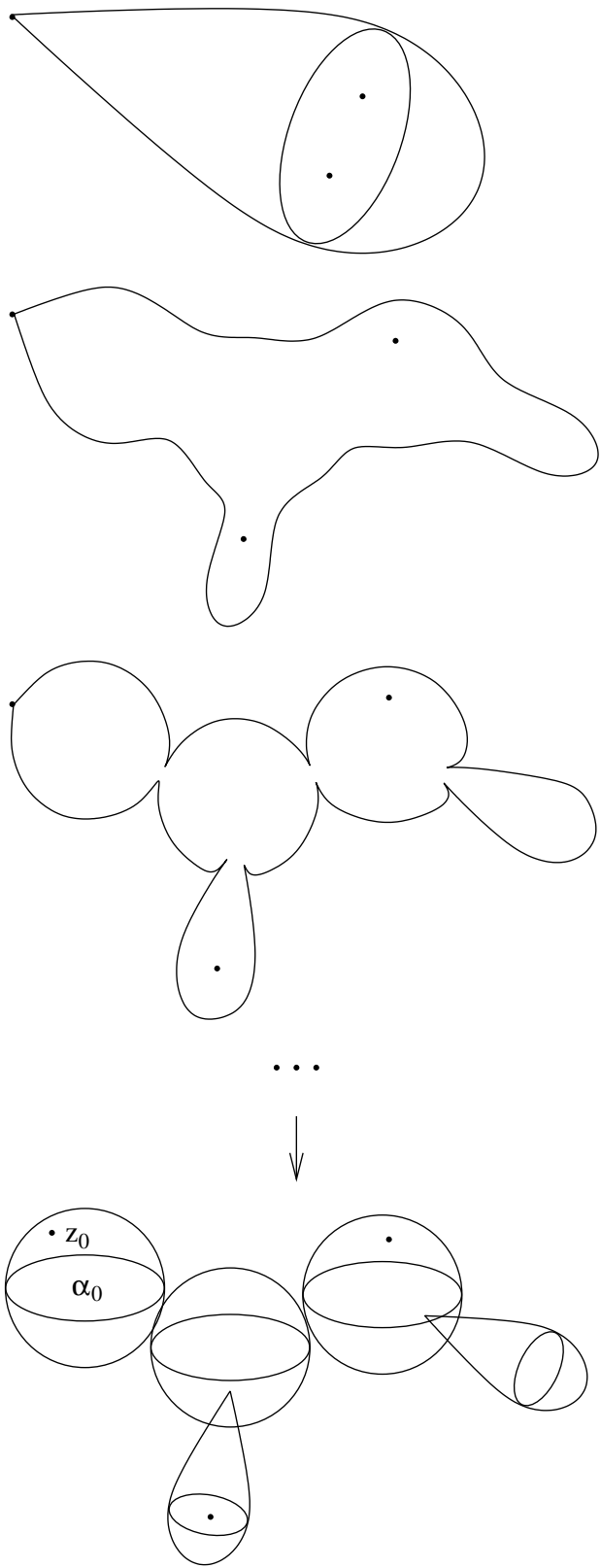}
\caption{Convergence of a sequence of vortex classes over $\C$ to a stable map.}
  \label{fig:convergence} % magnification factor = 45 %
\end{figure}
It is based on the notion of convergence of a sequence of pseudo-holomorphic spheres to a genus 0 pseudo-holomorphic stable map.%
\footnote{For that notion see for example \cite{MS04}.}
 An example in which it can be understood more explicitly, is discussed in the next section.\\

\begin{Rmk} The condition in the first part of (\ref{defi:conv w}), that $(\phi_\al^\nu)^*W_\nu|_{\BAR\Om}\to W_\al|_{\BAR\Om}$ with respect to $\tau_{\BAR\Om}$, is equivalent to the requirement that there exist representatives $w_\nu$ of $(\phi_\al^\nu)^*W_\nu|_{\BAR\Om}$ (for $\nu\in\N$) and $w$ of $W_\al|_{\BAR\Om}$ such that $w_\nu$ converges to $w$ in the $C^k$-topology, for every $k\in\N$. This follows from a straight-forward argument, using Lemma \ref{le:X G} (Appendix \ref{sec:add}). $\Box$
\end{Rmk}
\begin{Rmk} The last part of condition (\ref{defi:conv phi z}) and the second part of condition (\ref{defi:conv w}) capture the idea of catching a pseudo-holomorphic sphere in $\BAR M$ by ``zooming out'': Fix $\al\in T_\infty$, and consider the case $z_{\al,0}=\infty$. Then there exist $\lam_\al^\nu\in\C\wo \{0\}$ and $z_\al^\nu\in\C$ such that $\phi_\al^\nu(z)=\lam_\al^\nu z+ z_\al^\nu$. It follows from a direct calculation that $(\phi_\al^\nu)^*W_\nu$ is a vortex class with respect to the area form $\om_\Si=|\lam_\al^\nu|^2\om_0$, where $\om_0$ denotes the standard area form on $\C$. The last part of condition (\ref{defi:conv phi z}) means that $\lam_\al^\nu\to\infty$, for $\nu\to\infty$. Hence in the limit $\nu\to\infty$ we obtain the equations 
\[\bar\dd_{J,A}(u)=0,\quad\mu\circ u=0.\]
These correspond to the $\bar J$-Cauchy-Riemann equations for a map from $\C$ to $\BAR M$. (See Proposition \ref{prop:bar del J} in Appendix \ref{sec:add}.) The second part of (\ref{defi:conv w}) imposes that in fact the sequence of rescaled vortex classes converges (in a suitable sense) to a $\bar J$-holomorpic sphere and that this sphere equals $\bar u_\al$. 

It is unclear whether the bubbling result, Theorem \ref{thm:bubb}, remains true if we replace the $C^1$-convergence in this part of condition (\ref{defi:conv w}) by $C^\infty$-convergence. (Compare to Remark \ref{rmk:C infty} in Section \ref{sec:comp}.) $\Box$
\end{Rmk}
\begin{Rmk} The ``energy-conservation'' condition (\ref{eq:E V bar T}) has the important consequence that the stable map $(\W,\z)$ represents the same equivariant homology class as the vortex class $W_\nu$, for $\nu$ large enough. (See \cite[Proposition 5.4]{ZiPhD} and \cite{ZiConsEv}.) $\Box$ 
\end{Rmk}
\begin{rmk}\label{rmk:z 0}\rm The purpose of the additional marked point $(\al_0,z_0)$ is to be able to formulate the second part of condition (\ref{defi:conv w}). For $\al\in T_\infty$ and $\nu\in\N$ the map $G u_\nu\circ\phi_\al^\nu$ is only defined on the subset $(\phi_\al^\nu)^{-1}(\C)\sub S^2$. Since by condition (\ref{defi:conv phi z}) we have $\phi_\al^\nu(z_{\al,0})=\infty$, the composition $\bar u_{W_\nu}\circ\phi_\al^\nu:Q\to M/G$ is well-defined for each compact subset $Q\sub S^2\wo (Z_\al\cup\{z_{\al,0}\})$. Hence the second part of condition (\ref{defi:conv w}) makes sense. 

As another motivation, note that the bubbling result, Theorem \ref{thm:bubb}, is in general wrong, if we do not introduce the additional marked points $z_0^\nu:=\infty$ and $(\al_0,z_0)$. See Example \ref{xpl:conv} in Section \ref{SEC:EXAMPLE} below. $\Box$
\end{rmk}
\begin{rmk}\label{rmk:conv rot}\rm One conceptual difficulty in defining the notion of convergence is the following. Consider the group $\Isom^+(\Si)$ of orientation preserving isometries of $\Si$ (with respect to the metric $\om_\Si(\cdot,j\cdot)$).%
\footnote{This coincides with the group of diffeomorphisms of $\Si$ that preserve the pair $(\om_\Si,j)$.}
 This group acts on $\B_\Si$ (defined as in (\ref{eq:B})), as in (\ref{eq:phi []}). The set $\ME$ of finite energy vortex classes is invariant under this action. 

Hence naively, in the definition of convergence, for $\al\in T_1$ one would allow $\phi_\al^\nu$ to be an orientation preserving isometry of $\C$, rather than just a translation. The problem is that with this modification, there is no evaluation map on the set of stable maps, that is continuous with respect to convergence. Such a map is needed for the definition of the quantum Kirwan map.

Note here that we cannot define evaluation of a vortex class $W$ at some point $z\in\Si$ by choosing a representative $(P,A,u)$ of $W$ and evaluating $u$ at some point in the fiber over $z$, thus obtaining a point in $M$, since this point depends on the choices. Instead, $W$ evaluates at $z$ to a point in the Borel construction for the action of $G$ on $M$.%
\footnote{See \cite[Proposition 6.1]{ZiPhD} and \cite{ZiConsEv}.}

Another reason for allowing only translations as reparametrizations (for $\al\in T_1$) is that the action of $\Isom^+(\Si)$ on the set of vortex classes with positive energy is not always free. (See Example \ref{xpl:rot} in Section \ref{sec:repara} below.) This means that the action of the reparametrization group on the set of simple stable maps is not always free, if we allow reparametrizations in $\Isom^+(\Si)$. (Compare to Section \ref{sec:repara}.) $\Box$
\end{rmk}
\section{An example: the Ginzburg-Landau setting}\label{SEC:EXAMPLE}
Recall the definition (\ref{eq:ME}) of the set $\ME$ of finite energy vortex classes over the plane $\Si=\C$, whose image has compact closure. In this section we describe this set, stable maps, and convergence of a sequence of vortex classes to a stable map, in the simplest interesting example: Let $(M,\om,J,G):=(\R^2,\om_0,i,S^1)$. We equip $\g:=\Lie(S^1)=i\R$ with the standard inner product, and consider the action of $S^1\sub\C$ on $\R^2=\C$ by multiplication of complex numbers. We define a momentum map $\mu:\C\to\g$ for this action by
\[\mu(z):=\frac i2(1-|z|^2).\]
In this setting, the energy functional (\ref{eq:E B}) was introduced by V.~L.~Ginzburg and L.~D.~Landau \cite{GL}, in order to model superconductivity. Following the book \cite{JT} by A.~Jaffe and C.~Taubes, the set $\ME$ can be described as follows. Let $w:=(P,A,u)\in\BB_\C$ (defined as in (\ref{eq:BB})) and recall the definition (\ref{eq:e W}) of the energy density $e_w$. We denote by
\[E(w):=\int_\C e_w\]
the energy of $w$. In the present situation the condition on the image in the definition of $\ME$ is superfluous. This means that if $w$ solves the vortex equations (\ref{eq:BAR dd J A u},\ref{eq:F A mu}) and has finite energy, then the closure of the image of $u$ is compact. To see this, we fix such a solution. If the energy of $w$ vanishes, then $\mu\circ u\const0$, therefore the image of $u$ equals $S^1=\mu^{-1}(0)\sub M=\C$ and is thus compact. Hence the statement is a consequence of the following result.
\begin{prop}\label{prop:S 1 C image} If the energy of $w$ is positive (and finite) then the image of $u$ is contained in the open unit ball $B_1\sub\C$. 
\end{prop}
\begin{proof}[Proof of Proposition \ref{prop:S 1 C image}]\setcounter{claim}{0} An elementary calculation shows that $(A,u)$ solves the Euler-Lagrange equations corresponding to the energy functional $E:\BB_\C\to[0,\infty]$. Therefore, it follows from \cite[Chap.~III, Theorem 8.1]{JT} that $|u(p)|<1$ for every $p\in P$. This proves Proposition \ref{prop:S 1 C image}.
\end{proof}
In fact, under the hypothesis of this result the image of $u$ equals $B_1$, see Corollary \ref{cor:S 1 C image} below. (However, this fact will not be used.) Consider $W\in\ME$. We define the \emph{local degree map}
\begin{equation}\label{eq:deg W}\deg_W:\C\to\N_0\end{equation}
as follows. We choose a representative $(P,A,u)$ of $W$. Since $E(W)<\infty$, it follows from \cite[Chapter III, Theorem 2.2]{JT} that for every smooth section $\si:\C\to P$, the map $u\circ\si:\C\to\C$ has only finitely many zeros. Let $z\in\C$. We define 
\[\deg_u(z):=\deg\left(\frac{u\circ\si}{|u\circ\si|}:S^1_\eps(z)\to S^1\right),\]
where $\si:\C\to P$ is any smooth section and $\eps>0$ is so small that the closed ball $\BAR B_\eps(z)$ intersects $(u\circ\si)^{-1}(0)$ only in the point $z$. Here $S^1_\eps(z)\sub\R^2$ denotes the circle of radius $\eps$, centered at $z$. This definition does not depend on the choice of $\si$ nor $\eps$. We now define the map (\ref{eq:deg W}) by 
\begin{equation}\label{eq:deg W z}\deg_W(z):=\deg_u(z).\end{equation}
By an elementary argument the right hand side is independent of the choice of the representative $(P,A,u)$ of $W$, and hence $\deg_W(z)$ is well-defined. Furthermore, it follows from \cite[Theorem 2.2]{JT} that $\deg_W$ takes on nonnegative values. We define the \emph{(total) degree of $W$} to be
\begin{equation}\label{eq:deg W sum}\deg(W):=\sum_{z\in\C}\deg_W(z).\end{equation}
(Only finitely many terms in this sum are non-zero, hence the sum makes sense.) This number is proportional to the energy of $W$:
\begin{prop}\label{prop:E d} We have
\begin{equation}\label{eq:E d}
E(W)=\pi\deg(W).
\end{equation}
\end{prop}
\begin{proof}[Proof of Proposition \ref{prop:E d}]\setcounter{claim}{0} We choose a representative $w:=(P,A,u)$ of $W$ and a smooth section $\si:\C\to P$ as in Proposition \ref{prop:E int u om} in Appendix \ref{sec:vort}. We denote $v:=u\circ\si:\C\to\C$. Let $R>0$ be so large that $v(z)\neq0$ if $|z|\geq R$. We denote by $r$ the radial coordinate in $M=\C$ and by $\ga$ the standard angular one-form on $\R^2\wo\{0\}$.%
\footnote{By our convention this form integrates to $2\pi$ over any circle centered at the origin.}
 The one-form $\al:=\frac{r^2}2\ga$ is a primitive of $\om_0$, and therefore, by Stokes' theorem,
\begin{equation}\label{eq:int B R u om}\int_{B_R}v^*\om_0=\int_{S^1_R}v^*\al.\end{equation}
By elementary arguments, we have
\[\int_{S^1_R}v^*\ga=2\pi\deg\left(S^1_R\ni z\mapsto\frac{v(z)}{|v(z)|}\in S^1\right)=2\pi\sum_{z\in B_R}\deg_u(z)=2\pi\deg(W).\]
(Here in the last equality we used the assumption that $v(z)\neq0$ if $|z|\geq R$.) It follows that 
\begin{equation}\label{eq:pi deg W}\pi\deg(W)\min_{z\in S^1_R}|v(z)|^2\leq\int_{S^1_R}v^*\al\leq\pi\deg(W)\max_{z\in S^1_R}|v(z)|^2.\end{equation}
On the other hand, using the estimates $E(w)<\infty$ and $|\mu\circ u|\leq\sqrt{e_w}$, Lemma \ref{le:a priori} in Appendix \ref{sec:vort} implies that $|\mu\circ v(Rz)|=\frac12(1-|v(Rz)|^2)$ converges to 0, uniformly in $z\in S^1$, as $R\to\infty$. Combining this with (\ref{eq:pi deg W},\ref{eq:int B R u om},\ref{eq:int B R E}), equality (\ref{eq:E d}) follows. This proves Proposition \ref{prop:E d}.
\end{proof}
This result has the following consequence.
\begin{cor}\label{cor:S 1 C image} Let $w:=(P,A,u)$ be a smooth vortex over $\C$ with positive and finite energy. Then the image of $u$ contains the open unit ball $B_1\sub\C$. 
\end{cor}
\begin{proof}[Proof of Corollary \ref{cor:S 1 C image}]\setcounter{claim}{0} Consider the set 
\[X:=\big\{|u(p)|\,\big|\,p\in P\big\}.\]
This set is connected, and hence an interval. Since $E(w)<\infty$, for every $r<1$ there exists a point $p\in P$ such that $|u(p)|\geq r$. On the other hand, positivity of the energy and Proposition \ref{prop:E d} imply that $u$ vanishes somewhere. It follows that $X$ contains the interval $[0,1)$. Since the image of $u$ is invariant under the $S^1$-action, it follows that it contains the ball $B_1$. This proves Corollary \ref{cor:S 1 C image}. 
\end{proof}
Proposition \ref{prop:S 1 C image} and Corollary \ref{cor:S 1 C image} imply that the image of $u$ equals $B_1$, for every smooth vortex over $\C$ with positive and finite energy. Fix now $d\in\N_0$. We denote by $\Sym^d(\C)$ the $d$-fold symmetric product. By definition this is the quotient topological space for the action of the symmetric group $S_d$ on $\C^d$ given by
\[\si\cdot(z_1,\ldots,z_d):=\big(z_{\si^{-1}(1)},\ldots,z_{\si^{-1}(d)}\big).\]
We identify $\Sym^d(\C)$ with the set $\wt\Sym^d(\C)$ of all maps $m:\C\to\N_0$ such that $m(z)\neq0$ for only finitely many points $z\in\C$, and
\begin{equation}\label{eq:sum m}\sum_{z\in\C}m(z)=d,
\end{equation}
by assigning to $\z:=[z_1,\ldots,z_d]\in\Sym^d(\C)$ the \emph{multiplicity} map $m_\z:\C\to\N_0$, given by 
\[m_\z(z):=\#\big\{i\in\{1,\ldots,d\}\,|\,z_i=z\big\}.\]
We can now characterize vortex classes with energy $d\pi$ as follows.
\begin{prop}\label{prop:class} The map 
\begin{equation}
  \label{eq:M d Sym}\M_d:=\big\{W\in\M\,|\,E(W)=d\pi\big\}\ni W\mapsto\deg_W\in\wt\Sym^d(\C)=\Sym^d(\C) 
\end{equation}
is a bijection.   
\end{prop}
\begin{proof}[Proof of Proposition \ref{prop:class}]\setcounter{claim}{0} This follows from \cite[Chap.~III, Theorem 1.1]{JT}.
\end{proof}
As a consequence of Proposition \ref{prop:class}, we obtain a classification of stable maps in the sense of Definition \ref{defi:st} in the present setting: Here the symplectic quotient $\BAR M=\mu^{-1}(0)/S^1$ consists of a single point, the orbit $S^1\sub M=\C$. Hence every holomorphic sphere in $\BAR M$ is constant. Stable maps are thus classified in terms of their combinatorial structure $\big(T_0,T_1,T_\infty,E\big)$, the location of the special points, and for each $\al\in T_1$, a point in some symmetric product of $\C$.

Convergence to a stable map is explicitly described by the following result. Here we will use the inclusion 
\begin{eqnarray}
\label{eq:iota Sym}&\iota_d:\disj_{0\leq d'\leq d}\Sym^{d'}(\C)\to\Sym^d(S^2),&\\
\nn&\iota_d([z_1,\ldots,z_{d'}]):=\big[z_1,\ldots,z_{d'},\infty,\ldots,\infty\big],&
\end{eqnarray}
where we identify $S^2\iso \C\cup\{\infty\}$. We drop the constant maps to the symplectic quotient from the notation for a stable map, since no information gets lost.
\begin{prop}[Convergence in the Ginzburg-Landau setting]\label{PROP:CONV S 1 C} Let $k\in\N_0$, for $\nu\in\N$ let $W_\nu\in\ME$ be a vortex class and $z_1^\nu,\ldots,z_k^\nu\in\C$ be points, and let 
\[(\W,\z):=\Big(T_0,T_1,T_\infty,E,(W_\al)_{\al\in T_1},(z_{\al\be})_{\al E\be},(\al_i,z_i)_{i=0,\ldots,k}\Big)\]
be a stable map. Then the following conditions are equivalent.
\begin{enui}
\item\label{prop:conv S 1 C w} The sequence $\big(W_\nu,z_0^\nu:=\infty,z_1^\nu,\ldots,z_k^\nu\big)$ converges to $(\W,\z)$.
\item\label{prop:conv S 1 C deg} For large enough $\nu$ we have
\begin{equation}
\label{eq:deg w nu}\deg(W_\nu)=\sum_{\al\in T_1}\deg(W_\al)=:d.
\end{equation}
Furthermore, there exist M\"obius transformations $\phi_\al^\nu:S^2\to S^2$ for $\al\in T$ and $\nu\in\N$ such that conditions (\ref{defi:conv phi z},\ref{defi:conv al be},\ref{defi:conv z}) of Definition \ref{defi:conv} are satisfied, and for every $\al\in T_1$ the point in the symmetric product $\deg_{W_\nu}\circ\phi_\al^\nu\in\Sym^d(\C)\sub\Sym^d(S^2)$ converges to $\iota_d(\deg_{W_\al})\in\Sym^d(S^2)$, as $\nu\to\infty$.
\end{enui}
\end{prop}
The proof of Proposition \ref{PROP:CONV S 1 C} is based on Proposition \ref{prop:cpt mod} (Compactness modulo bubbling and gauge) in Section \ref{sec:comp} below, and will be given in Section \ref{sec:proof:prop:conv S 1 C}.
\begin{xpl}\label{xpl:conv} Let $k:=0$, and for $\nu\in\N$ let $W^\nu\in\M_7$ be the unique vortex satisfying 
\[\deg_{W^\nu}(-2-i)=1,\qquad\deg_{W^\nu}(3+4i)=2,\qquad \deg_{W^\nu}(\nu e^{i\nu})=4.\]
Let $W_1\in\M_3$ be the unique vortex satisfying 
\[\deg_{W_1}(-2-i)=1,\qquad \deg_{W_1}(3+4i)=2,\]
and $W_2\in\M_4$ be the unique vortex satisfying 
\[\deg_{W_2}(0)=4.\]
Then $\big(W^\nu,z_0^\nu:=\infty\big)$ converges to the stable map consisting of the sets $T_0:=\emptyset$, $T_1:=\{1,2\}$, $T_\infty:=\{0\}$, the tree relation $E:=\{(1,0),(0,1),(2,0),(0,2)\}$, the vortices $W_1$ and $W_2$, the unique constant $\bar J$-holomorphic sphere $\bar u_0$ in $\BAR M$, the nodal points $z_{1\,0}:=z_{2\,0}:=\infty$, $z_{0\,1}:=1$, $z_{0\,2}:=2\in S^2\iso\C\cup\{\infty\}$, and the marked point $(\al_0,z_0):=(0,\infty)\in T\x S^2$. This follows from Proposition \ref{PROP:CONV S 1 C}. 

It follows from this example that in general, the additional marked points $z_0^\nu:=\infty$ are needed in the bubbling result, Theorem \ref{thm:bubb}. Without these points, no subsequence of $W_\nu$ as above converges to a stable map.%
\footnote{For this statement to make sense, here we adjust the notion of ``stable map'' by discarding the marked point $(\al_0,z_0)$.}
 This follows from an elementary argument. $\Box$
\end{xpl}
\section{The action of the reparametrization group}\label{sec:repara}
This section covers an additional topic, which will not be used in this memoir, but will be relevant in a future definition of the quantum Kirwan map. Namely, we introduce a natural group of reparametrizations and show that this group acts freely on the set of simple stable maps consisting of vortex classes over $\C$ and pseudo-holomorphic spheres in the symplectic quotient. The relevance of this result is the following. For the definition of the quantum Kirwan map it will be necessary to show that a certain natural evaluation map on the set of vortex classes over $\C$ (see \cite[Proposition 6.1]{ZiPhD} and \cite{ZiConsEv}) is a pseudo-cycle. This will rely on the fact that its omega limit set has codimension at least two. In order to show this, one needs to cut down the dimension of each ``boundary stratum'' by dividing by the action of the reparametrization group. Heuristically, the freeness of the action of this group implies that the quotient is a smooth manifold, hence providing a meaning to this procedure.

We fix finite sets $T_0,T_1,T_\infty$ and a tree relation $E$ on the disjoint union $T:=T_0\disj T_1\disj T_\infty$. We define the \emph{reparametrization group} $G_T$ as follows. We define $\Aut(T):=\Aut\big(T_0,T_1,T_\infty,E\big)$ to be the subgroup of all automorphisms $f$ of the tree $(T,E)$, satisfying $f(T_i)=T_i$, for $i=0,1,\infty$. We denote by $\PSL$ the group of M\"obius transformations and by $\TR$ the group of translations of the plane $\C$. We define 
\[\Aut_\al:=\left\{\begin{array}{ll}
\TR,&\textrm{if }\al\in T_1,\\
\PSL,&\textrm{if }\al\in T_0\cup T_\infty.\end{array}\right.\] 
We denote by $\Aut_T$ the set of collections $(\phi_\al)_{\al\in T}$, such that $\phi_\al\in\Aut_\al$, for every $\al\in T$. The group $\Aut(T)$ acts on $\Aut_T$ by 
\[f\cdot(\phi_\al)_{\al\in T}:=(\phi_{f^{-1}(\al)})_{\al\in T}.\]
\begin{defi}\label{defi:G T} We define $G_T:=G_{T_0,T_1,T_\infty,E}$ to be the semi-direct product of $\Aut(T)$ and $\Aut_T$ induced by this action. 
\end{defi}
The group $\PSL$ acts on the set of $\bar J$-holomorphic maps $S^2\to\BAR M$ by 
\[\phi^*f:=f\circ\phi.\] 
Furthermore, the group $\TR$ acts on the set $\ME$ as in (\ref{eq:phi []}). By the \emph{combinatorial type} of a stable map $(\W,\z)$ as in (\ref{eq:W z}) we mean the tuple $\big(T_0,T_1,T_\infty,E\big)$. We denote by
\[\M(T):=\M\big(T_0,T_1,T_\infty,E\big)\]
the set of all \emph{stable maps of (combinatorial) type $T$}. $G_T$ acts on $\M(T)$ as follows. For every $(f,(\phi_\al))\in G_T$ and $(\W,\z)\in\M(T)$ we define
\begin{eqnarray}\nn&W'_\al:=\phi_{f(\al)}^*W_{f(\al)},\,\forall\al\in T_1,\quad\bar u'_\al:=\bar u_\al\circ\phi_{f(\al)},\,\forall\al\in T_\infty,&\\
\nn&z'_{\al\be}:=\phi_{f(\al)}^{-1}(z_{f(\al)f(\be)}),\quad\forall\al E\be,&\\
\nn&\al'_i:=f(\al_i),\quad z'_i:=\phi_{\al'_i}^{-1}(z_{\al'_i}),\quad\forall i=0,\ldots,k.&\end{eqnarray}
(Here we set $W_\al:=\bar u_\al$ if $\al\in T_\infty$. Furthermore, for $\phi\in\TR$ we set $\phi(\infty):=\infty$.) 
\begin{defi}\label{defi:action G T} We define 
\[(f,(\phi_\al))^*(\W,\z):=\Big(T_0,T_1,T_\infty,E,(W'_\al)_{\al\in T_1},(\bar u'_\al)_{\al\in T_\infty},(z'_{\al\be})_{\al E\be},(\al'_i,z'_i)_{i=0,\ldots,k}\Big).\]
\end{defi}
This defines an action of $G_T$ on $\M(T)$. Let now $(M,J)$ be an almost complex manifold. Recall that a $J$-holomorphic map $u:S^2\to M$ is called \emph{multiply covered} iff there exists a holomorphic map $\phi:S^2\to S^2$ of degree at least two, and a $J$-holomorphic map $v:S^2\to M$, such that $u=v\circ\phi$. Otherwise, $u$ is called \emph{simple}. 

We call a stable map $(\W,\z)$ \emph{simple} iff the following conditions hold: For every $\al\in T_\infty$ the $\bar J$-holomorphic map $\bar u_\al$ is constant or simple. Furthermore, if $\al,\be\in T_1$ are such that $\al\neq\be$ and $\E(W_\al)\neq0$, and $\phi\in\TR$, then $\phi^*W_\al\neq W_\be$. Moreover, if $\al,\be\in T_\infty$ are such that $\al\neq\be$ and $\bar u_\al$ is nonconstant, and if $\phi\in\PSL$, then $\bar u_\al\circ\phi\neq\bar u_\be$. We denote by
\[\M^*(T):=\M^*\big(T_0,T_1,T_\infty,E\big)\sub\M(T)\]
the subset of all \emph{simple stable maps}. The action of $G_T$ on $\M(T)$ leaves $\M^*(T)$ invariant. 
\begin{prop}\label{prop:simple} The action of $G_T$ on $\M^*(T)$ is free. 
\end{prop}
The proof of this result uses the following lemma.
\begin{lemma}\label{le:trans} The action of $\TR$ on $\ME$ is free.
\end{lemma}
\begin{proof}[Proof of Lemma \ref{le:trans}] \setcounter{claim}{0} Assume that $W$ is a smooth vortex class over $\C$ and $v\in\C$ is a vector, such that defining $\phi:\C\to\C$ by $\phi(z):=z+v$, we have $\phi^*W=W$. Then $e_W(z+nv)=e_W(z)$ for every $z\in\C$ and $n\in\Z$. It follows that $E(W)=\infty$, $e_W\const0$, or $v=0$. Lemma \ref{le:trans} follows from this.
\end{proof}
\begin{proof}[Proof of Proposition \ref{prop:simple}]\setcounter{claim}{0} This follows from an elementary argument, using the stability condition (\ref{defi:st st}), Lemma \ref{le:trans}, and the fact that every simple holomorphic sphere is somewhere injective (see \cite[Proposition 2.5.1]{MS04}).
\end{proof}
\begin{rmk}\label{rmk:Isom +}\rm The action of $\TR$ on $\ME$ extends to an action of the group $\Isom^+(\C)$ of orientation preserving isometries of $\C$. Hence one may be tempted to adjust the definition of the reparametrization group $G_T$ and its action on $\M^*(T)$ accordingly. However, for the purpose of defining the quantum Kirwan map, this is not possible. The reason is that in general there is no continuous evaluation map on $\ME$ that is invariant under the action of $\Isom^+(\C)$. By definition such an evaluation map $\ev$ assigns a point in the Borel construction $(M\x\EG)/G$ to each pair $(W,z)$, where $W$ is a vortex class and $z\in\C$, in such a way that $\pi\circ\ev(W,z)\in\bar u_W(z)$. Here $\bar u_W$ is defined as in (\ref{eq:BAR u W}) and $\pi:(M\x\EG)/G\to M/G$ denotes the canonical projection. Hence when studying the vortex equations over $\C$ we need to restrict our attention to a subgroup of the group of global symmetries of the equations. This is a crucial difference between vortices and pseudo-holomorphic curves. 

Note also that the action of $\Isom^+(\C)$ on the set of vortex classes of positive energy is not always free, as the next example shows. $\Box$
\end{rmk}
\begin{xpl}\label{xpl:rot} Consider the action of $G:=S^1\sub \C$ on $M:=\C$ by multiplication. Let $d\in\N_0$ be an integer. By Proposition \ref{prop:class} there exists a unique finite energy vortex class $W$ over $\C$ such that 
\[\deg_W(z)=\left\{\begin{array}{ll}
d,&\textrm{if }z=0,\\
0,&\textrm{otherwise.}
\end{array}\right.\]
For every rotation $R\in\SO(2)$ we have
\[\deg_{R^*W}=R^*\deg_W=\deg_W,\]
where $\SO(2)$ acts in a natural way on $\Sym^d(\C)$. Thus the action of $\Isom^+(\C)$ on the set of vortex classes of positive energy is not free. $\Box$
\end{xpl}
\section{Compactness modulo bubbling and gauge for rescaled vortices}\label{sec:comp}
In this section we formulate and prove a crucial ingredient (Proposition \ref{prop:cpt mod}) of the proof of Theorem \ref{thm:bubb}, which states the following. Consider a sequence of rescaled vortices over $\C$ with image in a fixed compact subset of $M$ and uniformly bounded energies. We assume that $(M,\om)$ is aspherical. Then there exists a subsequence that away from finitely many bubbling points and up to regauging, converges to a rescaled vortex over $\C$. If the rescalings converge to $\infty$, then the limit object corresponds to a $\bar J$-holomorphic sphere in $\BAR M$. 

The proof of this result is based on compactness for rescaled vortices over the punctured plane with uniformly bounded energy densities (Proposition \ref{prop:cpt bdd} below). It also uses the fact that at each bubbling point at least the energy $\Emin>0$ is lost, which is the minimal energy of a vortex over $\C$ or pseudo-holomorphic sphere in $\BAR M$. This is the content of Proposition \ref{prop:quant en loss} below, which is proved by a hard rescaling argument, using Proposition \ref{prop:cpt bdd} and Hofer's lemma. Another ingredient of the proof of Proposition \ref{prop:cpt mod} is Lemma \ref{le:conv e} below, which says that the energy densities of a convergent sequence of rescaled vortices converge to the density of the limit. 

In order to explain the main result of this section, let $M,\om,G,\g,\lan\cdot,\cdot\ran_\g,\mu,J$, $\Si,j$, and $\om_\Si$ be as in Chapter \ref{chap:main}. We fix a triple $w=(P,A,u)\in\BB_\Si$ (defined as in (\ref{eq:BB})). Recall the definition (\ref{eq:e W}) of the energy density $e^{\om_\Si,j}_w=e_w$.
\begin{rmk}\rm\label{rmk:trafo e vort} This density has the following transformation property: Let $\Si'$ be another surface, and $\phi:\Si'\to\Si$ a smooth immersion. Consider the pullback
\[\phi^*w:=\big(\phi^*P,\Phi^*A,u\circ\Phi\big),\] 
where the bundle isomorphism $\Phi:\phi^*P\to P$ is defined by $\Phi(z,p):=p$. A straight-forward calculation shows that 
\begin{equation}\label{eq:e phi *}e^{\phi^*(\om_\Si,j)}_{\phi^*w}=e^{\om_\Si,j}_w\circ\phi.
\end{equation} 
Note also that $w$ is a vortex with respect to $(\om_\Si,j)$ if and only if $\phi^*w$ is a vortex with respect to $\phi^*(\om_\Si,j)$. $\Box$
\end{rmk}
\begin{rmk}\rm\label{rmk:e vort} If $w$ is a vortex (with respect to $(\om_\Si,j)$) then 
\begin{equation}\label{eq:e vort}e^{\om_\Si,j}_w=|\dd_{J,A}u|^2+|\mu\circ u|^2,
\end{equation}
where $\dd_{J,A}u$ denotes the complex linear part of $d_Au=du+L_uA$, viewed as a one-form on $\Si$ with values in the complex vector bundle $(u^*TM)/G\to\Si$.%
\footnote{The complex structure on this bundle is induced by $J$.}
 This follows from the vortex equations (\ref{eq:BAR dd J A u},\ref{eq:F A mu}). $\Box$ 
\end{rmk}
Let $R\in[0,\infty]$ and $w\in\BB_\Si$. Consider first the case $0<R<\infty$. Then we define the \emph{$R$-energy density of $w$} to be
\begin{equation}\label{eq:e R W R}e^R_w:=R^2e^{R^2\om_\Si,j}_w.\end{equation}
This means that 
\begin{equation}\label{eq:e R W 1 2}e^R_w=\frac12\Big(|d_Au|_{\om_\Si}^2+R^{-2}|F_A|_{\om_\Si}^2+R^2|\mu\circ u|^2\Big),
\end{equation}
where the subscript ``$\om_\Si$'' means that the norms are taken with respect to the metric $\om_\Si(\cdot,j\cdot)$. 

If $R=0$ or $\infty$ then we define 
\begin{equation}\nn e^R_w:=\frac12|d_Au|_{\om_\Si}^2.
\end{equation}

We define the \emph{$R$-energy of $w$} on a measurable subset $X\sub\Si$ to be
\begin{equation}\label{eq:E R w X} E^R(w,X):=\int_Xe^R_w\om_\Si\in [0,\infty].\end{equation}
\begin{Rmk} Consider the case $(\Si,j)=\C$, equipped with the standard area form $\om_0$. Assume that $0<R<\infty$, and consider the map $\phi:\C\to\C$ defined by $\phi(z):=Rz$. Then equality (\ref{eq:e phi *}) implies that
\[e^R_{\phi^*w}=R^2e^{\om_0,i}_w\circ\phi.\]
Hence in the present setting the $R$-energy transforms via
\[E^R\big(\phi^*w,\phi^{-1}(X)\big)=E(w,X):=E^1(w,X).\Box\]
\end{Rmk}
The \emph{(symplectic) $R$-vortex equations} are the equations (\ref{eq:BAR dd J A u},\ref{eq:F A mu}) with $\om_\Si$ replaced by $R^2\om_\Si$, i.e., the equations 
\begin{eqnarray}\label{eq:R dd J A u}&\bar\dd_{J,A}(u)=0,&\\
\label{eq:F A R}&F_A+R^2(\mu\circ u)\om_\Si=0.&
\end{eqnarray}
In the case $R=\infty$ we interpret the equation (\ref{eq:F A R}) as
\[\mu\circ u=0.\]
We call a solution $(A,u)\in\WWW^p_\Si$ of equations (\ref{eq:R dd J A u},\ref{eq:F A R}) an \emph{$R$-vortex} over $\Si$. 
\begin{Rmks} Consider the case $(\Si,j)=\C$, equipped with $\om_0$, and let $0<R<\infty$. We define the map $\phi:\C\to\C$ by $\phi(z):=Rz$. Then $w\in\BB_{\C}$ is a vortex if and only if $\phi^*w$ is an $R$-vortex.

The rescaled energy density has the following important property. Let $R_\nu\in(0,\infty)$ be a sequence that converges to some $R_0\in[0,\infty]$, and for $\nu\in\N_0$ let $w_\nu=(P_\nu,A_\nu,u_\nu)$ be an $R_\nu$-vortex. If on compact sets $A_\nu$ converges to $A_0$ in $C^0$ and $u_\nu$ converges to $u_0$ in $C^1$ then 
\begin{equation}\nn e^{R_\nu}_{w_\nu}\to e^{R^0}_{w_0}\end{equation} 
in $C^0$ on compact sets. (See Lemma \ref{le:conv e} below.) In the proof of Theorem \ref{thm:bubb}, this will be used in order to show that locally on $\C$ no energy is lost in the limit $\nu\to\infty$. $\Box$ 
\end{Rmks}
We define the \emph{minimal energy $\Emin$} as follows. Recall from (\ref{eq:M}) that $\MM$ denotes the class consisting of smooth vortices, and that the energy of a $\bar J$-holomorphic map $f:S^2\to\BAR M$ is given by $E(f)=\int_{S^2}f^*\BAR\om$. We define%
\footnote{Here we use the convention that $\inf\emptyset=\infty$.}%
\begin{equation}\label{eq:E V}E_1:=\inf\big(\big\{E(P,A,u)\,\big|\,(P,A,u)\in\MM:\,\BAR{u(P)}\textrm{ compact}\big\}\cap(0,\infty)\big),
\end{equation}
\[E_\infty:=\inf\big(\big\{E(f)\,\big|\,f\in C^\infty(S^2,\BAR M):\,\bar\dd_{\bar J}(f)=0\big\}\cap(0,\infty)\big),\]
\begin{equation}\label{eq:Emin}\Emin:=\min\{E_1,E_\infty\}.\end{equation}
Assume that $M$ is equivariantly convex at $\infty$. Then Corollary \ref{cor:quant} in Appendix \ref{sec:vort} implies that $E_1>0$. Furthermore, our standing assumption (H) implies that $\BAR M$ is closed. It follows that $E_\infty>0$.%
\footnote{See e.g.~\cite[Proposition 4.1.4]{MS04}.}
 Hence the number $\Emin$ is positive. 
\begin{Rmk} The infima (\ref{eq:E V}) and (\ref{eq:Emin}) are attained, and hence the name ``minimal energy'' for $\Emin$ is justified. (This fact is not used anywhere in this memoir.) That (\ref{eq:Emin}) is attained follows from the fact that for every $C\in\R$ there exist only finitely many homotopy classes $\BAR B\in\pi_2(\BAR M)$ with $\lan[\BAR\om],\BAR B\ran\leq C$ that can be represented by a $\BAR J$-holomorphic map $S^2\to\BAR M$.%
\footnote{This is a corollary to Gromov compactness, see e.g.~\cite[Corollary 5.3.2]{MS04}.}
 That (\ref{eq:E V}) is attained follows from the fact that for every $C\in\R$ there exist only finitely many homology classes $B\in H^G_2(M,\Z)$ with $\big\lan[\om-\mu],B\big\ran\leq C$ that can be represented by a finite energy vortex whose image has compact closure. This is a consequence of Theorem \ref{thm:bubb} and \cite[Proposition 5.4]{ZiPhD} (Conservation of equivariant homology class). $\Box$
\end{Rmk}
The results of this and the next section are formulated for connections and maps of Sobolev regularity. This is a natural setup for the relevant analysis. Furthermore, we restrict our attention to the trivial bundle $\Si\x G$.%
\footnote{Since every smooth bundle over $\C$ is trivializable, this suffices for the proof of the main result.}

We fix $p>2$
\footnote{Recall that throughout this memoir, $p<\infty$, unless otherwise stated.}
 and naturally identify the affine space of connections on $\Si\x G$ of local Sobolev class $W^{1,p}_\loc$ with the space of one-forms on $\Si$ with values in $\g$, of class $W^{1,p}_\loc$. Furthermore, we identify the space of $G$-equivariant maps from $\Si\x G$ to $M$ of class $W^{1,p}_\loc$ with $W^{1,p}_\loc(\Si,M)$. Finally, we identify the gauge group%
\footnote{i.e., the group of gauge transformations}
 on $\Si\x G$ of class $W^{2,p}_\loc$ with $W^{2,p}_\loc(\Si,G)$. We denote
\[\WWW_\Si:=\Om^1(\Si,\g)\x C^\infty(\Si,M),\]
\[\WWW^p_\Si:=\big\{\textrm{one-form on }\Si\textrm{ with values in }\g\textrm{, of class }W^{1,p}_\loc\big\}\x W^{1,p}_\loc(\Si,M).\]
The gauge group $W^{2,p}_\loc(\Si,G)$ acts on $\WWW^p_\Si$ by
\[g^*(A,u):=\big(\Ad_{g^{-1}}A+g^{-1}dg,g^{-1}u\big),\]
where $\Ad_{g_0}:\g\to\g$ denotes the adjoint action of an element $g_0\in G$. For $r>0$ we denote by $B_r\sub\C$ the open ball of radius $r$, around 0.
\begin{prop}[Compactness modulo bubbling and gauge]\label{prop:cpt mod} Assume that $(M,\om)$ is aspherical. Let $R_\nu\in(0,\infty)$ be a sequence that converges to some $R_0\in (0,\infty]$, $r_\nu\in(0,\infty)$ a sequence that converges to $\infty$, and for every $\nu\in\N$ let $w_\nu=(A_\nu,u_\nu)\in\WWW^p_{B_{r_\nu}}$ be an $R_\nu$-vortex (with respect to $(\om_0,i)$). Assume that there exists a compact subset $K\sub M$ such that $u_\nu(B_{r_\nu})\sub K$, for every $\nu$. Suppose also that 
\[\sup_\nu E^{R_\nu}(w_\nu,B_{r_\nu})<\infty.\] 
Then there exist a finite subset $Z\sub\C$, an $R_0$-vortex $w_0:=(A_0,u_0)\in\WWW_{\C\wo Z}$, and, passing to some subsequence, gauge transformations $g_\nu\in W^{2,p}_\loc(\C\wo Z,G)$, such that the following conditions hold. 
\begin{enui}\item\label{prop:cpt mod:<} If $R_0<\infty$ then $Z=\emptyset$ and the sequence $g_\nu^*(A_\nu,u_\nu)$ converges to $w_0$ in $C^\infty$ on every compact subset of $\C$. 
\item\label{prop:cpt mod:=} If $R_0=\infty$ then on every compact subset of $\C\wo Z$, the sequence $g_\nu^*A_\nu$ converges to $A_0$ in $C^0$, and the sequence $g_\nu^{-1}u_\nu$ converges to $u_0$ in $C^1$.
\item\label{prop:cpt mod lim nu E eps} Fix a point $z\in Z$ and a number $\eps_0>0$ so small that $B_{\eps_0}(z)\cap Z=\{z\}$. Then for every $0<\eps<\eps_0$ the limit 
\[E_z(\eps):=\lim_{\nu\to\infty}E^{R_\nu}(w_\nu,B_\eps(z))\] 
exists and 
\[E_z(\eps)\geq\Emin.\] 
Furthermore, the function $(0,\eps_0)\ni\eps\mapsto E_z(\eps)\in[\Emin,\infty)$ is continuous.
\end{enui}
\end{prop}
\begin{Rmk} Convergence in conditions (\ref{prop:cpt mod:<},\ref{prop:cpt mod:=}) should be understood as convergence of the subsequence labelled by those indices $\nu$ for which $B_{r_\nu}$ contains the given compact set. $\Box$ 
\end{Rmk}
This proposition will be proved on page \pageref{prop:cpt mod proof}. The strategy of the proof is the following. Assume that the energy densities $e^{R_\nu}_{w_\nu}$ are uniformly bounded on every compact subset of $\C$. Then the statement of Proposition \ref{prop:cpt mod} with $Z=\emptyset$ follows from an argument involving Uhlenbeck compactness, an estimate for $\bar\dd_J$, elliptic bootstrapping (for statement (\ref{prop:cpt mod:<})), and a patching argument. 

If the densities are not uniformly bounded then we rescale the maps $w_\nu$ by zooming in near a bubbling point $z_0$ in a ``hard way'', to obtain a positive energy $\wt R_0$-vortex in the limit, with $\wt R_0\in\{0,1,\infty\}$. If $R_0<\infty$ then $\wt R_0=0$, and we obtain a $J$-holomorphic sphere in $M$. This contradicts symplectic asphericity, and thus this case is impossible. 

If $R_0=\infty$ then either $\wt R_0=1$ or $\wt R_0=\infty$, and hence either a vortex over $\C$ or a pseudo-holomorphic sphere in $\BAR M$ bubbles off. Therefore, at least the energy $\Emin$ is lost at $z_0$. Our assumption that the energies of $w_\nu$ are uniformly bounded implies that there can only be finitely many bubbling points. On the complement of these points a subsequence of $w_\nu$ converges modulo gauge. 

The bubbling part of this argument is captured by Proposition \ref{prop:quant en loss} below, whereas the convergence part is the content of the following result. 
\begin{prop}[Compactness with bounded energy densities]\label{prop:cpt bdd} Let $Z\sub\C$ be a finite subset, $R_\nu\geq 0$ be a sequence of numbers that converges to some $R_0\in [0,\infty]$, $\Om_1\sub\Om_2\sub\ldots\sub \C\wo Z$ open subsets such that $\bigcup_\nu\Om_\nu=\C\wo Z$, and for $\nu\in\N$ let $w_\nu=(u_\nu,A_\nu)\in\WWW^p_{\Om_\nu}$ be an $R_\nu$-vortex. Assume that there exists a compact subset $K\sub M$ such that for $\nu$ large enough
\begin{equation}
  \label{eq:u nu Om K}u_\nu(\Om_\nu)\sub K. 
\end{equation}
Suppose also that for every compact subset $Q\sub \C\wo Z$, we have 
\begin{equation}
    \label{eq:sup e}
\sup\big\{\Vert e_{w_\nu}^{R_\nu}\Vert_{L^{\infty}(Q)}\,\big|\,\nu\in\N:\,Q\sub\Om_\nu\big\}<\infty.
  \end{equation}
Then there exists an $R_0$-vortex $w_0:=(A_0,u_0)\in\WWW_{\C\wo Z}$, and passing to some subsequence, there exist gauge transformations $g_\nu\in W^{2,p}_\loc(\C\wo Z,G)$, such that the following conditions are satisfied.
\begin{enui}
\item\label{prop:cpt bdd < infty} If $R_0<\infty$ then $g_\nu^*w_\nu$ converges to $w_0$ in $C^\infty$ on every compact subset of $\C\wo Z$. 
\item\label{prop:cpt bdd = infty} If $R_0=\infty$ then on every compact subset of $\C\wo Z$, $g_\nu^*A_\nu$ converges to $A_0$ in $C^0$, and $g_\nu^{-1}u_\nu$ converges to $u_0$ in $C^1$. 
\end{enui}
\end{prop}
The proof of this result is an adaption of the argument of Step 5 in the proof of Theorem A in the paper by R.~Gaio and D.~A.~Salamon \cite{GS}. The proof of statement (\ref{prop:cpt bdd < infty}) is based on a compactness result for the case of a compact surface $\Si$ (possibly with boundary). (See Theorem \ref{thm:cpt cpt} in Appendix \ref{sec:vort}. That result follows from an argument by K.~Cieliebak et al.~in \cite{CGMS}.) The proof also involves a patching argument for gauge transformations, which are defined on an exhausting sequence of subsets of $\C\wo Z$. 

To prove statement (\ref{prop:cpt bdd = infty}), we will show that curvatures of the connections $A_\nu$ are uniformly bounded in $W^{1,p}$. This uses the second rescaled vortex equations and a uniform upper bound on $\mu\circ u_\nu$ (Lemma \ref{le:mu u} in Appendix \ref{sec:vort}), due to R.~Gaio and D.~A.~Salamon. The statement then follows from Uhlenbeck compactness with compact base, compactness for $\bar\dd_J$, and a patching argument.
\begin{proof}[Proof of Proposition \ref{prop:cpt bdd}]\setcounter{claim}{0} We may assume w.l.o.g.~that there exists a $G$-invariant compact subset $K\sub M$ satisfying (\ref{eq:u nu Om K}). (To see this, we choose a compact subset $K$ satisfying this condition and consider the set $GK$.) We choose $i_0\in\N$ so big that the balls $\bar B_{1/i_0}(z)$, $z\in Z$, are disjoint and contained in $B_{i_0}$. For every $i\in\N_0$ we define 
\[X^i:=\bar B_{i+i_0}\wo\bigcup_{z\in Z}B_{\frac1{i+i_0}}(z)\sub\C.\]

We prove {\bf statement (\ref{prop:cpt bdd < infty})}. Assume that $R_0<\infty$. Using the hypotheses (\ref{eq:u nu Om K},\ref{eq:sup e}), it follows from Theorem \ref{thm:cpt cpt} in Appendix \ref{sec:vort} (with $\Si:=X^2$) that there exist an infinite subset $I^1\sub\N$ and gauge transformations $g^1_\nu\in W^{2,p}(X^1,G)$ ($\nu\in I^1$), such that $X^1\sub\Om_\nu$ and
\[w^1_\nu:=(A^1_\nu,u^1_\nu):=(g^1_\nu)^*(w_\nu|X^1)\] 
is smooth, for every $\nu\in I^1$, and the sequence $(w^1_\nu)_{\nu\in I^1}$ converges to some $R_0$-vortex $w^1\in\WWW_{X^1}$, in $C^\infty$ on $X^1$.

Iterating this argument, for every $i\geq2$ there exist an infinite subset $I^i\sub I^{i-1}$ and gauge transformations $g^i_\nu\in W^{2,p}(X^i,G)$ ($\nu\in I^i$), such that $X^i\sub\Om_\nu$ and
\[w^i_\nu:=(A^i_\nu,u^i_\nu):=(g^i_\nu)^*(w_\nu|X^i)\] 
is smooth, for every $\nu\in I^i$, and the sequence $(w^i_\nu)_{\nu\in I^i}$ converges to some $R_0$-vortex $w^i\in\WWW_{X^1}$, in $C^\infty$ on $X^i$. 

Let $i\in\N$. For every $\nu\in I^i$ we define $h^i_\nu:=(g^{i+1}_\nu|X^i)^{-1}g^i_\nu$. We have $(h^i_\nu)^*(A^{i+1}_\nu|X^i)=A^i_\nu$. Furthermore, the sequences $(A^{i+1}_\nu)_{\nu\in I^{i+1}}$ and $(A^i_\nu)_{\nu\in I^{i+1}}$ are bounded in $W^{k,p}$ on $X^i$, for every $k\in\N$. Hence it follows from Lemma \ref{le:g smooth} (Appendix \ref{sec:add}) that the sequence $(h^i_\nu)_{\nu\in I^{i+1}}$ is bounded in $W^{k,p}$ on $X^i$, for every $k\in\N$. Hence, using Morrey's embedding theorem and the Arzel\`a-Ascoli theorem, it has a subsequence that converges to some gauge transformation $h^i\in C^\infty(X^i,G)$, in $C^\infty$ on $X^i$. Note that 
\begin{equation}\label{eq:h i w i}(h^i)^*(w^{i+1}|X^i)=w^i.
\end{equation}
We choose a map $\rho^i:X^{i+1}\to X^i$ such that $\rho^i=\id$ on $X^{i-1}$. We define%
\footnote{This patching construction follows the lines of the proofs of \cite[Theorems 3.6 and A.3]{FrPhD}.}
 $k^1:=h^1$, and recursively,
\begin{equation}\label{k i h i}k^i:=h^i(k^{i-1}\circ\rho^{i-1})\in C^\infty(X^i,G),\quad\forall i\geq2.
\end{equation}
Using (\ref{eq:h i w i}) and the fact $\rho^{i-1}=\id$ on $X^{i-2}$, we have, for every $i\geq2$, 
\[(k^i)^*w^{i+1}=(k^{i-1}\circ\rho^{i-1})^*w^i=(k^{i-1})^*w^i,\quad\textrm{on }X^{i-2}.\]
It follows that there exists a unique $w\in\WWW_{\C\wo Z}$ that restricts to $(k^{i+1})^*w^{i+2}$ on $X^i$, for every $i\in\N$. Let $i\in\N$. We choose $\nu_i\in I^{i+1}$ such that $\nu_i\geq i$ and a map $\tau^i:\C\wo Z\to X^i$ that is the identity on $X^{i-1}$. We define
\[g_i:=(g^{i+1}_{\nu_i}k^i)\circ\tau^i\in C^\infty(\C\wo Z,G).\] 
The sequence $g_i^*w_{\nu_i}$ converges to $w$, in $C^\infty$ on every compact subset of $\C\wo Z$. (Here we use the $C^\infty$-convergence on $X^i$ of $(w^i_\nu)_{\nu\in I^i}$ to $w^i$ and the facts $X_1\sub X_2\sub\cdots$ and $\bigcup_{i\in\N}X_i=\C\wo Z$.) Statement (\ref{prop:cpt bdd < infty}) follows. 

We prove {\bf statement (\ref{prop:cpt bdd = infty})}. Assume that $R_0=\infty$. 
\begin{claim}\label{claim:kappa nu p} For every compact subset $Q\sub \C\wo Z$ we have
  \begin{equation}
    \label{eq:sup kappa nu}\sup_\nu\big\{\Vert F_{A_\nu}\Vert_{L^p(Q)}\,\big|\,\nu\in\N:\,Q\sub\Om_\nu\big\}<\infty.
  \end{equation}
\end{claim}
\begin{proof}[Proof of Claim \ref{claim:kappa nu p}] Let $\Om\sub\C$ be an open subset containing $Q$ such that $\BAR\Om$ is compact and contained in $\C\wo Z$. Hypothesis (\ref{eq:sup e}) implies that
\begin{equation}\label{eq:sup nu Vert}\sup_\nu\Vert d_{A_\nu}u_\nu\Vert_{L^\infty(\BAR\Om)}<\infty. 
\end{equation}
It follows from our standing hypothesis (H) that there exists $\delta>0$ such that $G$ acts freely on
\[K:=\big\{x\in M\,\big|\,|\mu(x)|\leq\delta\big\}.\]
Since $\mu$ is proper the set $K$ is compact. Recall that $L_x:\g\to T_xM$ denotes the infinitesimal action at $x$. It follows that 
\begin{equation}\label{eq:sup xi}\sup\left\{\frac{|\xi|}{|L_x\xi|}\,\bigg|\,x\in K,\,0\neq\xi\in\g\right\}<\infty.\end{equation}
Using the second $R_\nu$-vortex equation, we have
\[|\mu\circ u_\nu|\leq\frac{\sqrt{e^{R_\nu}_{w_\nu}}}{R_\nu}.\] 
Hence by hypothesis (\ref{eq:sup e}) and the assumption $R_\nu\to\infty$, we have $\Vert\mu\circ u_\nu\Vert_{L^\infty(\Om)}<\delta$, for $\nu$ large enough. Using (\ref{eq:sup nu Vert},\ref{eq:sup xi}), Lemma \ref{le:mu u} in Appendix \ref{sec:vort} implies that 
\[\sup_\nu R_\nu^2\Vert\mu\circ u_\nu\Vert_{L^p(Q)}<\infty.\]
Estimate (\ref{eq:sup kappa nu}) follows from this and the second $R_\nu$-vortex equation. This proves Claim \ref{claim:kappa nu p}.
\end{proof}
Using Claim \ref{claim:kappa nu p}, Theorem \ref{thm:Uhlenbeck compact} (Uhlenbeck compactness) in Appendix \ref{sec:add} implies that there exist an infinite subset $I^1\sub\N$ and, for each $\nu\in I^1$, a gauge transformation $g^1_\nu\in W^{2,p}(X^1,G)$, such that $X^1\sub\Om_\nu$ and the sequence $A^1_\nu:=(g^1_\nu)^*(A_\nu|X^1)$ converges to some $W^{1,p}$-connection $A^1$ over $X^1$, weakly in $W^{1,p}$ on $X^1$. By Morrey's embedding theorem and the Arzel\`a-Ascoli theorem, shrinking $I^1$, we may assume that $A^1_\nu$ converges (strongly) in $C^0$ on $X^1$. 

Iterating this argument, for every $i\geq2$ there exist an infinite subset $I^i\sub I^{i-1}$ and gauge transformations $g^i_\nu\in W^{2,p}(X^1,G)$, for $\nu\in I^i$, such that $X^i\sub\Om_\nu$, for every $\nu\in I^i$, and the sequence $A^i_\nu:=(g^i_\nu)^*(A_\nu|X^i)$ converges to some $W^{1,p}$-connection $A^i$ over $X^i$, weakly in $W^{1,p}$ and in $C^0$ on $X^i$. 

Let $i\in\N$. For $\nu\in I^i$ we define $h^i_\nu:=(g^{i+1}_\nu|X^i)^{-1}g^i_\nu$. An argument as in the proof of statement (\ref{prop:cpt bdd < infty}), using Lemma \ref{le:g smooth} (Appendix \ref{sec:add}), implies that the sequence $(h^i_\nu)_{\nu\in I^i}$ has a subsequence that converges to some gauge transformation $h^i\in W^{2,p}(X^i,G)$, weakly in $W^{2,p}$ on $X^i$. 

Repeating the construction in the proof of statement (\ref{prop:cpt bdd < infty}) and using the weak $W^{1,p}$- and strong $C^0$-convergence of $A^i_\nu$ on $X^i$, we obtain an index $\nu_i\geq i+1$ and a gauge transformation $g_i\in W^{2,p}(\C\wo Z,G)$, for every $i\in\N$, such that $\nu_i\in I^{i+1}$ and $g_i^*A_{\nu_i}$ converges to some $W^{1,p}$-connection $A$ over $\C\wo Z$, weakly in $W^{1,p}$ and in $C^0$ on every compact subset of $\C\wo Z$. Passing to the subsequence $(\nu_i)_i$, we may assume w.l.o.g.~that $A_\nu$ converges to $A$, weakly in $W^{1,p}$ and in $C^0$ on every compact subset of $\C\wo Z$.
\begin{claim}\label{claim:hyp:prop:compactness delbar} The hypotheses of Proposition \ref{prop:compactness delbar} (Appendix \ref{sec:add}) with $k=1$ are satisfied.  
\end{claim}
\begin{proof}[Proof of Claim \ref{claim:hyp:prop:compactness delbar}] Let $\Om\sub \C\wo Z$ be an open subset with compact closure, and $\nu_0\in\N$ be such that $\Om\sub\Om_{\nu_0}$. Since the sequence $(A_\nu)$ converges to $A$, weakly in $W^{1,p}(\Om)$, we have 
\begin{equation}
  \label{eq:A nu W 1 p}
\sup_{\nu\geq\nu_0}\Vert A_\nu\Vert_{W^{1,p}(\Om)}<\infty.
\end{equation}
{\bf Condition (\ref{eq:u nu K})} is satisfied by assumption (\ref{eq:u nu Om K}). We check {\bf condition (\ref{eq:du nu})}: We denote by $|\Om|$ the area of $\Om$ and choose a constant $C>0$ such that $L_x\xi\leq C|\xi|$, for every $x\in K$ and $\xi\in\g$. For $\nu\geq\nu_0$, we have
\begin{eqnarray}\nn\Vert du_\nu\Vert_{L^p(\Om)}&\leq&\Vert d_{A_\nu}u_\nu\Vert_{L^p(\Om)}+\Vert L_{u_\nu}A_\nu\Vert_{L^p(\Om)}\\
&\leq&|\Om|^{\frac1p}\Vert d_{A_\nu}u_\nu\Vert_{L^\infty(\Om)}+C\Vert A_\nu\Vert_{L^p(\Om)}.
\end{eqnarray}
Here the second inequality uses the hypothesis (\ref{eq:u nu Om K}). Combining this with (\ref{eq:sup e}) and (\ref{eq:A nu W 1 p}), condition (\ref{eq:du nu}) follows. 

{\bf Condition (\ref{eq:sup k p})} follows from the first vortex equation, (\ref{eq:A nu W 1 p}), (\ref{eq:du nu}), and hypothesis (\ref{eq:u nu Om K}). This proves Claim \ref{claim:hyp:prop:compactness delbar}.
\end{proof}
By Claim \ref{claim:hyp:prop:compactness delbar}, we may apply Proposition \ref{prop:compactness delbar}, to conclude that, passing to some subsequence, $u_\nu$ converges to some map $u\in W^{2,p}(\C\wo Z)$, weakly in $W^{2,p}$ and in $C^1$ on every compact subset of $\C\wo Z$. The pair $w:=(A,u)$ solves the first vortex equation. Furthermore, multiplying the second $R_\nu$-vortex equation with $R_\nu^{-2}$, it follows that $\mu\circ u=0$. This means that $w$ is an $\infty$-vortex. 
\begin{claim}\label{claim:g A u smooth} There exists a gauge transformation $g\in W^{2,p}(\C\wo Z,G)$ such that $g^*(A,u)$ is smooth.
\end{claim}
\begin{proof}[Proof of Claim \ref{claim:g A u smooth}] Since $\C\wo Z$ continuously deformation retracts onto a wedge of circles, and $G$ is connected, there exists a continuous lift $\wt v:\C\wo Z\to\mu^{-1}(0)$ of $Gu$. By Proposition \ref{prop:bar del J} in Appendix \ref{sec:add} the map $Gu:\C\wo Z\to\BAR M$ is $\bar J$-holomorphic. Hence it is smooth. It follows that we may approximate $\wt v$ by some smooth lift $v$ of $Gu$. We define $g:\C\wo Z\to G$ to be the unique solution of $g(z)v=u(z)$, for every $z\in\C\wo Z$. Since the infinitesimal action at every $x\in\mu^{-1}(0)$ is injective, the equation $\bar\dd_{J,g^*A}(g^{-1}u)=0$ and smoothness of $v=g^{-1}u$ imply that $g^*A$ is smooth. This proves Claim \ref{claim:g A u smooth}.
\end{proof}
We choose $g$ as in Claim \ref{claim:g A u smooth}. Regauging $A_\nu$ by $g$, statement (\ref{prop:cpt bdd = infty}) follows. This completes the proof of Proposition \ref{prop:cpt bdd}.  
\end{proof}
\begin{Rmk} One can try to circumvent the patching argument for the gauge transformations in this proof by choosing an extension $\wt g^i_\nu$ of $g^i_\nu$ to $\C\wo Z$, and defining $g_\nu:=\wt g^\nu_\nu$. However, the sequence $(g_\nu)$ does not have the required properties, since $g_\nu^*w_\nu$ does not necessarily converge on compact subsets of $\C\wo Z$. The reason is that for $j>i$ the transformation $g^j_\nu$ does not need to restrict to $g^i_\nu$ on $X^i$. $\Box$
\end{Rmk}
\begin{rmk}\label{rmk:C infty} It is not clear if in the case $R_0=\infty$ the $g_\nu$'s can be chosen in such a way that $g_\nu^*w_\nu$ converges in $C^\infty$ on every compact subset of $\C\wo Z$. To prove this, one approach is to fix an open subset of $\C$ with smooth boundary and compact closure, which is contained in $\C\wo Z$. We can now try mimic the proof of \cite[Theorem 3.2]{CGMS}. In Step 3 of that proof the first and second vortex equations (and relative Coulomb gauge) are used iteratively in an alternating way. This iteration fails in our setting, because of the factor $R_\nu^2$ in the second vortex equations, which converges to $\infty$ by assumption. $\Box$
\end{rmk}
The second ingredient of the proof of Proposition \ref{prop:cpt mod} says that if the energy densities of a sequence of rescaled vortices are not uniformly bounded on some compact subset $Q$, then at least the energy $\Emin$ is lost in the limit, at some point in $Q$. Here $\Emin$ is defined as in (\ref{eq:Emin}). 
\begin{prop}[Quantization of energy loss]\label{prop:quant en loss} Assume that $(M,\om)$ is aspherical. Let $\Om\sub\C$ be an open subset, $0<R_\nu<\infty$ a sequence such that $\inf_\nu R_\nu>0$, and $w_\nu\in\WWW^p_\Om$ an $R_\nu$-vortex, for $\nu\in\N$. Assume that there exists a compact subset $K\sub M$ such that $u_\nu(\Om)\sub K$ for every $\nu$ and that $\sup_\nu E^{R_\nu}(w_\nu)<\infty$. Then the following conditions hold.
\begin{enui}
\item \label{prop:hard R e} For every compact subset $Q\sub\Om$ we have
\[\sup_\nu R_\nu^{-2}\Vert e_{w_\nu}^{R_\nu}\Vert_{C^0(Q)}<\infty.\]
\item \label{prop:hard limsup} If there exists a compact subset $Q\sub \Om$ such that $\sup_\nu||e_{w_\nu}^{R_\nu}||_{C^0(Q)}=\infty$ then there exists $z_0\in Q$ with the following property. For every $\eps>0$ so small that $B_\eps(z_0)\sub \Om$ we have
\begin{equation}
\label{eq:limsup Emin}\limsup_{\nu\to\infty} E^{R_\nu}(w_\nu,B_\eps(z_0))\geq\Emin.
\end{equation}
\end{enui}
\end{prop}
The proof of Proposition \ref{prop:quant en loss} is built on a bubbling argument as in Step 5 in the proof of \cite[Theorem A]{GS}. The idea is that under the assumption of (\ref{prop:hard limsup}) we may construct either a $\bar J$-holomorphic sphere in $\BAR M$ or a vortex over $\C$, by rescaling the sequence $w_\nu$ in a ``hard way''. This means that after rescaling the energy densities are bounded. We need the following two lemmata. 
\begin{lemma}[Hofer]\label{lemma:Hofer} Let $(X,d)$ be a metric space, $f:X\to[0,\infty)$ a continuous function, $x\in X$, and $\delta>0$. Assume that the closed ball $\bar B_{2\delta}(x)$ is complete. Then there exists $\xi\in X$ and a number $0<\eps\leq\delta$ such that
\[d(x,\xi)<2\delta,\qquad\sup_{B_\eps(\xi)}f\leq2f(\xi),\qquad\eps f(\xi)\geq\delta f(x).\]
\end{lemma}

\begin{proof}See \cite[Lemma 4.6.4]{MS04}.
\end{proof}
The next lemma ensures that for a suitably convergent sequence of rescaled vortices in the limit $\nu\to\infty$ no energy gets lost on any compact set. Apart from Proposition \ref{prop:quant en loss}, it will also be used in the proofs of Propositions \ref{prop:cpt mod} and \ref{prop:soft}, and Theorem \ref{thm:bubb}. 
\begin{lemma}[Convergence of energy densities]\label{le:conv e} Let $(\Si,\om_\Si,j)$ be a surface without boundary, equipped with an area form and a compatible complex structure, $R_\nu\in[0,\infty)$, $\nu\in\N$, a sequence of numbers that converges to some $R_0\in[0,\infty]$, and for $\nu\in\N_0$ let $w_\nu:=(A_\nu,u_\nu)\in\WWW^p_\Si$ be an $R_\nu$-vortex. Assume that on every compact subset of $\Si$, $A_\nu$ converges to $A_0$ in $C^0$ and $u_\nu$ converges to $u_0$ in $C^1$. Then we have
\begin{equation}\label{eq:e R nu W nu e f} e^{R_\nu}_{w_\nu}\to e^{R_0}_{w_0}
\end{equation}
in $C^0$ on every compact subset of $\Si$. 
\end{lemma}
\begin{proof}[Proof of Lemma \ref{le:conv e}]\setcounter{claim}{0} In the {\bf case $R_0<\infty$} the statement of the lemma is a consequence of equality (\ref{eq:e R W 1 2}) and the second rescaled vortex equation (\ref{eq:F A R}).

Consider the {\bf case $R_0=\infty$}. It follows from our standing hypothesis (H) that there exists a constant $\delta>0$ such that $G$ acts freely on 
\[K:=\{x\in M\,|\,|\mu(x)|\leq \delta\}.\] 
Properness of $\mu$ implies that $K$ is compact. 

Let $Q\sub\Si$ be a compact subset. The convergence of $u_\nu$ and the fact $\mu\circ u_0=0$ imply that for $\nu$ large enough, we have $u_\nu(Q)\sub K$. Furthermore, our hypotheses about the convergence of $A_\nu$ and $u_\nu$ imply that $\sup_\nu\Vert d_{A_\nu}u_\nu\Vert_{C^0(Q)}<\infty$. Finally, since $K$ is compact and $G$ acts freely on it, we have
\[\sup\left\{\frac{|\xi|}{|L_x\xi|}\,\bigg|\,x\in K,\,0\neq\xi\in\g\right\}<\infty.\]
Therefore, we may apply Lemma \ref{le:mu u} (Appendix \ref{sec:vort}), to conclude that 
\[\sup_QR_\nu^{2-\frac2p}|\mu\circ u_\nu|<\infty.\]
Since $p>2,R_\nu\to\infty$, and $e^\infty_{w_0}=\frac12|d_{A_0}u_0|^2$, the convergence (\ref{eq:e R nu W nu e f}) follows. This completes the proof of Lemma \ref{le:conv e}.
\end{proof}

In the proof of Proposition \ref{prop:quant en loss} we will also use the following.
\begin{rmk}\label{rmk:bar J}\rm Let $(A,u)\in\WWW^p_\C$ be an $\infty$-vortex, i.e., a solution of the equations $\bar\dd_{J,A}(u)=0$ and $\mu\circ u=0$. Then by Proposition \ref{prop:bar del J} (Appendix \ref{sec:add}) the map $Gu:\C\to\BAR M=\mu^{-1}(0)/G$ is $\bar J$-holomorphic, and $E^\infty(A,u)=E(\bar u)$. If this energy is finite, then by removal of singularities the map $\bar u$ extends to a $\bar J$-holomorphic map $\bar u:S^2\to\BAR M$.%
\footnote{See e.g.~\cite[Theorem 4.1.2]{MS04}.}
 It follows that $E^\infty(w)\geq\Emin$, provided that $E^\infty(w)>0$. $\Box$
\end{rmk}
\begin{proof}[Proof of Proposition \ref{prop:quant en loss}]\setcounter{claim}{0} We write $(A_\nu,u_\nu):=w_\nu$. Consider the function 
\[f_\nu:=|d_{A_\nu}u_\nu|+R_\nu|\mu\circ u_\nu|:\Om\to\R.\]
\begin{Claim}Suppose that the hypotheses of Proposition \ref{prop:quant en loss} are satisfied and that there exists a sequence $z_\nu\in\Om$ that converges to some $z_0\in\Om$, such that $f_\nu(z_\nu)\to\infty$. Then there exists a number
\begin{equation}
  \label{eq:r 0 limsup}0<r_0\leq\limsup_{\nu\to\infty} \frac{R_\nu}{f_\nu(z_\nu)}\,(\leq\infty)
\end{equation}
and an $r_0$-vortex $w_0\in\WWW_\C$, such that
\begin{equation}
  \label{eq:limsup E w 0}0<E^{r_0}(w_0)\leq\limsup_{\nu\to\infty} E^{R_\nu}(w_\nu,B_\eps(z_0)),
\end{equation}
for every $\eps>0$ so small that $B_\eps(z_0)\sub \Om$.
\end{Claim}
\begin{proof}[Proof of the claim] {\bf Construction of $r_0$:} We define $\de_\nu:=f_\nu(z_\nu)^{-\frac12}$. For $\nu$ large enough we have $\bar B_{2\de_\nu}(z_\nu)\sub\Om$. We pass to some subsequence such that this holds for every $\nu$. By Lemma \ref{lemma:Hofer}, applied with $(f,x,\de):=(f_\nu,z_\nu,\de_\nu)$, there exist $\ze_\nu\in B_{2\de_\nu}(z_0)$ and $\eps_\nu\leq\de_\nu$, such that 
\begin{eqnarray}
\label{eq:ze z nu}|\ze_\nu-z_\nu|&<&2\de_\nu,\\
\label{eq:sup B}\sup_{B_{\eps_\nu}(\ze_\nu)}f_\nu&\leq&2f_\nu(\ze_\nu),\\
\label{eq:eps nu}\eps_\nu f_\nu(\ze_\nu)&\geq&f_\nu(z_\nu)^{\frac12}.
\end{eqnarray}
Since by assumption $f_\nu(z_\nu)\to\infty$, it follows from (\ref{eq:ze z nu}) that the sequence $\ze_\nu$ converges to $z_0$. We define 
\begin{eqnarray*}\nn&c_\nu:=f_\nu(\ze_\nu),\quad\wt\Om_\nu:=\big\{c_\nu(z-\ze_\nu)\,\big|\,z\in\Om\big\},&\\
\nn&\phi_\nu:\wt\Om_\nu\to\Om,\quad\phi_\nu(\wt z):=c_\nu^{-1}\wt z+\ze_\nu,&\\
\nn&\wt w_\nu:=\phi_\nu^*w_\nu=(\phi_\nu^*A_\nu,u_\nu\circ\phi_\nu),\quad\wt R_\nu:=c_\nu^{-1}R_\nu.&
\end{eqnarray*}
Note that $\wt w_\nu$ is an $\wt R_\nu$-vortex. Passing to some subsequence we may assume that $\wt R_\nu$ converges to some $r_0\in[0,\infty]$. Since $\eps_\nu\leq\de_\nu=f_\nu(z_\nu)^{-\frac12}$, it follows from (\ref{eq:eps nu}) that $f_\nu(z_\nu)\leq f_\nu(\ze_\nu)$. It follows that the {\bf second inequality in (\ref{eq:r 0 limsup})} holds for the original sequence. 

{\bf Construction of $w_0$:} We choose a sequence $\Om_1\sub\Om_2\sub\ldots\sub\C$ of open sets such that $\bigcup_\nu\Om_\nu=\C$ and $\Om_\nu\sub\wt\Om_\nu$, for every $\nu\in\N$. We check the conditions of Proposition \ref{prop:cpt bdd} with these sets, $Z:=\emptyset$, and $R_\nu,w_\nu$ replaced by $\wt R_\nu,\wt w_\nu$: {\bf Condition (\ref{eq:u nu Om K})} is satisfied by hypothesis. 

We check {\bf condition (\ref{eq:sup e})}: A direct calculation involving (\ref{eq:sup B}) shows that 
\begin{equation}\label{eq:d wt A nu wt u nu}|d_{\wt A_\nu}\wt u_\nu|+\wt R_\nu|\mu\circ\wt u_\nu|=c_\nu^{-1}f_\nu\circ\phi_\nu\leq2,\quad\textrm{on }B_{\eps_\nu c_\nu}(0).\end{equation}
It follows from (\ref{eq:eps nu}) and the fact $f_\nu(z_\nu)\to\infty$, that $\eps_\nu c_\nu\to\infty$. Combining this with (\ref{eq:d wt A nu wt u nu}), condition (\ref{eq:sup e}) follows, for every compact subset $Q\sub\C$. 

Therefore, applying Proposition \ref{prop:cpt bdd}, there exists an $r_0$-vortex $w_0=(A_0,u_0)\in\WWW_\C$ and, passing to some subsequence, there exist gauge transformations $g_\nu\in W^{2,p}(\C,G)$, with the following properties. For every compact subset $Q\sub\C$, $g_\nu^*\wt A_\nu$ converges to $A_0$ in $C^0$ on $Q$, and $g_\nu^{-1}\wt u_\nu$ converges to $u_0$ in $C^1$ on $Q$. 

We prove the {\bf first inequality in (\ref{eq:limsup E w 0})}: By Lemma \ref{le:conv e} we have 
\begin{equation}
\label{eq:e R0 w0}e^{\wt R_\nu}_{\wt w_\nu}=e^{\wt R_\nu}_{g_\nu^*\wt w_\nu}\to e^{r_0}_{w_0},
\end{equation}
in $C^0(Q)$ for every compact subset $Q\sub\C$. Since
\[e^{\wt R_\nu}_{\wt w_\nu}(0)=c_\nu^{-2}e^{R_\nu}_{w_\nu}(\ze_\nu)\geq\frac12,\] 
it follows that $e^{r_0}_{w_0}(0)\geq1/2$. This implies that $E^{r_0}(w_0)>0$. This proves the first inequality in (\ref{eq:limsup E w 0}).

We prove the {\bf second inequality in (\ref{eq:limsup E w 0})}: Let $\eps>0$ be so small that $B_\eps(z_0)\sub \Om$, and $\de>0$. It follows from (\ref{eq:e R0 w0}) that $E^{r_0}(w_0)\leq\sup_\nu E^{R_\nu}(w_\nu)$. By hypothesis this supremum is finite. Therefore, there exists $R>0$ such that $E^{r_0}(w_0,\C\wo B_R)<\de.$ Since
\[E^{R_\nu}\big(w_\nu,B_{c_\nu^{-1}R}(\ze_\nu)\big)=E^{\wt R_\nu}(\wt w_\nu,B_R),\] 
the convergence (\ref{eq:e R0 w0}) implies that
\begin{equation}\label{eq:E R nu}\lim_{\nu\to\infty} E^{R_\nu}\big(w_\nu,B_{c_\nu^{-1}R}(\ze_\nu)\big)=E^{r_0}(w_0,B_R)>E^{r_0}(w_0)-\delta.
\end{equation}
On the other hand, since $c_\nu\to\infty$ and $\ze_\nu \to z_0$, for $\nu$ large enough the ball $B_{c_\nu^{-1}R}(\ze_\nu)$ is contained in $B_\eps(z_0)$. Combining this with (\ref{eq:E R nu}), we obtain 
\[\limsup_{\nu\to\infty} E^{R_\nu}(w_\nu,B_\eps(z_0))\geq E^{r_0}(w_0)-\delta.\]
Since this holds for every $\delta>0$, the second inequality in (\ref{eq:limsup E w 0}) (for the original sequence) follows.

It remains to prove the {\bf first inequality in (\ref{eq:r 0 limsup})}, i.e., that $r_0>0$. Assume by contradiction that $r_0=0$. For a map $u\in C^\infty(\C,M)$ we denote by 
\[E(u):=\frac12\int_{\C}|du|^2\]
its (Dirichlet-)energy.%
\footnote{Here the norm is taken with respect to the metric $\om(\cdot,J\cdot)$ on $M$.}
 By the second $R$-vortex equation with $R:=0$ we have $F_{A_0}=0$. Therefore, by Proposition \ref{prop:ka 0} (Appendix \ref{sec:add}) there exists $h\in C^\infty(\C,G)$ such that $h^*A_0=0$. By the first vortex equation the map $u'_0:=h^{-1}u_0:\C\to M$ is $J$-holomorphic. Let $\eps>0$ be such that $B_\eps(z_0)\sub\Om$. Using the second inequality in (\ref{eq:limsup E w 0}), we have
\[E(u'_0)=E^0(w_0)\leq\limsup_{\nu\to\infty} E^{R_\nu}(w_\nu,B_\eps(z_0)).\]
Combining this with the hypothesis $\sup_\nu E^{R_\nu}(w_\nu,\Om)<\infty$, it follows that the energy $E(u'_0)$ is finite. Hence by removal of singularities%
\footnote{see e.g.~\cite[Theorem 4.1.2]{MS04}}%
, $u'_0$ extends to a smooth $J$-holomorphic map $v:S^2\to M$. By the first inequality in (\ref{eq:limsup E w 0}) we have
\[\int_{S^2}v^*\om=E(v)=E^0(w_0)>0.\] 
This contradicts asphericity of $(M,\om)$. Hence $r_0$ must be positive. This concludes the proof of the claim.  
\end{proof}
{\bf Statement (\ref{prop:hard R e})} of Proposition \ref{prop:quant en loss} follows from the claim, considering a sequence $z_\nu\in Q$, such that $f_\nu(z_\nu)=\Vert f_\nu\Vert_{C^0(Q)}$, and using (\ref{eq:r 0 limsup}).

We prove {\bf statement (\ref{prop:hard limsup})}. Assume that there exists a compact subset $Q\sub \Om$ such that $\sup_\nu||e^{R_\nu}_{w_\nu}||_{C^0(Q)}=\infty$. Let $z_\nu\in Q$ be such that $f_\nu(z_\nu)\to\infty$. We choose a pair $(r_0,w_0)$ as in the claim. Using the first inequality in (\ref{eq:limsup E w 0}) and Remark \ref{rmk:bar J} (in the case $r_0=\infty$), we have $E^{r_0}(w_0)\geq\Emin$. Combining this with the second inequality in (\ref{eq:limsup E w 0}), inequality (\ref{eq:limsup Emin}) follows. This proves (\ref{prop:hard limsup}) and concludes the proof of Proposition \ref{prop:quant en loss}.
\end{proof}
We are now ready to prove Proposition \ref{prop:cpt mod} (p.~\pageref{prop:cpt mod}).
\begin{proof}[Proof of Proposition \ref{prop:cpt mod}]\setcounter{claim}{0}\label{prop:cpt mod proof} We abbreviate $e_\nu:=e_{w_\nu}^{R_\nu}$. 

\begin{Claim}For every $\ell\in\N_0$ there exists a finite subset $Z_\ell\sub\C$ such that the following holds. If $R_0<\infty$ then we have $Z_\ell=\emptyset$. Furthermore, if $|Z_\ell|<\ell$ then we have 
\begin{equation}
\label{eq:sup Vert e nu}\sup_{\nu\in\N}\big\{\Vert e_\nu\Vert_{C^0(Q)}\,\big|\,Q\sub B_{r_\nu}\big\}<\infty,
\end{equation}
for every compact subset $Q\sub \C\wo Z_\ell$. Moreover, for every $z_0\in Z_\ell$ and every $\eps>0$ the inequality (\ref{eq:limsup Emin}) holds. 
\end{Claim}
\begin{proof}[Proof of the claim] For $\ell=0$ the assertion holds with $Z_0:=\emptyset$. We fix $\ell\geq1$ and assume by induction that there exists a finite subset $Z_{\ell-1}\sub\C$ such that the assertion with $\ell$ replaced by $\ell-1$ holds. If (\ref{eq:sup Vert e nu}) is satisfied for every compact subset $Q\sub \C\wo Z_{\ell-1}$, then the statement for $\ell$ holds with $Z_\ell:=Z_{\ell-1}$. 

Assume that there exists a compact subset $Q\sub \C\wo Z_{\ell-1}$, such that (\ref{eq:sup Vert e nu}) does not hold. It follows from the induction hypothesis that 
\begin{equation}
  \label{eq:Z ell 1}
|Z_{\ell-1}|\geq\ell-1.  
\end{equation}
By statement (\ref{prop:hard limsup}) of Proposition \ref{prop:quant en loss} there exists a point $z_0\in Q$ such that inequality (\ref{eq:limsup Emin}) holds, for every $\eps>0$. We set $Z_\ell:=Z_{\ell-1}\cup \{z_0\}$. 

It follows from the fact that (\ref{eq:sup Vert e nu}) does not hold and statement (\ref{prop:hard R e}) of Proposition \ref{prop:quant en loss} that $R_0=\lim_{\nu\to\infty}R_\nu=\infty$. Furthermore, since $z_0\in Q\sub \C\wo Z_{\ell-1}$, (\ref{eq:Z ell 1}) implies that $|Z_\ell|\geq\ell$. It follows that the statement of the claim for $\ell$ is satisfied. By induction, the claim follows.
\end{proof}
We fix an integer
\[\ell>\frac{\sup_\nu E^{R_\nu}(w_\nu,B_{r_\nu})}{\Emin}\] 
and a finite subset $Z:=Z_\ell\sub\C$ that satisfies the conditions of the claim. It follows from the inequality (\ref{eq:limsup Emin}) that $\ell>|Z|$. Hence by the statement of the claim, the hypothesis (\ref{eq:sup e}) of Proposition \ref{prop:cpt bdd} is satisfied with $\Om_\nu:=B_{r_\nu}\wo Z$. Applying that result and passing to some subsequence, there exist an $R_0$-vortex $w_0\in\WWW_{\C\wo Z}$ and gauge transformations $g_\nu\in W^{2,p}_\loc(\C\wo Z,G)$, such that the {\bf statements (\ref{prop:cpt mod:<},\ref{prop:cpt mod:=})} of Proposition \ref{prop:cpt mod} are satisfied. (Here we use that $Z=\emptyset$ if $R_0<\infty$.)

We prove {\bf statement (\ref{prop:cpt mod lim nu E eps})}. Passing to some ``diagonal'' subsequence, the limit $\lim_{\nu\to\infty}E^{R_\nu}\big(w_\nu,B_{1/i}(z)\big)$ exists, for every $i\in\N$ and $z\in Z$. Let now $z\in Z$ and $\eps>0$. We choose $i\in\N$ bigger than $\eps^{-1}$. For $0<r<R$ we denote 
\[A(z,r,R):=\bar B_R(z)\wo B_r(z).\] 
By Lemma \ref{le:conv e} the limit
\[\lim_{\nu\to\infty}E^{R_\nu}\big(w_\nu,A(z,1/i,\eps)\big)\]
exists and equals $E^{R_0}(w_0,A(z,1/i,\eps))$. It follows that the limit
\[E_z(\eps):=\lim_{\nu\to\infty}E^{R_\nu}(w_\nu,B_\eps(z))\] 
exists. Inequality (\ref{eq:limsup Emin}) implies that $E_z(\eps)\geq\Emin$. Since $E^{R_0}(w_0,A(z,1/i,\eps))$ depends continuously on $\eps$, the same holds for $E_z(\eps)$. This proves statement (\ref{prop:cpt mod lim nu E eps}) and completes the proof of Proposition \ref{prop:cpt mod}.
\end{proof}
\begin{Rmk} In the above proof the set of bubbling points $Z$ is constructed by ``terminating induction''. Intuitively, this is induction over the number of bubbling points. The ``auxiliary index'' $\ell$ in the claim is needed to make this idea precise. Inequality (\ref{eq:limsup Emin}) ensures that the ``induction stops''. $\Box$
\end{Rmk} 
\section{Soft rescaling}\label{sec:soft}
The following proposition will be used inductively in the proof of Theorem \ref{thm:bubb} to find the next bubble in the bubbling tree, at a bubbling point of a given sequence of rescaled vortex classes. It is an adaption of \cite[Proposition 4.7.1.]{MS04} to vortices.
\begin{prop}[Soft rescaling]\label{prop:soft} Assume that $(M,\om)$ is aspherical. Let $r>0$, $z_0\in\C$, $R_\nu>0$ be a sequence that converges to $\infty$, $p>2$, and for every $\nu\in\N$ let $w_\nu:=(A_\nu,u_\nu)\in\WWW^p_{B_r(z_0)}$ be an $R_\nu$-vortex, such that the following conditions are satisfied.
\begin{enua}\item\label{prop:soft K} There exists a compact subset $K\sub M$ such that $u_\nu(B_r(z_0))\sub K$ for every $\nu$.
\item \label{prop:soft E} For every $0<\eps\leq r$ the limit
\begin{equation}\label{eq:E eps lim}E(\eps):=\lim_{\nu\to\infty} E^{R_\nu}(w_\nu,B_\eps(z_0))\end{equation}
exists and $\Emin\leq E(\eps)<\infty$. Furthermore, the function 
\begin{equation}\label{eq:eps E eps}(0,r]\ni\eps\mapsto E(\eps)\in\R
\end{equation} 
is continuous.
\end{enua}
Then there exist $R_0\in\{1,\infty\}$, a finite subset $Z\sub\C$, an $R_0$-vortex $w_0:=(A_0,u_0)\in\WWW_{\C\wo Z}$, and, passing to some subsequence, there exist sequences $\eps_\nu>0$, $z_\nu\in \C$, and $g_\nu\in W^{2,p}_\loc(\C\wo Z,G)$, such that, defining
\[\phi_\nu:\C\to\C,\quad\phi_\nu(\wt z):=\eps_\nu\wt z+z_\nu,\]
the following conditions hold.
\begin{enui}
\item\label{prop:soft Z} If $R_0=1$ then $Z=\emptyset$ and $E(w_0)>0$. If $R_0=\infty$ and $E^\infty(w_0)=0$ then $|Z|\geq2$. 
\item\label{prop:soft eps} The sequence $z_\nu$ converges to $z_0$. Furthermore, if $R_0=1$ then $\eps_\nu= R_\nu^{-1}$ for every $\nu$, and if $R_0=\infty$ then $\eps_\nu$ converges to 0 and $\eps_\nu R_\nu$ converges to $\infty$. 
\item\label{prop:soft conv} If $R_0=1$ then the sequence $g_\nu^*\phi_\nu^*w_\nu$ converges to $w_0$ in $C^\infty$ on every compact subset of $\C\wo Z$. Furthermore, if $R_0=\infty$ then on every compact subset of $\C\wo Z$, the sequence $g_\nu^*\phi_\nu^*A_\nu$ converges to $A_0$ in $C^0$, and the sequence $g_\nu^{-1}(u_\nu\circ\phi_\nu)$ converges to $u_0$ in $C^1$. 
\item\label{prop:soft lim nu E eps} Fix $z\in Z$ and a number $\eps_0>0$ such that $B_{\eps_0}(z)\cap Z=\{z\}$. Then for every $0<\eps<\eps_0$ the limit 
\[E_z(\eps):=\lim_{\nu\to\infty}E^{\eps_\nu R_\nu}\big(\phi_\nu^*w_\nu,B_\eps(z)\big)\]
exists and $\Emin\leq E_z(\eps)<\infty$. Furthermore, the function $(0,\eps_0)\ni\eps\mapsto E_z(\eps)\in\R$ is continuous.
\item \label{prop:soft en} We have
\begin{equation}
\label{eq:en cons}\lim_{R\to\infty}\limsup_{\nu\to\infty} E^{R_\nu}\big(w_\nu,B_{R^{-1}}(z_0)\wo B_{R\eps_\nu}(z_\nu)\big)=0.
\end{equation}
\end{enui}
\end{prop}
Roughly speaking, this result states that given a sequence of ``zoomed out'' vortices for which a positive amount of energy is concentrated around some point $z_0$, after ``zooming back in'', some subsequence converges either to a vortex over $\C$ or a holomorphic sphere in the symplectic quotient, up to bubbling at finitely many points. Furthermore, no energy is lost in the limit between the original sequence and the ``zoomed in'' subsequence. 

To explain this, note that condition (\ref{prop:soft E}) in the hypothesis implies that given an arbitrarily small ball around the point $z_0$, the energy of the vortex $w_\nu$ on the ball is bounded below by a positive constant arbitrarily close to $\Emin$, provided that $\nu$ is large enough. This means that in the limit $\nu\to\infty$, a positive amount of energy bubbles off around the point $z_0$. It follows that there exists a sequence of points $z_\nu\in\C$ converging to $z_0$, such that the energy density of $w_\nu$ at $z_\nu$ converges to $\infty$.

In the conclusion of the proposition, condition (\ref{prop:soft eps}) says that we are ``zooming in'' around $z_0$, and it specifies how fast we do so. Condition (\ref{prop:soft conv}) states that the ``zoomed in'' subsequence converges to either a vortex over $\C$ or a holomorphic sphere in $\BAR M$, depending on how fast we ``zoom back in''. Condition (\ref{prop:soft Z}) implies that the limit object is nontrivial or there are at least two new bubbling points. (The latter can only happen if the limit is a holomorphic sphere in $\BAR M$.) In the proof of Theorem \ref{thm:bubb}, this will guarantee that the new bubble is stable. 

Analogously to condition (\ref{prop:soft E}), condition (\ref{prop:soft lim nu E eps}) implies that at each point in $Z$ in the limit $\nu\to\infty$, at least the energy $\Emin$ bubbles off in the sequence of rescaled vortices $\phi_\nu^*w_\nu$. In the proof of Theorem \ref{thm:bubb}, this will be used to prove that the inductive construction of the bubble tree terminates after finitely many steps. Finally, condition (\ref{prop:soft en}) will ensure that at each step no energy is lost between the old and new bubble. 

\begin{Rmk} In condition (\ref{prop:soft conv}) the pullback $\phi_\nu^*w_\nu$ is defined over the set $\phi_\nu^{-1}(B_r(z_0))$. This set contains any given compact subset $Q\sub\C\wo Z$, provided that $\nu$ is large enough (depending on $Q$). Hence condition (\ref{prop:soft conv}) makes sense. $\Box$
\end{Rmk} 
The proof of Proposition \ref{prop:soft} is given on page \pageref{proof prop:soft}. It is based on the following result, which states that the energy of a vortex over an annulus is concentrated near the ends, provided that it is small enough. For $0\leq r,R\leq \infty$ we denote the open annulus around 0 with radii $r,R$ by 
\[A(r,R):=B_R\wo \bar B_r.\]
Note that $A(r,\infty)=\C\wo\bar B_r$, and $A(r,R)=\emptyset$ in the case $r\geq R$. We define
\[d:M\x M\to [0,\infty]\]
to be the distance function induced by the Riemannian metric $\om(\cdot,J\cdot)$.%
\footnote{If $M$ is disconnected then $d$ attains the value $\infty$.}
 We define 
\begin{equation}
  \label{eq:bar d}\bar d:M/G\x M/G\to [0,\infty],\qquad \bar d(\bar x,\bar y):=\min_{x\in\bar x,\,y\in\bar y}d(x,y).
\end{equation}
By Lemma \ref{le:metr} in Appendix \ref{sec:add} this is a distance function on $M/G$ which induces the quotient topology. 
\begin{prop}[Energy concentration near ends] \label{prop:en conc} There exists a constant $r_0>0$ such that for every compact subset $K\sub M$ and every $\eps>0$ there exists a constant $E_0$, such that the following holds. Assume that $r\geq r_0,R\leq\infty$, $p>2$, and $w:=(A,u)\in\WWW^p_{A(r,R)}$ is a vortex (with respect to $(\om_0,i)$), such that 
\begin{eqnarray}\nn&u(A(r,R))\sub K,&\\
\label{eq:E w A E 0}&E(w)=E\big(w,A(r,R)\big)\leq E_0.&
\end{eqnarray} 
Then we have%
\footnote{Note that for $a>\sqrt{R/r}$ we have $A(ar,a^{-1}R)=\emptyset$, hence (\ref{eq:E w A ar},\ref{eq:sup z z' bar d}) are non-trivial only for $a\leq\sqrt{R/r}$.}%
\begin{eqnarray}
\label{eq:E w A ar}&E\big(w,A(ar,a^{-1}R)\big)\leq4a^{-2+\eps}E(w),\quad\forall a\geq2,&\\
\label{eq:sup z z' bar d}&\sup_{z,z'\in A(ar,a^{-1}R)}\bar d(Gu(z),Gu(z'))\leq100a^{-1+\eps}\sqrt{E(w)},\,\,\forall a\geq4.&
\end{eqnarray}
(Here $Gx\in M/G$ denotes the orbit of a point $x\in M$.)
\end{prop}
The proof of this result is modeled on the proof of \cite[Theorem 1.3]{ZiA}, which in turn is based on the proof of \cite[Proposition 11.1]{GS}. It is based on an isoperimetric inequality for the invariant symplectic action functional (Theorem \ref{thm:isoperi} in Appendix \ref{sec:action}). It also relies on an identity relating the energy of a vortex over a compact cylinder with the actions of its end-loops (Proposition \ref{prop:en act} in Appendix \ref{sec:action}). The proof of (\ref{eq:sup z z' bar d}) also uses the following remark.
\begin{rmk}\rm\label{rmk:BAR d BAR x 0 BAR x 1} Let $\big(M,\lan\cdot,\cdot\ran_M\big)$ be a Riemannian manifold, $G$ a compact Lie group that acts on $M$ by isometries, $P$ a $G$-bundle over $[0,1]$
\footnote{Such a bundle is trivializable, but we do not fix a trivialization here.}
, $A\in\A(P)$ a connection, and $u\in C^\infty_G(P,M)$ a map. We define 
\[\ell(A,u):=\int_0^1|d_Au|dt,\]
where $d_Au=du+L_uA$, and the norm is taken with respect to the standard metric on $[0,1]$ and $\lan\cdot,\cdot\ran_M$. Furthermore, we define
\[\bar u:[0,1]\to M/G,\quad\bar u(t):=Gu(p),\]
where $p\in P$ is any point over $t$. We denote by $d$ the distance function induced by $\lan\cdot,\cdot\ran_M$, and define $\bar d$ as in (\ref{eq:bar d}). Then for every pair of points $\bar x_0,\bar x_1\in M/G$, we have
\[\bar d(\bar x_0,\bar x_1)\leq\inf\big\{\ell(A,u)\,\big|\,(P,A,u)\textrm{ as above: }\bar u(i)=\bar x_i,\,i=0,1\big\}.\]
This follows from a straight-forward argument. $\Box$
\end{rmk}
\begin{proof}[Proof of Proposition \ref{prop:en conc}]\setcounter{claim}{0} For every subset $X\sub M$ we define 
\[m_X:=\inf\big\{|L_x\xi|\,\big|\,x\in X,\,\xi\in\g:\,|\xi|=1\big\},\]
where the norms are with respect to $\om(\cdot,J\cdot)$ and $\lan\cdot,\cdot\ran_\g$. We set
\begin{equation}\label{eq:R 0 2 pi m}r_0:=m_{\mu^{-1}(0)}^{-1}.
\end{equation}
Let $K\sub M$ be a compact subset and $\eps>0$. Replacing $K$ by $GK$, we may assume w.l.o.g.~that $K$ is $G$-invariant. An elementary argument using our standing hypothesis (H) shows that there exists a number $\de_0>0$ such that $G$ acts freely on $K':=\mu^{-1}(\bar B_{\de_0})$, and 
\begin{equation}\label{eq:m K 1 eps}m_{K'}\geq\sqrt{1-\frac\eps2}m_{\mu^{-1}(0)}.
\end{equation}
We choose a constant $\de$ as in Theorem \ref{thm:isoperi} in Appendix \ref{sec:action} (Isoperimetric inequality), corresponding to
\[\lan\cdot,\cdot\ran_M:=\om(\cdot,J\cdot),\quad K',\quad c:=\frac1{2-\eps}.\]
Shrinking $\de$ we may assume that it satisfies the condition of Proposition \ref{prop:en act} in Appendix \ref{sec:action} (Energy action identity) for $K'$. We choose a constant $\wt E_0>0$ as in Lemma \ref{le:a priori} in Appendix \ref{sec:vort} (called $E_0$ there), corresponding to $K$. We define 
\begin{equation}\label{eq:E 0 de 0 de}E_0:=\min\left\{\wt E_0,\frac\pi{32}r_0^2\de_0^2,\frac{\de^2}{128\pi}\right\}.
\end{equation}
Assume that $r,R,p,w$ are as in the hypothesis. Without loss of generality, we may assume that $r<R$. 

Consider first the {\bf case $R<\infty$}, and assume that $w$ extends to a smooth vortex over the compact annulus of radii $r$ and $R$. We show that {\bf inequality (\ref{eq:E w A ar})} holds. We define the function 
\begin{equation}\label{eq:E s}E:[0,\infty)\to[0,\infty),\quad E(s):=E\big(w,A(re^s,Re^{-s})\big).
\end{equation}
\begin{Claim}For every $s\in\big[\log2,\log(R/r)/2\big)$ we have
\begin{equation}\label{eq:d ds E}\frac d{ds}E(s)\leq-(2-\eps)E(s).
\end{equation}
\end{Claim}
\begin{proof}[Proof of the claim] Using the fact $r\geq r_0$ and (\ref{eq:E w A E 0},\ref{eq:E 0 de 0 de}), it follows from Lemma \ref{le:a priori} in Appendix \ref{sec:vort} (with $r$ replaced by $|z|/2$) that
\begin{equation}\label{eq:e w z 0 8}e_w(z)\leq\min\left\{\de_0^2,\frac{\de^2}{4\pi^2|z|^2}\right\},\quad\forall z\in A(2r,R/2).\end{equation}
We define
\begin{eqnarray*}&\Si_s:=\big(s+\log r,-s+\log R\big)\x S^1,\,\forall s\in\R,&\\
&\phi:\Si_0\to\C,\,\phi(z):=e^z,\quad\wt w:=(\wt A,\wt u):=\phi^*w.&
\end{eqnarray*}
(Here we identify $\Si_0\iso\C/\sim$, where $z\sim z+2\pi in$, for every $n\in\Z$.) Let $s_0\in\big[\log(2r),\log(R/2)\big]$. Combining (\ref{eq:e w z 0 8}) with the fact $|\mu\circ u|\leq\sqrt{e_w}$ and Remark \ref{rmk:BAR d BAR x 0 BAR x 1}, it follows that
\begin{equation}\label{eq:wt u bar ell}\wt u(s_0,t)\in K'=\mu^{-1}(\bar B_{\de_0}),\,\forall t\in S^1,\quad\bar\ell(G\wt u(s_0,\cdot))\leq\de.
\end{equation}
Hence the hypotheses of Theorem \ref{thm:isoperi} (Appendix \ref{sec:action}) are satisfied with $K$ replaced by $K'$ and $c:=1/(2-\eps)$. By the statement of that result the loop $\wt u(s_0,\cdot)$ is admissible, and defining
\[\iota_{s_0}:S^1\to\Si_0,\quad\iota_{s_0}(t):=(s_0,t),\]
we have
\begin{equation}\label{eq:A iota s 0}\big|\A\big(\iota_{s_0}^*\wt w\big)\big|\leq\frac1{2-\eps}\Vert \iota_{s_0}^*d_{\wt A}\wt u\Vert_2^2+\frac1{2m_{K'}^2}\big\Vert\mu\circ\wt u\circ\iota_{s_0}\big\Vert_2^2.
\end{equation}
Here $\A$ denotes the invariant symplectic action, as defined in Appendix \ref{sec:action}. Furthermore, the $L^2$-norms are with respect to the standard metric on $S^1\iso\R/(2\pi\Z)$, the metric $\om(\cdot,J\cdot)$ on $M$, and the operator norm $|\cdot|_\op:\g^*\to\R$, induced by $\lan\cdot,\cdot\ran_\g$. 

By (\ref{eq:R 0 2 pi m},\ref{eq:m K 1 eps}) and the fact $2r\leq e^{s_0}$, we have
\begin{equation}\label{eq:frac 1 2 2 eps}\frac1{2-\eps}|\iota_{s_0}^*d_{\wt A}\wt u|_0^2(t)+\frac1{2m_{K'}^2}\big|\mu\circ\wt u\circ\iota_{s_0}\big|^2(t)\leq \frac1{2-\eps}e^{2s_0}e_w(e^{s_0+it}),\quad\forall t\in S^1.\end{equation}
Here the norm $|\cdot|_0$ is with respect to the standard metric on $S^1\iso\R/(2\pi\Z)$, and we used the fact that for every $\phi\in\g^*$,
\[|\phi|_\op\leq|\phi|:=\sqrt{\phi(\xi)}\]
where $\xi\in\g$ is determined by $\lan\xi,\cdot\ran_\g=\phi$. We fix $s\in\big[\log2,\log(R/r)/2\big)$. Recalling (\ref{eq:E s}), we have
\[E(s)=\int_{\Si_s}e^{2s_0}e_w(e^{s_0+it})dt\,ds_0.\]
Combining this with (\ref{eq:A iota s 0},\ref{eq:frac 1 2 2 eps}), it follows that
\begin{equation}\label{eq:A iota}-\A\big(\iota_{-s+\log R}^*\wt w\big)+\A\big(\iota_{s+\log r}^*\wt w\big)\leq-\frac1{2-\eps}\frac d{ds}E(s).\end{equation}
Using (\ref{eq:wt u bar ell}), the hypotheses of Proposition \ref{prop:en act} are satisfied with $K$ replaced by $K'$. Applying that result, we have
\[E(s)=-\A\big(\iota_{-s+\log R}^*\wt w\big)+\A\big(\iota_{s+\log r}^*\wt w\big).\]
Combining this with (\ref{eq:A iota}), inequality (\ref{eq:d ds E}) follows. This proves the claim.
\end{proof}
By the claim the derivative of the function
\[\left[\log2,\frac{\log\left(\frac Rr\right)}2\right)\ni s\mapsto E(s)e^{(2-\eps)s}\]
is non-positive, and hence this function is non-increasing. Inequality (\ref{eq:E w A ar}) follows.

{\bf We prove (\ref{eq:sup z z' bar d})}. Let $z\in A(4r,\sqrt{rR})$. Using 
(\ref{eq:E w A E 0}) and the fact $E_0\leq\wt E_0$, it follows from Lemma \ref{le:a priori} (with $r$ replaced by $|z|/2$) that 
\begin{equation}\label{eq:e w z}e_w(z)\leq\frac{32}{\pi|z|^2}E\left(w,B_{\frac{|z|}2}(z)\right).
\end{equation} 
We define $a:=|z|/(2r)$. Then $a\geq2$ and $B_{|z|/2}(z)$ is contained in $A(ar,a^{-1}R)$. Therefore, by (\ref{eq:E w A ar}) we have 
\[E\big(w,B_{\frac{|z|}2}(z)\big)\leq16r^{2-\eps}|z|^{-2+\eps}E(w).\]
Combining this with (\ref{eq:e w z}) and the fact $|d_Au|(z)\leq\sqrt{2e_w(z)}$, it follows that 
\begin{equation}\label{eq:d A u C 2r}|d_Au(z)v|\leq Cr^{1-\frac\eps2}|z|^{-2+\frac\eps2}\sqrt{E(w)}|v|,\quad\forall z\in A(4r,\sqrt{rR}),\,v\in\C.
\end{equation}
where $C:=32/\sqrt\pi$. A similar argument shows that 
\begin{equation}\label{eq:d A u C R}|d_Au(z)v|\leq CR^{-1+\frac\eps2}|z|^{-\frac\eps2}\sqrt{E(w)}|v|,\quad\forall z\in A(\sqrt{rR},R/4).
\end{equation}
Let now $a\geq4$ and $z,z'\in A(ar,a^{-1}R)$. Assume that $\eps\leq1$. (This is no real restriction.) We define $\ga:[0,1]\to\C$ to be the radial path of constant speed, such that $\ga(0)=z$ and $|\ga(1)|=|z'|$. Furthermore, we choose an angular path $\ga':[0,1]\to\C$ of constant speed, such that $\ga'(0)=\ga(1)$, $\ga'(1)=z'$, and $\ga'$ has minimal length among such paths. (See Figure \ref{fig:path}.)
\begin{figure}
\centering
\leavevmode\epsfbox{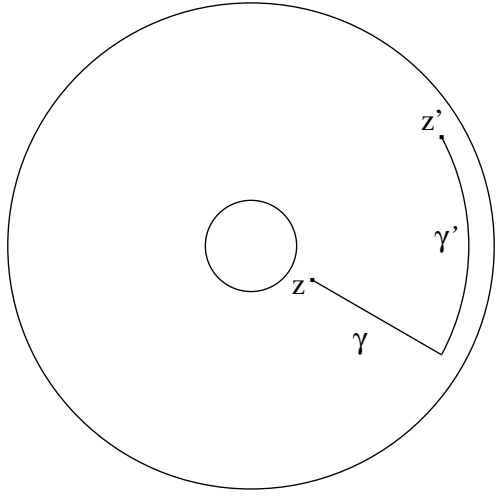}
\caption{The paths $\ga$ and $\ga'$ described in the proof of inequality (\ref{eq:sup z z' bar d}) (Proposition \ref{prop:en conc}).}
\label{fig:path} % magnification factor = 50 %
\end{figure}

Consider the ``twisted length'' of $\ga^*(A,u)$, given by
\[\int_0^1\big|d_Au\,\dot\ga(t)\big|dt.\]
It follows from (\ref{eq:d A u C 2r},\ref{eq:d A u C R}) and the fact $\eps\leq1$, that this length is bounded above by $4C\sqrt{E(w)}a^{-1+\eps/2}$. Similarly, it follows that the ``twisted length'' of ${\ga'}^*(A,u)$ is bounded above by $C\pi\sqrt{E(w)}a^{-1+\eps/2}$. Therefore, using Remark \ref{rmk:BAR d BAR x 0 BAR x 1}, inequality (\ref{eq:sup z z' bar d}) with $\eps$ replaced by $\eps/2$ follows. 

Assume now that $w$ is not smooth. By Theorem \ref{thm:reg gauge bdd} in Appendix \ref{sec:vort} the restriction of $w$ to any compact cylinder contained in $A(r,R)$ is gauge equivalent to a smooth vortex. Hence the inequalities (\ref{eq:E w A ar},\ref{eq:sup z z' bar d}) follow from what we just proved, using the $G$-invariance of $K$. 

In the {\bf case $R=\infty$} the inequalities (\ref{eq:E w A ar},\ref{eq:sup z z' bar d}) follow from what we just proved, by considering an arbitrarily large radius $R<\infty$. This completes the proof of Proposition \ref{prop:en conc}.
\end{proof}
We are now ready for the proof of Proposition \ref{prop:soft} (p.~\pageref{prop:soft}).
\begin{proof}[Proof of Proposition \ref{prop:soft}]\label{proof prop:soft}\setcounter{claim}{0} The function $E$ as in (\ref{eq:eps E eps}) is increasing and, by hypothesis (\ref{prop:soft E}), bounded below by $\Emin$. Hence the limit 
\begin{equation}
  \label{eq:m 0 lim}m_0:=\lim_{\eps\to0}E(\eps)
\end{equation}
exists and is bounded below by $\Emin$. We fix a compact subset $K\sub M$ as in hypothesis (\ref{prop:soft K}). We choose a constant $E_0>0$ as in Lemma \ref{le:a priori} in Appendix \ref{sec:vort} (Bound on energy density), depending on $K$. Considering the pullback of $w_\nu$ by the translation by $z_0$, we may assume w.l.o.g.~that $z_0=0$.
\begin{claim}\label{claim:soft wlog} We may assume w.l.o.g.~that 
\begin{equation}
  \label{eq:e C 0}\Vert e_{w_\nu}^{R_\nu}\Vert_{C^0(\bar B_r)}=e_{w_\nu}^{R_\nu}(0).
\end{equation}
\end{claim}
\begin{proof} [Proof of Claim \ref{claim:soft wlog}] Suppose that we have already proved the proposition under this additional assumption, and let $r,z_0=0,R_\nu,w_\nu$ be as in the hypotheses of the proposition. We choose $0<\hhat r\leq r/4$ so small that 
 \begin{equation}
   \label{eq:E hhat r}E(4\hhat r)=\lim_{\nu\to\infty}E^{R_\nu}(w_\nu,B_{4\hhat r})<m_0+E_0.
 \end{equation}
For $\nu\in\N$ we choose $\wt z_\nu\in\bar B_{2\hhat r}$ such that 
\begin{equation}
    \label{eq:e R ze}e^{R_\nu}_{w_\nu}(\wt z_\nu)=\Vert e^{R_\nu}_{w_\nu}\Vert_{C^0(\bar B_{2\hhat r})}.
\end{equation}
\begin{claim}\label{claim:z nu} The sequence $\wt z_\nu$ converges to $0$.
\end{claim}
\begin{proof}[Proof of Claim \ref{claim:z nu}] Recall that $A(r,R)$ denotes the open annulus of radii $r$ and $R$. Let $0<\eps\leq 2\hhat r$. Inequality (\ref{eq:E hhat r}) implies that there exists $\nu(\eps)\in\N$ such that 
\[E^{R_\nu}\big(w_\nu,A(\eps/2,4\hhat r)\big)<E_0,\]
for every $\nu\geq\nu(\eps)$. Hence by Lemma \ref{le:a priori} with $r=\eps/2$ we have
\begin{equation}\label{eq:e 4 C}e^{R_\nu}_{w_\nu}(z)<\frac{32E_0}{\pi\eps^2},\quad\forall\nu\geq\nu(\eps),\,\forall z\in A(\eps,2\hhat r).
\end{equation}
We define
\[\de_0:=\min\left\{2\hhat r,\eps\sqrt{\frac{m_0}{64E_0}}\right\}.\]
Increasing $\nu(\eps)$, we may assume that $E^{R_\nu}(w_\nu, B_{\de_0})>m_0/2$ for every $\nu\geq\nu(\eps)$, and therefore
\[\Vert e^{R_\nu}_{w_\nu}\Vert_{C^0(\bar B_{\de_0})}>\frac{32E_0}{\pi\eps^2}.\]
Combining this with (\ref{eq:e R ze},\ref{eq:e 4 C}) and the fact $\de_0\leq2\hhat r$, it follows that $\wt z_\nu\in B_\eps$, for every $\nu\geq\nu(\eps)$. This proves Claim \ref{claim:z nu}. 
\end{proof}
By Claim \ref{claim:z nu} we may pass to some subsequence such that $|\wt z_\nu|<\hhat r$ for every $\nu$. We define 
\[\psi_\nu:B_{\hhat r}\to\C,\quad\psi_\nu(z):=z+z_\nu,\quad\wt w_\nu:=(\wt A,\wt u):=\psi_\nu^*w_\nu.\]
Then (\ref{eq:e C 0}) with $w_\nu,r$ replaced by $\wt w_\nu,\hhat r$ is satisfied. By elementary arguments the hypotheses of Proposition \ref{prop:soft} are satisfied with $(w_\nu, r, z_0)$ replaced by $(\wt w_\nu,\hhat r,0)$. Assuming that we have already proved the statement of the proposition for $\wt w_\nu$, a straight-forward argument using Claim \ref{claim:z nu} shows that it also holds for $w_\nu$. This proves Claim \ref{claim:soft wlog}.
\end{proof}
{\bf So we assume w.l.o.g.~that (\ref{eq:e C 0}) holds.}\\

{\bf Construction of $R_0,Z,$ and $w_0$:} Recall the definition (\ref{eq:m 0 lim}) of $m_0$, and that we have chosen $E_0>0$ as in Lemma \ref{le:a priori}. We choose constants $r_0$ and $E_1$ as in Proposition \ref{prop:en conc}, the latter (called $E_0$ there) corresponding to the compact set $K$ and $\eps:=1$. We fix a constant 
\begin{equation}\label{eq:de min}0<\de<\min\left\{m_0,\frac{E_0}2,\frac{E_1}2\right\}.
\end{equation} 
We pass to some subsequence such that 
\begin{equation}\label{eq:E R nu w nu >}E^{R_\nu}(w_\nu,B_r(0))> m_0-\de,\quad\forall\nu\in\N.
\end{equation}
For every $\nu\in\N$, there exists $0<\hhat\eps_\nu<r$, such that
\begin{equation}
  \label{eq:E R =}E^{R_\nu}(w_\nu,B_{\hhat\eps_\nu})=m_0-\de.  
\end{equation}
It follows from the definition of $m_0$ that 
\begin{equation}
  \label{eq:eps nu 0}\hhat\eps_\nu\to0.%
\footnote{Otherwise, there exists a constant $\eps>0$, such that passing to some subsequence, we have $\hhat\eps_\nu\geq\eps$. Combining this with (\ref{eq:E R =}) and considering the limit $\nu\to\infty$, we obtain a contradiction to (\ref{eq:m 0 lim}).}%
\end{equation}
\begin{claim}\label{claim:inf nu hhat eps nu} We have
\begin{equation}
  \label{eq:hhat eps nu R nu}\inf_\nu \hhat\eps_\nu R_\nu>0.
\end{equation}
\end{claim}
\begin{proof}[Proof of Claim \ref{claim:inf nu hhat eps nu}] 
Equality (\ref{eq:e C 0}) implies that
\begin{equation}
  \label{eq:E eps nu C}E^{R_\nu}(w_\nu,B_{\hhat\eps_\nu})\leq \pi\hhat\eps_\nu^2e_{w_\nu}^{R_\nu}(0).
\end{equation}
The hypotheses $R_\nu\to\infty$, (\ref{prop:soft K}), and (\ref{prop:soft E}) imply that the hypotheses of Proposition \ref{prop:quant en loss} (Quantization of energy loss) are satisfied with $\Om:=B_r$. Thus by assertion (\ref{prop:hard R e}) of that proposition with $Q:=\{0\}$, we have 
\begin{equation}\nn\inf_\nu \frac{R_\nu^2}{e_{w_\nu}^{R_\nu}(0)}>0.
\end{equation}
Combining this with (\ref{eq:E eps nu C},\ref{eq:E R =}) and the fact $\de<m_0$, inequality (\ref{eq:hhat eps nu R nu}) follows. This proves Claim \ref{claim:inf nu hhat eps nu}.
\end{proof}
Passing to some subsequence, we may assume that the limit 
\begin{equation}
  \label{eq:hhat eps R}\hhat R_0:=\lim_{\nu\to\infty}\hhat \eps_\nu R_\nu\in[0,\infty] 
\end{equation}
exists. By Claim \ref{claim:inf nu hhat eps nu} we have $\hhat R_0>0$. We define
\begin{equation}\label{eq:R 0 eps nu}(R_0,\eps_\nu):=\left\{
  \begin{array}{ll}(\infty,\hhat\eps_\nu),&\textrm{if }\hhat R_0=\infty,\\
(1,R_\nu^{-1}),&\textrm{otherwise,}
  \end{array}
\right.  
\end{equation}
\[\wt R_\nu:=\eps_\nu R_\nu,\quad\phi_\nu:B_{\eps_\nu^{-1}r}\to B_r,\,\phi_\nu(z):=\eps_\nu z,\quad\wt w_\nu:=(\wt A_\nu,\wt u_\nu):=\phi_\nu^*w_\nu.\]
By Proposition \ref{prop:cpt mod} with $R_\nu,$ $w_\nu$ replaced by $\wt R_\nu$, $\wt w_\nu$ and $r_\nu:=r/\eps_\nu$ there exist a finite subset $Z\sub\C$ and an $R_0$-vortex $w_0=(A_0,u_0)\in\WWW_{\C\wo Z}$, and passing to some subsequence, there exist gauge transformations $g_\nu\in W^{2,p}_\loc(\C\wo Z,G)$, such that the conclusions of that proposition hold. 

We check the {\bf conditions of Proposition \ref{prop:soft}} with $z_\nu:=z_0:=0$: {\bf Condition \ref{prop:soft}(\ref{prop:soft eps})} holds by (\ref{eq:eps nu 0},\ref{eq:hhat eps R},\ref{eq:R 0 eps nu}). {\bf Condition \ref{prop:soft}(\ref{prop:soft conv})} follows from \ref{prop:cpt mod}(\ref{prop:cpt mod:<},\ref{prop:cpt mod:=}), and {\bf condition \ref{prop:soft}(\ref{prop:soft lim nu E eps})} follows from \ref{prop:cpt mod}(\ref{prop:cpt mod lim nu E eps}).

We prove {\bf condition \ref{prop:soft}(\ref{prop:soft en})}: We define
\[\psi_\nu:B_{R_\nu^{-1}r}\to B_r,\,\psi_\nu(z):=\hhat R_\nu^{-1}z,\quad w'_\nu:=\psi_\nu^*w_\nu.%
\footnote{In the case $R_0=1$ we have $\psi_\nu=\phi_\nu$ and hence $w'_\nu=\wt w_\nu$.}%
\]
We choose $0<\eps\leq r$ so small that
\[\lim_{\nu\to\infty}E^{R_\nu}(w_\nu,B_\eps)<m_0+\frac{E_1}2.\]
Furthermore, we choose an integer $\nu_0$ so large that for $\nu\geq\nu_0$, we have $E^{R_\nu}(w_\nu,B_\eps)<m_0+E_1/2$. We fix $\nu\geq\nu_0$. Using (\ref{eq:E R =},\ref{eq:de min}), it follows that
\[E\big(w'_\nu,A(\hhat\eps_\nu R_\nu,\eps R_\nu)\big)<E_1.\]
It follows that the requirements of Proposition \ref{prop:en conc} are satisfied with $r$, $R$, $w_\nu$ replaced by $\max\{r_0,\hhat\eps_\nu R_\nu\}$, $\eps R_\nu$, $w'_\nu$. Therefore, we may apply that result (with ``$\eps$'' equal to $1$), obtaining
\[E^{R_\nu}\Big(w_\nu,A\big(a\max\{R_\nu^{-1}r_0,\hhat\eps_\nu\},a^{-1}\eps\big)\Big)\leq4a^{-1}E_1,\quad\forall a\geq2.\]
Using (\ref{eq:hhat eps R},\ref{eq:R 0 eps nu}) and the fact $z_\nu=z_0=0$, the equality (\ref{eq:en cons}) follows. This proves \ref{prop:soft}(\ref{prop:soft en}).

To see that {\bf condition \ref{prop:soft}(\ref{prop:soft Z})} holds, assume first that $R_0=1$. Then $Z=\emptyset$ by statement (\ref{prop:cpt mod:<}) of Proposition \ref{prop:cpt mod}. The same statement and Lemma \ref{le:conv e} imply that
\[E(w_0,B_{2\hhat R_0})=\lim_{\nu\to\infty}E(\wt w_\nu,B_{2\hhat R_0}).\]
It follows from the convergence $\hhat \eps_\nu R_\nu\to \hhat R_0>0$ and (\ref{eq:E R =},\ref{eq:de min}) that this limit is positive. This proves condition \ref{prop:soft}(\ref{prop:soft Z}) in the case $R_0=1$. 

{\bf Assume now that $R_0=\infty$ and $E^\infty(w_0)=0$.} Then condition \ref{prop:soft}(\ref{prop:soft Z}) is a consequence of the following two claims.
\begin{claim}\label{claim:Z B 1} The set $Z$ is not contained in the open ball $B_1$.
\end{claim}
\begin{proof}[Proof of Claim \ref{claim:Z B 1}] By \ref{prop:soft}(\ref{prop:soft en}) (which we proved above) there exists $R>0$ so that
\begin{equation}
  \label{eq:limsup de}\limsup_{\nu\to\infty}E^{R_\nu}\big(w_\nu,A(R\eps_\nu,R^{-1})\big)<\de.
\end{equation}
(Here we used that $z_0=z_\nu=0$.) Since $R_0=\infty$, we have $\hhat \eps_\nu=\eps_\nu$. Hence it follows from (\ref{eq:E R =}) and (\ref{eq:E eps lim}) that 
\begin{equation}\label{eq:lim E R nu}\lim_{\nu\to\infty}E^{R_\nu}\big(w_\nu,A(\eps_\nu,R^{-1})\big)=E(R^{-1})-m_0+\de.
\end{equation}
By the definition (\ref{eq:m 0 lim}) of $m_0$, the right hand side is bounded below by $\de$. Therefore, combining (\ref{eq:lim E R nu}) with (\ref{eq:limsup de}), it follows that 
\begin{equation}\label{eq:liminf E A}\liminf_{\nu\to\infty}E^{R_\nu}(w_\nu,A(\eps_\nu,\eps_\nu R))>0.\end{equation}
Suppose by contradiction that $Z\sub B_1$. Then by \ref{prop:cpt mod}(\ref{prop:cpt mod:=}), the connection $g_\nu^*\wt A_\nu$ converges to $A_0$ in $C^0$ on $\bar A(1,R):=\BAR B_R\wo B_1$, and the map $g_\nu^{-1}\wt u_\nu$ converges to $u_0$ in $C^1$ on $\bar A(1,R)$. Hence Lemma \ref{le:conv e} implies that
\[E^\infty\big(w_0,A(1,R)\big)=\lim_{\nu\to\infty}E^{\wt R_\nu}\big(\wt w_\nu,A(1,R)\big).\]
Combining this with (\ref{eq:liminf E A}), we arrive at a contradiction to our assumption $E^\infty(w_0)=0$. This proves Claim \ref{claim:Z B 1}.
\end{proof}
\begin{claim}\label{claim:0 Z} The set $Z$ contains $0$.
\end{claim}
\begin{pf}[Proof of Claim \ref{claim:0 Z}] By Claim \ref{claim:Z B 1} the set $Z\wo B_1$ is nonempty. We choose a point $z\in Z\wo B_1$ and a number $\eps_0>0$ so small that $B_{\eps_0}(z)\cap Z=\{z\}$. We fix $0<\eps<\eps_0$. Since $\eps_\nu\to0$ (as $\nu\to\infty$), (\ref{eq:e C 0}) implies that
\[e^{\wt R_\nu}_{\wt w_\nu}(0)=\Vert e^{\wt R_\nu}_{\wt w_\nu}\Vert_{C^0(\bar B_\eps(z))},\]
for $\nu$ large enough. Combining this with condition \ref{prop:soft}(\ref{prop:soft lim nu E eps}) (which we proved above), it follows that
\[\liminf_{\nu\to\infty}e^{\wt R_\nu}_{\wt w_\nu}(0)\geq\frac{\Emin}{\pi\eps^2}.\]
Since $\eps\in(0,\eps_0)$ is arbitrary, it follows that 
\begin{equation}\label{eq:e wt R nu wt w nu 0}e^{\wt R_\nu}_{\wt w_\nu}(0)\to\infty,\quad\textrm{as }\nu\to\infty.
\end{equation}
If $0$ did not belong to $Z$, then by \ref{prop:cpt mod}(\ref{prop:cpt mod:=}) and Lemma \ref{le:conv e} the energy density $e^{\wt R_\nu}_{\wt w_\nu}(0)$ would converge to $e^\infty_{w_0}(0)$, a contradiction to (\ref{eq:e wt R nu wt w nu 0}). This proves Claim \ref{claim:0 Z}, and completes the proof of \ref{prop:soft}(\ref{prop:soft Z}) and therefore of Proposition \ref{prop:soft}.
\end{pf}
\end{proof}
\begin{rmk}\label{rmk:soft proof} \rm Assume that $R_0,Z,w_0$ are constructed as in the proof of condition (\ref{prop:soft Z}) of Proposition \ref{prop:soft}, and that $R_0=\infty$ and $E^\infty(w_0)=0$. Then $Z\sub\bar B_1$ (and hence $Z\cap S^1\neq\emptyset$ by Claim \ref{claim:Z B 1}). This follows from the inequalities
\[\lim_{\nu\to\infty}E^{\wt R_\nu}(\wt w_\nu,A(1,R))\leq\de<\Emin,\quad\forall R>1.\]
Here the first inequality is a consequence of condition (\ref{eq:E R =}). $\Box$
\end{rmk}
\section{Proof of the bubbling result}\label{sec:proof:thm:bubb}
Based on the results of the previous sections, we are now ready to prove the first main result of this memoir, Theorem \ref{thm:bubb}. The proof is an adaption of the proof of \cite[Theorem 5.3.1]{MS04} to the present setting. The strategy is the following: Consider first the case $k=0$, i.e., the only marked point is $z_0^\nu=\infty$. We rescale the sequence $W_\nu$ so rapidly that all the energy is concentrated around the origin in $\C$. Then we ``zoom back in'' in a soft way, to capture the bubbles (spheres in $\BAR M$ and vortices over $\C$) in an inductive way. (See Claim \ref{claim:tree} below.) 

Next we show that at each stage of this construction, the total energy of the components of the tree plus the energy loss at the unresolved bubbling points equals the limit of the energies $E(W^\nu)$. Here we call a bubbling point ``unresolved'' if we have not yet constructed any sphere or vortex which is attached to this point. (See Claim \ref{claim:f E}.) Furthermore, we prove that the number of vertices of the tree is uniformly bounded above. (See inequality (\ref{eq:N E}).) This implies that the inductive construction terminates at some point. 

We also show that the vortices over $\C$ and the bubbles in $\BAR M$ have the required properties and that the data fits together to a stable map, which is the limit of a subsequence of $W^\nu$. (See Claims \ref{claim:MMM bar u i} and \ref{claim:st conv}.)

For $k\geq1$ we then prove the statement of the theorem inductively, using the statement for $k=0$. At each induction step we need to handle one additional marked point in the sequence of vortex classes and marked points. In the limit there are three possibilities for the location of this point:\\

\noi{\bf (I)} It does not coincide with any special point.

\noi{\bf (II)} It coincides with the marked point $z_i$ (lying on the vertex $\al_i$), for some $i$.

\noi{\bf (III)} It lies between two already constructed bubbles.\\

In case (I) we just include the new marked point into the bubble tree. In case (II), we introduce either 
\begin{itemize}
\item a ``ghost bubble'', which carries the two marked points and is connected to $\al_i$, or 
\item a ``ghost vortex'', which is connected to $\al_i$, and a ``ghost bubble of type 0'', which is connected to the ``ghost vortex'' and carries the two marked points.
\end{itemize}
The second option will only occur if $\al_i\in T_\infty$. Which of the two options we choose depends on the speed at which the two marked points come together. In case (III) we introduce a ``ghost bubble'' between the two bubbles, which carries the new marked point. 
\begin{proof}[Proof of Theorem \ref{thm:bubb} (p.~\pageref{thm:bubb})]\label{thm:bubb proof}\setcounter{claim}{0}We consider first the {\bf case $k=0$.} Let $W_\nu$ be a sequence of vortex classes as in the hypothesis. For each $\nu\in\N$ we choose a representative $w_\nu:=(P_\nu,A_\nu,u_\nu)$ of $W_\nu$, such that $P_\nu=\C\x G$. Passing to some subsequence we may assume that $E(w_\nu)$ converges to some constant $E$. The hypothesis $E(W_\nu)>0$ (for every $\nu$) implies that $E\geq\Emin$. We choose a sequence $R_\nu\geq1$ such that 
\begin{equation}
  \label{eq:E B E}E(W_\nu,B_{R_\nu})\to E.  
\end{equation}
We define 
\begin{eqnarray*}&R_0^\nu:=\nu R_\nu,\quad \phi_\nu:\C\to\C,\,\phi_\nu(z):=R_0^\nu z,\quad w^\nu_0:=\phi_\nu^*w_\nu,&\\
&j_1:=0,\quad z_1:=0,\quad Z_0:=\{0\},\quad z_0^\nu:=0.&
\end{eqnarray*}
The next claim provides an inductive construction of the bubble tree. (Some explanations are given below. See also Figure \ref{fig:claim}.)
\begin{figure}
  \centering
\leavevmode\epsfbox{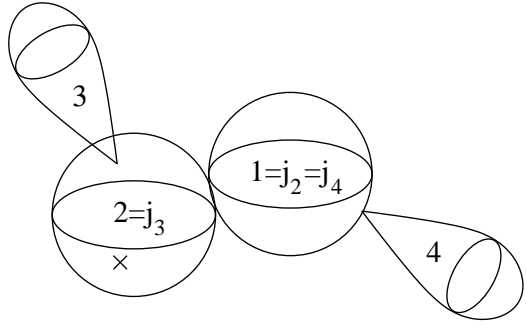}
\caption{This is a ``partial stable map'' as in Claim \ref{claim:tree}. It is a possible step in the construction of the stable map of Figure \ref{fig:stable map}. The cross denotes a bubbling point that has not yet been resolved. When adding marked points, a ``ghost sphere'' in $\BAR M$ will be introduced between the components 1 and 4.} 
\label{fig:claim} % magnification factor = 57 %
\end{figure}
\begin{claim}\label{claim:tree} For every integer $\ell\in\N$, passing to some subsequence, there exist an integer $N:=N(\ell)\in\N$ and tuples
\[(R_i,Z_i,w_i)_{i\in\{1,\ldots,N\}},\quad(j_i,z_i)_{i\in\{2,\ldots,N\}},\quad(R_i^\nu,z_i^\nu)_{i\in\{1,\ldots,N\},\,\nu\in\N},\]
where $R_i\in\{1,\infty\}$, $Z_i\sub \C$ is a finite subset, $w_i=(A_i,u_i)\in\WWW_{\C\wo Z_i}$ is an $R_i$-vortex, $j_i\in \{1,\ldots,i-1\}$, $z_i\in\C$, $R_i^\nu>0$, and $z_i^\nu\in\C$, such that the following conditions hold.
\begin{enui}
\item\label{claim:tree Z dist} For every $i=2,\ldots,N$ we have $z_i\in Z_{j_i}$. Moreover, if $i,i'\in \{2,\ldots,N\}$ are such that $i\neq i'$ and $j_i=j_{i'}$ then $z_i\neq z_{i'}$.
\item\label{claim:tree stab} Let $i=1,\ldots,N$. If $R_i=1$ then $Z_i=\emptyset$ and $E(w_i)>0$. If $R_i=\infty$ and $E^\infty(w_i)=0$ then $|Z_i|\geq2$.
\item\label{claim:tree R} Fix $i=1,\ldots,N$. If $R_i=1$ then $R_i^\nu=1$ for every $\nu$, and if $R_i=\infty$ then $R_i^\nu\to\infty$. Furthermore, 
\begin{equation}
\label{eq:R i nu}\frac{R_i^\nu}{R_{j_i}^\nu}\to 0,\qquad \frac{z_i^\nu-z_{j_i}^\nu}{R_{j_i}^\nu}\to z_i.
\end{equation}

In the following we set $\phi_i^\nu(z):=R_i^\nu z +z_i^\nu$, for $i=0,\ldots,N$ and $\nu\in\N$. 
\item\label{claim:tree conv} For every $i=1,\ldots,N$ there exist gauge transformations $g_i^\nu\in W^{2,p}_\loc(\C\wo Z_i,G)$ such that the following holds. If $R_i=1$ then $(g_i^\nu)^*(\phi_i^\nu)^*w_\nu$ converges to $w_i$ in $C^\infty$ on every compact subset of $\C$. Furthermore, if $R_i=\infty$ then on every compact subset of $\C\wo Z_i$ the sequence $(g_i^\nu)^*(\phi_i^\nu)^*A_\nu$ converges to $A_i$ in $C^0$, and the sequence $(g_i^\nu)^*(\phi_i^\nu)^*u_\nu$ converges to $u_i$ in $C^1$. 
\item\label{claim:tree lim nu E eps} Let $i=1,\ldots,N$, $z\in Z_i$ and $\eps_0>0$ be such that $B_{\eps_0}(z)\cap Z_i=\{z\}$. Then for every $0<\eps<\eps_0$ the limit 
\[E_z(\eps):=\lim_{\nu\to\infty} E^{R_i^\nu}\big((\phi_i^\nu)^*w_\nu,B_\eps(z)\big)\] 
exists, and $\Emin\leq E_z(\eps)<\infty$. Furthermore, the function
\[(0,\eps_0)\ni\eps\mapsto E_z(\eps)\in[\Emin,\infty)\]
is continuous.
\item\label{claim:tree B B}For every $i=1,\ldots,N$, we have
\[\lim_{R\to\infty}\limsup_{\nu\to\infty} E\big(w_\nu,B_{R_{j_i}^\nu/R}(z_{j_i}^\nu+R_{j_i}^\nu z_i)\wo B_{RR_i^\nu}(z_i^\nu)\big)=0.\]
\item\label{claim:tree compl} If $\ell>N$ then for every $j=1,\ldots,N$ we have 
\begin{equation}\label{eq:Z j}Z_j=\big\{z_i\,|\,j<i\leq N,j_i=j\big\}.
\end{equation}
\end{enui}
\end{claim}
To understand this claim, note that the collection $(j_i)_{i\in\{2,\ldots,N\}}$ describes a tree with vertices the numbers $1,\ldots,N$ and unordered edges $\big\{(i,j_i),(j_i,i)\big\}$. Attached to the vertices of this tree are vortex classes and $\infty$-vortex classes. (The latter will give rise to holomorphic spheres in $\BAR M$.) Each pair $(R_i^\nu,z_i^\nu)$ defines a rescaling $\phi_i^\nu$, which is used to obtain the $i$-th limit vortex or $\infty$-vortex. (See condition (\ref{claim:tree conv}).)

The point $z_i$ is the nodal point on the $j_i$-th vertex, at which the $i$-th vertex is attached. The corresponding nodal point on the $i$-th vertex is $\infty$. The set $Z_i$ consists of the nodal points except $\infty$ (if $i\geq2$) on the $i$-th vertex together with the bubbling points that have not yet been resolved. 

Condition (\ref{claim:tree Z dist}) implies that the nodal points at a given vertex are distinct. Condition (\ref{claim:tree stab}) guarantees that once all bubbling points have been resolved, the $i$-th component will be stable.%
\footnote{Note that in the case $i\geq2$ there is another nodal point at $\infty$, and for $i=1$ there will be a marked point at $\infty$, which is the limit of the sequence $z_0^\nu$.}

Condition (\ref{claim:tree R}) implies that the rescalings $\phi_i^\nu$ ``zoom out'' less than the rescalings $\phi_{j_i}^\nu$. A consequence of condition (\ref{claim:tree lim nu E eps}) is that at every nodal or unresolved bubbling point at least the energy $\Emin$ concentrates in the limit. Condition (\ref{claim:tree B B}) means that no energy is lost in the limit between each pair of adjacent bubbles. Finally, condition (\ref{claim:tree compl}) means that in the case $\ell>N$ all bubbling points have been resolved. 
\begin{proof}[Proof of Claim \ref{claim:tree}]We show that the statement holds for $\ell:=1$. We check the conditions of Proposition \ref{prop:soft} (Soft rescaling) with $z_0:=0$, $r:=1$ and $R_\nu$, $w_\nu$ replaced by $R_0^\nu$, $w_0^\nu$. Condition \ref{prop:soft}(\ref{prop:soft K}) follows from Proposition \ref{prop:bounded} in Appendix \ref{sec:vort}, using the hypothesis that $M$ is equivariantly convex at $\infty$. Condition \ref{prop:soft}(\ref{prop:soft E}) follows from the facts
\[\lim_{\nu\to\infty} E^{R_0^\nu}(w_0^\nu,B_\eps)=E,\,\forall\eps>0,\quad E\geq\Emin.\]
The first condition is a consequence of the facts $R_0^\nu=\nu R_\nu$, $E(w_\nu)\to E$, and (\ref{eq:E B E}).

Thus by Proposition \ref{prop:soft}, there exist $R_0\in \{1,\infty\}$, a finite subset $Z\sub\C$, and an $R_0$-vortex $w_0\in\WWW^p_{\C\wo Z_1}$, and passing to some subsequence, there exist sequences $\eps_\nu>0$, $z_\nu$, and $g_\nu$, such that the conclusions of Proposition \ref{prop:soft} with $R_\nu,w_\nu$ replaced by $R_0^\nu,w_0^\nu$ hold. We define
\[N:=N(1):=1,\quad R_1:=R_0,\quad Z_1:=Z,\quad w_1:=w_0,\quad R_1^\nu:=\eps_\nu R_0^\nu,\quad z_1^\nu:=R_0^\nu z_\nu.\] 

We check {\bf conditions (\ref{claim:tree Z dist})-(\ref{claim:tree compl})} of Claim \ref{claim:tree} with $\ell=1$: Conditions (\ref{claim:tree Z dist},\ref{claim:tree compl}) are void. Furthermore, conditions (\ref{claim:tree stab})-(\ref{claim:tree B B}) follow from \ref{prop:soft}(\ref{prop:soft Z})-(\ref{prop:soft en}). This proves the statement of the Claim for $\ell=1$. 

Let $\ell\in\N$ and assume, by induction, that we have already proved the statement of Claim \ref{claim:tree} for $\ell$. We show that it holds for $\ell+1$. By assumption there exists a number $N:=N(\ell)$ and there exist collections
\[(R_i,Z_i,w_i)_{i\in\{1,\ldots,N\}},\quad(j_i,z_i)_{i\in\{2,\ldots,N\}},\quad (R_i^\nu,z_i^\nu)_{i\in\{1,\ldots,N\},\,\nu\in\N},\]
such that conditions (\ref{claim:tree Z dist})-(\ref{claim:tree compl}) hold. If
\[Z_j=\big\{z_i\,|\,j< i\leq N,\,j_i=j\big\},\quad\forall j=1,\ldots,N\]
then conditions (\ref{claim:tree Z dist})-(\ref{claim:tree compl}) hold with $N(\ell+1):=N$, and we are done. Hence assume that there exists a $j_0\in\{1,\ldots,N\}$ such that 
\begin{equation}
  \label{eq:Z j 0}Z_{j_0}\neq\big\{z_i\,|\,j_0<i\leq N,\,j_i=j_0\big\}.  
\end{equation}
We set $N(\ell+1):=N+1$ and choose an element 
\begin{equation}\label{eq:z N 1}z_{N+1}\in Z_{j_0}\wo\big\{z_i\,|\,j<i\leq N,\,j_i=j_0\big\}.
\end{equation}
We fix a number $r>0$ so small that $B_r(z_{N+1})\cap Z_{j_0}=\{z_{N+1}\}$. We apply Proposition \ref{prop:soft} with $z_0:=z_{N+1}$ and $R_\nu$, $w_\nu$ replaced by $R_{j_0}^\nu$, $(\phi_{j_0}^\nu)^*w_\nu.$ Condition \ref{prop:soft}(\ref{prop:soft K}) holds by hypothesis. Furthermore, by condition (\ref{claim:tree lim nu E eps}) for $\ell$, condition \ref{prop:soft}(\ref{prop:soft E}) is satisfied. Hence passing to some subsequence, there exist $R_0\in\{1,\infty\}$, a finite subset $Z\sub \C$, an $R_0$-vortex $w_0\in\WWW^p_{\C\wo Z}$, and sequences $\eps_\nu>0$, $z_\nu$, such that the conclusion of Proposition \ref{prop:soft} holds. We define
\begin{eqnarray*}&R_{N+1}:=R_0,\quad Z_{N+1}:=Z,\quad w_{N+1}:=w_0,\quad j_{N+1}:=j_0,&\\
&R_{N+1}^\nu:=\eps_\nu R_{j_0}^\nu,\quad z_{N+1}^\nu:=R_{j_0}^\nu z_\nu+z_{j_0}^\nu.&
\end{eqnarray*}

We check {\bf conditions (\ref{claim:tree Z dist})-(\ref{claim:tree compl})} of Claim \ref{claim:tree} with $\ell$ replaced by $\ell+1$, i.e., $N$ replaced by $N+1$. Condition (\ref{claim:tree Z dist}) follows from the induction hypothesis and (\ref{eq:z N 1}). Conditions (\ref{claim:tree stab})-(\ref{claim:tree B B}) follow from \ref{prop:soft}(\ref{prop:soft Z})-(\ref{prop:soft en}).

We show that (\ref{claim:tree compl}) holds with $N$ replaced by $N+1$: By the induction hypothesis, it holds for $N$. Hence (\ref{eq:Z j 0}) implies that $N+1\geq\ell+1$. So there is nothing to check. This completes the induction and the proof of Claim \ref{claim:tree}. \end{proof}

Let $\ell\in\N$ be an integer and $N:=N(\ell)$, $(R_i,Z_i,w_i)$, $(j_i,z_i)$, $(R_i^\nu,z_i^\nu)$ be as in Claim \ref{claim:tree}. Recall that $Z_0=\{0\}$ and $z_0^\nu:=0$. We fix $i=0,\ldots,N$. We define $\phi_i^\nu(z):=R_i^\nu z+z_i^\nu$, for every measurable subset $X\sub\C$ we denote
\[E_i(X):=E^{R_i}(w_i,X),\quad E_i:=E_i(\C\wo Z_i),\quad E_i^\nu(X):=E^{R_i^\nu}((\phi_i^\nu)^*w_\nu,X).\]
Furthermore, for $z\in Z_i$ we define
\begin{equation}
  \label{eq:m i z}m_i(z):=\lim_{\eps\to 0}\lim_{\nu\to\infty}E_i^\nu(B_\eps(z)).
\end{equation}
For $i=0$ it follows from (\ref{eq:E B E}) and $R_0^\nu=\nu R_\nu$ that the limit $m_0(0)$ exists and equals $E$. For $i=1,\ldots,N$ it follows from condition (\ref{claim:tree lim nu E eps}) that the limit (\ref{eq:m i z}) exists and that $m_i(z)\geq \Emin$. 
For $j,k=0,\ldots, N$ we define 
\[Z_{j,k}:=Z_j\wo\{z_i\,|\,j<i\leq k,\,j_i=j\}\]
(This is the set of points on the $j$-th sphere that have not been resolved after the construction of the $k$-th bubble.) We define the function 
\begin{equation}\label{eq:f i}
f:\{1,\ldots,N\}\to [0,\infty),\quad f(i):=E_i+\sum_{z\in Z_{i,N}}m_i(z).  
\end{equation}
\begin{claim}\label{claim:f E} 
\begin{equation}\nn\sum_{i=1}^Nf(i)=E.
\end{equation}
\end{claim}
\begin{proof}[Proof of Claim \ref{claim:f E}]We show by induction that 
\begin{equation}
\label{eq:k f E}\sum_{i=1}^k\Bigg(E_i+\sum_{z\in Z_{i,k}}m_i(z)\Bigg)=E,
\end{equation}
for every $k=1,\ldots,N$. Claim \ref{claim:f E} is a consequence of this with $k=N$. For the proof of equality (\ref{eq:k f E}) we need the following.
\begin{claim}\label{claim:E Z} For every $i=1,\ldots,N$ we have
\begin{equation}
    \label{eq:m j i z i}m_{j_i}(z_i)=E_i+\sum_{z\in Z_i}m_i(z).
\end{equation}
\end{claim}
\begin{proof}[Proof of Claim \ref{claim:E Z}] Let $i=1,\ldots,N$. We choose a number $\eps>0$ so small that 
\[\bar B_{\eps}(z_i)\cap Z_{j_i}=\{z_i\},\qquad Z_i\sub B_{\eps^{-1}-\eps},\]
and if $z\neq z'$ are points in $Z_i$ then $|z-z'|>2\eps$. By condition (\ref{claim:tree lim nu E eps}) of Claim \ref{claim:tree}, for each $z\in Z_i$ the limit $\lim_{\nu\to\infty}E_i^\nu(B_\eps(z))$ exists. Lemma \ref{le:conv e} implies that
\[\lim_{\nu\to\infty}E_i^\nu(B_{\eps^{-1}})=E_i\left(B_{\eps^{-1}}\wo\bigcup_{z\in Z_i}B_\eps(z)\right)+\sum_{z\in Z_i}\lim_{\nu\to\infty}E_i^\nu(B_\eps(z)).\]
Combining this with condition (\ref{claim:tree B B}) of Claim \ref{claim:tree}, equality (\ref{eq:m j i z i}) follows from a straight-forward argument. This proves Claim \ref{claim:E Z}. 
\end{proof}
Since $Z_{1,1}=Z_1$, equality (\ref{eq:k f E}) for $k=1$ follows from Claim \ref{claim:E Z} and the fact $m_0(0)=E$. Let now $k=1,\ldots,N-1$ and assume that we have proved (\ref{eq:k f E}) for $k$. An elementary argument using Claim \ref{claim:E Z} with $i:=k+1$ shows (\ref{eq:k f E}) with $k$ replaced by $k+1$. By induction, Claim \ref{claim:f E} follows. 
\end{proof}
Consider the tree relation $E$ on $T:=\{1,\ldots,N\}$ defined by $i E i'$ iff $i=j_{i'}$ or $i'=j_i$. Lemma \ref{le:weight} in Appendix \ref{sec:add} with $f$ as in (\ref{eq:f i}), $k:=1$, $\al_1:=1\in T$, and $E_0:=\Emin$, implies that 
\begin{equation} \label{eq:N E} N\leq \frac{2E}\Emin+1.
\end{equation}
(Hypothesis (\ref{eq:f al E}) follows from conditions (\ref{claim:tree stab},\ref{claim:tree lim nu E eps}) of Claim \ref{claim:tree}.) Assume now that we have chosen $\ell>2E/\Emin+1$. By (\ref{eq:N E}) we have $\ell>N$, and therefore by condition (\ref{claim:tree compl}) of Claim \ref{claim:tree}, equality (\ref{eq:Z j}) holds, for every $j=1,\ldots,N$. We define
\[T:=\{1,\ldots,N\},\quad T_0:=\emptyset,\quad T_1:=\{i\in T\,|\,R_i=1\},\quad T_\infty:=T\wo T_1,\] 
and the tree relation $E$ on $T$ by 
\[i Ei'\iff i=j_{i'}\textrm{ or }i'=j_i.\] 
Furthermore, for $i,i'\in T$ such that $iEi'$ we define the nodal points 
\[z_{ii'}:=\left\{
\begin{array}{ll}\infty,&\textrm{if }i'=j_i,\\
z_{i'},&\textrm{if }i=j_{i'}.
  \end{array}
\right.\]
Moreover, we define the marked point 
\[(\al_0,z_0):=(1,\infty)\in T\x S^2.\]
\begin{claim}\label{claim:MMM bar u i} Let $i\in T$. If $i\in T_1$ then $E(w_i)<\infty$ and $u_i(\C\x G)$ has compact closure. Furthermore, if $i\in T_\infty$ then the map
\[Gu_i:\C\wo Z_i\to \BAR M=\mu^{-1}(0)/G\]
extends to a smooth $\bar J$-holomorphic map 
\[\bar u_i:S^2\iso\C\cup\{\infty\}\to \BAR M.\]
\end{claim}
\begin{proof}[Proof of Claim \ref{claim:MMM bar u i}] We choose gauge transformations $g_i^\nu\in W^{2,p}_\loc(\C\wo Z_i,G)$ as in condition (\ref{claim:tree conv}) of Claim \ref{claim:tree}, and define $w_i^\nu:=(g_i^\nu)^*(\phi_i^\nu)^*w_\nu$. 

{\bf Assume that $i\in T_1$.} It follows from Fatou's lemma that 
\[E(w_i)\leq\liminf_{\nu\to\infty}E(w_i^\nu)=E<\infty.\] 
Furthermore, since by hypothesis $M$ is equivariantly convex at $\infty$, by Proposition \ref{prop:bounded} in Appendix \ref{sec:vort} there exists a $G$-invariant compact subset $K_0\sub M$ such that $u_i^\nu(\C)\sub K_0$, for every $\nu\in\N$. Since $u_i^\nu$ converges to $u_i$ pointwise, it follows that $u_i(\C)\sub K_0$. Hence $w_i$ has the required properties.

{\bf Assume now that $i\in T_\infty$.} By Proposition \ref{prop:bar del J} in Appendix \ref{sec:add} the map
\[G u_i:\C\wo Z_i\to \BAR M=\mu^{-1}(0)/G\]
is $\bar J$-holomorphic, and $e_{G u_i}=e^\infty_{w_i}$. It follows from Fatou's lemma that
\[E^\infty(w_i,\C\wo Z_i)\leq\liminf_{\nu\to\infty}E^{R_i^\nu}(w_i^\nu)=E<\infty.\]
Therefore, by removal of singularities, it follows that $G u_i$ extends to a smooth $\bar J$-holomorphic map $\bar u_i:S^2\to\BAR M$.%
\footnote{See e.g.~\cite[Theorem 4.1.2]{MS04}.}
 This proves Claim \ref{claim:MMM bar u i}. \end{proof}
\begin{claim}\label{claim:st conv} The tuple
\[(\W,\z):=\Big(T_0,T_1,T_\infty,E,([w_i])_{i\in T_1},(\bar u_i)_{i\in T_\infty},(z_{ii'})_{iEi'},(\al_0:=1,z_0:=\infty)\Big)\]
is a stable map in the sense of Definition \ref{defi:st}, and the sequence $([w_\nu],z_0^\nu:=\infty)$ converges to $(\W,\z)$ in the sense of Definition \ref{defi:conv}. (Here $[w_i]$ denotes the gauge equivalence class of $w_i$.)
\end{claim}
\begin{proof}[Proof of Claim \ref{claim:st conv}] We check the conditions of Definition \ref{defi:st}. {\bf Condition (\ref{defi:st comb})} is a consequence of the definitions of $T_0,T_1$ and $T_\infty$. {\bf Condition (\ref{defi:st dist})} follows from condition (\ref{claim:tree Z dist}) of Claim \ref{claim:tree} and the fact $Z_i=\emptyset$, for $i\in T_1$. (This follows from condition (\ref{claim:tree stab}) of Claim \ref{claim:tree}.)

{\bf Condition (\ref{defi:st conn})} follows from an elementary argument using Claim \ref{claim:tree}(\ref{claim:tree R},\ref{claim:tree conv},\ref{claim:tree B B}) and Proposition \ref{prop:en conc}. {\bf Condition (\ref{defi:st st})} follows from Claim \ref{claim:tree}(\ref{claim:tree stab}). Hence all conditions of Definition \ref{defi:st} are satisfied.

We check the {\bf conditions of Definition \ref{defi:conv}}. {\bf Condition (\ref{eq:E V bar T})} follows from Claim \ref{claim:f E}, using condition (\ref{claim:tree compl}) of Claim \ref{claim:tree}. {\bf Condition \ref{defi:conv}(\ref{defi:conv phi z})} follows from a straight-forward argument, using Claim \ref{claim:tree}(\ref{claim:tree R}).

{\bf Condition \ref{defi:conv}(\ref{defi:conv al be})} follows from Claim \ref{claim:tree}(\ref{claim:tree R}) by an elementary argument. {\bf Condition \ref{defi:conv}(\ref{defi:conv w})} follows from Claim \ref{claim:tree}(\ref{claim:tree conv}). Finally, {\bf condition \ref{defi:conv}(\ref{defi:conv z})} is void, since $k=0$. This proves Claim \ref{claim:st conv}.
\end{proof}
Thus we have proved Theorem \ref{thm:bubb} in the case $k=0$.\\

{\bf We prove that the theorem holds for every $k\geq1$:} We prove by induction that for every $k\in\N_0$ and every tuple $\big(W_\nu,z_1^\nu,\ldots,z_k^\nu\big)$ as in the hypotheses of Theorem \ref{thm:bubb}, there exists a stable map $(\W,\z)$ as in (\ref{eq:W z}) and a collection $(\phi_\al^\nu)$ of M\"obius transformations, such that the conditions of Definition \ref{defi:conv} hold, 
\begin{equation}\label{eq:al T infty Q}
\forall\al\in T_\infty,\,Q\sub\C\wo Z_\al\textrm{ compact: }\lim_{\nu\to\infty}E\big(W_\nu,\phi_\al^\nu(Q)\big)=E(\bar u_\al,Q),
\end{equation}
\begin{eqnarray}
\label{eq:phi al nu infty}&\phi_\al^\nu(\infty)=\infty,&\\ 
\label{eq:lim de limsup nu}&\lim_{R\to\infty}\limsup_{\nu\to\infty}E\big(W_\nu,\C\wo\phi_{\al_0}^\nu(B_R)\big)=0,&
\end{eqnarray}
and for every edge $\al E\be$ such that $\al\in T_1\cup T_\infty$ and $\be$ lies in the chain of vertices from $\al$ to $\al_0$, we have
\begin{equation}\label{eq:lim limsup be al}\lim_{R\to\infty}\limsup_{\nu\to\infty}E\big(W_\nu,\phi_\be^\nu(B_{R^{-1}}(z_{\be\al}))\wo \phi_\al^\nu(B_R)\big)=0. 
\end{equation}
For $k=0$ we proved this above. In that construction condition (\ref{eq:al T infty Q}) follows from statement (\ref{prop:soft conv}) of Proposition \ref{prop:soft}. Furthermore, condition (\ref{eq:lim de limsup nu}) is a consequence of condition (\ref{claim:tree B B}) of Claim \ref{claim:tree} with $i=1$, and the facts $E(W_\nu,B_{R_\nu})\to E$ and $R_0^\nu=\nu R_\nu$. Finally, condition (\ref{eq:lim limsup be al}) follows from Claim \ref{claim:tree}(\ref{claim:tree B B}). 

Let now $k\in\N$ and assume that we have proved the statement for $k-1$. Passing to some subsequence, we may assume that for every $i=1,\ldots,k-1$, the limit 
\begin{equation}\label{eq:z k i}z_{ki}:=\lim_{\nu\to\infty}(z_k^\nu-z_i^\nu)\in\C\cup\{\infty\}\end{equation}
exists. We set 
\[z_{k0}:=\infty.\] 
Passing to a further subsequence, we may assume that the limit 
\[z_{\al k}:=\lim_{\nu\to\infty}(\phi_\al^\nu)^{-1}(z_k^\nu)\in S^2\] 
exists, for every $\al\in T:=T_0\disj T_1\disj T_\infty$. There are three cases.\\

\noi{\bf Case (I)} There exists a vertex $\al\in T$, such that $z_{\al k}$ is not a special point of $(\W,\z)$ at $\al$.\\

\noi{\bf Case (II)} There exists an index $i\in\{0,\ldots,k-1\}$ such that $z_{\al_ik}=z_i$.\\

\noi{\bf Case (III)} There exists an edge $\al E\be$ such that $z_{\al k}=z_{\al\be}$ and $z_{\be k}=z_{\be\al}$.\\

These three cases exclude each other. For the combination of the cases (II) and (III) this follows from the last part of condition (\ref{defi:st dist}) (distinctness of the special points) in Definition \ref{defi:st}. 

\begin{claim}\label{claim:I III} One of the three cases always applies. 
\end{claim}

\begin{proof}[Proof of Claim \ref{claim:I III}] This follows from an elementary argument, using that $T$ is finite and does not contain cycles. 
\end{proof}

{\bf Assume that Case (I) holds.} We fix a vertex $\al\in T$ such that $z_{\al k}$ is not a special point.%
\footnote{This vertex is unique, but we will not use this.}
 We define $\al_k:=\al$ and introduce a new marked point 
\[z_k:=z_{\al_kk}\]
on the $\al_k$-sphere. Then $(\W,\z)$ augmented by $(\al_k,z_k)$ is again a stable map and the sequence $(W_\nu,z_0^\nu,\ldots,z_k^\nu)$ converges to this new stable map via $(\phi_\al^\nu)_{\al\in T}$. The induction step in Case (I) follows.\\

{\bf Assume that Case (II) holds.} We fix an index $0\leq i\leq k-1$ such that $z_{\al_ik}=z_i$. (It is unique.)\\

Consider first {\bf Case (IIa): $i\neq0$ and the condition $z_{ki}\neq0$ or $\al_i\in T_0\cup T_1$ holds}. In this case we extend the tree $T$ as follows. We define 
\[j:=\left\{\begin{array}{ll}
0,&\textrm{if }z_{ki}=0,\\
1,&\textrm{if }z_{ki}\neq0,\infty,\\
\infty,&\textrm{if }z_{ki}=\infty,
\end{array}\right.
\]
and introduce an additional vertex $\ga$, which has type $j$ and carries the new marked point. This vertex is adjacent to $\al_i$ in the new tree. We move the $i$-th marked point from the vertex $\al_i$ to the vertex $\ga$ and introduce an additional marked point on $\ga$. More precisely, we define
\begin{eqnarray*}&T_j^\new:=T_j\disj\{\ga\},\quad T_{j'}^\new:=T_{j'},\,j'\in\{0,1,\infty\}\wo\{j\},&\\
&T^\new:=T\disj\{\ga\},\quad E^\new:=E\disj\big\{(\al_i,\ga),(\ga,\al_i)\big\},&\\
&\al_i^\new:=\al_k^\new:=\ga,\quad z_{\ga\al_i}^\new:=\infty,\quad z_{\al_i\ga}^\new:=z_i,&\\
&z_i^\new:=0,\quad z_k^\new:=\left\{\begin{array}{ll}
z_{ki},&\textrm{if }j=1,\\
1,&\textrm{if }j=0,\infty.\end{array}\right.&
\end{eqnarray*}
If $j=1$, i.e., $\ga\in T^\new_1$, then we introduce the following ``ghost vortex'': We choose a point $x_0$ in the orbit $\bar u_{\al_i}(z_i)\sub\mu^{-1}(0)$, and define $A_\ga:=0\in\Om^1(\C,\g)$ and $u_\ga:\C\to M$ to be the map which is constantly equal to $x_0$. We identify $A_\ga$ with a connection on $\C\x G$ and $u_\ga$ with a $G$-equivariant map $\C\x G\to M$, and set 
\begin{equation}\label{eq:W ga}W_\ga:=\big[\C\x G,A_\ga,u_\ga\big].
\end{equation}
If $j=\infty$, i.e., $\ga\in T_\infty^\new$, then we define $\bar u_{\ga}:S^2\to\BAR M$ to be the constant map
\[\bar u_{\ga}\const \bar u_{\al_i}(z_i).\] 

We denote by $(\W^\new,\z^\new)$ the tuple obtained from $(\W,\z)$ by making the changes described above. By elementary arguments this is again a stable map. 

We define the sequence of M\"obius transformations $\phi_{\ga}^\nu:S^2\to S^2$ by 
\begin{equation}
\label{eq:phi ga nu}\phi_\ga^\nu(z):=\left\{
\begin{array}{ll}z+z_i^\nu,&\textrm{if }j=1,\\ 
(z_k^\nu-z_i^\nu)z+z_i^\nu,&\textrm{if }j=0,\infty. 
\end{array}\right.
\end{equation}
\begin{claim}\label{claim:w new II} There exists a subsequence of $\big(W_\nu,z_0^\nu,\ldots,z_k^\nu\big)$ that converges to $(\W^\new,\z^\new)$ via the M\"obius transformations $(\phi_\al^\nu)_{\al\in T^\new,\,\nu\in\N}$. 
\end{claim}

\begin{proof}[Proof of Claim \ref{claim:w new II}] Condition (\ref{eq:E V bar T}) (energy conservation) holds for every subsequence, since the new component $\ga$ carries no energy. {\bf Conditions (\ref{defi:conv phi z},\ref{defi:conv al be}) of Definition \ref{defi:conv}} hold (for the new collection of M\"obius transformations), by elementary arguments. 

{\bf We check condition \ref{defi:conv}(\ref{defi:conv w})} up to some subsequence. For every $\al\in T^\new_1\cup T^\new_\infty$ we write  
\begin{equation}\label{eq:W al nu}W_\al^\nu:=(\phi_\al^\nu)^*W_\nu,\quad\bar u_\al^\nu:=\bar u_{W_\al^\nu}:\C\to M/G,
\end{equation}
where $\bar u_{W_\al^\nu}$ is defined as in (\ref{eq:BAR u W}). In the {\bf case $j=0$} condition \ref{defi:conv}(\ref{defi:conv w}) does not contain any new requirement. 

{\bf Assume that $j=1$.} It follows from Proposition \ref{prop:cpt mod} (Compactness modulo bubbling and gauge) with $R_\nu:=1$, $r_\nu:=\nu$ and $W_\nu$ replaced by $W_\ga^\nu$, that there exists a vortex class $\wt W_\ga$ over $\C$ such that passing to some subsequence, the sequence $W_\ga^\nu$ converges to $\wt W_\ga$, with respect to $\tau_{\C}$ (as in Definition \ref{defi:tau Si}), and the sequence $\bar u_\ga^\nu$ converges to $\bar u_{\wt W_\ga}$, uniformly on every compact subset of $\C$. 
\begin{claim}\label{claim:wt W ga W ga} We have $\wt W_\ga=W_\ga$ (defined as in (\ref{eq:W ga})).
\end{claim} 
\begin{proof}[Proof of Claim \ref{claim:wt W ga W ga}] A straight-forward argument implies that $\al_i\in T_\infty^\new$. Therefore, condition \ref{defi:conv}(\ref{defi:conv w}) for the sequence $\big(W_\nu,z_0^\nu,\ldots,z_{k-1}^\nu\big)$ implies that the sequence $\bar u_{\al_i}^\nu$ converges to $\bar u_{\al_i}$, in $C^1$ on some neighborhood of $z_i$. (Here we used that $i\neq0$.)

Since $\phi_{\al_i\ga}^\nu:=(\phi_{\al_i}^\nu)^{-1}\circ\phi_\ga^\nu$ converges to $z_{\al_i\ga}=z_i$, uniformly on every compact subset of $S^2\wo\{z_{\ga\al_i}^\new\}=\C$, it follows that 
\[\bar u_\ga^\nu=\bar u_{\al_i}^\nu\circ\phi_{\al_i\ga}^\nu\to\bar x_0:=\bar u_{\al_i}(z_i),\]
uniformly on every compact subset of $\C$. It follows that $\bar u_{\wt W_\ga}\const\bar x_0$. We choose representatives
\[(\wt A_\ga,\wt u_\ga)\in\Om^1(\C,G)\x C^\infty(\C,G)\]
of $\wt W_\ga$ and $x_0\in\mu^{-1}(0)$ of $\bar x_0$. By hypothesis (H) the action of $G$ on $\mu^{-1}(0)$ is free. Hence, after regauging, we may assume that $\wt u_\ga\const x_0$. (Here we use Lemma \ref{le:g smooth} (Appendix \ref{sec:add}), which ensures that the gauge transformation is smooth.) It follows from the first vortex equation that $\wt A_\ga=0$. This shows that $\wt W_\ga=W_\ga$ and hence proves Claim \ref{claim:wt W ga W ga}. 
\end{proof}
This proves condition \ref{defi:conv}(\ref{defi:conv w}) in the case $j=1$. 

{\bf Assume now that $j=\infty$.} We have again $\al_i\in T_\infty^\new$. Hence condition \ref{defi:conv}(\ref{defi:conv w}) for the sequence $\big(W_\nu,z_0^\nu,\ldots,z_{k-1}^\nu\big)$ implies that the sequence $\bar u_{\al_i}^\nu$ converges to $\bar u_{\al_i}$, in $C^1$ on some neighborhood of $z_i$. (Here we used that $i\neq0$.) Let $Q\sub\C=S^2\wo Z_\ga$ be a compact subset. Since $\phi_{\al_i\ga}^\nu$ converges to $z_i$, in $C^\infty$ on $Q$, it follows that $\bar u_\ga^\nu=\bar u_{\al_i}^\nu\circ\phi_{\al_i\ga}^\nu$ converges to $\bar u_\ga\const\bar u_{\al_i}(z_i)$ in $C^1$ on $Q$, as required. This proves condition \ref{defi:conv}(\ref{defi:conv w}) in all cases.

{\bf Condition \ref{defi:conv}(\ref{defi:conv z})} is a consequence of the definition (\ref{eq:phi ga nu}) of $\phi_\ga^\nu$. This proves Claim \ref{claim:w new II}. 
\end{proof}

To see that condition (\ref{eq:al T infty Q}) holds, observe that we need to prove this only in the case $j=\infty$ and then only for $\al=\ga$. In this case it follows from the same condition for $\al=\al_i$ and \ref{defi:conv}(\ref{defi:conv al be}) with $\al=\al_i$ and $\be=\ga$. (Here we used that $z_i\neq\infty$ and it does not coincide with any nodal point at $\al_i$.) The same conditions also imply (\ref{eq:lim limsup be al}) for $\al:=\ga$ and $\be:=\al_i$. (For the other pairs of adjacent vertices (\ref{eq:lim limsup be al}) follows from the induction hypothesis.)

The induction step in Case (IIa) follows.\\

Consider the {\bf Case (IIb): $i=0$}. We define 
\begin{eqnarray*}&T^\new_\infty:=T_\infty\disj\{\ga\},\quad T^\new_j:=T_j,\,j=0,1,\quad E^\new:=E\disj\big\{(\al_0,\ga),(\ga,\al_0)\big\},&\\
&\al_0^\new:=\al_k^\new:=\ga,\quad z_{\ga\al_0}^\new:=0,\quad z_{\al_0\ga}^\new:=z_0=\infty,\quad z^\new_0:=\infty,\quad z^\new_k:=1,&\\
&\bar u_{\ga}\const \bar u_{\al_0}(z_0).&
\end{eqnarray*}
The tuple $(\W^\new,\z^\new)$ obtained from $(\W,\z)$ by making these changes, is again a stable map. It follows from our assumption (\ref{eq:phi al nu infty}) that there exist $\lam_\nu\in\C\wo\{0\}$ and $z_\nu\in\C$, such that $\phi_{\al_0}^\nu(z)=\lam_\nu z+z_\nu$. We define the sequence of M\"obius transformations $\phi_\ga^\nu$ by 
\[\phi_\ga^\nu(z):=(z_k^\nu-z_\nu)z+z_\nu.\]
We show that some subsequence of $\big(W_\nu,z_0^\nu,\ldots,z_k^\nu\big)$ converges to $(\W^\new,\z^\new)$ via the M\"obius transformations $(\phi_\al^\nu)_{\al\in T^\new,\,\nu\in\N}$: Condition (\ref{eq:E V bar T}) holds with $T$ replaced by $T^\new$, since the map $\bar u_{\ga}$ is constant. Conditions \ref{defi:conv}(\ref{defi:conv phi z},\ref{defi:conv al be},\ref{defi:conv z}) follow from straight-forward arguments. 

We show that {\bf condition \ref{defi:conv}(\ref{defi:conv w})} holds. Recall the definition (\ref{eq:W al nu}) of $\bar u_\al^\nu$. An elementary argument using condition \ref{defi:conv}(\ref{defi:conv w}) with $\al:=\al_0$, our assumption (\ref{eq:lim de limsup nu}), and Proposition \ref{prop:en conc}, shows that for every $\eps>0$ there exist numbers $R\geq r_0$ and $\nu_0\in\N$ such that
\begin{equation}\nn\bar d\big(\barev_{z_0=\infty}(W_{\al_0}),\bar u_{\al_0}^\nu(z)\big)<\eps,\quad\forall\nu\geq\nu_0,\,z\in\C\wo B_R,
\end{equation}
where $W_{\al_0}:=\bar u_{\al_0}$ if $\al_0\in T_\infty$, and $\barev_\infty$ is defined as in (\ref{eq:barev},\ref{eq:barev w}). The sequence $\phi_{\al_0\ga}^\nu$ converges to $z^\new_{\al_0\ga}=z_0=\infty$, uniformly on every compact subset of $S^2\wo\{z^\new_{\ga\al_0}=0\}$. It follows that $\bar u_\ga^\nu$ converges to the constant map $\bar u_\ga\const\barev_\infty(W_{\al_0})$, uniformly on every compact subset of $\C\wo\{0\}=S^2\wo\{z^\new_{\ga\al_0},z^\new_0\}$. Condition \ref{defi:conv}(\ref{defi:conv w}) with $\al=\ga$ is a consequence of the following.
\begin{claim}\label{claim:conv C 1} Passing to some subsequence the convergence is in $C^1$ on every compact subset of $\C\wo\{0\}$.
\end{claim} 
\begin{proof}[Proof of Claim \ref{claim:conv C 1}] We denote $R_\nu:=|z_k^\nu-z_\nu|$, and choose a representative $w_\ga^\nu$ of $W_\ga^\nu$. This is an $R_\nu$-vortex. Our assumption $z_{\al_0k}=z_0=\infty$ implies that $R_\nu\to R_0:=\infty$. Hence by Proposition \ref{prop:cpt mod} there exists a finite subset $Z\sub\C$ and an $\infty$-vortex $w_\ga:=(A_\ga,u_\ga)$, and passing to some subsequence, there exist gauge transformations $g_\ga^\nu\in W^{2,p}_\loc(\C\wo Z,G)$, such that the assertions \ref{prop:cpt mod}(\ref{prop:cpt mod:=},\ref{prop:cpt mod lim nu E eps}) hold with $w_\nu:=w_\ga^\nu$. It follows that $\bar u_\ga^\nu$ converges in $C^1$ on every compact subset of $\C\wo Z$. Furthermore, it follows from (\ref{eq:lim de limsup nu}) and \ref{prop:cpt mod}(\ref{prop:cpt mod lim nu E eps}) that $Z\sub\{0\}$. Claim \ref{claim:conv C 1} follows. 
\end{proof}
Conditions (\ref{eq:al T infty Q},\ref{eq:lim de limsup nu}) (for $\al^\new_0=\ga$) and (\ref{eq:lim limsup be al}) (for $\al:=\al_0$ and $\be:=\ga$) follow from condition (\ref{eq:lim de limsup nu}) for $\al_0$, which holds by the induction hypothesis. This proves the induction step in Case (IIb).\\

Consider now {\bf Case (IIc): $z_{ki}=0$ and $\al_i\in T_\infty$}. Then we introduce a vertex $\ga\in T_1$ adjacent to $\al_i$ and a vertex $\de\in T_0$ adjacent to $\ga$, such that $\de$ carries the two new marked points. More precisely, we define
\begin{eqnarray*}&T_0^\new:=T_0\disj\{\de\},\quad T_1^\new:=T_1\disj\{\ga\},\quad T_\infty^\new:=T_\infty,&\\
&T^\new:=T\disj\{\ga,\de\},\quad E^\new:=E\disj\big\{(\al_i,\ga),(\ga,\al_i),(\ga,\de),(\de,\ga)\big\},&\\
&z_{\ga\al_i}^\new:=\infty,\quad z_{\al_i\ga}^\new:=z_i,\quad z_{\ga\de}:=0,\quad z_{\de\ga}:=\infty,&\\
&\al_i^\new:=\al_k^\new:=\de,\quad z_i^\new:=0,\quad z_k^\new:=1,&
\end{eqnarray*}
and $W_\ga$ as in case (IIa) with $j=1$. We thus obtain a new stable map. The proof of convergence to this stable map and of conditions (\ref{eq:al T infty Q}) (for $\ga$) and (\ref{eq:lim limsup be al}) (for $\al=\ga$ and $\be=\al_i$) is now analogous to the proof in Case (IIa). This proves the induction step in Case (IIc).\\

The Cases (IIa,b,c) cover all instances of Case (II), hence we have proved the induction step in Case (II).\\

{\bf Assume now that Case (III) holds.} In this case we introduce a new vertex $\ga$ between $\al$ and $\be$. Hence $\al$ and $\be$ are no longer adjacent, but are separated by $\ga$. We define $j:=0$ if $\al$ or $\be$ lies in $T_0$, and $j:=\infty$, otherwise.%
\footnote{In that case $\al$ or $\be$ lies in $T_\infty$.}
 We may assume w.l.o.g.~that $\be$ lies in the chain of vertices from $\al$ to $\al_0$. We define
\begin{eqnarray*}&T_j^\new:=T_j\disj\{\ga\},\quad T_{j'}^\new:=T_{j'},\,j'\in\{0,1,\infty\}\wo\{j\},&\\
&z_{\al\ga}^\new:=z_{\al\be},\,z_{\be\ga}^\new:=z_{\be\al},\,z_{\ga\al}^\new:=0,\,z_{\ga\be}^\new:=\infty,\,\al_k^\new:=\ga,\,z_k^\new:=1.&
\end{eqnarray*}
In the case $j=\infty$ we define $\bar u_\ga:S^2\to \BAR M$ to be the constant map equal to $\bar u_\be(z_{\be\al})$. Again, we obtain a new stable map. By condition (\ref{eq:phi al nu infty}) there exist $\lam_\al^\nu\in\C\wo\{0\}$ and $z_\al^\nu\in\C$, such that $\phi_\al^\nu(z)=\lam_\al^\nu z+z_\al^\nu$. We define 
\[\phi_\ga^\nu(z):=(z_k^\nu-z_\al^\nu)z+z_\al^\nu.\]
The proof of convergence proceeds now analogously to the proof in Case (IIb), using (\ref{eq:lim limsup be al}) for the proof of condition \ref{defi:conv}(\ref{defi:conv w}) rather than (\ref{eq:lim de limsup nu}). Furthermore, conditions (\ref{eq:al T infty Q}) with $\al=\ga$ and (\ref{eq:lim limsup be al}) with $(\al,\be)$ replaced by $(\al,\ga)$ and $(\ga,\be)$ follow from condition (\ref{eq:lim limsup be al}) for $(\al,\be)$, which holds by the induction hypothesis. This proves the induction step in Case (III), and hence terminates the proof of Theorem \ref{thm:bubb} in the case $k\geq1$.
\end{proof}
\begin{Rmk} In the above proof for $k=0$ the stable map $(\W,\z)$ is constructed by ``terminating induction''. Intuitively, this is induction over the integer $N$ occurring in Claim \ref{claim:tree}. The ``auxiliary index'' $\ell$ in Claim \ref{claim:tree} is needed to make this idea precise. Condition (\ref{claim:tree compl}) and the inequality (\ref{eq:N E}) ensure that the ``induction stops''. $\Box$
\end{Rmk}
\section{Proof of the result in Section \ref{SEC:EXAMPLE} characterizing convergence}\label{sec:proof:prop:conv S 1 C}
Using Compactness modulo bubbling and gauge for (rescaled) vortices, which was established in Section \ref{sec:comp} (Proposition \ref{prop:cpt mod}), in this section we prove Proposition \ref{PROP:CONV S 1 C}. The proof is based on the following result. Let $G:=S^1\sub\C$ act on $M:=\C$ by multiplication, with momentum map $\mu:\C\to i\R$ given by $\mu(z):=\frac i2(1-|z|^2)$, and let $d\in\N_0$. Recall the definition (\ref{eq:iota Sym}) of the map
\[\iota_d:\disj_{d'\leq d}\Sym^{d'}(\C)\to \Sym^d(S^2),\]
the definitions (\ref{eq:M d Sym},\ref{eq:deg W z}) of the set $\M_d$ of finite energy vortex classes of degree $d$ and of the local degree map $\deg_W:\C\to\N_0$, and the Definition \ref{defi:tau Si} of the ``$C^\infty$-topology $\tau_\Si$ on $\B_\Si$''.
\begin{prop}\label{prop:loc conv S 1 C} Let $0\leq d'\leq d$, $W\in\M_{d'}$ and $W_\nu\in\M_d$, for $\nu\in\N$. Then the following conditions are equivalent. 
\begin{enui}
\item\label{prop:loc conv g} For every open subset $\Om\sub\C$ with compact closure and smooth boundary the restriction $W_\nu|_{\BAR\Om}$ converges to $W|_{\BAR\Om}$ to $\BAR\Om$, as $\nu\to\infty$, with respect to the topology $\tau_{\BAR\Om}$.
\item\label{prop:loc conv deg} The point in the symmetric product
\[\deg_{W_\nu}\in\Sym^d(\C)\sub\Sym^d(S^2)\]
converges to $\iota_d(\deg_W)\in\Sym^d(S^2)$.
\end{enui}
\end{prop}
In the proof of this result we will use the following. Let $X$ be a topological space and $d\in\N_0$. We denote by $\Sym^d(X):=X^d/S_d$ the $d$-fold symmetric product of $X$, and canonically identify it with the set $\wt\Sym^d(X)$ of all maps $m:X\to\N_0$ such that $\sum_{x\in X}m(x)=d$.%
\footnote{Implicitly, here we require $m(x)$ to vanish, except for finitely many points $x\in X$.}
 We endow $\Sym^d(X)$ with the quotient topology. For every subset $Y\sub X$ and $d_0\in\N_0$, we define
\begin{equation}\label{eq:V Y d 0}V_Y^{d_0}:=\left\{m\in\wt\Sym^d(X)\,\Big|\,\sum_{x\in Y}m(x)=d_0,\,m(x)=0,\,\forall x\in\dd Y\right\},
\end{equation}
where $\dd Y\sub X$ denotes the boundary of $Y$. 
\begin{lemma}\label{le:subbasis} Let $\UU$ be a basis for the topology of $X$. If $X$ is Hausdorff then the sets $V_U^{d_0}$, $U\in\UU$, $d_0\in\N_0$, form a subbasis for the topology of $\wt\Sym^d(X)\iso\Sym^d(X)$. 
\end{lemma}
\begin{proof}[Proof of Lemma \ref{le:subbasis}]\setcounter{claim}{0} We denote by $\pi:X^d\to\Sym^d(X)$ the canonical projection. For every $U\in\UU$ and $d_0\in\N_0$, we have
\[\pi^{-1}(V_U^{d_0})=S_d\cdot\big(U^{d_0}\x(X\wo\BAR U)^{d-d_0}\big)\sub X^d,\]
where $\cdot$ denotes the action of $S_d$ on $X^d$. Since $|S_d|<\infty$, this set is open, i.e., $V_U^{d_0}$ is open. Now let $V\sub\wt\Sym^d(X)$ be an open set, and $m\in V$. It suffices to show that there exists a finite set $S\sub\UU\x\N_0$, such that the intersection $\bigcap_{(U,d_0)\in S}V_U^{d_0}$ contains $m$ and is contained in $V$. To see this, we denote by $x_1,\ldots,x_k\in X$ the distinct points at which $m$ does not vanish, and abbreviate $d_i:=m(x_i)$. For $i=1,\ldots,k$ we choose a neighborhood $U_i\in\UU$ of $x_i$, such that the sets $U_1,\ldots,U_k$ are disjoint, and 
\begin{equation}\label{eq:U 1 U k}U_1\x\cdots\x U_1\x\cdots\x U_k\x\cdots\x U_k\sub\pi^{-1}(V).
\end{equation}
Here the factor $U_i$ occurs $d_i$ times. (The sets exist by Hausdorffness of $X$, openness of $\pi^{-1}(V)$, the definition of the product topology on $X^d$, and the fact 
\[\big(x_1,\ldots,x_1,\ldots,x_k,\ldots,x_k\big)\in\pi^{-1}(m)\sub\pi^{-1}(V).)\] 
We define $S:=\big\{(U_i,d_i)\,\big|\,i=1,\ldots,k\big\}$. We have
\[\bigcap_{i=1}^kV_{U_i}^{d_i}=\pi\Big(U_1\x\cdots\x U_1\x\cdots\x U_k\x\cdots\x U_k\Big).\]
Combining this with (\ref{eq:U 1 U k}), it follows that $\bigcap_{i=1}^kV_{U_i}^{d_i}\sub V$. On the other hand, since $x_i\in U_i$, we have $m\in\bigcap_{i=1}^kV_{U_i}^{d_i}$. Hence the set $S$ has the required properties. This proves Lemma \ref{le:subbasis}.
\end{proof}
We will also use the following.
\begin{lemma}\label{le:x C 0} Let $x\in C\big(S^1,\C\wo\{0\}\big)$ be a loop. Then there exists a $C^0$-neighborhood $\UU\sub C\big(S^1,\C\wo\{0\}\big)$ of $x$, such that every $x'\in\UU$ is homotopic to $x$. 
\end{lemma}
\begin{proof}[Proof of Lemma \ref{le:x C 0}]\setcounter{claim}{0} We define $\UU$ to be the open $C^0$-ball around $x$, with radius $\min_{z\in S^1}|x(z)|$. Let $x'\in\UU$. We define
\[h:[0,1]\x S^1\to\C,\quad h(t,z):=tx'(z)+(1-t)x(z).\]
This map does not vanish anywhere, and therefore is a homotopy between $x$ and $x'$ inside $\C\wo\{0\}$. This proves Lemma \ref{le:x C 0}.
\end{proof}
\begin{proof}[Proof of Proposition \ref{prop:loc conv S 1 C}]\setcounter{claim}{0} We canonically identify the set $\B=\B_{\C}$ of gauge equivalence classes of smooth triples $(P,A,u)$ (as defined in (\ref{eq:B})) with the set of gauge equivalence classes of smooth pairs
\[(A,u)\in\WWW_\C:=\Om^1(\C,\g)\x C^\infty(\C,M).\] 

We show that {\bf (\ref{prop:loc conv g}) implies (\ref{prop:loc conv deg})}. Assume that (\ref{prop:loc conv g}) holds. We denote by $m:S^2=\C\cup\{\infty\}\to\N_0$ the map given by $\deg_W$ on $\C$, and $m(\infty):=d-\deg(W)$. We define 
\[\UU:=\big\{B_r(z)\,\big|\,z\in\C,\,r>0\big\}\cup\big\{(\C\wo\BAR B_r)\cup\{\infty\}\,\big|\,r\in[1,\infty)\big\}.\]
This is a basis for the topology of $S^2$. 
\begin{Claim}Let $U\sub\UU$ be such that $\deg_W$ vanishes on $\dd U$. Then there exists $\nu_0\in\N$ such that for every $\nu\geq\nu_0$ we have $\deg_{W_\nu}\in V_U^{d_0}$, where
\[d_0:=\sum_{z\in U}m(z).\]
\end{Claim}
\begin{proof}[Proof of the claim] Let $U$ be as in the hypothesis of the claim. Consider first the {\bf case $U=B_r(z)$}, for some $z\in\C$, $r>0$. We choose a (smooth) representative $w=(A,u)\in\WWW_{\BAR B_r(z)}$ of $W|_{\BAR B_r(z)}$. Since by assumption, $\deg_W$ vanishes on $\dd U=S^1_r(z)$, $u$ is nonzero on $S^1_r(z)$. 

The $C^\infty$-topology on the set of smooth connections on $\BAR B_r(z)$ is second-countable. Furthermore, the action of the gauge group on this space is continuous. Therefore, by assumption (\ref{prop:loc conv g}) and Lemma \ref{le:X G} in Appendix \ref{sec:add}, for each $\nu\in\N$, there exists a representative $w_\nu=(A_\nu,u_\nu)\in\WWW_{\BAR B_r(z)}$ of $W_\nu$, such that $w_\nu$ converges to $w$ in $C^\infty$, as $\nu\to\infty$. Hence, using Lemma \ref{le:x C 0}, there exists $\nu_0\in\N$ such that for $\nu\geq\nu_0$ the restriction of $u_\nu$ to $S^1_r(z)$ does not vanish anywhere, and is homotopic to $u|_{S^1_r(z)}$, as a map with target $\C\wo\{0\}$. It follows that
\[\deg\left(\frac{u_\nu}{|u_\nu|}:S^1_r(z)\to S^1\right)=\deg\left(\frac u{|u|}:S^1_r(z)\to S^1\right),\]
for $\nu\geq\nu_0$. By elementary arguments, the left hand side of this equality agrees with $\sum_{z'\in B_r(z)}\deg_{u_\nu}(z')$, and the right hand side equals $\sum_{z'\in B_r(z)}\deg_u(z')$. Since by definition, $\deg_{u_\nu}(z')=\deg_{W_\nu}(z')$ and $\deg_u(z')=\deg_W(z')$, it follows that
\[\sum_{z'\in B_r(z)}\deg_{W_\nu}(z')=\sum_{z'\in B_r(z)}\deg_W(z')=d_0,\]
and therefore, $\deg_{W_\nu}\in V_U^{d_0}$. This proves the claim in the case $U=B_r(z)\sub\C$. 

Consider now the {\bf case $U=(\C\wo\BAR B_r)\cup\{\infty\}$}, for some $r\geq1$. By what we just proved, there exists an index $\nu_0$, such that for every $\nu\geq\nu_0$ 
\[\sum_{z'\in B_r}\deg_{W_\nu}(z')=\sum_{z'\in B_r}\deg_W(z'),\]
It follows that 
\[\sum_{z'\in\C\wo\BAR B_r}\deg_{W_\nu}(z')=\sum_{z'\in(\C\wo\BAR B_r)\cup\{\infty\}}m(z')=d_0,\]
and therefore, $\deg_{W_\nu}\in V_U^{d_0}$, for every $\nu\geq\nu_0$. This proves the claim in the case $U=(\C\wo\BAR B_r)\cup\{\infty\}$.
\end{proof}
The claim implies that if $d_0\in\N$ and $U\in\UU$ are such that $m\in V_U^{d_0}$ then there exists $\nu_0\in\N$ such that for every $\nu\geq\nu_0$ we have $\deg_{W_\nu}\in V_U^{d_0}$. Combining this with Lemma \ref{le:subbasis} (with $X:=S^2$), condition (\ref{prop:loc conv deg}) follows. This proves that {\bf (\ref{prop:loc conv g}) implies (\ref{prop:loc conv deg})}.

We show the {\bf opposite implication}: Assume that (\ref{prop:loc conv deg}) is satisfied. We prove that given an open subset $\Om\sub\C$ with compact closure and smooth boundary, the class $W_\nu|_{\BAR\Om}$ converges to $W|_{\BAR\Om}$, as $\nu\to\infty$, with respect to the topology $\tau_{\BAR\Om}$. By Lemma \ref{le:subsubseq} (Appendix \ref{sec:add}) it suffices to show that for every such $\Om$ and every subsequence $(\nu_i)_{i\in\N}$ there exists a further subsequence $(i_j)_{j\in\N}$ such that $W_{\nu_{i_j}}|_{\BAR\Om}$ converges to $W|_{\BAR\Om}$, as $j\to\infty$. Let $(\nu_i)$ be a subsequence. Let $i\in\N$. We choose a representative $(A_i,u_i)\in\WWW_\C$ of $W_{\nu_i}$. It follows from Proposition \ref{prop:S 1 C image} that the image of $u_i$ is contained in the closed ball $\BAR B_1\sub\C$.%
\footnote{This is also true if $E(W_{\nu_i})=0$, since then the image of $u_i$ equals $S^1$.}
 Furthermore, by Proposition \ref{prop:E d}, we have $E(w_i)=d\pi$. Therefore, Proposition \ref{prop:cpt mod} (Compactness modulo bubbling and gauge) implies that there exist a smooth vortex $w_0:=(A_0,u_0)$ over $\C$, a subsequence $(i_j)_{j\in\N}$, and gauge transformations $g_j\in W^{2,p}_\loc(\C,G)$ (for $j\in\N$), such that the pair $w_j':=g_j^*(A_{i_j},u_{i_j})$ is smooth and converges to $w_0$, as $j\to\infty$, in $C^\infty$ on every compact subset of $\C$. We denote by $W_0$ the equivalence class of $w_0$. It follows from Lemma \ref{le:g smooth} (Appendix \ref{sec:add}) that $g_j$ is smooth, and hence $w_j'$ represents $W^j:=W_{\nu_{i_j}}$, for every $j$. It follows that $W^j|_{\BAR\Om}$ converges to $W_0|_{\BAR\Om}$ with respect to $\tau_{\BAR\Om}$, as $j\to\infty$, for every open subset $\Om\sub\C$ with compact closure and smooth boundary. 

It remains to show that $W_0=W$. The energy densities $e_{W^j}$ converge to $e_{W_0}$ in $C^\infty$ on compact subsets of $\C$. Therefore, we have
\[E(W_0)\leq\liminf_{j\to\infty}E(W^j)=d\pi.\]
By Proposition \ref{prop:E d} we have $E(W_0)=\pi\deg(W_0)$, and therefore, $d_0:=\deg(W_0)\leq d$. Hence $\iota_d(\deg_{W_0})\in\Sym^d(S^2)$ is well-defined. It follows from the implication (\ref{prop:loc conv g})$\then$(\ref{prop:loc conv deg}) (which we proved above), that the point
\[\deg_{W^j}\in\Sym^d(\C)\sub\Sym^d(S^2)\]
converges to $\iota_d(\deg_{W_0})$, as $j\to\infty$. By our assumption (\ref{prop:loc conv deg}), $\deg_{W^j}$ also converges to $\iota_d(\deg_W)$. It follows that $\iota_d(\deg_{W_0})=\iota_d(\deg_W)$, hence $\deg_{W_0}=\deg_W$, and therefore, $W_0=W$. (Here in the last step we used that the map (\ref{eq:M d Sym}) is injective.) This proves that {\bf (\ref{prop:loc conv deg}) implies (\ref{prop:loc conv g})} and concludes the proof of Proposition \ref{prop:loc conv S 1 C}.
\end{proof}
For the proof of Proposition \ref{PROP:CONV S 1 C} we also need the following lemma.
\begin{lemma}\label{le:Q Q'} Let $k$ be a positive integer, $\phi_0^\nu,\ldots,\phi_k^\nu$ be sequences of M\"obius transformations, and let
\[z_0,\ldots,z_{k-1},w_1,\ldots,w_k\in S^2\]
be points. Assume that $z_1\neq w_1,\ldots,z_{k-1}\neq w_{k-1}$, and that 
\[(\phi_i^\nu)^{-1}\circ\phi_{i+1}^\nu\to z_i,\]
uniformly on compact subsets of $S^2\wo \{w_{i+1}\}$, for $i=0,\ldots,k-1$. Let $Q\sub S^2\wo \{z_0\}$, $Q'\sub S^2\wo \{w_k\}$ be compact subsets. Then for $\nu$ large enough we have
\begin{equation}
\label{eq:phi 0 k}\phi_0^\nu(Q)\cap \phi_k^\nu(Q')=\emptyset.
\end{equation}
\end{lemma}

\begin{proof}[Proof of Lemma \ref{le:Q Q'}] This follows from an elementary argument. 
\end{proof}
\begin{proof}[Proof of Proposition \ref{PROP:CONV S 1 C}] \setcounter{claim}{0} {\bf Assume that (\ref{prop:conv S 1 C w}) holds}. Then Proposition \ref{prop:E d} and (\ref{eq:E V bar T}) imply that equality (\ref{eq:deg w nu}) holds for $\nu$ large enough, as claimed. The second part of condition (\ref{prop:conv S 1 C deg}) follows from Proposition \ref{prop:loc conv S 1 C}.

{\bf Suppose} now on the contrary that {\bf condition (\ref{prop:conv S 1 C deg})  holds}. Then Proposition \ref{prop:E d} and equality (\ref{eq:deg w nu}) imply (\ref{eq:E V bar T}). Let $\phi_\al^\nu$ (for $\al\in T$) be as in condition (\ref{prop:conv S 1 C deg}). We show \ref{defi:conv}(\ref{defi:conv w}): The first part of this condition (concerning $\al\in T_1$) follows from Proposition \ref{prop:loc conv S 1 C}. We prove the second part of the condition: Let $\al\in T_\infty$, and $Q \sub S^2\wo (Z_\al\cup\{z_{\al,0}\})$ be a compact subset. We denote by $\bar x_0$ the orbit $S^1\sub \C$. 
\begin{claim}\label{claim:u nu i j} For every subsequence $(\nu_i)_{i\in\N}$ there exists a further subsequence $(i_j)_{j\in\N}$, such that
\begin{eqnarray*}&u^{\nu_{i_j}}\circ\phi^{\nu_{i_j}}_\al(Q)\sub M^*=\C\wo\{0\},\quad\forall j\in\N,&\\
&G u^{\nu_{i_j}}\circ\phi^{\nu_{i_j}}_\al\to\bar x_0\textrm{ in }C^1\textrm{ on }Q\textrm{, as }j\to\infty.&
\end{eqnarray*}
\end{claim}
\begin{proof}[Proof of Claim \ref{claim:u nu i j}] We fix a subsequence $(\nu_i)_{i\in\N}$. We choose a M\"obius transformation $\psi$ such that $\psi(\infty)=z_{\al,0}$. By condition \ref{defi:conv}(\ref{defi:conv phi z}) we have
\[\phi_\al^\nu\circ\psi(\infty)=\phi_\al^\nu(z_{\al,0})=\infty,\]
and hence $\phi_\al^\nu\circ\psi$ restricts to an automorphism of $\C$. We define
\[W_\al^\nu:=(\phi_\al^\nu\circ\psi)^*W_\nu,\quad R_\al^\nu:=\big|\phi_\al^\nu\circ\psi(1)-\phi_\al^\nu\circ\psi(0)\big|\in(0,\infty).\]
Then $W_\al^\nu$ is an $R_\al^\nu$-vortex class over $\C$. We choose a representative $w_\al^\nu$ of $W_\al^\nu$. We check the hypotheses of Proposition \ref{prop:cpt mod} with the sequences $R_i:=R_\al^{\nu_i}$, $w_i:=w_\al^{\nu_i}$, and $R_0:=\infty$. By condition \ref{defi:conv}(\ref{defi:conv phi z}), $R_\al^{\nu_i}$ converges to $\infty$, as $i\to\infty$. Furthermore, by Proposition \ref{prop:S 1 C image} the image of $u_\al^{\nu_i}$ is contained in the compact set $K:=\BAR B_1\sub\C$. Finally, Proposition \ref{prop:E d} implies that the energies of the vortices $w_\al^{\nu_i}$ ($i\in\N$) are uniformly bounded. Hence the hypotheses of Proposition \ref{prop:cpt mod} are satisfied. Applying this result, there exist a finite subset $Z\sub \C$, an $\infty$-vortex $w_0:=(A_0,u_0)\in\WWW_{\C\wo Z}$, a subsequence $(i_j)$, and gauge transformations $g_j\in W^{2,p}_\loc\big(\C\wo Z,S^1\big)$, such that assertions of the proposition hold, with the sequence $(w_\nu)$ replaced by $(w_\al^{\nu_{i_j}})$. By condition \ref{prop:cpt mod}(\ref{prop:cpt mod:=}) the map $g_j^{-1}\big(u_{\nu_{i_j}}\circ\phi_\al^{\nu_{i_j}}\big)$ converges to $u_0\circ\psi^{-1}$ in $C^1$ on every compact subset of $\psi(\C\wo Z)=S^2\wo\psi\big(Z\cup\{\infty\}\big)$. Therefore, using the equality $\psi(\infty)=z_{\al,0}$, the statement of Claim \ref{claim:u nu i j} is a consequence of the following:
\begin{claim}\label{claim:Z Z al} The set $\psi(Z)$ is contained in $Z_\al$.
\end{claim}
\begin{pf}[Proof of Claim \ref{claim:Z Z al}] We choose a number $R>0$ so large that
\begin{equation}
\label{eq:be V Emin}\sum_{\be\in T_1}E\big(W_\be,\C\wo B_R\big)<\frac{\pi}4.
\end{equation}
Let $\be\in T_1$ be a vertex. We denote $W_\be^\nu:=(\phi_\be^\nu)^*W_\nu$. By what we have already proved, the restriction $W_\be^\nu|_{\BAR B_R}$ converges to $W_\be|_{\BAR B_R}$, as $\nu\to\infty$, with respect to the topology $\tau_{\BAR B_R}$. Therefore, using Lemmas \ref{le:X G} (Appendix \ref{sec:add}) and \ref{le:conv e}, it follows that 
\[E\big(W_\nu,\phi_\be^\nu(B_R)\big)=E(W_\be^\nu,B_R)\to E(W_\be,B_R),\]
as $\nu\to\infty$. We choose an index $j_0$ so large that for $j\geq j_0$ and every $\be\in T_1$ we have 
\begin{equation}
\label{eq:E w nu be}E(W_{\nu_{i_j}},\phi_\be^{\nu_{i_j}}(B_R))>E(W_\be,B_R)-\frac{\pi}{4|T_1|}.
\end{equation}
Let $\be\neq\ga\in T$ be vertices. We denote by $\big(\be,\be_1,\ldots,\be_{k-1},\ga\big)$ the chain of vertices connecting $\be$ with $\ga$, and write $\be_0:=\be$, $\be_k:=\ga$. By conditions \ref{defi:st}(\ref{defi:st dist}) and \ref{defi:conv}(\ref{defi:conv al be}) the hypothesis of Lemma \ref{le:Q Q'} with 
\[\phi_i^\nu:=\phi_{\be_i}^\nu,\quad z_i:=z_{\be_i\be_{i+1}},\quad w_i:=z_{\be_i\be_{i-1}}\]
are satisfied. It follows that for every compact subset $Q\sub S^2\wo Z_\be$ and $Q'\sub S^2\wo Z_\ga$, for $\nu$ large enough, we have
\begin{equation}
\label{eq:phi be nu Q}\phi_\be^\nu(Q)\cap \phi_\ga^\nu(Q')=\emptyset.
\end{equation}
Applying this several times with $\be,\ga\in T_1$ and $Q:=Q':=\bar B_R$, it follows that for $\nu$ large enough the sets $\phi_\be^\nu(\BAR B_R)$, $\be\in T_1$, are disjoint. Increasing $j_0$ we may assume w.l.o.g.~that this holds for $\nu\geq\nu_{i_{j_0}}$. Therefore, for $j\geq j_0$, we have
\begin{eqnarray}\nn E\Bigg(W_{\nu_{i_j}},\bigcup_{\be\in T_1}\phi_\be^{\nu_{i_j}}(B_R)\Bigg)&=&\sum_{\be\in T_1}E\big(W_{\nu_{i_j}},\phi_\be^{\nu_{i_j}}(B_R)\big)\\
\nn &>&\Bigg(\sum_{\be\in T_1}E(W_\be,B_R)\Bigg)-\frac\pi4\\
\label{eq:pi d Emin}&>&\sum_{\be\in T_1}E(W_\be,\C)-\frac\pi2.
\end{eqnarray}
Here in the second line we used (\ref{eq:E w nu be}) and in the third line we used (\ref{eq:be V Emin}). Our assumption (\ref{eq:deg w nu}) and Proposition \ref{prop:E d} imply that
\[\sum_{\be\in T_1}E(W_\be,\C)=E(W_\nu,\C),\]
for $\nu$ large enough. Hence (\ref{eq:pi d Emin}) implies that
\begin{equation}\label{eq:E W nu i j}E\Bigg(W_{\nu_{i_j}},\C\wo \bigcup_{\be\in T_1}\phi_\be^{\nu_{i_j}}(B_R)\Bigg)<\frac\pi2,
\end{equation}
for $j$ large enough. 

Let now $z\in S^2\wo (Z_\al\cup\{z_{\al,0}\})$. We show that $z$ does not belong to $\psi(Z)$. We choose a number $\eps>0$ so small that $\bar B_\eps(\psi^{-1}(z))\sub \C\wo \psi^{-1}(Z_\al)$. We define $Q:=\psi\big(\bar B_\eps(\psi^{-1}(z))\big)$. By (\ref{eq:phi be nu Q}), we have 
\[\phi_\al^\nu(Q)\cap \phi_\be^\nu(\bar B_R)=\emptyset,\]
for $\nu$ large enough. Therefore, by (\ref{eq:E W nu i j}) implies that 
\[E\big(W_{\nu_{i_j}},\phi_\al^{\nu_{i_j}}(Q)\big)<\frac\pi2,\]
for $j$ large enough. On the other hand, if $z$ belonged to $\psi(Z)$, then by condition \ref{prop:cpt mod}(\ref{prop:cpt mod lim nu E eps}) we would have 
\[\lim_{j\to\infty}E\big(W_{\nu_{i_j}},\phi_\al^{\nu_{i_j}}(Q)\big)=\lim_{j\to\infty}E\big(W_\al^{\nu_{i_j}},\bar B_\eps(\psi^{-1}(z))\big)\geq\Emin=\pi.\]
This contradiction proves that $z\not\in\psi(Z)$. It follows that $\psi(Z)\sub Z_\al\cup\{z_{\al,0}\}$. Since $\psi(\infty)=z_{\al,0}$, Claim \ref{claim:Z Z al} follows. This concludes the proof of Claim \ref{claim:u nu i j}.
\end{pf}
\end{proof}
Claim \ref{claim:u nu i j} and an elementary argument imply that $u_\nu\circ\phi_\al^\nu(Q)\sub M^*$ for $\nu$ large enough. Furthermore, it follows from the same claim and Lemma \ref{le:subsubseq} in Appendix \ref{sec:add} that $G u_\nu\circ\phi_\al^\nu\to\bar x_0$ in $C^1$ on $Q$, as $\nu\to\infty$. This proves the second part of \ref{defi:conv}(\ref{defi:conv w}). Hence all conditions of Definition \ref{defi:conv} are satisfied, and therefore, {\bf condition \ref{PROP:CONV S 1 C}(\ref{prop:conv S 1 C w}) holds}. This concludes the proof of Proposition \ref{PROP:CONV S 1 C}.
\end{proof}

\chapter{Fredholm theory for vortices over the plane}\label{chap:Fredholm}
In this chapter the equivariant homology class $[W]$ of an equivalence class $W$ of triples $(P,A,u)$ is defined, and the equivariant Chern number of $[W]$ is interpreted as a certain Maslov index. This number entered the index formula (\ref{eq:ind}). Furthermore, the second main result of this memoir, Theorem \ref{thm:Fredholm}, is proven. It states the Fredholm property for the vertical differential of the vortex equations over the plane $\C$, viewed as equations for equivalence classes $W$ of triples $(P,A,u)$.  Let $M,\om,G,\g,\lan\cdot,\cdot\ran_\g,\mu$ and $J$ be as in Chapter \ref{chap:main} %
\footnote{As always, we assume that hypothesis (H) (see Chapter \ref{chap:main}) is satisfied.}%
, and $(\Si,j):=\C$, equipped with the standard area form $\om_{\C}=\om_0$. 
\section{Equivariant homology, the Chern number, and the Maslov index}\label{SEC:HOMOLOGY CHERN} 
We fix numbers
\[p>2,\quad\lam>1-\frac2p.\] 
Then every equivalence class $W\in\B^p_\lam$ of triples $(P,A,u)\in\BB^p_\lam$ (defined as in (\ref{eq:BB p lam})) carries an equivariant homology class $[W]\in H_2^G(M,\Z)$, whose definition is based on the following lemma. Let $P$ be a topological (principal) $G$-bundle over $\C$
\footnote{Such a bundle is trivializable, but we do not fix a trivialization here.}%
, and $u:P\to M$ a continuous equivariant map. By an extension of the pair $(P,u)$ to $S^2$ we mean a triple $(\wt P,\iota,\wt u)$, where $\wt P$ is a topological $G$-bundle over $S^2$, $\iota:P\to\wt P$ a morphism of topological $G$-bundles which descends to the inclusion $\C\to\C\cup\{\infty\}=S^2$, and $\wt u:\wt P\to M$ a continuous $G$-equivariant map satisfying $\wt u\circ\iota=u$.
\begin{lemma}\label{le:P} The following statements hold.
\begin{enui}
\item\label{le:P:exists} For every triple $(P,A,u)\in\BB^p_\lam$ there exists an extension $(\wt P,\iota,\wt u)$ of the pair $(P,u)$, such that $G$ acts freely on $\wt u(\wt P_\infty)\sub M$, where $\wt P_\infty$ denotes the fiber of $\wt P$ over $\infty$.
\item\label{le:P:iso} Let $P$ be a topological $G$-bundle over $\C$, $u:P\to M$ a continuous $G$-equivariant map, and $(\wt P,\iota,\wt u)$ an extension of $(P,u)$, such that $G$ acts freely on $\wt u(\wt P_\infty)$. Furthermore, let $\big(P',u',\wt P',\iota',\wt u'\big)$ be another such tuple, and $\Phi:P'\to P$ an isomorphism of topological $G$-bundles, such that $u'=u\circ\Phi$. Then there exists a unique isomorphism of topological $G$-bundles $\wt\Phi:\wt P'\to\wt P$, satisfying 
\begin{equation}\label{eq:wt Phi}\wt\Phi\circ\iota'=\iota\circ\Phi,\quad\wt u'=\wt u\circ\wt\Phi.
\end{equation}
\end{enui}
\end{lemma}
The proof of this lemma is postponed to the appendix (page \pageref{proof:le:P}). For the definition of $[W]$ we also need the following. 
\begin{rmk}\label{rmk:homology} Let $G$ be a topological group, $X$ a closed%
\footnote{i.e., compact and without boundary}
 oriented topological (Hausdorff) manifold of dimension $k$, $P\to X$ a topological $G$-bundle, $Y$ a topological space, equipped with a continuous action of $G$, and $u:P\to Y$ a $G$-equivariant map. Then $u$ carries an equivariant homology class $[u]_G\in H_k^G(Y,\Z)$, defined as follows. We choose a universal $G$-bundle $\EG\to\BG$ and a continuous $G$-equivariant map $\theta:P\to\EG$.%
\footnote{Such a map descends to a classifying map $X\to\BG$.}
 The map $(u,\theta):P\to Y\x\EG$ descends to a map $u_G:X\to(Y\x\EG)/G$. We define 
\[[u]_G\in H_k^G(Y,\Z)=H_k\big((Y\x\EG)/G,\Z\big)\]
to be the push-forward of the fundamental class of $X$, under the map $u_G$. This class does not depend on the choice of $\theta$, nor on $\EG$ in the following sense. If $\EG'\to\BG'$ is another universal $G$-bundle, then the corresponding class $[u]_G'\in H_k\big((Y\x\EG')/G,\Z\big)$ is mapped to $[u]_G$ under the canonical isomorphism%
\footnote{This isomorphism is induced by an arbitrary continuous equivariant map from $\EG'$ to $\EG$.}%
\[H_k\big((Y\x\EG')/G,\Z\big)\to H_k\big((Y\x\EG)/G,\Z\big).\Box\]
\end{rmk}
\begin{defi}[Equivariant homology class]\label{defi:homology} We define the \emph{equivariant homology class $[W]\in H_2^G(M,\Z)$} of an element
\[W\in\Bigg(\bigcup_{p>2,\,\lam>1-\frac2p}\BB^p_\lam\Bigg)\Slash\sim_p\] 
to be the equivariant homology class of $\wt u:\wt P\to M$, where $(\wt P,\iota,\wt u)$ is an extension as in Lemma \ref{le:P}, corresponding to any representative $(P,A,u)$ of $W$. Here the equivalence relation $\sim_p$ is defined as in Section \ref{sec:Fredholm}.
\end{defi}
It follows from Lemma \ref{le:P} that this class does not depend on the choices of $(\wt P,\iota,\wt u)$ and $w$. 
\begin{rmk}\label{rmk:W well-defined} The condition $\lam>1-2/p$ is needed for $[W]$ to be well-defined, for every $W\in\B^p_\lam$. Consider for example the case $M:=\R^2,\om:=\om_0,J:=i$ and $G:=\{\one\}$. Let $p>2$. We choose a smooth map $u:\C\x\{\one\}\iso\C\to\R^2$ such that
\[u(z)=\left(\sin\left(\sqrt{\log|z|}\right),0\right),\quad\forall z\in\C\wo B_1.\]
Then the equivalence class $W$ of $\big(P:=\C\x\{\one\},0,u\big)$ lies in $\B^p_\lam$, for every $\lam\leq1-2/p$. However, there is no extension of $(P,u)$ as in Lemma \ref{le:P}, since $u(z)$ diverges, as $|z|\to\infty$. Therefore, $[W]$ is not well-defined. $\Box$
\end{rmk}
\begin{rmk}\label{rmk:homology vortex} Let $p>2$ and $W\in\B^p_\loc=\BB^p_\loc\Slash\sim_p$ be a gauge equivalence class of triples $(P,A,u)$ of Sobolev class.%
\footnote{Recall from Section \ref{sec:Fredholm} that each such triple consists of a $G$-bundle of class $W^{2,p}_\loc$, and a connection $A$ and a $G$-equivariant map $u:P\to M$ of class $W^{1,p}_\loc$.}
 Assume that every representative of $W$ satisfies the vortex equations (\ref{eq:BAR dd J A u},\ref{eq:F A mu}), the energy of $W$ is finite, and the closure of its image compact. As explained in Remark \ref{rmk:geometry} in Section \ref{sec:remarks}, we have $W\in\B^p_\lam$, for every $\lam<2-2/p$. Therefore, the equivariant homology class of $W$ is well-defined. $\Box$
\end{rmk}
The contraction appearing in formula (\ref{eq:ind}) has a concrete geometric meaning. It can be interpreted as the Maslov index of a certain loop of linear symplectic transformations, as follows. Let $(M,\om)$ be a symplectic manifold and $G$ a connected compact Lie group, acting on $M$ in a Hamiltonian way, with momentum map $\mu$. Assume that $G$ acts freely on $\mu^{-1}(0)$. Let $\Si$ be a compact, connected, oriented topological surface%
\footnote{i.e., real two-dimensional topological manifold}
 with non-empty boundary. We associate to this data a map $m_{\Si,\om,\mu}$, called ``Maslov index'', as follows. The domain of this map consists of the weak $(\om,\mu)$-homotopy classes of $(\om,\mu)$-admissible maps from $\Si\to M$, and it takes values in the even integers. 

Here we call a continuous map $u:\Si\to M$ $(\om,\mu)$-admissible, iff for every connected component $X$ of the boundary $\dd\Si$ there exists a $G$-orbit which is contained in $\mu^{-1}(0)$, and contains the set $u(X)$. We denote by $C\big(\Si,M;\om,\mu\big)$ the set of such maps. We call two maps $u_0,u_1\in C\big(\Si,M;\om,\mu\big)$ \emph{weakly $(\om,\mu)$-homotopic}, iff they are homotopic through such maps.%
\footnote{This means that there exists a continuous map $u:[0,1]\x\Si\to M$, such that $u(i,\cdot)=u_i$, for $i=0,1$, and for every $t\in[0,1]$ we have $u(t,\cdot)\in C\big(\Si,M;\om,\mu\big)$.}
 This defines an equivalence relation on $C\big(\Si,M;\om,\mu\big)$. We call the corresponding equivalence classes \emph{weak $(\om,\mu)$-homotopy classes}. 
\begin{xpl}\label{xpl:S 1 C n} Consider $\R^{2n}=\C^n$, equipped with the standard symplectic form $\om:=\om_0$, and the action of $S^1\sub\C$ given by 
\[z\cdot(z_1,\ldots,z_n):=\big(zz_1,\ldots,zz_n\big),\]
with momentum map 
\begin{equation}\label{eq:mu C n}\mu:\C^n\to(\Lie S^1)^*\iso\Lie S^1=i\R,\quad\mu(z_1,\ldots,z_n):=\frac i2\Bigg(1-\sum_{j=1}^n|z_j|^2\Bigg).\end{equation}
We have $\mu^{-1}(0)=S^{2n-1}\sub\C^n$, and the $S^1$-orbits contained in the unit sphere are the fibers of the Hopf fibration $S^{2n-1}\to\CP^{n-1}$. Consider the case in which $\Si$ is the unit disk $\DDD\sub\C$. Then for each integer $d\in\Z$ and each vector $v\in S^{2n-1}$ the map 
\begin{equation}\label{eq:u}u_{d,v}:\DDD\to\C^n,\quad u_{d,v}(z):=z^dv
\end{equation}
is $(\om_0,\mu)$-admissible. Furthermore, $u_{d,v}$ and $u_{d',v'}$ are weakly $(\om_0,\mu)$-homotopic if and only if $d=d'$. To see this, note that if $d=d'$ then a homotopy between the two maps is given by $[0,1]\x\DDD\ni(t,z)\mapsto z^dv(t)$, where $v\in C\big([0,1],S^{2n-1}\big)$ is any path joining $v$ and $v'$. Conversely, assume that $u_{d,v}$ and $u_{d',v'}$ are weakly $(\om_0,\mu)$-homotopic. Then by an elementary argument, there exists a \emph{smooth} weak $(\om_0,\mu)$-homotopy $h:[0,1]\x\DDD\to\C^n$ between these maps. It follows that 
\[d'\pi=\int_\DDD u_{d',v'}^*\om_0=\int_\DDD u_{d,v}^*\om_0+\int_{[0,1]\x\dd\DDD}h^*\om_0=d\pi+0,\]
where in the last equality we used the fact that for every $v\in T\dd\DDD$, $dh(t,\cdot)v$ is tangent to the characteristic distribution on $S^{2n-1}$
\footnote{This follows from the fact that by assumption, $h(t,\dd\DDD)$ is contained in a Hopf circle, for every $t\in[0,1]$.}%
, and therefore $\om_0\big(\cdot,dh(t,\cdot)v\big)$ vanishes on vectors tangent to $S^{2n-1}$. Therefore, we have $d=d'$. This proves the claimed equivalence. $\Box$
\end{xpl}

The definition of the map $m_{\Si,\om,\mu}$ is based on the following Maslov index of a regular symplectic transport over a curve. Let $X$ be an oriented, connected, closed topological curve%
\footnote{i.e., a real one-dimensional topological manifold}%
, and $(V,\om)$ a symplectic vector space. We denote by $\Aut\om$ the group of linear symplectic automorphisms of $V$. We define a \emph{regular symplectic transport over $X$} to be a continuous map $\Phi:X\x X\to\Aut\om$ satisfying 
 \[\Phi(z,z)=\one,\quad\Phi(z'',z)=\Phi(z'',z')\Phi(z',z),\quad z,z',z''\in X.\] 
 We define the \emph{Maslov index} of such a map to be 
\begin{equation}\label{eq:m X om}m_{X,\om}(\Phi):=2\deg\left(X\ni z\mapsto\rho_\om\circ\Phi(z,z_0)\in S^1\right).
\end{equation}
Here $z_0\in X$ is an arbitrary point, $\rho_\om:\Aut(\om)\to S^1$ denotes the Salamon-Zehnder map (see \cite[Theorem 3.1.]{SZ}), and $\deg$ the degree of a map from $X$ to $S^1$. This definition does not depend on the choice of $z_0$, by some homotopy argument. 

Let now $M,\om,G,\mu,\Si$ be as before, and $a$ a weak $(\om,\mu)$-homotopy class. To define $m_{\om,\mu}(a)$, we choose a representative $u:\Si\to M$ of $a$, a symplectic vector space $(V,\Om)$ of dimension $\dim M$ and a symplectic trivialization $\Psi:\Si\x V\to u^*TM$.%
\footnote{Such a trivialization exists, since the group $\Aut\Om$ is connected, and the surface $\Si$ deformation retracts onto a wedge of circles. Here we use that $\Si$ is connected and has non-empty boundary.}
 Let $X$ be a connected component of $\dd\Si$. We define the map 
\[\Phi_X:X\x X\to\Aut\Om,\quad\Phi_X(z',z):=\Psi_{z'}^{-1}g_{z',z}\cdot\Psi_z.\]
Here $g_{z',z}\in G$ denotes the unique element satisfying $u(z')=g_{z',z}u(z)$, $g_{z',z}\cdot$ denotes the induced action of $g_{z',z}$ on $TM$, and $\Psi_z:=\Psi(z,\cdot)$. The map $\Phi_X$ is a regular symplectic transport. 
\begin{defi}[Maslov index]\label{defi:Maslov} We define 
\begin{equation}\label{m Si om mu a}m_{\Si,\om,\mu}(a):=\sum m_{X,\Om}(\Phi_X),
\end{equation}
where the sum runs over all connected components $X$ of $\dd\Si$. 
\end{defi}
This definition does not depend on the choices of the symplectic vector space $(V,\Om)$ and the trivialization $\Psi$. This follows from the fact that (\ref{m Si om mu a}) agrees with the Maslov index associated with $a$ and the coisotropic submanifold $\mu^{-1}(0)\sub M$, as defined in \cite{ZiMaslov}. See \cite[Lemma 45]{ZiMaslov}.
\begin{Rmk} This Maslov index is based on a definition by D.~A.~Salamon and R.~Gaio (for $\Si=\DDD$), see \cite{GS}. $\Box$
\end{Rmk}
\begin{xpl}\label{xpl:u d v} Consider the situation of Example \ref{xpl:S 1 C n}, with $u_{d,v}$ defined as in (\ref{eq:u}). This map carries the Maslov index 
\[m_{\DDD,\om_0,\mu}([u_{d,v}])=2dn,\]
where $[\cdots]$ denotes the $(\om_0,\mu)$-homotopy class. This follows from a straight-forward calculation. $\Box$
\end{xpl}
We now return to the setting of the beginning of this section. We define $\Si$ to be the compact oriented topological surface obtained from $\C$ by ``gluing a circle at $\infty$'' to it.%
\footnote{This surface is homeomorphic to the closed disk $\DDD\sub\C$.}
 The class $W$ carries a weak $(\om,\mu)$-homotopy class, whose definition relies on the following lemma.
\begin{lemma}\label{le:section} Let $p>2$, $\lam>1-2/p$, and $(P,A,u)\in\BB^p_\lam$. There exists a section $\si:\C\to P$ of class $W^{1,p}_\loc$, such that the map $u\circ\si:\C\to M$ continuously extends to $\Si$. Furthermore, for every such section $\si$ the corresponding extension $\wt u:\Si\to M$ is $(\om,\mu)$-admissible. 
\end{lemma}
The proof of this lemma is postponed to the appendix (page \pageref{proof:le:section}). Let 
\[W\in\Bigg(\bigcup_{p>2,\,\lam>1-\frac2p}\BB^p_\lam\Bigg)\Slash\sim_p.\]
\begin{defi}\label{defi:homotopy W} We define the \emph{$(\om,\mu)$-homotopy class of $W$} to be the weak $(\om,\mu)$-homotopy class of the continuous extension $\wt u:\Si\to M$ of the map $u\circ\si:\C\to M$. Here $(P,A,u)$ is a representative of $W$, and $\si$ a section of $P$ as in the previous lemma.
\end{defi}
By Lemma \ref{le:homotopic} in Appendix \ref{sec:proofs homology} this homotopy class is independent of the choice of the section $\si$. Furthermore, it follows from a straight-forward argument that it is independent of the choice of the representative $(P,A,u)$. The contraction appearing in formula (\ref{eq:ind}) can now be expressed as follows. 
\begin{prop}[Chern number and Maslov index]\label{prop:Chern Maslov} We have
\[2\big\lan c_1^G(M,\om),[W]\big\ran=m_{\Si,\om,\mu}\big((\om,\mu)\textrm{-homotopy class of }W\big).\]
\end{prop}
The proof of this result is postponed to the appendix (page \pageref{proof:prop:Chern Maslov}). 
\begin{Xpl}[Maslov index of a vortex] Let $(M,\om,J,G):=(\R^2,\om_0,i,S^1)$, $\lan\cdot,\cdot\ran_\g$ be the standard inner product on $\g:=\Lie(S^1)=i\R$, the action of $S^1\sub\C$ on $\R^2=\C$ be given by multiplication, with momentum map as in (\ref{eq:mu C n}) with $n=1$, $\Si:=\C$, equipped with the area form $\om_0$, and $W$ a gauge equivalence class of smooth finite energy vortices. By Proposition \ref{prop:S 1 C image} and Remark \ref{rmk:homology vortex} the equivariant homology class $[W]$ is well-defined. Furthermore, we have
\[\big\lan c_1^G(M,\om),[W]\big\ran=\deg(W),\] 
where the degree $\deg(W)$ is defined as in (\ref{eq:deg W sum}). This follows from Proposition \ref{prop:Chern Maslov} and a straight-forward calculation of the Maslov index of the $(\om,\mu)$-homotopy class of $W$. $\Box$
\end{Xpl}
\section{Proof of the Fredholm result}\label{sec:proof:Fredholm}
In this section Theorem \ref{thm:Fredholm} is proved, based on a Fredholm theorem for the augmented vertical differential and the existence of a bounded right inverse for $L_w^*$, the formal adjoint of the infinitesimal action of the gauge group on pairs $(A,u)$. In the present section we always assume that 
\[\bar n:=(\dim M)/2-\dim G>0.\] 
\subsection{Reformulation of the Fredholm theorem}\label{subsec:reform}
In this subsection we state the two results mentioned above and deduce Theorem \ref{thm:Fredholm} from them. To formulate the first result, let $p>2$, $\lam\in\R$, $\BB^p_\lam$ be defined as in (\ref{eq:BB p lam}), and $w:=(P,A,u)\in\BB^p_\lam$. We denote 
\[\im L:=\{(x,L_x\xi)\,|\,x\in M,\,\xi\in\g\},\] 
and by $\PR:TM\to TM$ the orthogonal projection onto $\im L$. $\PR$ induces an orthogonal projection 
\[{\PR}^u:TM^u:=(u^*TM)/G\to TM^u\] 
onto $(u^*\im L)/G$. Recall that $\wone(\g_P)$ denotes the bundle of one-forms on $\R^2$ with values in $\g_P$. We write
\[{\PR^u}\ze:=(\al,{\PR^u}v),\quad\forall\ze=(\al,v)\in\wone(\g_P)\oplus TM^u.\]
Note that $\im L$ is in general not a subbundle of $TM$, since the dimension of $\im L_x$ may vary with $x\in M$. Recall that $\Ga^{1,p}_\loc(E)$ denotes the space of $W^{1,p}_\loc$-sections of a vector bundle $E$. For $\ze\in\Ga^{1,p}_\loc\big(\wone(\g_P)\oplus TM^u\big)$ we define 
\[\Vert\ze\hhat\Vert_{w,p,\lam}:=\Vert\ze\Vert_{w,p,\lam}+\Vert {\PR}^u\ze\Vert_{p,\lam},\] 
where $\Vert\ze\Vert_{w,p,\lam}$ is defined as in (\ref{eq:ze w p lam}). Recall the definition (\ref{eq:YY w}) of $\YY_w^{p,\lam}$. We denote by $\Ga^p_\lam(\g_P)$ the space of $L^p_\lam$-sections of $\g_P$. We define 
\begin{eqnarray}\label{eq:X w p lam}&\XXX_w:=\XXX_w^{p,\lam}:=\big\{\ze\in\Ga^{1,p}_\loc\big(\wone(\g_P)\oplus TM^u\big)\,\big|\,\Vert\ze\hhat\Vert_{w,p,\lam}<\infty\big\},&\\ 
\label{eq:Y w p lam}&\YYY_w:=\YYY_w^{p,\lam}:=\YY_w^{p,\lam}\oplus\Ga^p_\lam(\g_P),&
\end{eqnarray}
Recall the definition (\ref{eq:L w *}) of the formal adjoint map for the infinitesimal action of the gauge group on pairs $(A,u)$. Restricting the domain and target, this becomes the operator
\[L_w^*:\XXX_w^{p,\lam}\to\Ga^p_\lam(\g_P),\quad L_w^*(\al,v):=-d_A^*\al+L_u^*v.\]
It follows from the fact $L_x^*=L_x^*\PR_x$ (for every $x\in M$) and compactness of $\BAR{u(P)}$ that this operator is well-defined and bounded. We define the \emph{augmented vertical differential of the vortex equations over $\C$ at $w$} to be the map
\begin{eqnarray}
\label{eq:DD w p lam}&\DD_w:=\DD_w^{p,\lam}:\XXX_w^{p,\lam}\to\YYY_w^{p,\lam},&\\
\label{eq:DD w ze}&\DD_w\ze:=\left(\begin{array}{c}
\big(\na^Av+L_u\al\big)^{0,1}-\frac12J(\na_vJ)(d_Au)^{1,0}\\
d_A\al+\om_0\, d\mu(u)v\\
L_w^*\ze
\end{array}\right).&
\end{eqnarray}
\begin{thm}[Fredholm property for the augmented vertical differential]\label{thm:Fredholm aug} Let $p>2$ and $\lam>-2/p+1$ be real numbers, and $w:=(P,A,u)\in\BB^p_{\lam}$. Then the following statements hold. 
  \begin{enui}\item\label{thm:Fredholm X w} The normed spaces $\big(\XXX^{p,\lam}_w,\Vert\cdot\hhat\Vert_{w,p,\lam}\big),\YY^{p,\lam}_w$ and $\Ga^p_\lam(\g_P)$ are complete. 
\item\label{thm:Fredholm DD w} Assume that $-2/p+1<\lam<-2/p+2$. Then the operator $\DD_w$ (as in (\ref{eq:DD w p lam},\ref{eq:DD w ze})) is Fredholm of real index 
\[\ind\DD_w=2\bar n+2\big\lan c_1^G(M,\om),[[w]]\big\ran,\]
where $[w]$ denotes the equivalence class of $w$, and $[[w]]$ denotes the equivariant homology class of $[w]$.
\end{enui}
\end{thm}
This theorem will be proved in Section \ref{subsec:proof:Fredholm aug}, based on the existence of a suitable trivialization of $\wone(\g_P)\oplus TM^u$ in which the operator $\DD_w$ becomes standard up to a compact perturbation. 

The second ingredient of the proof of Theorem \ref{thm:Fredholm} is the following.
\begin{thm}\label{thm:L w * R} Let $p>2$, $\lam>1-2/p$, and $w:=(P,A,u)\in\BB^p_\lam$. Then the map $L_w^*:\XXX_w^{p,\lam}\to\Ga^p_\lam(\g_P)$ admits a bounded (linear) right inverse. 
\end{thm}
The proof of Theorem \ref{thm:L w * R} is postponed to Section \ref{subsec:proof:thm:L w * R} (page \pageref{proof:thm:L w * R}). It is based on the existence of a bounded right inverse for the operator $d_A^*$ over a compact subset of $\R^n$ diffeomorphic to $\bar B_1$ (Proposition \ref{PROP:RIGHT}) and the existence of a neighborhood $U\sub M$ of $\mu^{-1}(0)$, such that
\[\inf\big\{|L_x\xi|\,\big|\,x\in U,\,\xi\in\g:\,|\xi|=1\big\}>0.\]
Recall the definition (\ref{eq:M *}) of the subset $M^*$ of $M$. Theorem \ref{thm:Fredholm} can be reduced to Theorems \ref{thm:Fredholm aug} and \ref{thm:L w * R}, as follows:
\begin{proof}[Proof of Theorem \ref{thm:Fredholm} (p.~\pageref{thm:Fredholm})]\setcounter{claim}{0} Let $p>2$, $\lam>1-2/p$, and $W\in\B^p_\lam$. 

{\bf We prove statement (\ref{thm:Fredholm X w})}. 
\begin{Claim}We have
\[\XX_w:=\XX_w^{p,\lam}=K:=\ker\big(L_w^*:\XXX_w^{p,\lam}\to\Ga^p_\lam(\g_P)\big),\]
and the restriction of the norm $\Vert\cdot\hhat\Vert_{w,p,\lam}$ to $\XX_w$ is equivalent to $\Vert\cdot\Vert_{w,p,\lam}$. 
\end{Claim}
\begin{proof}[Proof of the claim] It suffices to prove that $\XX_w\sub K$ and this inclusion is bounded. Hypothesis (H) implies that there exists $\de>0$ such that $\mu^{-1}(\bar B_\de)\sub M^*$. We have 
\[c:=\min\big\{|L_x\xi|\,\big|\,x\in\mu^{-1}(\bar B_\de),\,\xi\in\g:\,|\xi|=1\big\}>0.\] 
It follows from Lemma \ref{le:si} in Appendix \ref{sec:proofs homology} that there exists $R>0$ such that $u(P|_{\C\wo B_R})\sub \mu^{-1}(\bar B_\de)$. Let $\ze=(\al,v)\in\XX_w$. Then $L_u^*v=d_A^*\al$, and thus, using the last assertion of Remark \ref{rmk:c} in Appendix \ref{sec:add},
\[\Vert{\PR}^u v\Vert_{p,\lam}\leq c^{-1}\Vert L_u^*v\Vert_{p,\lam}\leq c^{-1} \Vert\na^A\al\Vert_{p,\lam}\leq c^{-1} \Vert\ze\Vert_{w,p,\lam}<\infty.\] 
Hence $\XX_w\sub K$, and this inclusion is bounded. This proves the claim.
\end{proof}
{\bf Statement (\ref{thm:Fredholm X w})} follows from Theorem \ref{thm:Fredholm aug}(\ref{thm:Fredholm:X Y}) and the claim.

{\bf Statement (\ref{thm:Fredholm:DDD})} follows from Theorem \ref{thm:Fredholm aug}(\ref{thm:Fredholm DD w}), Theorem \ref{thm:L w * R} and Lemma \ref{le:X Y Z} (Appendix \ref{sec:add}) with
\[X:=\XXX_w,\quad Y:=\YY_w,\quad Z:=\Ga^p_\lam(\g_P),\quad T:=L_w^*,\quad D':\XXX_w\to\YY_w,\]
where $D'\ze$ is defined to be the vector consisting of the first and second rows of $\DD_w\ze$ (as in (\ref{eq:DD w ze})). This proves Theorem \ref{thm:Fredholm}.
\end{proof}
\subsection{Proof of Theorem \ref{thm:Fredholm aug} (augmented vertical differential)}\label{subsec:proof:Fredholm aug}
This subsection contains the core of the proof of Theorem \ref{thm:Fredholm aug}. Here we introduce the notion of a good complex trivialization, and state an existence result for such a trivialization (Proposition \ref{prop:triv}). Furthermore, we formulate a result saying that every good trivialization transforms $\D_w$ into a compact perturbation of the direct sum of copies of $\dd_{\bar z}$ and a certain matrix operator (Proposition \ref{prop:X X w}). The results of this subsection will be proved in Subsection \ref{subsec:proofs}. 

We denote by $s$ and $t$ the standard coordinates in $\R^2=\C$. For $v\in\R^n$ and $d\in \Z$ we denote 
\[\lan v\ran:=\sqrt{1+|v|^2},\quad p_d:\C\to \C,\,p_d(z):=z^d.\] 
We equip the bundle $\wone(\g_P)$ with the (fiberwise) complex structure $J_P$ defined by $J_P\al:=-\al\,i$. Furthermore, we denote 
\[\g^\C:=\g\otimes_\R\C,\quad V:=\C^{\bar n}\oplus\g^\C\oplus\g^\C,\] 
and for $a\in\C$ we use the notation
\[a\cdot\oplus\id:V\to V,\quad\big(v^1,\ldots,v^{\bar n},\al,\be\big)\mapsto \big(av^1,v^2,\ldots,v^{\bar n},\al,\be\big).\] 
For $x\in M$ we write 
\[L^\C_x:\g^\C\to T_xM\] 
for the complex linear extension of $L_x$. We define 
\[H_x:=\ker d\mu(x)\cap(\im L_x)^\perp,\,\forall x\in M.\]
Note that in general, the union $H$ of all the $H_x$'s is not a smooth subbundle of $TM$, since the dimension of $H_x$ may depend on $x$. However, there exists an open neighborhood $U\sub M$ of $\mu^{-1}(0)$ such that $H|_U$ is a subbundle of $TM|_U$. Let $p>2$, $\lam>-2/p+1$ and $w:=(P,A,u)\in\BB^p_\lam$. We denote
\begin{equation}\label{eq:d lan}d:=\big\lan c_1^G(M,\om),[W]\big\ran.
\end{equation}
For $z\in\C$ we define 
\[H^u_z:=\big\{G\cdot(p,v)\,\big|\,p\in\pi^{-1}(z)\sub P,\,v\in H_{u(p)}\big\}.\] 
Consider a complex trivialization (i.e, a bundle isomorphism descending to the identity on the base $\C$)
\[\Psi:\C\x V\to\wone(\g_P)\oplus TM^u.\]
\begin{defi}\label{defi:triv} We call $\Psi$ \emph{good}, if the following properties are satisfied.
\begin{enui}\item {\bf (Splitting)}\label{defi:triv split} For every $z\in\C$ we have
  \begin{align}
    \label{eq:Psi z C bar n}\Psi_z(\C^{\bar n}\oplus\g^\C\oplus\{0\})=&\,\{0\}\oplus TM^u_z,\\
\label{eq:Psi z 0} \Psi_z(\{0\}\oplus\{0\}\oplus\g^\C)=&\,\wone(\g_P)\oplus\{0\}.
  \end{align}
Furthermore, there exists a number $R>0$, a section $\si$ of $P\to\C\wo B_1$, of class $W^{1,p}_\loc$, and a point $x_\infty\in\mu^{-1}(0)$, such that the following conditions are satisfied. For every $z\in\C\wo B_R$ we have
\begin{eqnarray}\label{eq:Psi infty x H}&\Psi_z(\C^{\bar n}\oplus\{0\}\oplus\{0\})=H^u_z,&\\
\nn&u\circ\si(re^{i\phi})\to x_\infty,\textrm{ uniformly in }\phi\in\R,\textrm{ as }r\to\infty,&\\
\nn&\si^*A\in L^p_\lam(\C\wo B_1,\g),& 
\end{eqnarray}
and for every $z\in\C\wo B_R$ and $\big(\al,\be=\phi+i\psi\big)\in\g^\C\oplus\g^\C$, we have
\begin{equation}\label{eq:Psi 0 al be}\Psi_z(0,\al,\be)=\big(G\cdot\big(\si(z),\phi ds+\psi dt\big),G\cdot\big(u\circ\si(z),L^\C_{u\circ\si(z)}(\al)\big)\big).
\end{equation}
\item\label{defi:triv C} There exists a number $C>0$ such that for every $(z,\ze)\in\C\x V$
  \begin{equation}
    \label{eq:C}C^{-1}|\ze|\leq \big|\Psi_z(\lan z\ran^d\cdot\oplus\id)\ze\big|\leq C|\ze|.
  \end{equation}
\item\label{defi:triv na} We have $\big|\na^A\big(\Psi(p_d\cdot\oplus\id)\big)\big|\in L^p_\lam(\C\wo B_1)$.
\end{enui}
\end{defi}
The first ingredient of the proof of Theorem \ref{thm:Fredholm aug} is the following result.
\begin{prop}\label{prop:triv} If $p>2$, $\lam>-2/p+1$, and $w:=(P,A,u)\in\BB^p_\lam$, then there exists a good complex trivialization of $\wone(\g_P)\oplus TM^u$. 
\end{prop}
The proof of this proposition is postponed to subsection \ref{subsec:proofs} (page \pageref{proof:triv}). The next result shows that a good trivialization transforms $\DD_w$ into a compact perturbation of some standard operator. We denote $\N_0:=\N\cup\{0\}$ and
\[|\al|:=\sum_{i=1}^n\al_i,\quad\dd^\al:=\dd_1^{\al_1}\cdots\dd_n^{\al_n},\quad\forall\al=(\al_1,\ldots,\al_n)\in\N_0^n.\] 
Let $1\leq p\leq \infty$, $n\in\N$, $k\in\N_0$, $\lam\in\R$, $\Om\sub\R^n$ be an open subset, $W$ a real or complex vector space,  and $u:\Om\to W$ a $k$-times weakly differentiable map. We define 
\begin{eqnarray}\nn\Vert u\Vert_{L^{k,p}_\lam(\Om,W)}&:=&\sum_{|\al|\leq k}\big\Vert\langle \cdot\rangle^{\lam+|\al|}\dd^\al u\big\Vert_{L^p(\Om,W)}\in[0,\infty],\\
\nn\Vert u\Vert_{W^{k,p}_\lam(\Om,W)}&:=&\sum_{|\al|\leq k}\Vert\langle\cdot\rangle^\lam\dd^\al u\Vert_{L^p(\Om,W)}\in[0,\infty],\\
\label{eq:L k p}L^{k,p}_\lam(\Om,W)&:=&\big\{u\in W^{k,p}_\loc(\Om,W)\,|\,\Vert u\Vert_{L^{k,p}_\lam(\Om,W)}<\infty\big\}\\
\label{eq:W k p} W^{k,p}_\lam(\Om,W)&:=&\big\{u\in W^{k,p}_\loc(\Om,W)\,|\,\Vert u\Vert_{W^{k,p}_\lam(\Om,W)}<\infty\big\}.
\end{eqnarray}
If $(X_i,\Vert\cdot\Vert_i)$, $i=1,\ldots,k$, are normed vector spaces then we endow $X_1\oplus\cdots\oplus X_k$ with the norm $\Vert(x_1,\ldots,x_k)\Vert:=\sum_i\Vert x_i\Vert_i$. Let $d\in\Z$. If $d<0$ then we choose $\rho_0\in C^\infty(\C,[0,1])$ such that $\rho_0(z)=0$ for $|z|\leq1/2$ and $\rho_0(z)=1$ for $|z|\geq1$. In the case $d\geq0$ we set $\rho_0:=1$. The isomorphism of Lemma \ref{le:X d iso} (Appendix \ref{sec:weighted}) induces norms on the vector spaces
\begin{equation}\label{eq:XXX' p lam d XXX'' p lam}\XXX'_{p,\lam,d}:=\C\rho_0p_d+L^{1,p}_{{\lam-1}-d}(\C,\C),\quad\XXX''_{p,\lam}:=\C^{\bar n-1}+L^{1,p}_{\lam-1}(\C,\C^{\bar n-1}).\end{equation} 
We define
\[\XXX_d:=\XXX^{p,\lam}_d:=\XXX'_{p,\lam,d}\oplus\XXX''_{p,\lam}\oplus W^{1,p}_{\lam}(\C,\g^\C\oplus\g^\C),\] 
\[\YYY_d:=\YYY^{p,\lam}_d:=L^p_{\lam-d}(\C,\C)\oplus L^p_{\lam}\big(\C,\C^{\bar n-1}\oplus\g^\C\oplus\g^\C\big).\] 
For a complex vector space $W$ we denote by $\dd^W_{\bar z}$ ($\dd^W_z$) the operator $\frac12(\dd_s+i\dd_t)$ ($\frac12(\dd_s-i\dd_t)$) acting on functions from $\C$ to $W$. We denote by $\lan\cdot,\cdot\ran_\g^\C$ the hermitian inner product on $\g^\C$ (complex anti-linear in its first argument) extending $\lan\cdot,\cdot\ran_\g$. Furthermore, we denote by $\wzeroone(TM^u)$ the bundle of complex anti-linear one-forms on $\C$ with values in $TM^u$, and define the isomorphisms 
\begin{eqnarray*}&F_1:TM^u\to \wzeroone(TM^u),\quad F_1(v):=(ds-Jdt)v,&\\
&F_2:\wone(\g_P)\to \wtwo(\g_P)\oplus\g_P,\quad F_2(\phi ds+\psi dt):=(\psi ds\wedge dt,\phi),&\\ 
&F:\wone(\g_P)\oplus TM^u\to\wzeroone(TM^u)\oplus\wtwo(\g_P)\oplus\g_P,\quad F(\al,v):=(F_1v,F_2\al).&
\end{eqnarray*}
We are now ready to formulate the second ingredient of the proof of Theorem \ref{thm:Fredholm aug}:
\begin{prop}[Operator in good trivialization]\label{prop:X X w} Let $p>2$, $\lam>-2/p+1$, $w:=(P,A,u)\in \BB^p_{\lam}$, and
\[\Psi:\C\x V\to\wone(\g_P)\oplus TM^u\]
be a good trivialization. We define $d$ as in (\ref{eq:d lan}). The following statements hold.
  \begin{enui}\item \label{prop:X X w iso} The following maps are well-defined isomorphisms of normed spaces:
  \begin{equation} \label{eq:X Y Psi}\XXX_d\ni\ze\mapsto \Psi\ze\in \XXX_w,\quad\YYY_d\ni\ze\mapsto F\Psi\ze\in \YYY_w
  \end{equation}
\item\label{prop:X X w Fredholm} There exists a positive%
\footnote{This means that $\lan S_\infty v,v\ran_\g^\C>0$ for every $0\neq v\in\g^\C$.}
 $\C$-linear map $S_\infty:\g^\C\to \g^\C$ such that the following operator is compact:
\begin{equation}\label{eq:F Psi DD}S:=(F\Psi)^{-1}\DD_w \Psi-\dd^{\C^{\bar n}}_{\bar z}\oplus\left(
  \begin{array}{cc}
\dd^{\g^\C}_{\bar z}&\id/2\\
S_\infty&2\dd^{\g^\C}_z
  \end{array}
\right):\XXX_d\to\YYY_d
\end{equation}
\end{enui}
\end{prop}
The proof of Proposition \ref{prop:X X w} is postponed to subsection \ref{subsec:proofs} (page \pageref{proof:X X w}). It is based on some inequalities and compactness properties for weighted Sobolev spaces and a Hardy-type inequality (Propositions \ref{prop:lam d Morrey} and \ref{prop:Hardy} in Appendix \ref{sec:weighted}).
\begin{proof}[Proof of Theorem \ref{thm:Fredholm aug} (p.~\pageref{thm:Fredholm aug})]\setcounter{claim}{0} \label{proof:Fredholm aug} We fix $p>2$, $\lam>-2/p+1$, and a triple $w:=(P,A,u)\in \BB^p_\lam$. We prove {\bf part (\ref{thm:Fredholm X w})}. The space (\ref{eq:L k p}) is complete, see \cite{Lockhart PhD}. The same holds for the space (\ref{eq:W k p}) by Proposition \ref{prop:lam d Morrey}(\ref{prop:lam W}) (Appendix \ref{sec:weighted}). Combining this with Propositions \ref{prop:triv} and \ref{prop:X X w}(\ref{prop:X X w iso}), part (\ref{thm:Fredholm X w}) follows. 

{\bf Part (\ref{thm:Fredholm DD w})} follows from Propositions \ref{prop:triv} and \ref{prop:X X w}(\ref{prop:X X w Fredholm}), Corollary \ref{cor:d L L} and Proposition \ref{prop:d A B d} (Appendix \ref{sec:weighted}). This proves Theorem \ref{thm:Fredholm aug}.\end{proof}
\begin{rmk}\label{rmk:Fredholm aug} An alternative approach to prove Theorem \ref{thm:Fredholm aug} is to switch to ``logarithmic'' coordinates $\tau+i\phi$ (defined by $e^{\tau+i\phi}=s+it\in\C\wo\{0\}$). In these coordinates and a suitable trivialization the operator $\DD_w$ is of the form $\dd_\tau+A(\tau)$. Hence one can try to apply the results of \cite{RoSa}. However, this is not possible, since $A(\tau)$ contains the operator $v\mapsto e^{2\tau} d\mu(u)v\,d\tau\wedge d\phi$, which diverges for $\tau\to\infty$. $\Box$
\end{rmk}
\subsection{Proofs of the results of subsection \ref{subsec:proof:Fredholm aug}}\label{subsec:proofs}
\begin{proof}[Proof of Proposition \ref{prop:triv} (p.~\pageref{prop:triv})]\label{proof:triv}\setcounter{claim}{0} Let $p,\lam$ and $w$ be as in the hypothesis. We choose a section $\si$ of $P|_{\C\wo B_1}$ and a point $x_\infty\in\mu^{-1}(0)$ as in Lemma \ref{le:si} in Appendix \ref{sec:proofs homology}. 
\begin{claim}\label{claim:U triv} There exists an open $G$-invariant neighborhood $U\sub M$ of $x_\infty$ such that $H|_U$ is a smooth subbundle of $TM$ with the following property. There exists a smooth complex trivialization $\Psi^U:U\x\C^{\bar n}\to H|_U$ satisfying 
\[\Psi^U_{gx}v_0=g\Psi^U_xv_0:=g\Psi^U(x,v_0),\quad\forall g\in G,\,x\in U,\,v_0\in \C^{\bar n}.\] 
\end{claim}
\begin{proof}[Proof of Claim \ref{claim:U triv}] By hypothesis (H) we have $x_\infty\in M^*$ (where $M^*$ is defined as in (\ref{eq:M *})). We choose a $G$-invariant neighborhood $U_0\sub M^*$ of $x_\infty$ so small that $\ker d\mu(x)$ and $(\im L_x)^\perp$ intersect transversely, for every $x\in U_0$. Then $H|_{U_0}$ is a smooth subbundle of $TM|_{U_0}$. Furthermore, by the local slice theorem there exists a pair $(U,N)$, where $U\sub U_0$ is a $G$-invariant neighborhood of $x_\infty$ and $N\sub U$ is a submanifold of dimension $\dim M-\dim G$ that intersects $Gx$ transversely in exactly one point, for every $x\in U$. We choose a complex trivialization of $H|_N$ and extend it in a $G$-equivariant way, thus obtaining a trivialization $\Psi^U$ of $H|_U$. This proves Claim \ref{claim:U triv}.
\end{proof}
We choose $U$ and $\Psi^U$ as in Claim \ref{claim:U triv}. It follows from Lemma \ref{le:si} that there exists $R>1$ such that $u(p)\in U$, for $p\in\pi^{-1}(z)\sub P$, if $z\in\C\wo B_R$. We define 
\begin{eqnarray*}&\wt\Psi^\infty:(\C\wo B_R)\x(\C^{\bar n}\oplus\g^\C)\to TM^u=(u^*TM)/G,&\\
&\wt\Psi^\infty_z(v_0,\al)=G\cdot\left(u\circ\si(z),\Psi^U_{u\circ\si(z)}(z^{-d}\cdot\oplus\id)v_0+L^\C_{u\circ\si(z)}\al\right).&
\end{eqnarray*}
This is a complex trivialization of $TM^u|_{\C\wo B_R}$ (of class $W^{1,p}_\loc$). 
\begin{claim}\label{claim:wt Psi 1} $\wt\Psi^\infty|_{\C\wo B_{R+1}}$ extends to a complex trivialization of $TM^u$. 
\end{claim}
We define $f:\C\wo\{0\}\to S^1$ by $f(z):=z/|z|$. 
\begin{proof}[Proof of Claim \ref{claim:wt Psi 1}] We choose a complex trivialization  
\[\Psi^0:\bar B_R\x (\C^{\bar n}\oplus\g^\C)\to TM^u|_{\bar B_R}\]
of class $W^{1,p}_\loc$.%
\footnote{To see that such a trivialization exists, we first choose a continuous trivialization $\wt\Psi^0$ of the bundle. An argument using local trivializations of class $W^{1,p}_\loc$, shows that we may regularize $\wt\Psi^0$, so that it becomes of class $W^{1,p}_\loc$.}
 We define 
\begin{eqnarray}\label{eq:Phi S 1 R}&\Phi:S^1_R:=\{z\in\C\,|\,|z|=R\}\to \Aut(\C^{\bar n}\oplus\g^\C),&\\ 
\nn&\Phi_z(v_0,\al):=(\Psi^0_z)^{-1}\left(G\cdot\big(u\circ\si(z),\Psi^U_{u\circ\si(z)}v_0+L^\C_{u\circ\si(z)}\al\big)\right).&
\end{eqnarray}
For a continuous map $x:S^1_R\to S^1$ we denote by $\deg(x)$ its degree.
\begin{claim}\label{claim:lan deg} We have 
\begin{equation}\label{eq:c 1 G deg}\big\lan c_1^G(M,\om),[[w]]\big\ran=\deg(f\circ\det\circ\Phi).
\end{equation}
\end{claim}
\begin{proof}[Proof of Claim \ref{claim:lan deg}] We define $\wt P$ to be the quotient of $P\disj \big((S^2\wo\{0\})\x G\big)$ under the equivalence relation generated by $p\sim (z,g)$, where $g\in G$ is determined by $\si(z)g=p$, for $p\in \pi^{-1}(z)\sub P$, $z\in \C\wo\{0\}$. Furthermore, we define
\[\wt u:\wt P\to M,\quad\wt u([p]):=\left\{\begin{array}{ll}u(p),&\textrm{for }p\in P,\\
\wt u([\infty,g]):=g^{-1}x_\infty,&\textrm{for }g\in G.\end{array}\right.\]
The statement of Lemma \ref{le:si} implies that this map is continuous and extends $u$. The fiberwise pullback form $\wt u^*\om$ on $\wt P$ descends to a symplectic form $\wt\om$ on the vector bundle $TM^{\wt u}=(\wt u^*TM)/G\to S^2$. Similarly, the structure $J$ induces an complex structure $\wt J$ on $TM^{\wt u}$. The structures $\wt\om$ and $\wt J$ are compatible, and therefore, we have
\[c_1(TM^{\wt u},\wt\om)=c_1\big(TM^{\wt u},\wt J\big).\]
Using Lemma \ref{le:Chern bundle} in Appendix \ref{sec:proofs homology}, it follows that
\begin{equation}\label{eq:c 1 G}\big\lan c_1^G(M,\om),[[w]]\big\ran=\big\lan c_1\big(TM^{\wt u},\wt J\big),[S^2]\big\ran.
\end{equation}
We define the local complex trivialization 
\begin{eqnarray*}&\Psi^\infty:(S^2\wo B_R)\x(\C^{\bar n}\oplus\g^\C)\to TM^{\wt u},&\\
&\Psi^\infty_z(v_0,\al):=\left\{\begin{array}{ll}
G\cdot\big([u\circ\si(z)],\Psi^U_{u\circ\si(z)}v_0+L^\C_{u\circ\si(z)}\al\big),&\textrm{if }z\in \C\wo B_R,\\
 G\cdot\big([\infty,\one],\Psi^U_{x_\infty}v_0+L^\C_{x_\infty}\al\big),&\textrm{if }z=\infty.
\end{array}\right.
\end{eqnarray*}
Recalling the definition (\ref{eq:Phi S 1 R}), we have 
\[\Phi_z=(\Psi^0_z)^{-1}\Psi^\infty_z,\quad\forall z\in S^1_R.\] 
Therefore, $\Phi$ is the transition map between $\Psi^0$ and $\Psi^\infty$. It follows that 
\[\big\lan c_1\big(TM^{\wt u},\wt J\big)[S^2]\big\ran=\deg(f\circ\det\circ\Phi).\]
Combining this with (\ref{eq:c 1 G}), equality (\ref{eq:c 1 G deg}) follows. This proves Claim \ref{claim:lan deg}.
\end{proof}
We denote $d:=\big\lan c_1^G(M,\om),[[w]]\big\ran$. By Claim \ref{claim:lan deg} and Lemma \ref{le:Aut} in Appendix \ref{sec:add} the maps $\Phi$ and 
\[S^1_R\ni z\mapsto (z^d\cdot\oplus\id)\in\Aut(\C^{\bar n}\oplus\g^\C)\] 
are homotopic. Hence there exists a continuous map 
\[h:\bar B_R\wo B_1\to \Aut(\C^{\bar n}\oplus\g^\C),\] 
such that $h_z:=h(z)=(z^d\cdot\oplus\id)$, if $z\in S^1_1$, and $h_z=\Phi(z)$, if $z\in S^1_R$. We define
\[\wt\Psi:\C\x(\C^{\bar n}\oplus\g^\C)\to TM^u,\quad\wt\Psi_z:=\left\{\begin{array}{ll}
\wt\Psi^\infty_z,&\textrm{for }z\in\C\wo B_R,\\ 
\Psi^0_zh_z(z^{-d}\cdot\oplus\id),&\textrm{for }z\in B_R\wo B_1,\\
\Psi^0_z,&\textrm{for }z\in B_1.
\end{array}\right.\]
Regularizing $\wt\Psi$ on the ball $B_{R+1}$, we obtain the required extension of $\wt\Psi^\infty|_{\C\wo B_{R+1}}$, of class $W^{1,p}_\loc$. This proves Claim \ref{claim:wt Psi 1}.
\end{proof}
We define the trivialization 
\begin{equation}\label{eq:hhat Psi infty}\hhat\Psi^\infty:(\C\wo B_R)\x\g^\C\to\wone\big(\g_P|_{\C\wo B_R}\big),\quad\hhat\Psi^\infty_z(\phi+i\psi):=G\cdot\big(\si(z),\phi ds + \psi dt\big).
\end{equation}
\begin{claim}\label{claim:hhat Psi 1} $\hhat\Psi^\infty|_{\C\wo B_{R+1}}$ extends to a complex trivialization of the bundle $\wone(\g_P)$ over $\C$. 
\end{claim}
\begin{proof}[Proof of Claim \ref{claim:hhat Psi 1}] We denote by $\Ad$ and $\Ad^\C$ the adjoint representations of $G$ on $\g$ and $\g^\C$ respectively. We have 
\[\det(\Ad^\C_g)=\det(\Ad_g)\in\R,\quad\forall g\in G.\] 
We choose a continuous section $\wt\si$ of the restriction $P|_{\bar B_R}$. We define $g:S^1_R\to G$ to be the unique map such that $\si(z)=\wt\si(z) g(z)$, for every $z\in S^1_R$. Since $f\circ\det(\Ad^\C_g)\const \pm 1$, we have
\[\deg\left(S^1_R\ni z\mapsto f\circ\det(\Ad^\C_{g(z)})\right)=0.\] 
Hence Lemma \ref{le:Aut} (Appendix \ref{sec:add}) implies that there exists a continuous map $\Phi:\bar B_R\to \Aut(\g^\C)$ satisfying $\Phi_z:=\Phi(z)=\Ad^\C_{g(z)}$, for every $z\in S^1_R$. We define $\hhat\Psi:\C\x\g^\C\to \wone(\g_P)$ to be the trivialization that equals $\hhat\Psi^\infty$ on $\C\wo B_R$, and satisfies
\[\hhat\Psi_z\al:=G\cdot\big(\wt\si,\phi'ds + \psi' dt\big),\]
where $\phi'+ i\psi':=\Phi_z\al$, for every $z\in B_R$, $\al\in \g^\C$. Regularizing $\hhat\Psi$ on the ball $B_{R+1}$, we obtain an extension of $\hhat\Psi^\infty|_{\C\wo B_{R+1}}$, of class $W^{1,p}_\loc$. This proves Claim \ref{claim:hhat Psi 1}.
\end{proof}
We choose extensions $\wt\Psi$ and $\hhat\Psi$ of $\wt\Psi^\infty$ and $\hhat\Psi^\infty$ as in Claims \ref{claim:wt Psi 1} and \ref{claim:hhat Psi 1}, and we define
\[\Psi:\C\x V\to\wone(\g_P)\oplus TM^u,\quad\Psi(z;v_0,\al,\be):=\big(\hhat\Psi_z\be,\wt\Psi_z(v_0,\al)\big).\]
\begin{claim}\label{claim:Psi good} The map $\Psi$ is a good complex trivialization.
\end{claim} 
\begin{pf}[Proof of Claim \ref{claim:Psi good}] Condition {\bf (\ref{defi:triv split})} of Definition \ref{defi:triv} follows from the construction of $\Psi$. (The condition $\si^*A\in L^p_\lam(\C\wo B_1,\g)$ follows from the statement of Lemma \ref{le:si}.) To prove {\bf (\ref{defi:triv C})}, note that for $z\in\C\wo B_{R+1}$ and $(v_0,\al,\be)\in V$, we have
\begin{equation}\label{eq:Psi z z}\big|\Psi_z(z^d\cdot\oplus\id)(v_0,\al,\be)\big|^2=|\be|^2+\big|\Psi^U_{u\circ\si(z)}v_0\big|^2+\big|L^\C_{u\circ\si(z)}\al\big|^2.\end{equation}
Here we used the fact $H_x=(\im L^\C_x)^\perp$, for every $x\in M$. By our choice of $U$, $H|_U\sub TM|_U$ is a smooth subbundle of rank $\dim M-2\dim G$. It follows that $\im L^\C|_U=H^\perp|_U$ is a smooth subbundle of $TM|_U$ of rank $2\dim G$. Hence $L^\C_x:\g^\C\to T_xM$ is injective, for every $x\in U$. Since by assumption $\BAR{u(P)}\sub M$ is compact, the same holds for the set $\BAR{u(P|_{\C\wo B_{R+1}})}\sub\BAR{u(P)}$. It follows that there exists a constant $C>0$ such that 
\[C^{-1}|v_0|\leq\big|\Psi^U_{u\circ\si(z)}v_0\big|\leq C|v_0|,\quad C^{-1}|\al|\leq \big|L^\C_{u\circ\si(z)}\al\big|\leq C|\al|,\]
for every $z\in \C\wo B_{R+1}$, $v_0\in\C^{\bar n}$, and $\al\in\g^\C$. Combining this with equality (\ref{eq:Psi z z}), condition (\ref{defi:triv C}) follows. 

We check condition {\bf (\ref{defi:triv na})}. Let 
\[\ze:=\big(v_0,\al,\be=\phi+i\psi\big)\in V,\quad z\in \C\wo B_{R+1},\quad v\in T_z\C.\]
We choose a point $p\in\pi^{-1}(z)\sub P$ and a vector $\wt v\in T_pP$ such that $\pi_*\wt v=v$. Then 
\[\na^A_v\big(\wt\Psi(p_d\cdot\oplus\id)(v_0,\al)\big)=G\cdot\big(u(p),\wt\na^A_{\wt v}(\Psi^U_uv_0+L^\C_u\al)\big),\] 
where $\wt\na^A_{\wt v}$ is defined as in (\ref{eq:wt na A}). Furthermore, for every smooth vector field $X$ on $U$ we have 
\[\wt\na^A_{\wt v}X=(u^*\na)_{\wt v-p(A\wt v)}X=\na_{d_Au\cdot\wt v}X.\]
We define 
\[C:=\max\big|\na_{v'}(\Psi^U_xv''+L^\C_x\al)\big|,\]
where the maximum is taken over all $v'\in T_xM$, $x\in \BAR{u(P|_{\C\wo B_{R+1}})}$ and $(v'',\al)\in \C^{\bar n}\oplus\g^\C$ such that $|v'|\leq1$, $|(v'',\al)|\leq1$. Furthermore, we define 
\[C':=\Vert d_Au\Vert_{L^p_\lam(\C\wo B_{R+1})}.\] 
It follows that 
\begin{equation}\label{eq:Vert na A v wt Psi}\big\Vert\na^A_v\big(\wt\Psi(p_d\cdot\oplus\id)(v_0,\al)\big)\big\Vert_{L^p_\lam(\C\wo B_{R+1})}\leq CC'|v||(v_0,\al)|.
\end{equation}
We now define $\wt\phi,\wt\psi:P\to\g$ to be the unique equivariant maps such that $\wt\phi\circ\si\const\phi$, $\wt\psi\circ\si\const\psi$. We have $d_A\wt\phi\, \si_*v=[(\si^*A)v,\phi]$, and similarly for $\wt\psi$. Since 
\[\wt\na^A_{\si_*v}(\wt\phi ds+\wt\psi dt)=(d_A\wt\phi\si_*v)ds+(d_A\wt\psi\si_*v)dt,\] 
using (\ref{eq:hhat Psi infty}), it follows that
\[\big|\na^A_v\big(\hhat\Psi(\phi ds+\psi dt)\big)\big|=\big|G\cdot\big(\si(z),\wt\na^A_{\si_*v}(\wt\phi ds+\wt\psi dt)\big)\big|=\big|\big[(\si^*A)v,\be\big]\big|.\]
Using the fact $\Vert\si^*A\Vert_{p,\lam}<\infty$ and inequality (\ref{eq:Vert na A v wt Psi}), condition (\ref{defi:triv na}) follows. This proves Claim \ref{claim:Psi good} and concludes the proof of Proposition \ref{prop:triv}.
\end{pf}
\end{proof}
\begin{proof}[Proof of Proposition \ref{prop:X X w} (p.~\pageref{prop:X X w})]\label{proof:X X w} \setcounter{claim}{0}
Let $p,\lam,w=(P,A,u)$ and $\Psi$ be as in the hypothesis. We choose $\rho_0\in C^\infty(\C,[0,1])$ such that $\rho_0(z)=0$ for $z\in B_{1/2}$ and $\rho_0(z)=1$ for $z\in\C\wo B_1$. We fix $R\geq1$, $\si$ and $x_\infty$ as in Definition \ref{defi:triv}(\ref{defi:triv split}). 

{\bf We prove statement (\ref{prop:X X w iso}).} For every $\ze\in W^{1,1}_\loc(\C,V)$ Leibnitz' rule implies that 
\begin{equation}\label{eq:na Phi Psi F ze}\na^A(\Psi\ze)=\big(\na^A\big(\Psi(p_d\cdot\oplus\id)\big)\big)(p_{-d}\cdot\oplus\id)\ze+\Psi(p_d\cdot\oplus\id)D\big((p_{-d}\cdot\oplus\id)\ze\big).
\end{equation}
\begin{claim}\label{claim:F X bdd} The first map in (\ref{eq:X Y Psi}) is well-defined and bounded.
\end{claim}
\begin{proof}[Proof of Claim \ref{claim:F X bdd}] Proposition \ref{prop:lam d Morrey}(\ref{prop:lam Morrey}) in Appendix \ref{sec:weighted} and the fact $\lam>-2/p+1$ imply that there exists a constant $C_1$ such that 
\begin{equation}\label{eq:lan ran - d id ze}\big\Vert(\lan\cdot\ran^{-d}\cdot\oplus\id)\ze\big\Vert_\infty\leq C_1\Vert\ze\Vert_{\XXX_d},\quad\forall \ze\in\XXX_d\,.\end{equation}
We choose a constant $C_2:=C$ as in part (\ref{defi:triv C}) of Definition \ref{defi:triv}. Then by (\ref{eq:C}) and (\ref{eq:lan ran - d id ze}), we have 
\begin{equation}\label{eq:Psi ze infty}\Vert\Psi\ze\Vert_\infty\leq C_1C_2\Vert\ze\Vert_{\XXX_d}\quad\forall \ze\in\XXX_d\,.\end{equation}
It follows from (\ref{eq:Psi infty x H}) and (\ref{eq:Psi 0 al be}), the definition $H_x:=\ker d\mu(x)\cap \im L_x^\perp$ and the compactness of $\BAR{u(P)}$ that there exists $C_3\in\R$ such that, for every $\ze\in\XXX_d$,
\begin{equation}\label{eq:d mu u v'}\big\Vert \,|d\mu(u)v'|+|{\PR}^uv'|+|\al'|\,\big\Vert_{p,\lam}\leq C_3\Vert\ze\Vert_{\XXX_d},
\end{equation}
where $(v',\al'):=\Psi\ze$. For $r>0$ we denote 
\[B_r^C:=\C\wo B_r,\quad\Vert\cdot\Vert_{p,\lam;r}:=\Vert\cdot\Vert_{L^p_\lam(B_r^C)}.\] 
We define 
\[C_4:=\max\big\{\big\Vert\na^A\big(\Psi(p_d\cdot\oplus\id)\big)\big\Vert_{p,\lam;1},C_2\big\}.\] 
By condition (\ref{defi:triv na}) of Definition \ref{defi:triv} we have $C_4<\infty$. Let $\ze\in\XXX_d$. Then by (\ref{eq:na Phi Psi F ze}) we have
\begin{equation}\label{eq:na A Psi ze p lam}\Vert\na^A(\Psi\ze)\Vert_{p,\lam;1}\leq C_4\Big(\Vert(p_{-d}\cdot\oplus\id)\ze\Vert_{L^\infty(B_1^C)}+\big\Vert D\big((p_{-d}\cdot\oplus\id)\ze\big)\big\Vert_{p,\lam;1}\Big).
\end{equation}
Defining 
\[C_5:=\max\big\{-d2^{(-d+3)/2},2\big\},\]
we have, by Proposition \ref{prop:lam d Morrey}(\ref{prop:lam d iso}),
\[\big\Vert D\big((p_{-d}\cdot\oplus\id)\ze\big)\big\Vert_{p,\lam;1}\leq C_5\Vert\ze\Vert_{\XXX_d}.\] 
Combining this with (\ref{eq:na A Psi ze p lam}) and (\ref{eq:lan ran - d id ze}), we get
\begin{equation}\label{eq:na A Psi ze} \Vert\na^A(\Psi\ze)\Vert_{p,\lam;1}\leq C_4\big(2^{\frac{|d|}2}C_1+C_5\big)\Vert\ze\Vert_{\XXX_d}.\end{equation}
By a direct calculation there exists a constant $C_6$ such that 
\[\Vert\na^A(\Psi\ze)\Vert_{L^p(B_1)}\leq C_6\Vert\ze\Vert_{\XXX_d},\quad\forall\ze\in\XXX_d.\]
Claim \ref{claim:F X bdd} follows from this and (\ref{eq:Psi ze infty},\ref{eq:d mu u v'},\ref{eq:na A Psi ze}).
\end{proof} 
\begin{claim}\label{claim:F -1 X bdd} The map $\XXX_w\ni\ze'\mapsto \Psi^{-1}\ze'\in \XXX_d$ is well-defined and bounded. 
\end{claim}
\begin{proof}[Proof of Claim \ref{claim:F -1 X bdd}] We choose a neighborhood $U\sub M$ of $\mu^{-1}(0)$ as in Lemma \ref{le:U} (Appendix \ref{sec:add}), and define $c$ as in (\ref{eq:c inf}), and $C_1:=\max\{c^{-1},1\}$. Since $u\circ\si(re^{i\phi})$ converges to $x_\infty$, uniformly in $\phi\in\R$, as $r\to\infty$, there exists $R'\geq R$ such that $u(p)\in U$, for every $p\in\pi^{-1}(B_{R'}^C)\sub P$. Then (\ref{eq:Psi infty x H},\ref{eq:Psi 0 al be}) and (\ref{eq:c inf}) imply that 
\begin{equation}\label{eq:Vert al be Vert}\big\Vert(\al,\be)\big\Vert_{p,\lam;R'}\leq C_1\big\Vert\Psi(0,\al,\be)\big\Vert_{p,\lam;R'}\leq C_1\Vert\ze'\Vert_w,\end{equation}
where $(v_0,\al,\be):=\Psi^{-1}\ze'$, for every $\ze'\in\XXX_d$. 
\begin{claim}\label{claim:d rho 0 p} There exists a constant $C_2$ such that for every $\ze'\in\XXX_w$, we have
\begin{equation}
\label{eq:d rho 0 p}\big\Vert D\big((\rho_0p_{-d}\cdot\oplus\id)\Psi^{-1}\ze'\big)\big\Vert_{L^p_{\lam}(\C)}\leq C_2\Vert\ze'\Vert_w.
\end{equation}
\end{claim}
\begin{proof}[Proof of Claim \ref{claim:d rho 0 p}] It follows from equality (\ref{eq:na Phi Psi F ze}) and conditions (\ref{defi:triv C}) and (\ref{defi:triv na}) of Definition \ref{defi:triv} that there exist constants $C$ and $C'$ such that 
\begin{eqnarray}\label{eq:p - d id F}&\big\Vert D\big((p_{-d}\cdot\oplus\id)\Psi^{-1}\ze'\big)\big\Vert_{p,\lam;1}&\\
\nn&\leq C\Big(\Vert\na^A\ze'\Vert_{p,\lam}+\big\Vert\na^A\big(\Psi(p_d\cdot\oplus\id)\big)\big\Vert_{p,\lam}\Vert\ze'\Vert_\infty\Big)\leq C'\Vert\ze'\Vert_w,\quad \forall\ze'\in\XXX_w\,.&
\end{eqnarray}
On the other hand, Leibnitz' rule implies that 
\[D(\Psi^{-1}\ze')=\Psi^{-1}\big(\na^A\ze'-(\na^A\Psi)\Psi^{-1}\ze'\big).\] 
Hence by a short calculation, using Leibnitz' rule again, it follows that there exists a constant $C''$ such that 
\[\big\Vert D\big((\rho_0p_{-d}\cdot\oplus\id)\Psi^{-1}\ze'\big)\big\Vert_{L^p(B_1)}\leq C''\Vert\ze'\Vert_w,\quad\ze'\in \XXX_w.\] 
Combining this with (\ref{eq:p - d id F}), Claim \ref{claim:d rho 0 p} follows. 
\end{proof}
Let $\ze'\in\XXX_w$. We denote 
\[\wt\ze:=(\wt v_0,\wt\al,\wt\be):=(\rho_0p_{-d}\cdot\oplus\id)\Psi^{-1}\ze'.\] 
By inequality (\ref{eq:d rho 0 p}) the hypotheses of Proposition \ref{prop:Hardy} in Appendix \ref{sec:weighted} with $n:=2$ and $\lam$ replaced by $\lam-1$ are satisfied. It follows that there exists 
\[\ze_\infty:=\big(v_\infty,\al_\infty,\be_\infty\big)\in V=\C^{\bar n}\oplus\g^\C\oplus\g^\C,\] 
such that $\wt\ze(re^{i\phi})\to\ze_\infty$, uniformly in $\phi\in\R$, as $r\to\infty$, and 
\begin{equation}\label{eq:wt ze ze infty}\Vert(\wt\ze-\ze_\infty)|\cdot|^{\lam-1}\big\Vert_{L^p(\C)}\leq (\dim M+2\dim G)\frac p{\lam-1+\frac2p}\big\Vert D\wt\ze|\cdot|^{\lam}\big\Vert_{L^p(\C)}
\,.\end{equation}
Since $\lam>-2/p+1$, we have
\[\int_{B_{R'}^C}\lan\cdot\ran^{p\lam}=\infty.\]
Hence the convergence $(\wt\al,\wt\be)\to (\al_\infty,\be_\infty)$ and the estimate (\ref{eq:Vert al be Vert}) imply that $(\al_\infty,\be_\infty)=(0,0)$. We choose a constant $C>0$ as in part (\ref{defi:triv C}) of Definition \ref{defi:triv}. The convergence $\wt v_0\to v_\infty$ and the first inequality in (\ref{eq:C}) imply that 
\begin{equation}\label{eq:v infty 1}|v_\infty|\leq\Vert\wt v_0\Vert_\infty\leq 2^{\frac{|d|}2}C\Vert\ze'\Vert_\infty  
\,.\end{equation}
We define 
\[\big(v^1,\ldots,v^{\bar n},\al,\be\big):=\Psi^{-1}\ze'-\big(\rho_0p_dv_\infty^1,v_\infty^2,\ldots,v_\infty^{\bar n},0,0\big).\] 
Proposition \ref{prop:lam d Morrey}(\ref{prop:lam d iso}) in Appendix \ref{sec:weighted} and inequalities (\ref{eq:wt ze ze infty}) and (\ref{eq:d rho 0 p}) imply that there exists a constant $C_6$ (depending on $p,\lam,d$ and $\Psi$, but not on $\ze'$) such that 
\begin{equation}\label{eq:v 1 v 2 v bar n}\Vert v^1\Vert_{L^{1,p}_{{\lam-1}-d}(B_1^C)}+\big\Vert\big(v^2,\ldots,v^{\bar n},\al,\be\big)\big\Vert_{L^{1,p}_{\lam-1}(B_1^C)}\leq C_6\Vert\ze'\Vert_w\,.
\end{equation}
Finally, by a straight-forward argument, there exists a constant $C_7$ (independent of $\ze'$) such that
\[\Vert\Psi^{-1}\ze'\Vert_{W^{1,p}(B_{R'})}\leq C_7\Vert\ze'\Vert_w.\]
Combining this with (\ref{eq:Vert al be Vert},\ref{eq:v infty 1},\ref{eq:v 1 v 2 v bar n}), Claim \ref{claim:F -1 X bdd} follows.\end{proof}
Claims \ref{claim:F X bdd} and \ref{claim:F -1 X bdd} imply that the first map in (\ref{eq:X Y Psi}) is an isomorphism (of normed vector spaces). It follows from condition (\ref{defi:triv C}) of Definition \ref{defi:triv} that the second map in (\ref{eq:X Y Psi}) is an isomorphism. This completes the proof of {\bf (\ref{prop:X X w iso})}.

{\bf We prove statement (\ref{prop:X X w Fredholm}).} Recall that we have chosen $R>0,\si$ and $x_\infty$ as in Definition \ref{defi:triv}(\ref{defi:triv split}). We define $S_\infty:\g^\C\to\g^\C$ to be the complex linear extension of $L_{x_\infty}^*L_{x_\infty}:\g\to\g$. By our hypothesis (H) the Lie group $G$ acts freely on $\mu^{-1}(0)$. It follows that $L_{x_\infty}$ is injective. Therefore $S_\infty$ is positive with respect to $\lan\cdot,\cdot\ran_\g^\C$. By (\ref{eq:Psi z C bar n}) and (\ref{eq:Psi z 0}) there exist complex trivializations 
\[\Psi_1:\C\x(\C^{\bar n}\oplus\g^\C)\to TM^u,\quad\Psi_2:\C\x\g^\C\to \wone(\g_P),\] 
such that $\Psi_z(v_0,\al,\be)=\big((\Psi_2)_z\be,(\Psi_1)_z(v_0,\al)\big)$. We denote by 
\[\iota:\g^\C\to \C^{\bar n}\oplus\g^\C,\quad\pr:\C^{\bar n}\oplus\g^\C\to \g^\C\] 
the canonical inclusion and projection. We define 
\begin{eqnarray*}&\XXX_d^1:=L^{1,p}_{{\lam-1}-d}(\C,\C)\oplus L^{1,p}_{\lam-1}(\C,\C^{\bar n-1})\oplus W^{1,p}_{\lam}(\C,\g^\C),\quad\XXX_d^2:=W^{1,p}_{\lam}(\C,\g^\C),&\\
&\XXX_d':=\XXX_d^1\oplus\XXX_d^2,\quad \XXX_d^0:=\C\rho_0p_d\oplus\C^{\bar n-1}\oplus\{(0,0)\}\sub\XXX_d,&\\
&\YYY_d^1:=L^p_{\lam-d}(\C,\C)\oplus L^p_{\lam}(\C,\C^{\bar n-1}\oplus\g^\C),\quad \YYY_d^2:=L^p_\lam(\C,\g^\C).&
\end{eqnarray*}
Note that $\XXX_d=\XXX_d^0+\XXX_d'$ and $\YYY_d=\YYY_d^1\oplus\YYY_d^2$. We define $S:\XXX_d\to\YYY_d$ as in (\ref{eq:F Psi DD}). Since $\XXX_d^0$ is finite dimensional, $S|_{\XXX_d^0}$ is compact. Hence it suffices to prove that $S|_{\XXX_d'}$ is compact. To see this, we denote 
\[Q:=\left(\begin{array}{c}ds\wedge dt\,d\mu(u)\\
L_u^*
\end{array}\right),\quad T:=\left(\begin{array}{c}d_A\\
  -d_A^*
\end{array}\right),\]
and we define $S^i_{\phantom{i}j}:\XXX_d^j\to\YYY_d^i$ (for $i,j=1,2$) and $\wt S^1_{\phantom{1}1}:\XXX_d^1\to\YYY_d^1$ by 
\begin{eqnarray*}&S^1_{\phantom{1}1}v:=(F_1\Psi_1)^{-1}((\na^A\Psi_1)v)^{0,1},\,\wt S^1_{\phantom{1}1}v:=-(F_1\Psi_1)^{-1}\big(J(\na_{\Psi_1v}J)(d_Au)^{1,0}/2\big),&\\
&S^1_{\phantom{1}2}\al:=(F_1\Psi_1)^{-1}(L_u\Psi_2\al)^{0,1}-\iota\al/2,\, S^2_{\phantom{2}1}v:=\big((F_2\Psi_2)^{-1}Q\Psi_1- S_\infty\pr\big)v,&\\
&S^2_{\phantom{2}2}\al:=(F_2\Psi_2)^{-1}(T\Psi_2)\al.&
\end{eqnarray*}
Here $(T\Psi_2)\al:=T(\Psi_2\al)$, for $\al\in\g^\C$ (viewed as a constant section of $\C\x\g^\C$). (Recall also that $S_\infty:\g^\C\to\g^\C$ is the complex linear extension of $L_{x_\infty}^*L_{x_\infty}:\g\to\g$.) A direct calculation shows that 
\[S(v,\al)=\big(S^1_{\phantom{1}1}v+\wt S^1_{\phantom{1}1}v+S^1_{\phantom{1}2}\al,S^2_{\phantom{2}1}v+S^2_{\phantom{2}2}\al\big).\] 
For a subset $X\sub \C$ we denote by $\chi_X:\C\to\{0,1\}$ its characteristic function. It follows that $\chi_{B_R}S|_{\XXX_d'}$ is of 0-th order. Since it vanishes outside $B_R$, it follows that this map is compact. 
\begin{claim}\label{claim:S i j compact} The operators $\chi_{B_R^C}S^i_{\phantom{i}j}$, $i,j=1,2$, and $\chi_{B_R^C}\wt S^1_{\phantom{1}1}$ are compact. 
\end{claim}
\begin{pf}[Proof of Claim \ref{claim:S i j compact}] To see that $\chi_{B_R^C}S^1_{\phantom{1}1}$ is compact, note that Leibnitz' rule and holomorphicity of $p_d$ imply that 
\[(\na^A\Psi_1)^{0,1}=\big(\na^A\big(\Psi_1(p_d\cdot\oplus\id)\big)\big)^{0,1}(p_{-d}\cdot\oplus\id),\quad\textrm{on }\C\wo\{0\}.\] 
For a normed vector space $V$ we denote by $C_b(\C,V)$ the space of bounded continuous maps from $\C$ to $V$. Since $\lam>1-2/p$, assertions (\ref{prop:lam d iso}) and (\ref{prop:lam Morrey}) of Proposition \ref{prop:lam d Morrey} imply that the map 
\[(\rho_0p_{-d}\cdot\oplus\id):\XXX_d^1\to C_b\big(\C,\C^{\bar n}\oplus\g^\C\big)\] 
is well-defined and compact. By Definition \ref{defi:triv}(\ref{defi:triv na}), the map 
\[\chi_{B_R^C}\big(\na^A\big(\Psi_1(p_d\cdot\oplus\id)\big)\big)^{0,1}:C_b\big(\C,\C^{\bar n}\oplus\g^\C\big)\to\Ga^p_{\lam}\big(\wzeroone(TM^u)\big)\] 
is bounded. Condition (\ref{defi:triv C}) of Definition \ref{defi:triv} implies boundedness of the map 
\[(F_1\Psi_1)^{-1}:\Ga^p_{\lam}\big(\wzeroone(TM^u)\big)\to\YYY_d^1.\] 
Compactness of $\chi_{B_R^C}S^1_{\phantom{1}1}$ follows. 

By the definition of $\BB^p_\lam$, we have $|d_Au|\in L^p_{\lam}(\C)$. This together with Proposition \ref{prop:lam d Morrey}(\ref{prop:lam d iso}) and (\ref{prop:lam Morrey}) and  Definition \ref{defi:triv}(\ref{defi:triv C}) implies that the map $\chi_{B_R^C}\wt S^1_{\phantom{1}1}$ is compact. Furthermore, it follows from Definition \ref{defi:triv}(\ref{defi:triv split}) that $\chi_{B_R^C}S^1_{\phantom{1}2}=0$.

To see that $\chi_{B_R^C}S^2_{\phantom{2}1}$ is compact, we define $f:B_R^C\to \End(\g^\C)$ by setting $f(z):\g^\C\to\g^\C$ to be the complex linear extension of the map 
\[L_{u\circ\si(z)}^*L_{u\circ\si(z)}-L_{x_\infty}^*L_{x_\infty}:\g\to\g.\] 
Since $u\circ\si(re^{i\phi})$ converges to $x_\infty$, uniformly in $\phi$, as $r\to\infty$, the map $f(re^{i\phi})$ converges to 0, uniformly in $\phi$, as $r\to \infty$. Hence by Proposition \ref{prop:lam d Morrey}(\ref{prop:lam W cpt}), the map 
\[W^{1,p}_{\lam}(\C,\g^\C)\ni\al\mapsto \chi_{B_R^C}f\al\in L^p_{\lam}(\C,\g^\C)\] 
is compact. Definition \ref{defi:triv}(\ref{defi:triv split}) implies that $\chi_{B_R^C}S^2_{\phantom{2}1}=\chi_{B_R^C}f\pr$. It follows that this operator is compact. 

Finally, Proposition \ref{prop:lam d Morrey}(\ref{prop:lam Morrey}) and parts (\ref{defi:triv na}) and (\ref{defi:triv C}) of Definition \ref{defi:triv} imply that the map $\chi_{B_R^C}S^2_{\phantom{2}2}$ is compact. Claim \ref{claim:S i j compact} follows. It follows that the operator $S:\XXX_d\to\YYY_d$ (as in (\ref{eq:F Psi DD})) is compact. This completes the proof of statement {\bf (\ref{prop:X X w Fredholm})} and hence of Proposition \ref{prop:X X w}.
\end{pf}
\end{proof}
\subsection{Proof of Theorem \ref{thm:L w * R} (Right inverse for $L_w^*$)}\label{subsec:proof:thm:L w * R}
For the proof of this result we need the following. Let $n\in\N$, $\ell\in\N_0$, $p>n/(\ell+1)$, $G$ be a compact Lie group, $\lan\cdot,\cdot\ran_\g$ an invariant inner product on $\g:=\Lie(G)$, $X$ a manifold (possibly with boundary), and $P\to X$ a $G$-bundle of class $W^{\ell+1,p}_\loc$. We denote by $\g_P:=(P\x\g)/G\to X$ the adjoint bundle, and by $\A^{\ell,p}_\loc(P)$ the affine space of connections on $P$ of class $W^{\ell,p}_\loc$, i.e., of class $W^{\ell,p}$ on every compact subset of $X$. If $X$ is compact then we abbreviate $\A^{\ell,p}(P):=\A^{\ell,p}_\loc(P)$. Let $\lan\cdot,\cdot\ran_X$ be a Riemannian metric on $X$ and $A\in\A^{\ell,p}_\loc(P)$. The connection $A$ induces a connection $d_A$ on the adjoint bundle $\g_P=(P\x\g)/G$. For every $i\in\N$ this connection and the Levi-Civita connection of $\lan\cdot,\cdot\ran_X$ induce a connection $\na_{\lan\cdot,\cdot\ran_X}^A$ on the bundle $(T^*X)^{\otimes i}\otimes\g_P$. We abbreviate these connections by $\na^A$. Let $k\in\{0,\ldots,\ell+1\}$. For a vector bundle $E$ over $X$ we denote by $\Ga^{k,p}_\loc(E)$ the space of $W^{k,p}_\loc$-sections of $E$. For $\al\in\Ga^{k,p}_\loc\big((T^*X)^{\otimes i}\otimes\g_P\big)$ we define 
\begin{equation}\label{eq:Vert al Vert}\Vert\al\Vert_{k,p,A}:=\Vert\al\Vert_{k,p,X,\lan\cdot,\cdot\ran_X,A}:=\sum_{j=0,\ldots,k}\big\Vert\,\big|(\na^A)^j\al\big|_{\lan\cdot,\cdot\ran_X,\lan\cdot,\cdot\ran_\g}\big\Vert_{L^p(X,\lan\cdot,\cdot\ran_X)},
\end{equation}
where $|\cdot|_{\lan\cdot,\cdot\ran_X,\lan\cdot,\cdot\ran_\g}$ denotes the pointwise norm induced by $\lan\cdot,\cdot\ran_X$ and $\lan\cdot,\cdot\ran_\g$, and $\Vert\cdot\Vert_{L^p(X,\lan\cdot,\cdot\ran_X)}$ denotes the $L^p$-norm of a function on $X$, induced by $\lan\cdot,\cdot\ran_X$. We denote by $\wi(\g_P)$ the bundle of $i$-forms on $X$ with values in $\g_P$, and 
\begin{eqnarray}\label{eq:Om i k p A}&\Om^i_{k,p,A}(\g_P):=\big\{\al\in\Ga^{k,p}_\loc\big(\wi(\g_P)\big)\,\big|\,\Vert\al\Vert_{k,p,A}<\infty\big\},&\\
\label{eq:Ga k p A}&\Ga^{k,p}_A(\g_P):=\Om^0_{k,p,A}(\g_P),\quad\Ga^p(\g_P):=\Ga^{0,p}_A(\g_P).&
\end{eqnarray} 
We denote by 
\begin{equation}\label{eq:d A *} d_A^*=-*d_A*:\Om^1_{1,p,A}(\g_P)\to\Ga^p(\g_P)
\end{equation}
the formal adjoint of $d_A$, with respect to the $L^2$-metrics induced by $\lan\cdot,\cdot\ran_X$ and $\lan\cdot,\cdot\ran_\g$. Here $*$ denotes the Hodge star operator. 
\begin{prop}[Right inverse for $d_A^*$]\label{PROP:RIGHT} Let $n,G$ and $\lan\cdot,\cdot\ran_\g$ be as above, $p>n$, and $(X,\lan\cdot,\cdot\ran_X)$ a Riemannian manifold. Assume that $X$ is diffeomorphic to $\BAR B_1\sub\R^n$. Then the following statements hold.
\begin{enui}
\item\label{prop:right:exists} For every $G$-bundle $P\to X$ of class $W^{2,p}$ and every connection one-form $A$ on $P$ of class $W^{1,p}$ there exists a bounded right inverse of the operator (\ref{eq:d A *}). 
\item\label{prop:right:eps C} There exist constants $\eps>0$ and $C>0$ such that for every $G$-bundle $P\to X$ of class $W^{2,p}$, and every $A^{1,p}\in\A(P)$ satisfying $\Vert F_A\Vert_p\leq\eps$, there exists a right inverse $R$ of the operator $d_A^*$ (as in (\ref{eq:d A *})), satisfying
\[\Vert R\Vert:=\sup\big\{\Vert R\xi\Vert_{1,p,A}\,\big|\,\xi\in\Ga^p(\g_P):\,\Vert\xi\Vert_p\leq1\big\}\leq C.\] 
\end{enui}
\end{prop}
We postpone the proof of Proposition \ref{PROP:RIGHT} to the appendix (page \pageref{proof:prop:right}).
\begin{proof}[Proof of Theorem \ref{thm:L w * R} (p.~\pageref{thm:L w * R})]\label{proof:thm:L w * R}\setcounter{claim}{0}
Let $p,\lam$ and $w=(P,A,u)$ be as in the hypothesis. We construct a map 
\begin{equation}\label{eq:R}R:L^p_\loc(\g_P)\to\Ga^p_\loc\big(\wone(\g_P)\oplus TM^u\big)
\end{equation}
such that $L_w^*R$ is well-defined and equals $\id$, and we show that $R$ restricts to a bounded map from $\Ga^p_\lam(\g_P)$ to $\XXX_w^{p,\lam}$. (See Claim \ref{claim:ze X w} below.) It follows from hypothesis (H) that there exists $\de>0$ such that $\mu^{-1}(\bar B_\de)\sub M^*$ (defined as in (\ref{eq:M *})). It follows from Lemma \ref{le:si} in Appendix \ref{SEC:HOMOLOGY CHERN} that there exists a number $a>0$ such that $u(p)\in \mu^{-1}(\bar B_\de)$, for every $p\in\pi^{-1}\big(\C\wo (-a,a)^2\big)\sub P$. We choose constants $\eps_1$ and $C_1$ as in the second assertion of Proposition \ref{PROP:RIGHT} (corresponding to $\eps$ and $C$, for $n=2$). Furthermore, we choose constants $C_2$ and $\eps_2$ as in Lemma \ref{le:infty 1 p A} in Appendix \ref{sec:proof:prop:right}  (corresponding to $C$ and $\eps$). We define $\eps:=\min\{\eps_1,\eps_2\}$. By assumption we have $|F_A|\in L^p_\lam(\C)$. Hence there exists an integer $N>a$ such that 
\begin{equation}\label{eq:F A N}\Vert F_A\Vert_{L^p_\lam\big(\C\wo(-N,N)^2\big)}<\eps.
\end{equation}
We choose a smooth function $\rho:[-1,1]\to[0,1]$ such that $\rho=0$ on $[-1,-3/4]\cup[3/4,1]$, $\rho=1$ on $[-1/4,1/4]$, and $\rho(-t)=\rho(t)$ and $\rho(t)+\rho(t-1)=1$, for all $t\in[0,1]$. We choose a bijection 
\[(\phi,\psi):\Z\wo\{0\}\to \Z^2\wo\{-N,\ldots,N\}^2.\] 
We define $\wt\rho:\R\to[0,1]$ by 
\[\wt\rho(t):=\left\{\begin{array}{ll}1,&\textrm{if } |t|\leq N,\\
\rho(|t|-N),&\textrm{if }N\leq|t|\leq N+1,\\
0,&\textrm{if }|t|\geq N+1,
\end{array}\right.\]
and $\rho_0:\C\to[0,1]$ by $\rho_0(s,t):=\wt\rho(s)\wt\rho(t)$. Furthermore, for $i\in\Z\wo\{0\}$ we define 
\[\rho_i:\C\to[0,1],\quad\rho_i(s,t):=\rho(s-\phi(i))\rho(t-\psi(i)).\] 
We choose a compact subset $K_0\sub [-N-1,N+1]^2$ diffeomorphic to $\bar B_1$, such that
\[[-N-3/4,N+3/4]^2\sub \Int K_0,\]
and we denote $\Om_0:=\Int K_0$. Furthermore, we choose a compact subset $K\sub [-1,1]^2$ diffeomorphic to $\bar B_1$, such that $[-3/4,3/4]^2\sub \Int K$. For $i\in\Z\wo\{0\}$ we define 
\[\Om_i:=\Int K+(\phi(i),\psi(i)).\] 
For $i\in\Z$ we define 
\[T_i:=d_A^*:\Om^1_{1,p,A}\big(\g_P|_{\Om_i})\to\Ga^p(\g_P|_{\Om_i}).\]
By the first assertion of Proposition \ref{PROP:RIGHT} there exists a bounded right inverse $R_0$ of $T_0$. We fix $i\in\Z\wo\{0\}$. Since $\lam>1-2/p>0$, we have $\Vert F_A\Vert_{L^p(\Om_i)}\leq\Vert F_A\Vert_{L^p_\lam(\Om_i)}$, and by inequality (\ref{eq:F A N}), the right hand side is bounded by $\eps$. Hence it follows from the statement of Proposition \ref{PROP:RIGHT} that there exists a right inverse $R_i$ of $T_i$, satisfying 
\begin{equation}\label{eq:R i 1 p}\Vert R_i\xi\Vert_{1,p,\Om_i,A}\leq C_1\Vert\xi\Vert_{L^p(\Om_i)},\quad\forall\xi\in\Ga^{1,p}_A(\g_P|_{\Om_i}). 
\end{equation}
We define 
\[\hhat R:L^p_\loc(\g_P)\to\Ga^{1,p}_\loc\big(\wone(\g_P)\big),\quad\hhat R\xi:=\sum_{i\in\Z}\rho_i\cdot R_i(\xi|_{\Om_i}).\] 
Each section $\xi:\C\to\g_P$ induces a section $L_u\xi:\C\to TM^u$. For every $p\in\pi^{-1}\big(\C\wo(-N,N)^2\big)\sub P$ we have $u(p)\in\mu^{-1}(\bar B_\de)\sub M^*$, and therefore the map $L_{u(p)}^*L_{u(p)}:\g\to\g$ is invertible. We define $\wt R:L^p_\loc(\g_P)\to\Ga^p_\loc(TM^u)$ by 
\begin{equation}\label{eq:wt R xi z}(\wt R\xi)(z):=\left\{\begin{array}{ll}
0,&\textrm{for }z\in (-N,N)^2,\\
L_u(L_u^*L_u)^{-1}\big(\xi-d_A^*\hhat R\xi\big)(z),&\textrm{for }z\in\C\wo(-N,N)^2.
\end{array}\right.\end{equation}
Furthermore, we define the map (\ref{eq:R}) by 
\[R\xi:=(-\hhat R\xi,\wt R\xi).\] 
The operator $L_w^*R$ is well-defined and equals $\id$. The statement of Theorem \ref{thm:L w * R} is now a consequence of the following. We denote by $\Ga^p_\lam(\g_P)$ the space of $L^p_\lam$-sections of $\g_P$. 
\begin{claim}\label{claim:ze X w} The map $R$ restricts to a bounded operator from $\Ga^p_\lam(\g_P)$ to $\XXX_w^{p,\lam}$ (defined as in (\ref{eq:X w p lam})). 
\end{claim}
\begin{pf}[Proof of Claim \ref{claim:ze X w}] We choose a constant $C_3$ so big that 
\begin{equation}\label{eq:sup C 3}\sup_{z\in \Om_i}\lan z\ran^{p\lam}\leq C_3\inf_{z\in \Om_i}\lan z\ran^{p\lam},\quad\forall i\in\Z.
\end{equation}
For a weakly differentiable section $\xi:\C\to \g_P$ we denote 
\[\Vert\xi\Vert_{1,p,\lam,A}:=\Vert\xi\Vert_{p,\lam}+\Vert d_A\xi\Vert_{p,\lam}.\]
\begin{claim}\label{claim:d A * wt al} There exists a constant $C_4$ such that 
\[\Vert \xi-d_A^*\hhat R\xi\Vert_{1,p,\lam,A}\leq C_4\Vert\xi\Vert_{p,\lam},\quad\forall\xi\in\Ga^p_\lam(\g_P).\]
\end{claim}
\begin{proof}[Proof of Claim \ref{claim:d A * wt al}] Let $\xi\in\Ga^p_\lam(\g_P)$. We denote $\al:=\hhat R\xi$ and $\al_i:=R_i(\xi|_{\Om_i})$. Since $\sum_{i\in\Z}\rho_i=1$, a straight-forward calculation shows that 
\begin{equation}\label{eq:d A * wt al}d_A^*\al=\xi-\sum_{i\in\Z}*\big((d\rho_i)\wedge*\al_i\big).\end{equation} 
Fix $z\in\C$. Then $\big|\big\{i\in\Z\,|\,\rho_i(z)\neq0\big\}\big|\leq4$. Hence equality (\ref{eq:d A * wt al}) implies that 
\[\big|\big(\xi-d_A^*\al\big)(z)\big|^p\leq 4^{p-1}\Vert\rho'\Vert_\infty^p\sum_{i\in\Z}|\al_i(z)|^p.\] 
Combining this with (\ref{eq:R i 1 p},\ref{eq:sup C 3}), we obtain 
\[\Vert\xi-d_A^*\al\Vert_{p,\lam}^p\leq 4^p\Vert\rho'\Vert_\infty^p\max\big\{C_1^p,\Vert R_0\Vert^p\big\}C_3\sum_{i\in\Z}\Vert\xi\Vert_{L^p_\lam(\Om_i)}^p.\] 
By (\ref{eq:d A * wt al}), we have
\begin{equation}\nn\label{eq:d A d A *}\big|d_A\big(\xi-d_A^*\al\big)(z)\big|^p\leq 8^{p-1}\max\big\{\Vert\rho''\Vert_\infty^p,\Vert\rho'\Vert_\infty^p\big\}\sum_{i\in\Z}\big(|\al_i|^p+|\na^A\al_i|^p\big)(z).\end{equation}
Using again (\ref{eq:R i 1 p}), Claim \ref{claim:d A * wt al} follows.
\end{proof}
We choose $C_4$ as in Claim \ref{claim:d A * wt al}. The following will be used in the proofs of Claims \ref{claim:hat R wt R} and \ref{claim:na A p lam} below. Let $\xi\in\Ga^p_\lam(\g_P)$. We abbreviate $\wt\xi:=\xi-d_A^*\hhat R\xi$. By the fact $\wt\xi|_{(-N,N)^2}=0$, Lemma \ref{le:infty 1 p A} in Appendix \ref{sec:proof:prop:right} (Twisted Morrey's inequality, using (\ref{eq:F A N})), the fact $\lam>1-2/p>0$ and Claim \ref{claim:d A * wt al}, we have
\begin{equation}\label{eq:wt v infty}\Vert\wt\xi\Vert_\infty\leq C_2\Vert\wt\xi\Vert_{1,p,\lam,A}\leq C_2C_4\Vert\xi\Vert_{p,\lam}.\end{equation} 
Recall that we have chosen $\de>0$ such that $\mu^{-1}(\bar B_\de)\sub M^*$. We define 
\begin{equation}\label{eq:c inf L x xi}c:=\inf\left\{\frac{|L_x\xi|}{|\xi|}\,\bigg|\,x\in\mu^{-1}(\bar B_\de),\,0\neq\xi\in\g\right\}.
\end{equation}
Recall that ${\PR}^u:TM^u\to TM^u$ denotes the orthogonal projection onto $(u^*\im L)/G$. Claim \ref{claim:ze X w} is now a consequence of the following three claims. 
\begin{claim}\label{claim:hat R wt R} We have
\[\sup\Big\{\big\Vert R\xi\big\Vert_\infty+\big\Vert |d\mu(u)\wt R\xi|+|{\PR}^u \wt R\xi|+|\hhat R\xi|\big\Vert_{p,\lam}\,\big|\,\xi\in\Ga^p_\lam(\g_P):\,\Vert\xi\Vert_{p,\lam}\leq1\Big\}<\infty.\]
\end{claim}
\begin{proof}[Proof of Claim \ref{claim:hat R wt R}] Let $\xi\in\Ga^p_\lam(\g_P)$ be such that
\[\Vert\xi\Vert_{p,\lam}\leq1.\]
We denote $\wt\xi:=\xi-d_A^*\hhat R\xi$. Inequality (\ref{eq:wt v infty}), Remark \ref{rmk:c}, and the assumption $\Vert\xi\Vert_{p,\lam}\leq1$ imply that 
\begin{equation}\label{eq:wt R xi}\Vert \wt R\xi\Vert_\infty\leq c^{-1}\Vert\wt\xi\Vert_{L^\infty(\C\wo (-N,N)^2)}\leq c^{-1}C_2C_4,
\end{equation}
where $c$ is defined as in (\ref{eq:c inf L x xi}). We fix $i\in\Z$ and denote $\al_i:=R_i(\xi|_{\Om_i})$. For $i\neq0$ Lemma \ref{le:infty 1 p A} in Appendix \ref{sec:proof:prop:right}, (\ref{eq:R i 1 p}), the fact $\lam>1-2/p>0$, and the assumption $\Vert\xi\Vert_{p,\lam}\leq1$ imply that
\[\Vert \al_i\Vert_\infty\leq C_2\Vert\al_i\Vert_{1,p,\Om_i,A}\leq C_2C_1\Vert\xi\Vert_{L^p(\Om_i)}\leq C_2C_1.\] 
Furthermore, we have
\[\Vert\al_0\Vert_\infty\leq C_5\Vert R_0\Vert\,\Vert\xi|_{\Om_0}\Vert_p\leq C_5\Vert R_0\Vert,\]
where 
\[C_5:=\sup\Big\{\Vert\al\Vert_\infty\,\Big|\,\al\in\Ga^{1,p}_A\big(\wone(\g_P|_{\Om_0})\big):\,\Vert\al\Vert_{1,p,A}\leq1\Big\}.\]
An argument involving a finite cover of $\Om_0$ by small enough balls and Lemma \ref{le:infty 1 p A}, implies that $C_5<\infty$. It follows that
\begin{equation}\label{eq:hhat R xi infty}\Vert\hhat R\xi\Vert_\infty\leq\sup_i\Vert \al_i\Vert_\infty\leq \max\big\{C_2C_1,C_5\Vert R_0\Vert\big\}<\infty.
\end{equation}
We define 
\[C:=\max\left\{|d\mu(x)|\,\big|\,x\in\BAR{u(P)}\right\}.\] 
The definition (\ref{eq:wt R xi z}), the second inequality in (\ref{eq:L x * L x c}) (Remark \ref{rmk:c} in Appendix \ref{sec:add}), the statement of Claim \ref{claim:d A * wt al}, and the assumption $\Vert\xi\Vert_{p,\lam}\leq1$ imply that
\begin{equation}\label{eq:Vert d mu u}\big\Vert d\mu(u)\wt R\xi\big\Vert_{p,\lam}\leq Cc^{-1}\Vert\wt\xi\Vert_{L^p_\lam(\C\wo (-N,N)^2)}\leq Cc^{-1}C_4.
\end{equation}
By the last equality in (\ref{eq:L x * L x c}), we have
\[{\PR^u}\wt R\xi=L_u(L_u^*L_u)^{-1}\wt\xi.\]
Hence using again the second inequality in (\ref{eq:L x * L x c}) and the statement of Claim \ref{claim:d A * wt al}, we obtain
\begin{equation}\label{eq:Vert PR wt R xi}\big\Vert {\PR}^u \wt R\xi\big\Vert_{p,\lam}\leq c^{-1}C_4.
\end{equation}
For every $z\in\C$ there are at most four indices $i\in\Z$ for which $\rho_i(z)\neq0$. Therefore, we have 
\[\Vert\hhat R\xi\Vert_{p,\lam}^p\leq4^p\sum_i\Vert \al_i\Vert_{L^p_\lam(\Om_i)}^p.\] 
Using (\ref{eq:R i 1 p},\ref{eq:sup C 3}) and the assumption $\Vert\xi\Vert_{p,\lam}\leq1$, it follows that 
\[\Vert\hhat R\xi\Vert_{p,\lam}^p\leq4^p\max\{C_1^p,\Vert R_0\Vert^p\}C_3.\]
Combining this with (\ref{eq:wt R xi},\ref{eq:hhat R xi infty},\ref{eq:Vert d mu u},\ref{eq:Vert PR wt R xi}), Claim \ref{claim:hat R wt R} follows.  
\end{proof}
\begin{claim}\label{claim:na A p lam} We have 
\[\sup\big\{\Vert\na^A(\wt R\xi)\Vert_{p,\lam}\,\big|\,\xi\in\Ga^p_\lam(\g_P):\,\Vert\xi\Vert_{p,\lam}\leq1\big\}<\infty.\] 
\end{claim}
\begin{proof}[Proof of Claim \ref{claim:na A p lam}] Let $\xi\in\Ga^p_\lam(\g_P)$. We define 
\[\wt\xi:=\xi-d_A^*\hhat R\xi,\quad\eta:=(L_u^*L_u)^{-1}\wt\xi,\] 
and $\rho\in\Om^2(M,\g)$ as in (\ref{eq:xi rho v v'}) in Appendix \ref{sec:add}. By Lemma \ref{le:na A L u} in the same appendix, we have 
\begin{equation}\label{eq:na A L u eta}\na^A(L_u\eta)=L_ud_A\eta+\na_{d_Au}X_\eta,\end{equation} 
where $X_{\xi_0}$ denotes the vector field on $M$ generated by an element $\xi_0\in\g$. Using the second part of Lemma \ref{le:na A L u} (with $v:=L_u\eta$), it follows that 
\begin{equation}\label{eq:d A wt xi}L_u^*L_ud_A\eta =d_A\wt\xi-\rho(d_Au,L_u\eta)-L_u^*\na_{d_Au}X_\eta.\end{equation}
We choose a constant $C$ so big that 
\[|\rho(v,v')|\leq C|v\Vert v'|,\quad|\na_vX_{\xi_0}|\leq C|v\Vert\xi_0|,\quad\forall x\in\mu^{-1}(\bar B_\de),\,v,v'\in T_xM,\,\xi_0\in\g.\]
We define $C_0:=\max\big\{c^{-1},3Cc^{-2}\big\}$. Since $\wt R\xi=L_u\eta$, equalities (\ref{eq:na A L u eta},\ref{eq:d A wt xi}) and Remark \ref{rmk:c} imply that 
\[\Vert\na^A(\wt R\xi)\Vert_{p,\lam}\leq C_0\big(\big\Vert d_A\wt\xi\big\Vert_{p,\lam}+\Vert d_Au\Vert_{p,\lam}\big\Vert\wt\xi\big\Vert_\infty\big).\] 
Since $\Vert d_Au\Vert_{p,\lam}<\infty$, Claim \ref{claim:d A * wt al} and (\ref{eq:wt v infty}) now imply Claim \ref{claim:na A p lam}.
\end{proof}
\begin{claim}\label{claim:na A al p lam} We have 
\[\sup\big\{\Vert\na^A(\hhat R\xi)\Vert_{p,\lam}\,\big|\,\xi\in\Ga^p_\lam(\g_P):\,\Vert\xi\Vert_{p,\lam}\leq1\big\}<\infty.\]
\end{claim}
\begin{pf}[Proof of Claim \ref{claim:na A al p lam}] Let $\xi\in\Ga^p_\lam(\g_P)$ be such that $\Vert\xi\Vert_{p,\lam}\leq1$. We write $\al_i:=R_i(\xi|_{\Om_i})$. Then we have 
\[\na^A(\hhat R\xi)=\sum_i\big(\rho_i\na^A\al_i + d\rho_i\otimes\al_i\big).\] 
Setting 
\[C:=8^p\Vert\rho'\Vert_\infty^pC_3\max\{C_1^p,\Vert R_0\Vert^p\},\] 
it follows that 
\[\Vert\na^A(\hhat R\xi)\Vert_{p,\lam}^p\leq8^{p-1}\sum_i\Big(\Vert\na^A\al_i\Vert_{p,\lam}^p + \Vert\rho'\Vert_\infty^p\Vert\al_i\Vert_{p,\lam}^p\Big)\leq C.\] 
Here in the second inequality we used the fact $\Vert\rho'\Vert_\infty\geq1$, and (\ref{eq:R i 1 p}). This proves Claim \ref{claim:na A al p lam}, and completes the proof of Claim \ref{claim:ze X w} and hence of Theorem \ref{thm:L w * R}. 
\end{pf}
\end{pf}
\end{proof}

\appendix

\chapter[Auxiliary results]{Auxiliary results about vortices, weighted spaces, and other topics}\label{chap:vort weight other}
In the appendix some additional results are recollected, which were used in the proofs of the main theorems of this memoir. As always, we denote $\N_0:=\N\cup\{0\}$, by $B_r\sub\C$ the open ball of radius $r$, by $S^1_r\sub\C$ the circle of radius $r$, centered at $0$, and by $\A(P)$ the affine space of smooth connection one-forms on a smooth principal bundle $P$. Let $M,\om,G,\g,\lan\cdot,\cdot\ran_\g,\mu,J$ be as in Chapter \ref{chap:main}, and $(\Si,j)$ a Riemann surface, equipped with a compatible area form $\om_\Si$. For $\xi\in\g$ and $x\in M$ we denote by $L_x\xi\in T_xM$ the (infinitesimal) action of $\xi$ at $x$. Let $p>2$, $P$ be a (principal) $G$-bundle over $\Si$ of class $W^{2,p}_\loc$, $A$ a connection one-form on $P$ of class $W^{1,p}_\loc$, and $u:P\to M$ a $G$-equivariant map of class $W^{1,p}_\loc$. We denote $w:=(P,A,u)$, define the energy density $e_w$ as in (\ref{eq:e W}), and denote by
\[E(w):=\int_\Si e_w\om_\Si\]
the energy of $w$.
\section{Auxiliary results about vortices}\label{sec:vort}
Let $M,\om,G,\g,\lan\cdot,\cdot\ran_\g,\mu,J$ be as in Chapter \ref{chap:main}.%
\footnote{As always, we assume that hypothesis (H) is satisfied.}
 The following result was used in the proof of Proposition \ref{prop:E d} in Section \ref{SEC:EXAMPLE}. 
\begin{prop}\label{prop:E int u om} Let $w:=(P,A,u)$ be a smooth vortex over $\C$ with finite energy, such that the closure of the image $u(P)\sub M$ is compact. Then there exists a smooth section $\si:\C\to P$, such that 
\begin{equation}\label{eq:int B R E}\int_{B_R}(u\circ\si)^*\om\to E(w),\quad\textrm{as }R\to\infty.
\end{equation}
\end{prop}
This result is a consequence of \cite[Proposition 11.1]{GS}. For the convenience of the reader we include a proof here.%
\footnote{This proof is similar to the one of \cite[Proposition 11.1]{GS}, however, it relies on the isoperimetric inequality for the invariant symplectic action functional proved in \cite{ZiA} (Theorem 1.2) rather than the earlier inequality \cite[Lemma 11.3]{GS}.}
 We need the following. We denote by $\ga$ the standard angular one-form on $\R^2\wo\{0\}$.%
\footnote{By our convention this form integrates to $2\pi$ over any circle centered at the origin.}%
\begin{lemma}\label{le:int S R si A}Let $P$ be a smooth $G$-bundle over $\R^2$ and $A\in\A(P)$. Then there exists a smooth section $\si:\R^2\to P$, such that 
\begin{equation}\label{eq:int S 1 R |}\int_{S^1_R}|\si^*A|R\ga\leq\int_{B_R\wo B_1}|F_A|d^2x+\int_{S^1_1}|\si^*A|\ga,\end{equation}
where the norms are with respect to the standard metric on $\R^2$. 
\end{lemma}
\begin{proof}[Proof of Lemma \ref{le:int S R si A}]\setcounter{claim}{0} We choose $\si$ such that
\begin{equation}\label{eq:si A x x}(\si^*A)_xx=0,\quad\forall x\in\R^2\wo B_1.
\end{equation}
(This means that $\si^*A$ is in radial gauge on $\R^2\wo B_1$. Such a section exists, since (\ref{eq:si A x x}) corresponds to a family of ordinary differential equations, one for each direction $x\in S^1$.) We identify $\R/(2\pi\Z)$ with $S^1$ and define 
\[\phi:\Si:=[0,\infty)\x S^1\to\C=\R^2,\quad\phi(s,t):=e^{s+it},\quad\Psi:=\big((\si\circ\phi)^*A\big)_t:\Si\to\g,\]
where the subscript ``$t$'' refers to the $t$-component of the one-form $(\si\circ\phi)^*A$. It follows from (\ref{eq:si A x x}) that the $s$-component of this form vanishes, and therefore, 
\[(\si\circ\phi)^*A=\Psi dt,\quad(\si\circ\phi)^*F_A=\dd_s\Psi\,ds\,dt.\]
Using the estimate
\[|\Psi(s,t)|\leq\int_0^s|\dd_s\Psi(s',t)|ds'+|\Psi(0,t)|,\]
it follows that 
\[\int_{\{s\}\x S^1}\big|(\si\circ\phi)^*A\big|_\Si dt=\int_{[0,s]\x S^1}\big|(\si\circ\phi)^*F_A\big|_\Si ds'dt+\int_{\{0\}\x S^1}\big|(\si\circ\phi)^*A\big|_\Si dt,\]
for every $s\geq0$, where the subscript ``$\Si$'' indicates that the norms are taken with respect to the standard metric on $\Si$. Inequality (\ref{eq:int S 1 R |}) follows from this by a straight-forward calculation. This proves Lemma \ref{le:int S R si A}.
\end{proof}
\begin{proof}[Proof of Proposition \ref{prop:E int u om}]\setcounter{claim}{0} Let $w:=(P,A,u)$ be as in the hypothesis and $R>0$. We choose a section $\si:\C\to P$ as in Lemma \ref{le:int S R si A}. Denoting by $E(w,B_R)$ the energy of $w$ over $B_R$, we have, by \cite[Proposition 3.1]{CGS},
\begin{eqnarray}\nn E(w,B_R)&=&\int_{B_R}\big((u\circ\si)^*\om-d\big\lan\mu\circ u\circ\si,\si^*A\big\ran_\g\big)\\
\label{eq:E w B R int}&=&\int_{B_R}(u\circ\si)^*\om-\int_{S^1_R}\big\lan\mu\circ u\circ\si,\si^*A\big\ran_\g.
\end{eqnarray}
Let $\eps>0$. By \cite[Corollary 1.4]{ZiA} there exists a constant $C_1$ such that
\begin{equation}\label{eq:sqrt e w z}\sqrt{e_w(z)}\leq C_1|z|^{-2+\eps},\quad\forall z\in\C\wo B_1.\end{equation}
Combining this with the estimate $|F_A|\leq\sqrt{e_w}$, and using (\ref{eq:int S 1 R |}), it follows that
\[\int_{S^1_R}|\si^*A|R\ga\leq C_2R^\eps+C_3,\quad\textrm{where }C_2:=\frac{2\pi C_1}\eps,\,C_3:=\int_{S^1_1}|\si^*A|\ga.\]
Combining this with the inequality $|\mu\circ u|\leq\sqrt{e_w}$ and (\ref{eq:sqrt e w z}), it follows that 
\[\left|\int_{S^1_R}\big\lan\mu\circ u\circ\si,\si^*A\big\ran_\g\right|\leq C_1C_2R^{-2+2\eps}+C_1C_3R^{-2+\eps}.\]
By choosing $\eps\in(0,1)$ and using (\ref{eq:E w B R int}), the convergence (\ref{eq:int B R E}) follows. This proves Proposition \ref{prop:E int u om}.
\end{proof}
The next result was used in the proofs of Propositions \ref{prop:soft} and \ref{prop:en conc} (Section \ref{sec:soft}).
\begin{lemma}[Bound on energy density]\label{le:a priori} Let $K\sub M$ be a compact subset. Then there exists a constant $E_0>0$ such that the following holds. Let $r>0$, $P$ be a smooth $G$-bundle over $B_r$, $p>2$, and $(A,u)$ a vortex on $P$ of class $W^{1,p}_\loc$, such that
\begin{eqnarray}
\nn &u(P)\sub K,&\\
\nn&E(w,B_r)\leq E_0.&
\end{eqnarray}
(where $E(w,B_r)$ denotes the energy of $w$ over the ball $B_r$). Then we have 
\begin{equation}
\nn e_w(0)\leq \frac8{\pi r^2}E(w,B_r).
\end{equation}
\end{lemma}
For the proof of Lemma \ref{le:a priori} we need the following lemma.
\begin{lemma}[Heinz]\label{le:Heinz} Let $r>0$ and $c\geq0$. Then for every function $f\in C^2(B_r,\R)$ satisfying the inequalities
\begin{equation}
\nn f\geq0,\qquad \La f\geq-cf^2,\qquad \int_{B_r}f<\frac\pi{8c}
\end{equation}
we have
\begin{equation}
\nn f(0)\leq\frac8{\pi r^2}\int_{B_r}f.
\end{equation}
\end{lemma}
\begin{proof}[Proof of Lemma \ref{le:Heinz}] This is \cite[Lemma 4.3.2]{MS04}.
\end{proof}
\begin{proof}[Proof of Lemma \ref{le:a priori}] Since $G$ is compact, we may assume w.l.o.g.~that $K$ is $G$-invariant. The result then follows from Theorem \ref{thm:reg gauge bdd} below, the calculation in Step 1 of the proof of \cite[Proposition 11.1.]{GS}, and Lemma \ref{le:Heinz}.
\end{proof}
Lemma \ref{le:a priori} has the following consequence.
\begin{cor}[Quantization of energy]\label{cor:quant} If $M$ is equivariantly convex at $\infty$
\footnote{as defined in Chapter \ref{chap:main} on p.~\pageref{convex}}%
, then we have
\[\inf_wE(w)>0,\] 
where $w=(P,A,u)$ ranges over all vortices over $\C$ with $P$ smooth and $(A,u)$ of class $W^{1,p}_\loc$ for some $p>2$, such that $E(w)>0$ and $\bar u(P)$ is compact.
\end{cor}
\begin{proof}[Proof of Corollary \ref{cor:quant}] This is an immediate consequence of Proposition \ref{prop:bounded} below and Lemma \ref{le:a priori}.
\end{proof}
This corollary implies that the minimal energy $E_1$ of a vortex over $\C$ (defined as in (\ref{eq:E V})) is positive, and therefore $\Emin>0$ (defined as in (\ref{eq:Emin})). 

The next lemma was used in the proofs of Proposition \ref{prop:cpt bdd} and Lemma \ref{le:conv e} (Section \ref{sec:comp}). It is a consequence of \cite[Lemma 9.1]{GS}. Let $(\Si,\om_\Si,j)$ be a surface with an area form and a compatible complex structure.
\begin{lemma}[Bounds on the momentum map component]\label{le:mu u} Let $c>0$, $Q\sub\Si\wo\dd\Si$ and $K\sub M$ be compact subsets, and $p>2$. Then there exist positive constants $R_0$ and $C_p$ such that the following holds. Let $R\geq R_0$, $P$ be a smooth $G$-bundle over $\Si$, and $(A,u)$ an $R$-vortex on $P$ of class $W^{1,p}_\loc$, such that 
\begin{eqnarray}\nn &u(P)\sub K,&\\
\nn &\Vert d_Au\Vert_{L^\infty(\Si)}\leq c,&\\
\nn &|\xi|\leq c|L_{u(p)}\xi|,\quad\forall p\in P,\forall \xi\in\g.&
\end{eqnarray}
Then 
\[\int_Q|\mu\circ u|^p\om_\Si\leq C_pR^{-2p},\qquad \sup_Q|\mu\circ u|\leq C_pR^{\frac2p-2},\]
where we view $|\mu\circ u|$ as a function from $\Si$ to $\R$. 
\end{lemma}

\begin{proof}[Proof of Lemma \ref{le:mu u}] This follows from the proof of \cite[Lemma 9.1]{GS}, using Theorem \ref{thm:reg gauge bdd} below. 
\end{proof}
The next result was also used in the proofs of Proposition \ref{prop:en conc} (Section \ref{sec:soft}) and Lemma \ref{le:a priori}, and will be used in the proof of Proposition \ref{prop:reg gauge}.
\begin{thm}[Regularity modulo gauge over a compact surface]\label{thm:reg gauge bdd} Assume that $\Si$ is compact. Let $P$ be a smooth $G$-bundle over $\Si$, $p>2$, and $(A,u)$ a vortex on $P$ of class $W^{1,p}$. Then there exists a gauge transformation $g\in W^{2,p}(\Si,G)$ such that $g^*w$ is smooth over $\Si\wo\dd\Si$. 
\end{thm}
\begin{proof}[Proof of Theorem \ref{thm:reg gauge bdd}] This follows from the proof of \cite[Theorem 3.1]{CGMS}, using a version of the local slice theorem allowing for boundary (see \cite[Theorem 8.1]{We}).
\end{proof}
The next result will be used in the proof of Proposition \ref{prop:bounded} below.
\begin{prop}[Regularity modulo gauge over $\C$]\label{prop:reg gauge} Let $R\geq0$ be a number, $P$ a smooth $G$-bundle over $\C$, $p>2$, and $w:=(A,u)$ an $R$-vortex on $P$ of class $W^{1,p}_\loc$. Then there exists a gauge transformation $g$ on $P$ of class $W^{2,p}_\loc$ such that $g^*w$ is smooth.
\end{prop}
\begin{proof}[Proof of Proposition \ref{prop:reg gauge}]\setcounter{claim}{0}%
\footnote{This proof follows the lines of the proofs of \cite[Theorems 3.6 and A.3]{FrPhD}.}%
\begin{Claim}There exists a collection $(g_j)_{j\in\N}$, where $g_j$ is a gauge transformation over $B_{j+1}$ of class $W^{2,p}$, such that for every $j\in\N$, we have
\begin{eqnarray}\label{eq:g j w}&g_j^*w\textrm{ is smooth over }B_{j+1},\\
\label{eq:g j g j-1}&g_{j+1}=g_j\textrm{ over }B_j.&
\end{eqnarray}
\end{Claim}
\begin{proof}[Proof of the claim] By Theorem \ref{thm:reg gauge bdd} there exists a gauge transformation $g_1\in W^{2,p}(B_2,G)$ such that $g_1^*w$ is smooth. Let $\ell\in\N$ be an integer and assume by induction that there exist gauge transformations $g_j\in W^{2,p}(B_{j+1},G)$, for $j=1,\ldots,\ell$, such that (\ref{eq:g j w}) holds for $j=1,\ldots,\ell$, and (\ref{eq:g j g j-1}) holds for $j=1,\ldots,\ell-1$. We show that there exists a gauge transformation $g_{\ell+1}\in W^{2,p}(B_{\ell+2},G)$ such that 
\begin{eqnarray}\label{eq:g ell+1}&g_{\ell+1}^*w\textrm{ smooth over }B_{\ell+2},&\\
\label{eq:g ell+1 g ell}&g_{\ell+1}=g_\ell\textrm{ over }B_\ell.& 
\end{eqnarray}
We choose a smooth function $\rho:\bar B_{\ell+2}\to B_{\ell+1}$ such that $\rho(z)=z$ for $z\in B_\ell$. By Theorem \ref{thm:reg gauge bdd} there exists a gauge transformation $h\in W^{2,p}(B_{\ell+2},G)$ such that $h^*w$ is smooth over $\bar B_{\ell+2}$. We define
\[g_{\ell+1}:=h\big((h^{-1}g_\ell)\circ\rho\big).\]
Then $g_{\ell+1}$ is of class $W^{2,p}$ over $B_{\ell+2}$, and (\ref{eq:g ell+1 g ell}) is satisfied. Furthermore, $g_\ell^*w=(h^{-1}g_\ell)^*h^*w$ is smooth over $B_{\ell+1}$. Therefore, using smoothness of $h^*w$ over $B_{\ell+2}$, Lemma \ref{le:g smooth}(\ref{le:g k+1 p}) below implies that $h^{-1}g_\ell$ is smooth over $B_{\ell+1}$. It follows that
\[g_{\ell+1}^*w=\big((h^{-1}g_\ell)\circ\rho\big)^*h^*w\]
is smooth over $B_{\ell+2}$. This proves (\ref{eq:g ell+1}), terminates the induction, and concludes the proof of the claim.  
\end{proof}
We choose a collection $(g_j)$ as in the claim, and define $g$ to be the unique gauge transformation on $P$ that restricts to $g_j$ over $B_j$. This makes sense by (\ref{eq:g j g j-1}). Furthermore, (\ref{eq:g j w}) implies that $g^*w$ is smooth. This proves Proposition \ref{prop:reg gauge}. 
\end{proof}
The next result was used in the proof of Proposition \ref{prop:cpt bdd} (Section \ref{sec:comp}).
\begin{thm}[Compactness for vortex classes over compact surface]\label{thm:cpt cpt} Let $\Si$ be a compact surface (possibly with boundary), $\om_\Si$ an area form, $j$ a compatible complex structure on $\Si$, $P$ a $G$-bundle over $\Si$, $K\sub M$ a compact subset, $R_\nu\in[0,\infty)$, $p>2$, and $(A_\nu,u_\nu)$ an $R_\nu$-vortex on $P$ of class $W^{1,p}$, for every $\nu\in\N$. Assume that $R_\nu$ converges to some $R_0\in[0,\infty)$, and
\[u_\nu(P)\sub K,\quad\sup_\nu\Vert d_{A_\nu}u_\nu\Vert_{L^p(\Si)}<\infty.\]
Then there exist a smooth $R_0$-vortex $(A_0,u_0)$ on $P|(\Si\wo\dd\Si)$ and gauge transformations $g_\nu$ on $P$ of class $W^{2,p}$, such that $g_\nu^*(A_\nu,u_\nu)$ converges to $(A_0,u_0)$, in $C^\infty$ on every compact subset of $\Si\wo\dd\Si$. 
\end{thm}
\begin{proof}[Proof of Theorem \ref{thm:cpt cpt}] This follows from a modified version of the proof of \cite[Theorem 3.2]{CGMS}: We use a version of Uhlenbeck compactness for a compact base with boundary, see Theorem \ref{thm:Uhlenbeck compact} below, and a version of the local slice theorem allowing for boundary, see \cite[Theorem 8.1]{We}. Note that the proof carries over to the case in which $R_\nu=0$ for some $\nu\in\N$, or $R_0=0$. 
\end{proof}
The following result was used in the proofs of Theorem \ref{thm:bubb} (Section \ref{sec:proof:thm:bubb}) and Corollary \ref{cor:quant}. 
\begin{prop}[Boundedness of image]\label{prop:bounded} Assume that $M$ is equivariantly convex at $\infty$. Then there exists a $G$-invariant compact subset $K_0\sub M$ such that the following holds. Let $p>2$, $P$ a $G$-bundle over $\C$, and $(A,u)$ a vortex on $P$ of class $W^{1,p}_\loc$, such that $E(w)<\infty$ and $\BAR{u(P)}$ is compact. Then we have $u(P)\sub K_0$.
\end{prop}
\begin{proof}[Proof of Proposition \ref{prop:bounded}] Let $P$ be a $G$-bundle over $\C$. By an elementary argument every smooth vortex on $P$ is smoothly gauge equivalent to a smooth vortex that is in radial gauge outside $B_1$. Using Proposition \ref{prop:reg gauge}, it follows that every vortex on $P$ of class $W^{1,p}_\loc$ is gauge equivalent to a smooth vortex that is in radial gauge outside $B_1$. Hence the statement of Proposition \ref{prop:bounded} follows from \cite[Proposition 11.1]{GS}.
\end{proof}
\section{The invariant symplectic action}\label{sec:action}
The proof of Proposition \ref{prop:en conc} (Energy concentration near ends) in Section \ref{sec:soft} was based on an isoperimetric inequality and an energy action identity for the invariant action functional (Theorem \ref{thm:isoperi} and Proposition \ref{prop:en act} below). Building on work by D.~A.~Salamon and R.~Gaio \cite{GS}, we define this functional as follows.%
\footnote{This is the definition from \cite{ZiA}, written in a more intrinsic way.}
 We first review the usual symplectic action functional: Let $(M,\om)$ be a symplectic manifold without boundary. We fix a Riemannian metric $\lan\cdot,\cdot\ran_M$ on $M$, and denote by $d,\exp,|v|,\iota_x>0$, and $\iota_X:=\inf_{x\in X}\iota_x\geq0$ the distance function, the exponential map, the norm of a vector $v\in TM$, and the injectivity radii of a point $x\in M$ and a subset $X\sub M$, respectively. We define the symplectic action of a loop $x:S^1\to M$ of length $\ell(x)<2\iota_{x(S^1)}$ to be
\begin{equation}
\nn \A(x):=-\int_\DDD u^*\om.
\end{equation}
Here $\DDD\sub\C$ denotes the (closed) unit disk, and $u:\DDD\to M$ is any smooth map such that 
\begin{equation}\nn u(e^{it})=x(t),\,\forall t\in \R/(2\pi\Z)\iso S^1,\quad d\big(u(z),u(z')\big)<\iota_{x(S^1)},\,\forall z,z'\in\DDD.
\end{equation}
\begin{lemma}\label{le:A} The action $\A(x)$ is well-defined, i.e., a map $u$ as above exists, and $\A(x)$ does not depend on the choice of $u$. 
\end{lemma}
\begin{proof} The lemma follows from an elementary argument, using the exponential map $\exp_{x(0+\Z)}:T_{x(0+\Z)}M\to M$. \end{proof}
Let now $G$ be a compact connected Lie group with Lie algebra $\g$. Suppose that $G$ acts on $M$ in a Hamiltonian way, with (equivariant) momentum map $\mu:M\to\g^*$, and that $\lan\cdot,\cdot\ran_M$ is $G$-invariant. We denote by $\lan\cdot,\cdot\ran:\g^*\x\g\to\R$ the natural contraction. Let $P$ be a smooth $G$-bundle over $S^1$ and $x\in C^\infty_G(P,M)$. We call $(P,x)$ \emph{admissible} iff there exists a section $s:S^1\to P$ such that $\ell(x\circ s)<2\iota_{x(P)}$, and 
\begin{equation}\nn\A(g\cdot(x\circ s))-\A(x\circ s)=\int_{S^1}\big\lan\mu\circ x\circ s,g^{-1}dg\big\ran,
\end{equation}
for every $g\in C^\infty(S^1,G)$ satisfying $\ell(g\cdot(x\circ s))\leq \ell(x\circ s)$. 
\begin{defi}\label{defi:action} Let $(P,x)$ be an admissible pair, and $A$ be a connection on $P$. We define the \emph{invariant symplectic action of $(P,A,x)$} to be 
\begin{equation}\nn\A(P,A,x):=\A(x\circ s)+\int_{S^1}\big\lan\mu\circ x\circ s,A\,ds\big\ran,
\end{equation}
where $s:S^1\to P$ is a section as above.
\end{defi}
To formulate the isoperimetric inequality, we need the following. If $X$ is a manifold, $P$ a $G$-bundle over $X$ and $u\in C^\infty_G(P,M)$, then we define $\bar u:X\to M$ by $\bar u(y):=Gu(p)$, where $p\in P$ is any point in the fiber over $y$. We define $M^*$ as in (\ref{eq:M *}). For a loop $\bar x:S^1\to M^*/G$ we denote by $\bar\ell(\bar x)$ its length \wrt the Riemannian metric on $M^*/G$ induced by $\lan\cdot,\cdot\ran_M$. Furthermore, for each subset $X\sub M$ we define 
\[m_X:=\inf\big\{|L_x\xi|\,\big|\,x\in X,\,\xi\in\g:\,|\xi|=1\big\}.\]
The first ingredient of the proof of Proposition \ref{prop:en conc} is the following.
\begin{thm}[Isoperimetric inequality]\label{thm:isoperi} Assume that there exists a $G$-invariant $\om$-compatible almost complex structure $J$ such that $\lan\cdot,\cdot\ran_M=\om(\cdot,J\cdot)$. Then for every compact subset $K\sub M^*$ and every constant $c>\frac12$ there exists a constant $\de>0$ with the following property. Let $P$ be a $G$-bundle over $S^1$ and $x\in C^\infty_G(P,M)$, such that $x(P)\sub K$ and $\bar\ell(\bar x)\leq\de$. Then $(P,x)$ is admissible, and for every connection $A$ on $P$ we have
\begin{equation}\nn|\A(P,A,x)|\leq c\Vert d_Ax\Vert_2^2+\frac1{2m_K^2}\Vert\mu\circ x\Vert_2^2.
\end{equation}
\end{thm}
Here we view $d_Ax$ as a one-form on $S^1$ with values in the bundle $(x^*TM)/G\to S^1$, and $\mu\circ x$ as a section of the co-adjoint bundle $(P\x\g^*)/G\to S^1$. Furthermore, $S^1$ is identified with $\R/(2\pi\Z)$, and the norms are taken with respect to the standard metric on $\R/(2\pi\Z)$, the metric $\lan\cdot,\cdot\ran_M$ on $M$, and the operator norm on $\g^*$ induced by $\lan\cdot,\cdot\ran_\g$.%
\footnote{Note that in \cite[Theorem 1.2]{ZiA} $S^1$ was identified with $\R/\Z$ instead. Note also that hypothesis (H) is not needed for Theorem \ref{thm:isoperi}.}%
\begin{proof}[Proof of Theorem \ref{thm:isoperi}] This is \cite[Theorem 1.2]{ZiA}. 
\end{proof}

The second ingredient of the proof of Proposition \ref{prop:en conc} is the following. For $s\in\R$ we denote by $\iota_s:S^1\to\R\x S^1$ the map given by $\iota_s(t):=(s,t)$. Let $X,X'$ be manifolds, $f\in C^\infty(X',X)$, $P$ a $G$-bundle over $X$, $A$ a connection on $P$, and $u\in C^\infty_G(P,M)$. Then the pullback triple $f^*(P,A,u)$ consists of a $G$-bundle $P'$ over $X'$, a connection on $P'$, and an equivariant map from $P'$ to $M$. 
\begin{prop}[Energy action identity]\label{prop:en act} For every compact subset $K\sub M^*$ there exists a constant $\de>0$ with the following property. Let $s_-\leq s_+$ be numbers, $\Si:=[s_-,s_+]\x S^1$, $\om_\Si$ an area form on $\Si$, $j$ a compatible complex structure, and $w:=(A,u)$ a smooth vortex over $\Si$, such that $u(P)\sub K$ and $\bar\ell(\bar u\circ\iota_s)<\de$, for every $s\in [s_-,s_+]$. Then the pairs $\iota_{s_\pm}^*(P,u)$ are admissible, and 
\begin{equation}
\nn E(w)=-\A\big(\iota_{s_+}^*(P,A,u)\big)+\A\big(\iota_{s_-}^*(P,A,u)\big).
\end{equation}
\end{prop}
\begin{proof}[Proof of Proposition \ref{prop:en act}] This follows from \cite[Proposition 3.1]{ZiA}.
\end{proof}
\section{Proofs of the results of Section \ref{SEC:HOMOLOGY CHERN}}\label{sec:proofs homology}
This section contains the proofs of Lemmas \ref{le:P},\ref{le:section} and Proposition \ref{prop:Chern Maslov}, which were stated in Section \ref{SEC:HOMOLOGY CHERN}. We also state and prove Lemma \ref{le:homotopic}, which was used in Definition \ref{defi:homotopy W} in that section. Let $M,\om,G,,\g,\lan\cdot,\cdot\ran_\g,\mu$ and $J$ be as in Chapter \ref{chap:main}, $\Si:=\C$, $\om_\Si$ the standard area form $\om_0$, $p>2$, and $\lam>1-2/p$. Assume that hypothesis (H) holds.

Lemma \ref{le:P} was used in the definition of the equivariant homology class of an equivalence class of triples $(P,A,u)$ (Definition \ref{defi:homology}). Its proof is based on the following result, which was also used in the proofs of Theorem \ref{thm:Fredholm} (Section \ref{subsec:reform}), \ref{thm:L w * R} (Section \ref{subsec:proof:thm:L w * R}) and Proposition \ref{prop:triv} (Section \ref{subsec:proof:Fredholm aug}). 
\begin{lemma}\label{le:si} For every $(P,A,u)\in\BB^p_\lam$ there exists a section $\si$ of the restriction of the bundle $P$ to $B_1^C:=\C\wo B_1$, of class $W^{1,p}_\loc$, and a point $x_\infty\in\mu^{-1}(0)$, such that $u\circ \si(re^{i\phi})$ converges to $x_\infty$, uniformly in $\phi\in\R$, as $r\to\infty$, and $\si^*A\in L^p_\lam(B_1^C)$. 
\end{lemma}
\begin{proof}[Proof of Lemma \ref{le:si}]\label{proof:le:si}\setcounter{claim}{0} 
\begin{claim}\label{claim:R u} The expression $|\mu\circ u|(re^{i\phi})$ converges to 0, uniformly in $\phi\in\R$, as $r\to\infty$. 
\end{claim}
\begin{proof}[Proof of Claim \ref{claim:R u}] We define the function $f:=|\mu\circ u|^2:M\to\R$. It follows from the ad-invariance of $\lan\cdot,\cdot\ran_\g$ that 
\begin{equation}\label{eq:d mu u 2}df=2\big\lan d_A(\mu\circ u),\mu\circ u\big\ran_\g=2\big\lan d\mu(u)d_Au,\mu\circ u\big\ran_\g.\end{equation}
Since $\BAR{u(P)}\sub M$ is compact, we have $\sup_{\C}|d\mu(u)|<\infty$ and $\sup_{\C}|\mu\circ u|<\infty$. Furthermore, $|d_Au|\leq \sqrt{e_w}\in L^p_\lam$. Combining this with (\ref{eq:d mu u 2}), it follows that $df\in L^p_\lam$. Therefore, by Proposition \ref{prop:Hardy} in the next section (Hardy-type inequality, applied with $u$ replaced by $f$ and $\lam$ replaced by $\lam-1$) the expression $f(re^{i\phi})$ converges to some number $y_\infty\in \R$, as $r\to\infty$, uniformly in $\phi\in\R$. Since $|\mu\circ u|\leq\sqrt{e_w}\in L^p_\lam$, it follows that $y_\infty=0$. 
This proves Claim \ref{claim:R u}.
\end{proof}
It follows from hypothesis (H) that there exists a $\de>0$ such that $\mu^{-1}(\bar B_\de)\sub M^*$ (defined as in (\ref{eq:M *})). We choose $R>0$ so big that $|\mu\circ u|(z)\leq\de$ if $z\in B_{R-1}^C=\C\wo B_{R-1}$. Since $G$ is compact, the action of it on $M$ is proper. Hence the local slice theorem implies that $M^*/G$ carries a unique manifold structure such that the canonical projection $\pi^{M^*}:M^*\to M^*/G$ is a submersion. Consider the map $\bar u:B_{R-1}^C\to M^*/G$ defined by $\bar u(z):=G u(p)$, where $p\in\pi^{-1}(z)\sub P$ is arbitrary. 
\begin{claim}\label{claim:bar u} The point $\bar u(re^{i\phi})$ converges to some point $\bar x_\infty\in \mu^{-1}(0)/G$, uniformly in $\phi\in\R$, as $r\to\infty$. 
\end{claim}
\begin{proof}[Proof of Claim \ref{claim:bar u}] We choose $n\in\N$ and an embedding $\iota:M^*/G\to \R^n$. Furthermore, we choose a smooth function $\rho:\C\to\R$ that vanishes on $B_{R-1}$ and equals 1 on $B_R^C$. We define $f:\C\to \R^n$ to be the map given by $\rho \cdot\iota\circ\bar u$ on $B_{R-1}^C$ and by 0 on $B_{R-1}$. It follows that 
\begin{equation}\label{eq:df}\Vert df\Vert_{L^p_\lam(B_R^C)}\leq\big\Vert d\iota(\bar u)d\bar u\big\Vert_{L^p_\lam(B_R^C)}+\Vert\iota\circ\bar u\,d\rho\Vert_{L^p_\lam(B_R\wo B_{R-1})}.\end{equation}
A short calculation shows that $|d\bar u|\leq|d_Au|$, and therefore,
\begin{equation}\label{eq:d iota bar u}\big\Vert d\iota(\bar u)d\bar u\big\Vert_{L^p_\lam(B_R^C)}\leq\Vert d\iota(\bar u)\Vert_{L^\infty(B_R^C)}\Vert d_Au\Vert_{L^p_\lam(B_R^C)}.
\end{equation} 
Our assumption $w=(P,A,u)\in\BB^p_\lam$ implies that $\Vert d_Au\Vert_{L^p_\lam(B_R^C)}<\infty$. Furthermore, $\mu$ is proper by the hypothesis (H), hence the set $\mu^{-1}(\bar B_\de)$ is compact. Thus the same holds for the set $\pi^{M^*}(\mu^{-1}(\bar B_\de))$. This set contains the image of $\bar u$. It follows that $\Vert d\iota(\bar u)\Vert_{L^\infty(B_R^C)}<\infty$. Combining this with (\ref{eq:df},\ref{eq:d iota bar u}), we obtain the estimate
\[\Vert df\Vert_{L^p_\lam(\C)}\leq\Vert df\Vert_{L^p_\lam(B_R)}+\Vert df\Vert_{L^p_\lam(B_R^C)}<\infty.\]
Hence the hypotheses of Proposition \ref{prop:Hardy} (with $\lam$ replaced by $\lam-1$) are satisfied. It follows that the point $f(re^{i\phi})$ converges to some point $y_\infty\in\R^n$, uniformly in $\phi\in\R$, as $r\to\infty$. Claim \ref{claim:bar u} follows.
\end{proof}
Let $\bar x_\infty$ be as in Claim \ref{claim:bar u}. We choose a local slice around $\bar x_\infty$, i.e., a pair $(\bar U,\wt\si)$, where $\bar U\sub M^*/G$ is an open neighborhood of $\bar x_\infty$, and $\wt\si:\bar U\to M^*$ is a smooth map satisfying $\pi^{M^*}\circ \wt\si=\id_{\bar U}$. Then there exists a unique section $\si'$ of $P|_{B_R^C}$, of class $W^{1,p}_\loc$, such that $\wt\si\circ\bar u=u\circ\si'$. By the continuous homotopy lifting property of $P$ we may extend this to a continuous section $\si''$ of $P|_{B_1^C}$. Regularizing $\si''$ on $B_{R+1}\wo B_1$, we obtain a section $\si$ of $P|_{B_1^C}$, of class $W^{1,p}_\loc$.%
\footnote{For this we regularize $\si''$ suitably in local trivializations.}
 We define $x_\infty:=\wt\si(\bar x_\infty)$. It follows from Claim \ref{claim:bar u} that $u\circ\si(re^{i\phi})$ converges to $x_\infty$, uniformly in $\phi$, for $r\to\infty$. Furthermore, 
\begin{eqnarray*}&\Vert\si^*L_uA\Vert_{L^p_\lam(B_{R+1}^C)}\leq \Vert\si^*du\Vert_{L^p_\lam(B_{R+1}^C)}+\Vert\si^*d_Au\Vert_{L^p_\lam(B_{R+1}^C)},&\\
&\si^*du=du\,d\si=d\wt\si\,d\bar u,\,\textrm{on }B_{R+1}^C,\quad |d\bar u|\leq |d_Au|.\end{eqnarray*}
Since
\[\inf\Big\{|L_u(p)\xi|\,\big|\,p\in P|_{B_{R+1}^C},\,\xi\in\g:\,|\xi|=1\Big\}>0,\quad\Vert d_Au\Vert_{p,\lam}<\infty,\]
it follows that $\si^*A\in L^p_\lam(B_1^C)$. This proves Lemma \ref{le:si}.
\end{proof}
\begin{proof}[Proof of Lemma \ref{le:P}]\setcounter{claim}{0}\label{proof:le:P} We prove {\bf statement (\ref{le:P:exists})}: We denote by $\pi:P\to\C$ the bundle projection, and identify $S^2\iso\C\cup\{\infty\}$. We choose $\si$ and $x_\infty$ as in Lemma \ref{le:si}, and we define $\wt P$ to be the quotient of $P\disj\big((S^2\wo\{0\})\x G\big)$ under the equivalence relation generated by $p\sim (\pi(p),g)$, where $g\in G$ is determined by $(\si\circ\pi(p))g=p$, for $p\in P$. We define $\iota$ to be the canonical map from $P$ to $\wt P$, and $\wt u:\wt P\to M$ to be the unique map satisfying 
\[\wt u\circ\iota=u,\quad\wt u([(\infty,g)]):=g^{-1}x_\infty,\,\forall g\in G.\]
The statement of Lemma \ref{le:si} implies that this map is continuous. This proves (\ref{le:P:exists}).

We prove {\bf statement (\ref{le:P:iso})}. Uniqueness of $\wt\Phi$ follows from the condition $\wt\Phi\circ\iota'=\iota\circ\Phi$ and continuity of $\wt\Phi$. We prove existence: We define the map $\phi:\wt P'_\infty\to\wt P_\infty$ as follows. The map $\wt u$ descends to a continuous map $\wt f:S^2\to M/G$. Recall that $M^*$ denotes the set of points in $M$ where $G$ acts freely. We denote $\wt U:=\wt f^{-1}(M^*/G)$. Since $M^*$ is open, the set $\wt U$ is, as well. Since $G$ is compact, the canonical map $M^*\to M^*/G$ defines a smooth $G$-bundle. We denote by $\wt\pi:\wt P\to S^2$ the projection map. We define
\begin{equation}\label{eq:Psi wt P}\wt\Psi:\wt\pi^{-1}(\wt U)\to\wt f|_{\wt U}^*M^*,\quad\wt\Psi(\wt p):=\big(\wt\pi(\wt p),\wt u(\wt p)\big).
\end{equation}
Furthermore, we define $\wt f',\wt U',\wt\pi',\wt\Psi'$ in an analogous way, using $\wt P'$ and $\wt u'$. Since $\Phi$ descends to the identity on $\C$, and $u'=u\circ\Phi$, the maps $u$ and $u'$ descend to the same map $\C\to M/G$. Since $\wt u\circ\iota=u$, $\wt u'\circ\iota'=u'$, and $\wt u$ and $\wt u'$ are continuous, it follows that $\wt f=\wt f'$. We claim that there exists a unique map $\wt\Phi:\wt P'\to\wt P$, satisfying
\begin{equation}\label{eq:wt Phi wt Psi}\wt\Phi=\wt\Psi^{-1}\circ\wt\Psi'\,\textrm{on }(\wt\pi')^{-1}(\wt U),\quad\wt\Phi\circ\iota'=\iota\circ\Phi\,\textrm{on }P'.
\end{equation}
To see uniqueness of this map, note that the hypothesis $\wt u'(\wt P'_\infty)\sub M^*$ implies that $\wt P'_\infty\sub(\wt\pi')^{-1}(\wt U)$. Hence conditions (\ref{eq:wt Phi wt Psi}) determine $\wt\Phi$ on the whole of $\wt P'$. Existence of this map follows from the equality $\wt\Psi'\circ\iota'=\wt\Psi\circ\iota\circ\Phi$ on ${\pi'}^{-1}(\wt U\cap\C)\sub P'$, which follows from the assumptions $\wt u'\circ\iota'=u'$, $\wt u\circ\iota=u$, and $u\circ\Phi=u'$. The map $\wt\Phi$ has the required properties. This proves (\ref{le:P:iso}) and completes the proof of Lemma \ref{le:P}.
\end{proof}
We now prove Lemma \ref{le:section} (p.~\pageref{le:section}), which was used in Section \ref{SEC:HOMOLOGY CHERN}, in order to define the $(\om,\mu)$-homotopy class of an equivalence class $W$ of triples $(P,A,u)$. (See Definition \ref{defi:homotopy W}.)
\begin{proof}[Proof of Lemma \ref{le:section}]\setcounter{claim}{0}\label{proof:le:section} To prove the first statement, we choose a section $\si_0$ of the restriction of $P$ to the disk $\DDD$, of class $W^{2,p}_\loc$. By Lemma \ref{le:si} there exists a section $\wt\si$ of the restriction of the bundle $P$ to $B_1^C:=\C\wo B_1$, of class $W^{1,p}_\loc$, and a point $x_\infty\in\mu^{-1}(0)$, such that $u\circ\wt\si(re^{i\phi})$ converges to $x_\infty$, uniformly in $\phi\in\R$, as $r\to\infty$. We define $g_\infty:S^1\to G$ to be the unique map satisfying $\wt\si=\si_0g_\infty$, on $S^1$, and $\si$ to be the continuous section of $P$ that agrees with $\si_0$ on $B_1$, and satisfies
\[\si(z)=\wt\si(z)g_\infty(z/|z|)^{-1},\quad\forall z\in B_1^C.\]
By regularizing $\si$, we may assume that it is of class $W^{1,p}_\loc$. This section satisfies the requirements of the first part of the lemma. 

To prove the second statement, let $\si$ be a section of $P$ of class $W^{1,p}_\loc$, such that the map $u\circ\si:\C\to M$ continuously extends to a map $\wt u:\Si\to M$. It follows from Claim \ref{claim:bar u} in the proof of Lemma \ref{le:si} that there exists a point $\bar x_\infty\in\BAR M=\mu^{-1}(0)/G$, such that $G\wt u(z)=\bar x_\infty$, for every $z\in\dd\Si$. The second statement follows from this. This proves Lemma \ref{le:section}.
\end{proof}
The next lemma was used in Definition \ref{defi:homotopy W}.
\begin{lemma}\label{le:homotopic} Let $p>2$, $\lam>1-2/p$, $(P,A,u)\in\BB^p_\lam$ be a triple, and $\si,\si'$ sections of $P$ as in Lemma \ref{le:section}. Then the continuous extensions $\wt u,\wt{u'}:\Si\to M$ of $u\circ\si,u\circ\si'$ are weakly $(\om,\mu)$-homotopic. 
\end{lemma}
\begin{proof}[Proof of Lemma \ref{le:homotopic}]\setcounter{claim}{0} Let $R\in(0,\infty)$. We denote by $B_R$ and $\BAR{B_R}$ the open and closed balls in $\C$, of radius $R$, centered around 0. 
\begin{Claim}There exists a continuous map $h:[0,1]\x\Si\to M$ such that 
\begin{equation}\label{eq:h 0}h(0,\cdot)=\wt u,\quad h(1,z)=\wt{u'}(z),\,\forall z\in\BAR{B_R},\quad h(t,z)=h(0,z),\,\forall z\in\dd\Si,t\in[0,1].
\end{equation}
\end{Claim}
\begin{proof}[Proof of the claim] We define
\[g_0:\C\to G,\quad g_0(z)\si'(z):=\si(z).\]
There exists a continuous map $g:[0,1]\x\C\to G$ such that 
\begin{equation}\label{eq:g 0}g(0,\cdot)=g_0,\quad g(1,z)=\one,\,\forall z\in\BAR{B_R},\quad g(t,z)=g_0(z),\,\forall z\in\C\wo B_{R+1}.
\end{equation}
To see this, observe that we may assume without loss of generality that $g_0(0)=\one$. (Here we use the assumption that $G$ is connected.) We choose a continuous map $f:[0,1]\x\C\to\C$ such that 
\[f(0,\cdot)=\id,\quad f(1,z)=0,\,\forall z\in\BAR{B_R},\quad f(t,z)=z,\,\forall z\in\C\wo B_{R+1}.\]
We define $g:=g_0\circ f$. This map satisfies (\ref{eq:g 0}). We now define 
\[h(t,z):=\left\{\begin{array}{ll}
u\big(g(t,z)\si(z)\big),&\textrm{if }z\in\C,\\
\wt u(z),&\textrm{if }z\in\dd\Si.
\end{array}\right.\]
This map satisfies the conditions (\ref{eq:h 0}). This proves the claim.
\end{proof}
It follows from hypothesis (H) that there exists a number $R>0$ such that $\wt u(\Si\wo B_R)\sub M^*$. By the claim, we may assume without loss of generality that $\wt u=\wt{u'}$ on $\BAR B_R$. Since $G$ is compact, the canonical projection $\pi:M^*\to M^*/G$ naturally defines a smooth $G$-bundle. It follows that the map 
\[[0,1]\x(\Si\wo B_R)\ni(t,z)\mapsto\pi\circ\wt u(z)\in M^*/G\]
has a continuous lift $h:[0,1]\x(\Si\wo B_R)\to M^*$ that agrees with the map $(0,z)\mapsto\wt u(z)$ on $\{0\}\x(\Si\wo B_R)$, with the map $(1,z)\mapsto\wt{u'}(z)$ on $\{1\}\x(\Si\wo B_R)$, and with the map $(t,z)\mapsto\wt u(z)$ on $[0,1]\x S^1_R$, where $S^1_R\sub\C$ denotes the circle of radius $R$, around 0. The map 
\[[0,1]\x\Si\ni(t,z)\mapsto\left\{\begin{array}{ll}
\wt u(z),&\textrm{if }z\in\BAR{B_R},\\
h(t,z),&\textrm{otherwise,}
\end{array}\right.\]
is a weak $(\om,\mu)$-homotopy from $\wt u$ to $\wt{u'}$. This proves Lemma \ref{le:homotopic}.
\end{proof}
We now prove Proposition \ref{prop:Chern Maslov}. We need the following. Let $w$ be a representative of $W$. Let $(\wt P,\iota,\wt u)$ be an extension as in Lemma \ref{le:P} (p.~\pageref{le:P}). Then $\om$ induces a fiberwise symplectic form $\wt\om$ on the topological vector bundle $TM^{\wt u}=(\wt u^*TM)/G$ over $S^2$. We denote by $c_1(TM^{\wt u},\wt\om)$ its first Chern class.
\begin{lemma}[Chern number]\label{le:Chern bundle} We have
\begin{equation}\label{eq:c 1 G W}\big\lan c_1^G(M,\om),[W]\big\ran=\big\lan c_1(TM^{\wt u},\wt\om),[S^2]\big\ran.
\end{equation} 
\end{lemma}
For the proof of this lemma we need the following remark.
\begin{rmk}\label{rmk:u E G} Let $G$ be a topological group, $X$ and $X'$ topological spaces, $P\to X$ and $P'\to X'$ topological $G$-bundles, $M$ a topological $G$-space, $E$ a $G$-equivariant symplectic vector bundle over $M$, and $u:P\to M$, $\theta:P\to P'$ continuous $G$-equivariant maps. We define $f:X\to(M\x P')/G$ to be the map induced by $u$ and $\theta$. Then the symplectic vector bundles
\begin{equation}\label{eq:u E G}(u^*E)/G,\quad f^*\big((E\x P')/G\big)
\end{equation}
are isomorphic. Here we denote by $(u^*E)/G$ the symplectic vector bundle over $X$ obtained from the pullback bundle $u^*E\to P$ as the quotient by the $G$-action, by $E\x P'$ the natural symplectic vector bundle over $M\x P'$, and by $(E\x P')/G$ the induced symplectic vector bundle over $(M\x P')/G$. An isomorphism between the bundles in (\ref{eq:u E G}) is given by the map
\[G(p,v)\mapsto\big(\pi(p),G(v,\theta(p))\big).\Box\]
\end{rmk}
\begin{proof}[Proof of Lemma \ref{le:Chern bundle}]\setcounter{claim}{0} We choose a continuous $G$-equivariant map $\theta:\wt P\to\EG$. We denote by
\[f:S^2\iso\C\disj\{\infty\}\to(M\x\EG)/G\]
the map induced by $(\wt u,\theta):\wt P\to M\x\EG$. By definition, $c_1^G(M,\om)$ is the first Chern class of the vector bundle
\[(TM\x\EG)/G\to(M\x\EG)/G,\]
equipped with the fiberwise symplectic form induced by $\om$. Furthermore, we have $[W]=f_*[S^2]$. Therefore, equality (\ref{eq:c 1 G W}) follows from naturality of the first Chern class under pullback by the map $f$, and Remark \ref{rmk:u E G}, with $X:=S^2$, $(P,u)$ replaced by $(\wt P,\wt u)$, and $(E,P'):=(TM,\EG)$. 
This proves Lemma \ref{le:Chern bundle}.
\end{proof}
For the proof of Proposition \ref{prop:Chern Maslov} we also need the following.
\begin{rmk}\label{rmk:c 1 m} Let $\Si$ be an oriented topological surface homeomorphic to the closed disk. We denote by $\Si'$ the surface obtained from $\Si$ by collapsing its boundary to a point%
\footnote{The surface $\Si'$ is homeomorphic to $S^2$.}%
, and by $f:\Si\to\Si'$ the canonical ``collapsing'' map. Let $(E,\om)$ be a symplectic vector bundle over $\Si'$, $(V,\Om)$ a symplectic vector space of dimension the rank of $E$, and $\Psi:\Si\x V\to f^*E$ a symplectic trivialization. Then we have 
\begin{equation}\label{eq:c 1 m}\big\lan c_1(E,\om),[\Si']\big\ran=m_{\dd\Si,\Om}\Big(\dd\Si\x\dd\Si\ni(z,z')\mapsto\Psi_{z'}^{-1}\Psi_z\in\Aut\Om\Big).
\end{equation}
Here $[\Si']$ denotes the fundamental class of $\Si'$ and $\Aut\Om$ the group of linear symplectic automorphisms of $V$, and $m_{\dd\Si,\Om}$ is defined as in (\ref{eq:m X om}). Furthermore, we denote by $\pi:\Si\to\Si'$ the canonical map, and we use the canonical identification $(f^*E)_z=E_{\pi(z)}$, for $z\in\dd\Si$. Equality (\ref{eq:c 1 m}) follows from an elementary argument (e.g.~an argument as in the proof of \cite[Theorem 2.69]{MS98}.) $\Box$
\end{rmk}
\begin{proof}[Proof of Proposition \ref{prop:Chern Maslov} (p.~\pageref{prop:Chern Maslov})]\label{proof:prop:Chern Maslov}\setcounter{claim}{0} We choose an extension $(\wt P,\iota,\wt u)$ as in Lemma \ref{le:P}(\ref{le:P:exists}), such that $\wt u(\wt P_\infty)\sub\mu^{-1}(0)$. (It follows from the proof of Lemma \ref{le:P}(\ref{le:P:exists}) that $\wt u$ may be chosen to satisfy this condition.) By Lemma \ref{le:Chern bundle}, equality (\ref{eq:c 1 G W}) holds. We denote by $\wt\pi:\wt P\to S^2$ the canonical projection, by $\Si$ the compact surface obtained from $\C$ by ``gluing a circle at $\infty$'', and by
\[f:\Si\to\C\disj\{\infty\}\iso S^2\]
the map that is the identity on the interior $\C=\Int\Si$, and maps the boundary $\dd\Si\iso S^1$ to $\infty$. We choose a continuous map $\wt\si:\Si\to\wt P$, such that $\wt\pi\circ\wt\si=f$. We define $v:=\wt u\circ\wt\si:\Si\to M$. This map continuously extends the map $u\circ\si$, where $\si:=\wt\si|_{\C}$. Hence, by definition, the $(\om,\mu)$-homotopy class of $W$ equals the $(\om,\mu)$-homotopy class of $v$. We choose a symplectic trivialization $\Psi:\Si\x V\to v^*TM$. For $z,z'\in\dd\Si$ we define $g_{z',z}\in G$ to be the unique element satisfying 
\begin{equation}\label{eq:wt si g}\wt\si(z')g_{z',z}=\wt\si(z).
\end{equation} 
It follows that $v(z')=g_{z',z}v(z)$, and hence, using the definition of the Maslov index,
\begin{eqnarray}\nn&m_{\om,\mu}\big((\om,\mu)\textrm{-homotopy class of }W\big)&\\
\label{eq:m m}&=m_{\dd\Si,\Om}\Big(\dd\Si\x\dd\Si\ni(z',z)\mapsto\Psi_{z'}^{-1}g_{z',z}\cdot\Psi_z\in\Aut\Om\Big).&
\end{eqnarray}
The statement of Proposition \ref{prop:Chern Maslov} is now a consequence of the following claim.
\begin{Claim}The number (\ref{eq:m m}) agrees with $\big\lan c_1\big(TM^{\wt u},\wt\om\big),[S^2]\big\ran$.
\end{Claim}
\begin{pf}[Proof of the claim] We define the map
\[\wt\Psi:\Si\x V\to f^*TM^{\wt u},\quad\wt\Psi_zw:=\wt\Psi(z,w):=\big(z,G\big(\wt\si(z),\Psi_zw\big)\big).\]
This is a continuous symplectic trivialization. We denote by $\Si'$ the surface obtained from $\Si$ by collapsing its boundary $\dd\Si$ to a point. There is a canonical homeomorphism $\C\disj\{\infty\}\to\Si'$, and the composition of $f$ with this map agrees with the collapsing map. Therefore, applying Remark \ref{rmk:c 1 m} with $E:=TM^{\wt u}$ and $\om,\Psi$ replaced by $\wt\om,\wt\Psi$, we have
\begin{equation}\label{eq:c 1 m wt Psi}\big\lan c_1\big(TM^{\wt u},\wt\om\big),[S^2]\big\ran=m_{\dd\Si,\Om}\big(\dd\Si\x\dd\Si\ni(z,z')\mapsto\wt\Psi_{z'}^{-1}\wt\Psi_z\in\Aut\Om\big).
\end{equation}
Equality (\ref{eq:wt si g}) implies that
\[\wt\Psi_{z'}^{-1}\wt\Psi_z=\Psi_{z'}^{-1}g_{z',z}\cdot\Psi_z,\quad\forall z,z'\in\dd\Si.\]
Combining this with (\ref{eq:c 1 m wt Psi}), the claim follows.
\end{pf}
This proves Proposition \ref{prop:Chern Maslov}.
\end{proof}
\section{Weighted Sobolev spaces and a Hardy-type inequality}\label{sec:weighted}
Let $d\in \Z$. The following lemma was used in Section \ref{subsec:proof:Fredholm aug} in order to define norms on the vector spaces $\XXX'_{p,\lam,d}$ and $\XXX''_{p,\lam}$ defined in (\ref{eq:XXX' p lam d XXX'' p lam}). If $d<0$ then let $\rho_0\in C^\infty(\C,[0,1])$ be such that $\rho_0(z)=0$ for $|z|\leq1/2$ and $\rho_0(z)=1$ for $|z|\geq1$. In the case $d\geq0$ we set $\rho_0:=1$. Recall the definitions
\[p_d:\C\to \C,\,p_d(z):=z^d,\quad\lan x\ran:=\sqrt{1+|v|^2},\,\forall v\in\R^n.\]
\begin{lemma}\label{le:X d iso} For every $1<p<\infty$ and $\lam>-2/p$ the map
\[\C\oplus L^{1,p}_{\lam-d}(\C,\C)\to \C\cdot\rho_0p_d+L^{1,p}_{\lam-d}(\C,\C),\quad (v_\infty,v)\mapsto v_\infty\rho_0p_d+v\]
is an isomorphism of vector spaces.  
\end{lemma}
\begin{proof}[Proof of Lemma \ref{le:X d iso}]\setcounter{claim}{0} This follows from a straight-forward argument. 
\end{proof}
The following proposition was used in the proofs of Theorem \ref{thm:Fredholm aug} (Section \ref{subsec:proof:Fredholm aug}) and Proposition \ref{prop:X X w} (Section \ref{subsec:proofs}). For every normed vector space $V$ we denote by $C_b(\R^n,V)$ the space of bounded continuous maps from $\R^n$ to $V$. We denote by $B_r$ the ball of radius $r$ in $\R^n$, and by $X^C$ the complement of a subset $X\sub\R^n$. Recall the definitions (\ref{eq:L k p},\ref{eq:W k p}) of the weighted Sobolev spaces $L^{k,p}_\lam(\Om,W)$ and $W^{k,p}_\lam(\Om,W)$.
\begin{prop}[Weighted Sobolev spaces]\label{prop:lam d Morrey} Let $n\in\N$. Then the following statements hold.
\begin{enui} \item\label{prop:lam Morrey} Let $n<p<\infty$. Then for every $\lam\in\R$ there exists $C>0$ such that 
\begin{equation}
    \label{eq:lam Morrey}\Vert u\lan\cdot\ran^{\lam+\frac np}\Vert_{L^\infty(\R^n)}\leq C\Vert u\Vert_{L^{1,p}_\lam(\R^n)},\quad \forall u\in W^{1,1}_\loc(\R^n).
\end{equation}
If $\lam>-n/p$ then $L^{1,p}_\lam(\R^n)$ is compactly contained in $C_b(\R^n)$. 
\item\label{prop:lam W} For every $k\in\N_0$, $1<p<\infty$ and $\lam\in\R$ the map
\[W^{k,p}_\lam(\R^n)\ni u\mapsto \lan\cdot\ran^\lam u\in W^{k,p}(\R^n)\]
is a well-defined isomorphism (of normed spaces).
\item\label{prop:lam W cpt} Let $p>1$, $\lam\in\R$, and $f\in L^\infty(\R^n)$ be such that $\Vert f\Vert_{L^\infty(\R^n\wo B_i)}\to 0$, for $i\to \infty$. Then the operator
\[W^{1,p}_\lam(\R^n)\ni u\mapsto fu\in L^p_\lam(\R^n)\]
is compact.
\item \label{prop:lam d iso} For every $1<p<\infty$, $\lam\in\R$, $d\in\Z$, and $u\in L^{1,p}_\lam(B_1^C)$ the following inequality holds:
\[\Vert p_du\Vert_{L^{1,p}_{\lam-d}(B_1^C)}\leq \max\big\{-d2^{(-d+3)/2},2\big\}\Vert u\Vert_{L^{1,p}_\lam(B_1^C)}.\]
\end{enui}
\end{prop}

\begin{proof} [Proof of Proposition \ref{prop:lam d Morrey}]\setcounter{claim}{0} {\bf Proof of statement (\ref{prop:lam Morrey}):} Inequality (\ref{eq:lam Morrey}) follows from inequality (1.11) in Theorem 1.2 in the paper \cite{Ba} by R.~Bartnik. Assume now that $\lam>-n/p$. Then it follows from Morrey's embedding theorem that there exists a canonical bounded inclusion $L^{1,p}_\lam(\R^n)\inj C_b(\R^n)$. In order to show that this inclusion is compact, let $u_\nu\in L^{1,p}_\lam(\R^n)$ be a sequence such that
\begin{equation}\label{eq:C sup nu}C:=\sup_\nu\Vert u_\nu\Vert_{L^{1,p}_\lam(\R^n)}<\infty.
\end{equation}
By Morrey's embedding theorem and the Arzel\`a-Ascoli theorem on $\bar B_j$ (for $j\in\N$) and a diagonal subsequence argument, there exists a subsequence $u_{\nu_j}$ of $u_\nu$ that converges to some map $u\in W^{1,p}_\loc(\R^n)$, weakly in $W^{1,p}(B_j)$, and strongly in $C(\bar B_j)$, for every $j\in\N$. 
\begin{Claim}We have $u\in C_b(\R^n)$ and $u_{\nu_j}$ converges to $u$ in $C_b(\R^n)$. 
\end{Claim}
\begin{pf}[Proof of the claim] We choose a constant $C'$ as in the first part of (\ref{prop:lam Morrey}). For every $R>0$ we have
\[\Vert u\Vert_{L^{1,p}_\lam(B_R)}\leq\limsup_j\Vert u_{\nu_j}\Vert_{L^{1,p}_\lam(B_R)}\leq C.\]
Hence $u\in L^{1,p}_\lam(\R^n)$.  Since $\lam>-n/p$, by inequality (\ref{eq:lam Morrey}), this implies $u\in C_b(\R^n)$. To see the second statement, we choose a smooth function $\rho:\R^n\to[0,1]$ such that $\rho(x)=0$ for $x\in B_1$, $\rho(x)=1$ for $x\in B_3^C$, and $|D\rho|\leq1$. Let $R\geq1$ and $j\in\N$. We define $\rho_R:=\rho(\cdot/R):\R^n\to [0,1]$. Abbreviating $v_j:=u_{\nu_j}-u$, we have 
\begin{equation}\label{eq:Vert v j}\Vert v_j\Vert_\infty\leq\big\Vert v_j(1-\rho_R)\big\Vert_\infty+\big\Vert v_j\rho_R\big\Vert_\infty.\end{equation} 
Inequality (\ref{eq:lam Morrey}) and the fact $\rho_R=0$ on $B_R$ imply that 
\begin{equation}\label{eq:Vert v j rho R}\big\Vert v_j\rho_R\big\Vert_\infty\leq C'R^{-\lam-\frac np}\Vert v_j\rho_R\Vert_{1,p,\lam}. 
\end{equation} 
Furthermore, using (\ref{eq:C sup nu}), we have
\[\Vert v_j\rho_R\Vert_{1,p,\lam}\leq 2\Vert v_j\Vert_{1,p,\lam}\leq4C.\]
Combining this with (\ref{eq:Vert v j}) and (\ref{eq:Vert v j rho R}), and the fact $\lim_{j\to\infty}\Vert v_j\Vert_{L^\infty(B_{3R})}=0$, it follows that
\[\limsup_{j\to\infty}\Vert v_j\Vert_\infty\leq 4CC'R^{-\lam-\frac np}.\]
Since $\lam>-n/p$ and $R\geq1$ is arbitrary, it follows that $u_{\nu_j}$ converges to $u$ in $C_b(\R^n)$. This proves the claim and completes the proof of statement {\bf (\ref{prop:lam Morrey})}.
\end{pf}

{\bf Statement (\ref{prop:lam W})} follows from a straight-forward calculation.

{\bf Proof of statement (\ref{prop:lam W cpt}):} Let $f\in L^\infty(\R^n)$ be as in the hypothesis. Let $u_\nu\in W^{1,p}_\lam(\R^n)$ be a sequence such that
\[C:=\sup_\nu\Vert u_\nu\Vert_{W^{1,p}_\lam(\R^n)}<\infty.\]
By the Rellich-Kondrashov compactness theorem on $\bar B_j$ (for $j\in\N$) and a diagonal subsequence argument there exists a subsequence $(\nu_j)$ and a map $v\in L^p_\loc(\R^n)$, such that $fu_{\nu_j}$ converges to $v$, strongly in $L^p(K)$, as $j\to\infty$, for every compact subset $K\sub \R^n$. Elementary arguments show that $v\in L^p_\lam(\R^n)$ and $fu_{\nu_j}$ converges to $v$ in $L^p_\lam(\R^n)$. (For the latter we use the hypothesis that $\Vert f\Vert_{L^\infty(\R^n\wo B_i)}\to 0$, as $i\to \infty$.) This proves {\bf (\ref{prop:lam W cpt})}. 

{\bf Statement (\ref{prop:lam d iso})} follows from a straight-forward calculation. This completes the proof of Proposition \ref{prop:lam d Morrey}.\end{proof}
The next result was used in the proofs of Proposition \ref{prop:X X w} (Section \ref{subsec:proofs}) and Lemma \ref{le:si} (Appendix \ref{sec:proofs homology}).
\begin{prop}[Hardy-type inequality]\label{prop:Hardy} Let $n\in\N$, $p>n$, $\lam>-n/p$ and $u\in W^{1,1}_\loc(\R^n,\R)$ be such that $\Vert Du|\cdot|^{\lam+1}\Vert_{L^p(\R^n)}<\infty$. Then $u(rx)$ converges to some $y_\infty\in \R$, uniformly in $x\in S^{n-1}$, as $r\to\infty$, and 
  \begin{equation}
    \label{eq:Hardy y}\big\Vert(u-y_\infty)|\cdot|^\lam\big\Vert_{L^p(\R^n)}\leq\frac p{\lam+\frac np}\left\Vert Du|\cdot|^{\lam+1}\right\Vert_{L^p(\R^n)}.
  \end{equation}
\end{prop}
For the proof of this proposition we need the following. We denote by $B_r$ the ball of radius $r$ in $\R^n$, and by $X^C$ the complement of a subset $X\sub\R^n$.
\begin{lemma}[Hardy's inequality]\label{le:Hardy} Let $n\in\N$, $1<p<\infty$, $\lam>-n/p$ and $u\in W^{1,1}_\loc(\R^n,\R)$. If there exists $R>0$ such that $u|_{B_R^C}=0$ then
\[\Vert u|\cdot|^\lam\Vert_{L^p(\R^n)}\leq\frac p{\lam+\frac np}\left\Vert Du|\cdot|^{\lam+1}\right\Vert_{L^p(\R^n)}\,(\in[0,\infty]).\]
\end{lemma}
\begin{proof}[Proof of Lemma \ref{le:Hardy}]\setcounter{claim}{0} If $u$ is smooth then the stated inequality follows from \cite[Chapter 6, Exercise 21]{Kav}. The general case can be reduced to this case by mollifying the function $u$. This proves the lemma. \end{proof}
\begin{proof}[Proof of Proposition \ref{prop:Hardy}] \setcounter{claim}{0} Let $n,p,\lam$ be as in the hypothesis. We define $\eps:=\lam+\frac np$. 
\begin{Claim}There exists a constant $C_1$ such that for every weakly differentiable function $u:\R^n\to\R$ and $x,y\in\R^n$ satisfying $0<|x|\leq|y|$, we have
\[|u(x)-u(y)|\leq C_1|x|^{-\eps}\big\Vert D u|\cdot|^{\lam+1}\big\Vert_{L^p(B_{|x|}^C)}.\] 
\end{Claim}
\begin{proof}[Proof of the claim] By Morrey's theorem there is a constant $C$ such that
\[|u(0)-u(x)|\leq Cr^{1-\frac np}\Vert Du\Vert_{L^p(B_r)},\]
for every $r>0$, weakly differentiable function $u:B_r\to\R$, and $x\in B_r$. Let $u,x$ and $y$ be as in the hypothesis of the claim. Let $N\in \N$ be such that $2^{N-1}|x|\leq |y|\leq 2^N|x|$. For $i=0,\ldots,N$ we define $x_i:=2^ix\in \R^n$. Furthermore, we set
\[x_{N+7}:=2^N|x|\frac y{|y|},\quad x_{N+8}:=y,\]
and we choose points
\[x_i\in S^{n-1}_{2^N|x|}:=\left\{y\in\R^n\,\big|\,|y|=2^N|x|\right\},\quad i=N+1,\ldots,N+6,\]
such that
\[|x_i-x_{i-1}|\leq 2^{N-1}|x|,\quad\forall i=N+1,\ldots,N+7.\]
For $i=0,\ldots,N-1$ we have $x_i\in \bar B_{2^i|x|}(x_{i+1})$. Hence it follows from the statement of Morrey's theorem that
\[|u(x_{i+1})-u(x_i)|\leq C(2^i|x|)^{-\eps}\big\Vert Du|\cdot|^{\lam+1}\big\Vert_{L^p(B_{|x|}^C)}.\]
Moreover, for $i=N,\ldots,N+7$ we have $x_{i+1}\in \bar B_{2^{N-1}|x|}(x_i)$, and hence analogously,
\[|u(x_{i+1})-u(x_i)|\leq C(2^{N-1}|x|)^{-\eps}\Vert Du|\cdot|^{\lam+1}\Vert_{L^p(B_{|x|}^C)}.\]
Using the inequality
\[|u(y)-u(x)|\leq\sum_{i=0,\ldots,N+7}|u(x_{i+1})-u(x_i)|,\]
the claim follows. 
\end{proof}
Let $u\in W^{1,1}_\loc(\R^n,\R)$ be such that $\Vert Du|\cdot|^{\lam+1}\Vert_{L^p(\R^n)}<\infty$. It follows from the claim that there exists $y_\infty\in\R$ such that $u(rx)$ converges to $y_\infty$, as $r\to\infty$, uniformly in $x\in S^{n-1}$. To prove inequality (\ref{eq:Hardy y}), we choose a smooth map $\rho:[0,\infty)\to [0,1]$ such that $\rho(t)=1$ for $0\leq t\leq 1$, $\rho(t)=0$ for $t\geq2$ and $|\rho'(t)|\leq 2$. We fix a number $R>0$ and define
\[\rho_R:\R\to [0,1],\quad\rho_R(x):=\rho(|x|/R).\]
We abbreviate $v:=u-y_\infty$. Using Lemma \ref{le:Hardy} with $u$ replaced by $\rho_Rv$, we have
\[\Vert v|\cdot|^\lam\Vert_{L^p(B_R)}\leq \big\Vert\rho_R v|\cdot|^\lam\big\Vert_{L^p(\R^n)}\leq\frac p{\lam+\frac np}\big\Vert D(\rho_Rv)|\cdot|^{\lam+1}\big\Vert_{L^p(\R^n)}.\] 
Combining this with a calculation using Leibnitz' rule, it follows that 
\begin{equation}\label{eq:Vert v lam}\Vert v|\cdot|^\lam\Vert_{L^p(B_R)}\leq \frac p{\lam+\frac np}\left(4\Vert v|\cdot|^\lam\Vert_{L^p(B_{2R}\wo B_R)}+\big\Vert Du|\cdot|^{\lam+1}\big\Vert_{L^p(\R^n)}\right).\end{equation} 
The above claim implies that
\[|v(x)|\leq C_1|x|^{-\eps}\Vert Du|\cdot|^{\lam+1}\Vert_{L^p(B_R^C)},\quad\forall x\in B_R^C.\]
Using the equalities
\[\int_{B_{2R}\wo B_R}|x|^{-n}dx=\log 2|S^{n-1}|\]
and $\eps=\lam+n/p$, it follows that 
\[\Vert v|\cdot|^\lam\Vert_{L^p(B_{2R}\wo B_R)}^p\leq C_1^p\log 2|S^{n-1}|\big\Vert Du|\cdot|^{\lam+1}\big\Vert^p_{L^p(B_R^C)}.\]
Inequality (\ref{eq:Hardy y}) follows by inserting this into the right hand side of (\ref{eq:Vert v lam}) and sending $R$ to $\infty$. This proves Proposition \ref{prop:Hardy}. \end{proof}
The next result will be used to prove Corollary \ref{cor:d L L} below, which was used in the proof of Theorem \ref{thm:Fredholm aug} (Section \ref{subsec:proof:Fredholm aug}). For every $d\in\Z$ we define $P_d$ and $\bar P_d$ to be the spaces of polynomials in $z\in\C$ and $\bar z$ of degree \emph{less than} $d$.%
\footnote{Hence if $d\leq 0$ we have $P_d=\{0\}$.}
 We abbreviate
\[L^{1,p}_\lam:=L^{1,p}_\lam(\C,\C),\quad L^p_\lam:=L^p_\lam(\C,\C),\quad\dd_{\bar z}:=\dd^\C_{\bar z},\quad\dd_z:=\dd^\C_z.\]
Let $X$ be a normed vector space and $Y\sub X$ a closed subspace. We denote by $X^*$ the dual space of $X$ and equip $X/Y$ with the quotient norm.
\begin{prop}[Fredholm property for $\dd_{\bar z}$]\label{prop:d L L} For every $d\in\Z$, $1<p<\infty$ and $-2/p+1<\lam<-2/p+2$ the following conditions hold.
\begin{enui}\item\label{prop:d L L:Fredholm} The operator $T:=\dd_{\bar z}:L^{1,p}_{\lam-1-d}\to L^p_{\lam-d}$ is Fredholm.
\item\label{prop:d L L:kernel} We have $\ker T=P_d$.
\item\label{prop:d L L:coker} The map 
\begin{equation}\label{eq:bar P -d}\bar P_{-d}\to \big(L^p_{\lam-d}/\im T\big)^*,\quad u\mapsto\left(v+\im T\mapsto \int_\C uv\,ds\,dt\right)\end{equation}
is well-defined and a $\C$-linear isomorphism.
\end{enui}
\end{prop}
The proof of this proposition is based on the following result, which is due to R.~B.~Lockhart.
\begin{thm}\label{thm:T Fredholm} Let $n,k,m\in\N_0$ be such that $n\geq2$, $k\geq m$, $\lam\in\R$, and 
\[T:L^{k,p}_\lam(\R^n,\C)\to L^{k-m,p}_{\lam+m}(\R^n,\C)\] 
a constant coefficient homogeneous elliptic linear operator of order $m$ on $\R^n$. If $\lam+n/p\not\in\Z$ then $T$ is Fredholm. 
\end{thm}
\begin{proof}[Proof of Theorem \ref{thm:T Fredholm}] This is an immediate consequence of R.~B.~Lockhart's result \cite[Theorem 4.3]{Lockhart Fred}, using that a bounded linear operator between Banach spaces is Fredholm if its adjoint operator is Fredholm.  
\end{proof}
\begin{rmk}\label{rmk:X Y} Let $X$ be a normed vector space and $Y\sub X$ a closed subspace. We equip $X/Y$ with the quotient norm. The map
\[Y^\perp:=\big\{\phi\in X^*\,\big|\,\phi(x)=0,\,\forall x\in Y\big\}\to (X/Y)^*,\quad\phi\mapsto \big(x+Y\mapsto \phi(x)\big),\] 
is well-defined and an isometric isomorphism. This follows from a straight-forward argument. $\Box$
\end{rmk}
We denote by $\SSS$ the space of Schwartz functions on $\C$ and by $\SSS'$ the space of temperate distributions. By $\HAT:\SSS'\to\SSS'$ we denote the Fourier transform, and by $\Unhat:\SSS'\to\SSS'$ the inverse transform. 
\begin{proof}[Proof of Proposition \ref{prop:d L L}]\setcounter{claim}{0} \label{proof:d L L}\setcounter{claim}{0} Let $d,p,\lam,$ and $T$ be as in the hypothesis.

{\bf Statement (\ref{prop:d L L:Fredholm})} follows from Theorem \ref{thm:T Fredholm}, observing that $\dd_{\bar z}$ is elliptic, i.e., its principal symbol 
\[\si_T:\R^2=\C\to \C,\quad \si_T(\ze)= \frac\ze2\] 
does not vanish on $\R^2\wo\{0\}$. 

We prove {\bf statement (\ref{prop:d L L:kernel}).} A calculation in polar coordinates shows that for every polynomial $u$ in $z$ we have
\begin{eqnarray}\label{eq:u L}u\in L^{1,p}_{{\lam-1}-d}\iff \deg u<d-\lam+1-\frac2p.
\end{eqnarray}
Hence our assumption $\lam<-2/p+2$ implies that $\ker T\cont P_d$. Therefore, statement (\ref{prop:d L L:kernel}) is a consequence of the following claim.
\begin{claim}\label{claim:ker T P d} We have $\ker T\sub P_d$. 
\end{claim}
\begin{proof}[Proof of Claim \ref{claim:ker T P d}] Let $u\in \ker T$. Then $0=\hhat{\dd_{\bar z}u}(\ze)=\frac i2\ze\hhat u$ (as temperate distributions). It follows that the support of $\hhat u$ is either empty or consists of the point $0\in\C$. Hence the Paley-Wiener theorem implies that $u$ is real analytic in the variables $s$ and $t$, where $z=s+it$, and there exists $N\in\N$ such that $\sup_{z\in\C}|u(z)|\lan z\ran^N<\infty$.%
\footnote{See e.g.~\cite[Theorem IX.12]{ReSi}.}
 Therefore, by Liouville's Theorem $u$ is a polynomial in the variable $z$. Since by our assumption $\lam>-2/p+1$, it follows from (\ref{eq:u L}) that $u\in P_d$. This proves Claim \ref{claim:ker T P d}.
\end{proof}
To prove {\bf statement (\ref{prop:d L L:coker})}, we define $p':=p/(p-1)$. Consider the isometric isomorphism
\[\Phi:L^{p'}_{-\lam+d}\to (L^p_{\lam-d})^*,\quad\Phi(u):=\left(v\mapsto \int_\C uv\right).\]
Denoting by $T^*$ the adjoint operator of $T$, we have
\[T^*\Phi=\dd_z:L^{p'}_{-\lam+d}\to (L^{1,p}_{{\lam-1}-d})^*,\] 
where the derivatives are taken in the sense of distributions. 
\begin{claim}\label{claim:T P} We have $\ker(T^*\Phi)=\bar P_{-d}$.
\end{claim}
\begin{proof}[Proof of Claim \ref{claim:T P}] For every polynomial $u$ in $\bar z$ we have
\begin{equation}
  \label{eq:u L *}u\in L^{p'}_{-\lam+d} \iff \deg u<-d+\lam-\frac2{p'}=-d+\lam-2+\frac2p.
\end{equation}
Our assumption $\lam>-2/p+1$ and (\ref{eq:u L *}) imply that $\ker T^*\cont\bar P_{-d}$. Furthermore, the inclusion $\ker T^*\sub \bar P_{-d}$ is proved analogously to the inclusion $\ker T\sub P_d$, using $\lam<-2/p+2$ and (\ref{eq:u L *}). This proves Claim \ref{claim:T P}.
\end{proof}
It follows from Claim \ref{claim:T P} that the map $\Phi$ restricts to a $\C$-linear isomorphism between $\bar P_{-d}$ and $\ker T^*=(\im T)^\perp$. The composition of this map with the canonical isomorphism $(\im T)^\perp\to\big(L^p_{\lam-d}/\im T\big)^*$ described in Remark \ref{rmk:X Y}, equals the map (\ref{eq:bar P -d}). Statement (\ref{prop:d L L:coker}) follows. This completes the proof of Proposition \ref{prop:d L L}.
\end{proof}
Let $d\in\Z$, $1<p<\infty$, $-2/p+1<\lam<-2/p+2$, and $\rho_0:\C\to [0,1]$ be a smooth function that vanishes on $B_{1/2}$ and equals 1 on $B_1^C$. We equip $\C\rho_0p_d+L^{1,p}_{{\lam-1}-d}$ with the norm induced by the isomorphism of Lemma \ref{le:X d iso}. This norm is complete. (See e.g. \cite{Lockhart PhD}.) 
\begin{cor}\label{cor:d L L} The map $\dd_{\bar z}:\C\rho_0p_d+L^{1,p}_{{\lam-1}-d}\to L^p_{\lam-d}$ is Fredholm, with real index $2+2d$.
\end{cor}
\begin{proof}[Proof of Corollary \ref{cor:d L L}]\setcounter{claim}{0} The composition of the isomorphism of Lemma \ref{le:X d iso} with the above map is given by
\[T+S:\C\oplus L^{1,p}_{{\lam-1}-d}\to L^p_{\lam-d},\quad T(x_\infty,u):=\dd_{\bar z}u,\quad S(x_\infty,u):=x_\infty(\dd_{\bar z}\rho_0)p_d.\]
The map $T$ is the composition of the canonical projection $\pr:\C\oplus L^{1,p}_{{\lam-1}-d}\to L^{1,p}_{{\lam-1}-d}$ with the operator $\dd_{\bar z}:L^{1,p}_{{\lam-1}-d}\to L^p_{\lam-d}$. Using Proposition \ref{prop:d L L}, it follows that $T$ is Fredholm of real index $2+2d$. Furthermore, $S$ is compact, since it equals the composition of the canonical projection $\C\oplus L^{1,p}_{{\lam-1}-d}\to\C$ (which is compact) with a bounded operator. Corollary \ref{cor:d L L} follows.
\end{proof}
The next result was used in the proof of Theorem \ref{thm:Fredholm aug} (Fredholm property for the augmented vertical differential) in Section \ref{subsec:proof:Fredholm aug}. Let $(V,\lan\cdot,\cdot\ran)$ be a finite dimensional hermitian vector space, $A,B:V\to V$ positive  linear maps, $\lam\in\R$ and $1<p<\infty$. We define
\[T_\lam:=\left(
  \begin{array}{cc}
\dd_{\bar z}& A\\
B & \dd_z
  \end{array}
\right):W^{1,p}_\lam(\C,V\oplus V)\to L^p_\lam(\C,V\oplus V).\]
\begin{prop}\label{prop:d A B d} The operator $T_\lam$ is Fredholm of index 0.
\end{prop}
For the proof of Proposition \ref{prop:d A B d} we need the following result.
\begin{prop}\label{prop:Calderon} Let $(V,\lan\cdot,\cdot\ran),p$ and $A$ be as above, and $n\in\N$. Then the map
\[-\La+A:W^{2,p}(\R^n,V)\to L^p(\R^n,V)\]
is an isomorphism (of Banach spaces).
\end{prop}
\begin{proof}[Proof of Proposition \ref{prop:Calderon}]\setcounter{claim}{0} Consider first the case $\dim_\C V=1$ and $A=1$. We define
\[G:=(2\pi)^{\frac n2}\big(\langle\cdot\rangle^{-2}\big)\Unhat\in \SSS'.\]
The map $\SSS\ni u\mapsto G*u\in\SSS$ is well-defined. By Calder\'on's Theorem this map extends uniquely to an isomorphism
\begin{equation}
  \label{eq:L G u}L^p(\R^n,\C)\ni u\mapsto G*u\in W^{2,p}(\R^n,\C).
\end{equation}
(See \cite[Theorem 1.2.3.]{Ad}.) Note that 
\[(-\La+1)(G*u)=\big(\lan\cdot\ran^2(G*u)\HAT\big)\Unhat=u,\] 
for every $u\in\SSS$. It follows that the inverse of (\ref{eq:L G u}) is given by
\[-\La+1:W^{2,p}(\R^n,\C)\to L^p(\R^n,\C).\]
Hence this is an isomorphism.

The general case can be reduced to the above case by diagonalizing the map $A$. This proves Proposition \ref{prop:Calderon}.
\end{proof}
\begin{proof}[Proof of Proposition \ref{prop:d A B d}]\label{proof:d A B d} \setcounter{claim}{0} We abbreviate $L^p:=L^p(\C,V\oplus V)$, etc. 

{\bf Assume first that $\lam=0$.} We denote by $A^{1/2},B^{1/2}:V\to V$ the unique positive  linear maps satisfying $(A^{\frac12})^2=A$, $(B^{\frac12})^2=B$. We define 
\begin{eqnarray}\nn L:=\left(
    \begin{array}{cc}\dd_{\bar z}&A^{\frac12}B^{\frac12}\\
B^{\frac12}A^{\frac12}&\dd_z
    \end{array}
\right):W^{1,p}\to L^p.
\end{eqnarray}
A short calculation shows that 
\begin{equation}
  \label{eq:A B T}T_0=\big(A^{\frac12}\oplus B^{\frac12}\big)L\big(A^{-\frac12}\oplus B^{-\frac12}\big).
\end{equation}
\begin{Claim}The operator $L$ is an isomorphism. 
\end{Claim}
\begin{proof}[Proof of the claim] We define
\[L':=\left(
  \begin{array}{cc}-\dd_z&A^{\frac12}B^{\frac12}\\
B^{\frac12}A^{\frac12}&-\dd_{\bar z}\end{array}
\right):W^{2,p}\to W^{1,p}.\]
By a short calculation we have 
\[LL'=\big(-\frac\La4 +A^{\frac12}B A^{\frac12}\big)\oplus\big(-\frac\La4 + B^{\frac12}A B^{\frac12}\big):W^{2,p}\to L^p.\] 
Since the linear maps
\[A^{\frac12}BA^{\frac12},B^{\frac12}AB^{\frac12}:V\to V\]
are positive, Proposition \ref{prop:Calderon} implies that $LL'$ is an isomorphism. We denote by $(LL')^{-1}:L^p\to W^{2,p}$ its inverse and define
\[R:=L'(LL')^{-1}:L^p\to W^{1,p}.\]
Then $R$ is bounded and $LR=\id_{L^p}$. 

By a short calculation, we have $LL'(u,v)=L'L(u,v)$, for every Schwartz function $(u,v)\in\SSS$. This implies that
\[(LL')^{-1}L|_\SSS=L(LL')^{-1}|_\SSS,\]
and therefore $RL|_\SSS=\id_\SSS$. Since $RL:W^{1,p}\to W^{1,p}$ is continuous and $\SSS\sub W^{1,p}$ is dense, it follows that $RL=\id_{W^{1,p}}$. The claim follows. 
\end{proof}
The maps
\[A^{\frac12}\oplus B^{\frac12}:L^p\to L^p,\quad A^{-\frac12}\oplus B^{-\frac12}:W^{1,p}\to W^{1,p}\]
are automorphisms. Therefore, (\ref{eq:A B T}) and the claim imply that $T_0$ is an isomorphism. 

Consider now the {\bf general case $\lam\in\R$}. The map
\[L^p\ni (u,v)\mapsto \lan\cdot\ran^{-\lam}(u,v)\in L^p_\lam\]
is an isometric isomorphism. Furthermore, by Proposition \ref{prop:lam d Morrey}(\ref{prop:lam W}) the map
\[W^{1,p}_\lam\ni (u,v)\mapsto \lan\cdot \ran^\lam(u,v)\in W^{1,p}\]
is well-defined and an isomorphism. We define 
\[S:=\lan\cdot\ran^\lam(\dd_{\bar z}\lan\cdot\ran^{-\lam})\oplus\lan\cdot\ran^\lam(\dd_z\lan\cdot\ran^{-\lam}):W^{1,p}\to L^p.\] 
Direct calculations show that 
\[T_\lam=\lan\cdot\ran^{-\lam}(T_0+S)\lan\cdot\ran^\lam,\quad|\dd_{\bar z}\lan\cdot\ran^{-\lam}|\leq |\lam|\lan\cdot\ran^{-\lam-1}/2,\quad|\dd_z\lan\cdot\ran^{-\lam}|\leq|\lam|\lan\cdot\ran^{-\lam-1}/2.\] 
Therefore, Proposition \ref{prop:lam d Morrey}(\ref{prop:lam W cpt}) implies that the operator $S$ is compact. Since we proved that $T_0$ is an isomorphism, it follows that $T_\lam$ is a Fredholm map of index 0. This proves Proposition \ref{prop:d A B d} in the general case.
\end{proof}
\section{Smoothening a principal bundle}\label{sec:smooth}
The main result of this section states that a principal bundle of Sobolev class $W^{2,p}_\loc$ is Sobolev isomorphic to a smooth bundle, if $p$ is large enough. This will be used in the proofs of Propositions \ref{PROP:RIGHT}, \ref{prop:Uhlenbeck}, and \ref{prop:right d A * d A} in the next section. Let $n\in\N$ be an integer, $p>n/2$ a real number, $X$ a smooth manifold (possibly with boundary) of dimension $n$, $G$ a compact Lie group, and $P$ a $G$-bundle over $X$ of class $W^{2,p}_\loc$.%
\footnote{By definition, this means the following. Let $U,U'\sub X$ be open subsets, and $\Phi:U\x G\to P$ and $\Phi':U'\x G\to P$ local trivializations. Then the corresponding transition function $g_{U',U}:U\cap U'\to G$ is bounded in $W^{2,p}$ on every compact subset $K\sub U\cap U'$. Note here that $K$ may intersect the boundary $\dd X$.}%
\begin{thm}[Smoothening a principal bundle]\label{thm:reg bundle} If $p>\frac n2$ and $P$ is as above then there exists an isomorphism of principal $G$-bundles of class $W^{2,p}_\loc$, from $P$ to a smooth bundle over $X$.
\end{thm}
The proof of this result relies on the facts that there exists a smooth $G$-bundle that is ``$C^0$-close'' to $P$, and that if two ``Sobolev bundles'' are ``$C^0$-close'' then they are ``Sobolev-isomorphic''. In order to explain this, let $\UU$ be an open cover of $X$. We call a collection of functions
\[g_{U',U}:U\cap U'\to G\quad(U,U'\in\UU)\]
\emph{compatible} iff it satisfies 
\begin{equation}\label{eq:g U U' U''}g_{U,U}=\one,\quad g_{U'',U'}g_{U',U}=g_{U'',U}\textrm{ on }U\cap U'\cap U'',\quad\forall U,U',U''\in\UU.
\end{equation}
\begin{Rmk}The second condition means that the collection $(g_{U',U})$ is a \v{C}ech 1-cocycle. $\Box$
\end{Rmk}
Recall that a cover $\UU$ of $X$ is called \emph{locally finite} iff every point $x\in X$ possesses a neighborhood which intersects only finitely many sets in $\UU$. By a \emph{refinement} of the cover $\UU$ we mean a cover $\VV$ of $X$, together with a map $\VV\ni V\mapsto U_V\in\UU$, such that $V\sub U_V$, for every $V\in\VV$. The first ingredient of the proof of Theorem \ref{thm:reg bundle} is the following.
\begin{prop}[Smoothening a compatible collection of maps]\label{prop:smooth collection} Let
\[g_{U,U'}\in W^{2,p}_\loc(U\cap U',G)\quad(U,U'\in\UU)\]
be a compatible collection, and $W$ a neighborhood of the diagonal in $G\x G$. Then there exists a locally finite refinement $\VV\ni V\mapsto U_V\in\UU$ consisting of open precompact sets, and a compatible collection of smooth maps
\[h_{V,V'}:V\cap V'\to G\quad(V,V'\in\VV),\]
such that 
\begin{equation}\label{eq:wt g g}\big(h_{V,V'}(x),g_{U_V,U_{V'}}(x)\big)\in W,\quad\forall V,V'\in\VV,\,x\in V\cap V'.
\end{equation}
\end{prop}
\begin{Rmk} This result says that there exists a smooth \v{C}ech $1$-cocycle, which is arbitrarily close in the uniform sense to a given \v{C}ech $1$-cocycle of Sobolev class. Condition (\ref{eq:wt g g}) may look unconventional, but it is a natural way of stating the closeness condition. An alternative would be to introduce a Riemannian metric on $G$ and formulate the condition in terms of the induced distance function on $G$. $\Box$
\end{Rmk}
\begin{proof}[Proof of Proposition \ref{prop:smooth collection}] We may assume w.l.o.g.~that $\dd X=\emptyset$, by considering the double $X\#X$. If $X$ is compact then the statement follows from \cite[Theorem 2.1]{Is}, using the fact that $W^{2,p}$-maps are continuous, since $p>n/2$. In the general case, it follows from an adaption of that proof: We first choose a locally finite refinement $\VV\ni V\mapsto U_V\in\UU$ by precompact sets, and define $g'_{V,V'}:=g_{U_V,U_{V'}}|_{V\cap V'}$. Then we follow the argument of the proof of \cite[Theorem 2.1]{Is}, for the cover $\VV$.%
\footnote{Note that we can choose the sets in the refinement of $\VV$ as in that proof to be precompact, since the sets in $\VV$ are precompact.}
 This proves Proposition \ref{prop:smooth collection}.
\end{proof}
The next lemma will also be used in the proof of Theorem \ref{thm:reg bundle}.
\begin{lemma}\label{le:k V} There exists a neighborhood $\NN$ of the diagonal $G\x G$ with the following property. Let $n\in\N$ be an integer, $p>\frac n2$ a real number, $X$ a smooth manifold of dimension $n$, and $\UU$ a locally finite cover of $X$ by precompact sets. Then there exists a refinement $\VV\ni V\mapsto U_V\in\UU$ such that the following holds. Let
\[g_{U,U'}\in W^{2,p}\big(U\cap U',G\big),\quad h_{U,U'}\in W^{2,p}\big(U\cap U',G\big)\quad(U,U'\in\UU)\]
be compatible collections of maps satisfying 
\[\big(g_{U,U'}(x),h_{U,U'}(x)\big)\in\NN,\quad\forall x\in U\cap U',\,U,U'\in \UU.\]
Then there exists a collection of maps $k_V\in W^{2,p}(V,G)$ ($V\in\VV$), such that 
\begin{equation}\label{eq:k h k g} k_{V'}^{-1}h_{V',V}k_V=g_{V',V}\textrm{ on }V\cap V'.
\end{equation}
\end{lemma}
\begin{proof}[Proof of Lemma \ref{le:k V}] This is a direct consequence of \cite[Lemma 7.2]{We}. 
\end{proof}
For the proof of Theorem \ref{thm:reg bundle}, we also need the following.
\begin{rmk}\label{rmk:P} Let $\UU$ be an open cover of $X$, and $g_{U',U}$ ($U,U'\in\UU$) be a compatible collection of maps. We define the set
\[P_{(g_{U',U})}:=\big\{(U,x,g)\,\big|\,U\in\UU,\,x\in U,\,g\in G\big\}/\sim,\]
where the equivalence relation $\sim$ is defined by
\[(U,x,g)\sim(U',x',g')\textrm{ iff }x=x'\in U\cap U'\textrm{ and }g'=gg_{U',U}(x).\]
If the maps $g_{U',U}$ are smooth, then this set naturally is a smooth $G$-bundle, and if the maps $g_{U',U}$ are of Sobolev class $W^{k,p}_\loc$ with $kp>n:=\dim X$, then it naturally is a $G$-bundle of class $W^{k,p}_\loc$. In either case, a system of local trivializations is given by 
\begin{equation}\label{eq:Phi U U G}\Phi_U:U\x G\to P,\quad(\Phi_U)_x(g):=\Phi_U(x,g):=[U,x,g],
\end{equation}
where $[U,x,g]$ denotes the equivalence class of $(U,x,g)$. The map $g_{U',U}$ is the transition map from $\Phi_U$ to $\Phi_{U'}$. This means that 
\[(\Phi_{U'})_x^{-1}(\Phi_U)_x(g)=g\,g_{U',U}(x),\quad\forall x\in U\cap U',\,g\in G.\]
Let $kp>n$, and $(g_{U',U})$ and $(h_{U',U})$ be compatible collections of maps of class $W^{k,p}_\loc$. Then the bundles $P_{(g_{U',U})}$ and $P_{(h_{U',U})}$ are $W^{k,p}_\loc$-isomorphic, if there exists a collection of maps $k_U\in W^{k,p}_\loc(U,G)$, satisfying the equation
\begin{equation}\label{eq:k U'}k_{U'}(x)^{-1}h_{U',U}(x)k_U(x)=g_{U',U}(x),\quad\forall x\in U\cap U',\forall U,U'\in\UU.
\end{equation}
Defining $\Phi_U$ as in (\ref{eq:Phi U U G}) and $\Psi_U$ similarly, with $g_{U',U}$ replaced by $h_{U',U}$, an isomorphism $P_{(g_{U',U})}\to P_{(h_{U',U})}$ is given by
\[[U,x,g]\mapsto(\Psi_U)_x\big(k_U(x)(\Phi_U)_x^{-1}(g)\big).\]
(It follows from the compatibility condition (\ref{eq:g U U' U''}) and (\ref{eq:k U'}) that this map is well-defined, i.e., the right hand side above does not depend on the choice of the representative $(U,x,g)$.) $\Box$
\end{rmk}
\begin{proof}[Proof of Theorem \ref{thm:reg bundle} (p.~\pageref{thm:reg bundle})]\setcounter{claim}{0} We choose a cover $\UU$ of $X$, and a system of local trivializations $\Phi_U:U\x G\to P$ of class $W^{2,p}_\loc$ ($U\in\UU$). For $U,U'\in\UU$ we denote by $g_{U',U}:U\cap U'\to G$ the corresponding transition map, defined by 
\[\cdot g_{U',U}(x):=(\Phi_{U'})_x^{-1}(\Phi_U)_x,\]
where for $g\in G$, $\cdot g:G\to G$ denotes right multiplication. We choose a neighborhood $\NN$ of the diagonal in $G\x G$ as in Lemma \ref{le:k V}, and a refinement $\VV$ and a collection of smooth maps $h_{V',V}$ ($V,V'\in\VV$) as in Proposition \ref{prop:smooth collection}. Using (\ref{eq:wt g g}), we may apply Lemma \ref{le:k V} with $\UU,g_{U',U},h_{U',U}$ replaced by $\VV,g_{U_{V'},U_V},h_{V',V}$, to conclude that there exists a refinement $\WW\ni W\mapsto V_W\in\VV$, and a collection of maps $k_W\in W^{2,p}(W,G)$ ($W\in\WW$), such that 
\[k_{W'}^{-1}h_{V_{W'},V_W}k_W=g_{U_{V_{W'}},U_{V_W}}\quad\textrm{on }W\cap W',\,\forall W,W'\in\WW.\]
Therefore, the statement of Theorem \ref{thm:reg bundle} follows from Remark \ref{rmk:P}. 
\end{proof}
\section{Proof of the existence of a right inverse for $d_A^*$}\label{sec:proof:prop:right}
In this section we prove Proposition \ref{PROP:RIGHT} (Section \ref{subsec:proof:thm:L w * R}). We need the following results. Let $n\in\N$, $G$ be a compact Lie group with Lie algebra $\g$, $\lan\cdot,\cdot\ran_\g$ an invariant inner product on $\g$, and $(X,\lan\cdot,\cdot\ran_X)$ a Riemannian manifold (possibly with boundary) of dimension $n$. Recall the definitions (\ref{eq:Vert al Vert},\ref{eq:Om i k p A},\ref{eq:Ga k p A}) of $\Vert\al\Vert_{k,p,A},\Om^i_{k,p,A}(\g_P),\Ga^{k,p}_A(\g_P),\Ga^p(\g_P)$. If $X$ is compact, $p>n/2$, and $P$ is a $G$-bundle over $X$ of Sobolev class $W^{2,p}$, then we denote by $\A^{1,p}(P)$ the affine space of connection one-forms on $P$ of class $W^{1,p}$.
\begin{prop}[Uhlenbeck gauge]\label{prop:Uhlenbeck} Let $n/2<p<\infty$. Assume that $X$ is compact and diffeomorphic to the closed ball $\bar B_1\sub\R^n$. Then there exist constants $\eps>0$ and $C$ with the following property. Let $P_0$ and $P$ be $G$-bundles over $X$ of class $W^{2,p}$, and $A_0,A\in\A^{1,p}(P)$ connections, such that $A_0$ is flat and
\[\Vert F_A\Vert_p\leq\eps.\]
Then there exists an isomorphism of $G$-bundles $\Phi:P_0\to P$, of class $W^{2,p}$, such that 
\begin{equation}\label{eq:Phi A A 0}\Vert\Phi^*A-A_0\Vert_{1,p,A_0}\leq C\Vert F_A\Vert_p.
\end{equation}
\end{prop}
The proof of this result is based on the following.
\begin{lemma}[Uhlenbeck gauge with trivial connection]\label{le:Uhlenbeck trivial} Let $n/2<p<\infty$. Assume that $X=\BAR B_1\sub\R^n$, equipped with the standard metric $\lan\cdot,\cdot\ran_0$. We denote by $P_0$ the trivial $G$-bundle $X\x G$, and by $A_0$ the trivial connection on this bundle. Then there exist constants $\eps>0$ and $C>0$, such that for every connection $A\in\A^{1,p}(P_0)$, satisfying $\Vert F_A\Vert_p\leq\eps$, there exists a gauge transformation $g$ on $P_0$ of class $W^{2,p}$, such that 
\[\Vert g^*A-A_0\Vert_{1,p,A_0,\lan\cdot,\cdot\ran_0}\leq C\Vert F_A\Vert_p.\]
\end{lemma}
\begin{proof}[Proof of Lemma \ref{le:Uhlenbeck trivial}] This follows e.g.~from \cite[Theorem 6.3]{We}.
\end{proof}
In the proof of Proposition \ref{prop:Uhlenbeck} we also use the following.
\begin{rmk}\label{rmk:metric} Let
\[\ell\in\N_0,\quad p\in\left(\frac n{\ell+1},\infty\right),\]
and $\lan\cdot,\cdot\ran'_X$ be a Riemannian metric on $X$. Assume that $X$ is compact. Then there exists a constant $C>0$ such that 
\[C^{-1}\Vert\al\Vert_{k,p,\lan\cdot,\cdot\ran_X,A}\leq \Vert\al\Vert_{k,p,\lan\cdot,\cdot\ran_X',A}\leq C\Vert\al\Vert_{k,p,\lan\cdot,\cdot\ran_X,A},\]
for every $G$-bundle $P$ over $X$, of class $W^{\ell+1,p}$, connection $A\in\A^{\ell,p}(P)$, integer $k\in\{0,\ldots,\ell+1\}$, and every differential form $\al$ on $X$ with values in $\g_P$, of class $W^{k,p}$. (Here the degree of $\al$ is arbitrary.) This follows from an elementary argument, using induction over $k$. $\Box$
\end{rmk}
\begin{proof}[Proof of Proposition \ref{prop:Uhlenbeck}]\setcounter{claim}{0} Consider first the {\bf case $X=\BAR B_1$} together with the standard metric. We choose constants $\eps,C$ as in Lemma \ref{le:Uhlenbeck trivial}. Let $P_0,P,A_0,A$ be as in the hypothesis. We show that the required isomorphism $\Phi$ exists. It follows from Theorem \ref{thm:reg bundle} that we may assume without loss of generality that $P_0$ and $P$ are smooth. Since $X$ is smoothly retractible to a point, $P_0$ and $P$ are smoothly isomorphic to the trivial bundle over $X$. Hence we may assume without loss of generality that they are the trivial bundle. Since $A_0$ is flat, it follows from Lemma \ref{le:Uhlenbeck trivial} that $A_0$ is $W^{2,p}$-gauge equivalent to the trivial connection on $X\x G$. Therefore, the statement of Proposition \ref{prop:Uhlenbeck} is a consequence of Lemma \ref{le:Uhlenbeck trivial}.

The situation in which $\lan\cdot,\cdot\ran_X$ is a general metric, can be reduced to the above case, using Remark \ref{rmk:metric}. This proves Proposition \ref{prop:Uhlenbeck}.
\end{proof}
The proof of Proposition \ref{PROP:RIGHT} is based on the following. Recall from (\ref{eq:d A *}) that 
\[d_A^*=-*d_A*:\Om^1_{1,p,A}(\g_P)\to\Ga^p(\g_P)\]
denotes the formal adjoint of the operator $d_A$. If $(X,\Vert\cdot\Vert_X)$ and $(Y,\Vert\cdot\Vert_Y)$ are normed vector spaces and $T:X\to Y$ a bounded linear map then we denote by
\[\Vert T\Vert:=\sup\big\{\Vert Tx\Vert_Y\,\big|\,x\in X:\,\Vert x\Vert_X\leq1\big\}\] 
the operator norm of $T$. 
\begin{prop}\label{prop:right d A * d A} Let $p>n$. Assume that $X$ is diffeomorphic to $\BAR B_1\sub\R^n$, and $(X,\lan\cdot,\cdot\ran_X)$ can be embedded (as a Riemannian manifold) into $\R^n$, together with the standard metric. Then there exist constants $\eps>0$ and $C>0$, such that for every $G$-bundle $P\to X$ of class $W^{2,p}$, and every connection $A\in\A^{1,p}(P)$ satisfying $\Vert F_A\Vert_p\leq\eps$, there exists a right inverse $R$ of the operator 
\begin{equation}\label{eq:d A * d A}d_A^*d_A:\Ga^{2,p}_A(\g_P)\to\Ga^p(\g_P),\end{equation}
with operator norm $\Vert R\Vert$ bounded above by $C$.
\end{prop}
For the proof of this result, we need the following three lemmas.
\begin{lemma}[Twisted Morrey's inequality]\label{le:infty 1 p A} Let $p>n$. Assume that $X$ is diffeomorphic to $\BAR B_1$. Then there exist constants $C$ and $\eps>0$ such that for every $G$-bundle $P\to X$ of class $W^{2,p}$, and $A\in\A^{1,p}(P)$, the following holds. If $\Vert F_A\Vert_p\leq\eps$ then 
\begin{equation}\label{eq:al infty C}\Vert\al\Vert_\infty\leq C\Vert\al\Vert_{1,p,A},\quad\forall\al\in\Om^i_{1,p,A}(\g_P),\,i\in\{0,\ldots,n\}.
\end{equation}
\end{lemma}
\begin{proof}[Proof of Lemma \ref{le:infty 1 p A}]\setcounter{claim}{0} We denote by $A_0$ the trivial connection on the trivial bundle $P_0:=X\x G$. 
\begin{Claim}There exists a constant $C_1>0$ such that the inequality (\ref{eq:al infty C}) holds with $C=C_1$ and $A=A_0$.
\end{Claim}
\begin{proof}[Proof of the claim] This follows from Morrey's theorem, using the hypothesis $p>n$.
\end{proof}
We choose constants $\eps>0$ and $C_2:=C$ as in Proposition \ref{prop:Uhlenbeck}. Let $P$ be a $G$-bundle over $X$, of class $W^{2,p}$, and $A\in\A^{1,p}(P)$, such that $\Vert F_A\Vert_p\leq\eps$. By the statement of Proposition \ref{prop:Uhlenbeck}, there exists an isomorphism $\Phi:P_0\to P$ of class $W^{2,p}$, such that 
\begin{equation}\label{eq:Phi A A 0 C 2}\Vert\Phi^*A-A_0\Vert_{1,p,A_0}\leq C_2\Vert F_A\Vert_p.
\end{equation}
Let $i\in\{0,\ldots,n\}$ and $\al\in\Om^i_{1,p,A}(\g_P)$. We set 
\[A':=\Phi^*A,\quad\al':=\Phi^*\al,\quad C_3:=\max\big\{|[\xi,\eta]|\,\big|\,\xi,\eta\in \g:\,|\xi|\leq1,\,|\eta|\leq1\big\},\] 
where the norm $|\cdot|$ is with respect to $\lan\cdot,\cdot\ran_\g$. A direct calculation shows that
\[(\na^{A_0}-\na^{A'})\al'=[(A'-A_0)\otimes\al'],\]
where $[\cdot]:\g_P\otimes\g_P\to\g_P$ denotes the map induced by the Lie bracket on $\g$. Using the above claim, it follows that 
\begin{equation}\label{eq:al infty}\Vert \al\Vert_\infty=\Vert \al'\Vert_\infty\leq C_1\Vert\al'\Vert_{1,p,A_0}\leq C_1\big(\Vert\al'\Vert_{1,p,A'}+C_3\Vert A'-A_0\Vert_\infty\Vert\al'\Vert_p\big).\end{equation}
The above claim, inequality (\ref{eq:Phi A A 0 C 2}), and the assumption $\Vert F_A\Vert_p\leq\eps$ imply the estimate $\Vert A'-A_0\Vert_\infty\leq C_1C_2\eps$. Combining this with (\ref{eq:al infty}), the statement of Lemma \ref{le:infty 1 p A} follows. 
\end{proof}
\begin{lemma}\label{le:A 0 A} Let $p>n$ and $\eps>0$. Assume that $X$ is diffeomorphic to $\BAR B_1$. Then there exists a constant $\de>0$ with the following property. Let $P\to X$ be a $G$-bundle of class $W^{2,p}$, $A_0,A\in\A^{1,p}(P)$, and $\xi\in\Ga^{2,p}_{A_0}(\g_P)$, such that $A_0$ is flat 
\begin{equation}\label{eq:A A 0 de}\Vert A-A_0\Vert_{1,p,A_0}\leq\de.
\end{equation}
Then the following inequalities hold:
\begin{eqnarray}\label{eq:d A * d A A 0}&\big\Vert\big(d_A^*d_A-d_{A_0}^*d_{A_0}\big)\xi\big\Vert_p\leq\eps\Vert\xi\Vert_{1,p,A_0},&\\
\label{eq:xi 2 p A}&\Vert\xi\Vert_{2,p,A}\leq(1+\eps)\Vert\xi\Vert_{2,p,A_0}.&
\end{eqnarray}
\end{lemma}
\begin{proof}[Proof of Lemma \ref{le:A 0 A}]\setcounter{claim}{0} We have
\begin{eqnarray}\label{eq:d A * d A -}&d_A^*d_A-d_{A_0}^*d_{A_0}&\\
\nn&=(d_A-d_{A_0})^*(d_A-d_{A_0})+d_{A_0}^*(d_A-d_{A_0})+(d_A-d_{A_0})^*d_{A_0},&\\
\label{eq:d A d A 0}&d_A-d_{A_0}=\big[(A-A_0)\wedge\cdot\big].&
\end{eqnarray}
It follows that there exists a constant $C>0$ such that 
\begin{eqnarray*}&\Vert\big(d_A^*d_A-d_{A_0}^*d_{A_0}\big)\xi\Vert_p&\\
&\leq C\Big(\Vert A-A_0\Vert_\infty^2\Vert\xi\Vert_p+\left\Vert\na^{A_0}(A-A_0)\right\Vert_p\Vert\xi\Vert_\infty+\Vert A-A_0\Vert_\infty\Vert d_{A_0}\xi\Vert_p\Big),&
\end{eqnarray*}
for every $G$-bundle $P\to X$ of class $W^{2,p}$, $A_0,A\in\A^{1,p}(P)$, and $\xi\in\Ga^{2,p}_{A_0}(\g_P)$.%
\footnote{The constant $C$ depends on the Lie bracket on $\g$ and the metric on $X$, but not on the bundle $P$.}
 Using Lemma \ref{le:infty 1 p A} and flatness of $A_0$, it follows that there exists a constant $\de>0$ such that inequality (\ref{eq:d A * d A A 0}) holds for every $P,A_0,A,\xi$ as in the hypothesis. 

Inequality (\ref{eq:xi 2 p A}) follows from an analogous argument, involving formulas for $\na^A\na^A-\na^{A_0}\na^{A_0}$ and $\na^A-\na^{A_0}$ similar to (\ref{eq:d A * d A -},\ref{eq:d A d A 0}). This proves Lemma \ref{le:A 0 A}.
\end{proof}
\begin{lemma}\label{le:right inverse} Let $X$ and $Y$ be Banach spaces, and $T_0,S:X\to Y$ and $R_0:Y\to X$ bounded linear maps, such that
\[T_0R_0=\id,\quad\Vert S\Vert\leq\frac1{2\Vert R_0\Vert}.\]
Then there exists a right inverse $R$ of $T_0+S$, satisfying $\Vert R\Vert\leq2\Vert R_0\Vert$.
\end{lemma} 
The proof of this lemma will use the following remark.
\begin{rmk}\label{rmk:X T} Let $X$ be a Banach space, and $T:X\to X$ a linear map satisfying $\Vert T\Vert<1$. Then $\id+T$ is invertible, and 
\[\Vert(\id+T)^{-1}\Vert\leq\frac1{1-\Vert T\Vert}.\]
This follows from a standard argument, using the Neumann series $\sum_{n\in\N_0}(-1)^nT^n$. $\Box$
\end{rmk}
\begin{proof}[Proof of Lemma \ref{le:right inverse}] By hypothesis, we have
\[\Vert SR_0\Vert\leq\Vert S\Vert\,\Vert R_0\Vert\leq\frac12.\]
Hence using Remark \ref{rmk:X T}, it follows that the map
\[R:=R_0(\id+SR_0)^{-1}:Y\to X\]
is well-defined and has the desired properties. This proves Lemma \ref{le:right inverse}.
\end{proof}
\begin{proof}[Proof of Proposition \ref{prop:right d A * d A} (p.~\pageref{prop:right d A * d A})]\setcounter{claim}{0} 
\begin{claim}\label{claim:flat right} Let $1<p<\infty$. There exists a constant $C>0$ such that for every $G$-bundle $P\to X$ of class $W^{2,p}$, and every \emph{flat} connection $A$ on $P$ of class $W^{1,p}$, there exists a right inverse $R$ of the operator (\ref{eq:d A * d A}), satisfying $\Vert R\Vert\leq C$. 
\end{claim}
\begin{proof}[Proof of Claim \ref{claim:flat right}] By Theorem \ref{thm:reg bundle} we may assume without loss of generality that $P$ is smooth. We choose a smooth flat connection $A_0$ on $P$. By Proposition \ref{prop:Uhlenbeck} there exists an automorphism $\Phi$ of $P$, of class $W^{2,p}$, such that inequality (\ref{eq:Phi A A 0}) holds. Since $A$ is flat, it follows that $\Phi^*A=A_0$, and thus $\Phi^*A$ is smooth. Hence we may assume w.l.o.g.~that $A$ is smooth.

For an open subset $U\sub\R^n$ we denote by $C^\infty_0(U)$ the compactly supported smooth real valued functions on $U$. We define the operator $\wt T:C^\infty_0(B_1)\to C^\infty(B_1)$ as follows. We denote by $\Phi:\R^n\wo\{0\}\to\R$ the fundamental solution for the Laplace equation. It is given by
\[\Phi(x)=\left\{\begin{array}{ll}
\frac1{2\pi}\log|x|,&\textrm{if }n=2,\\
\frac{\Ga(n/2)}{2(2-n)\pi^{\frac n2}}|x|^{2-n},&\textrm{if }n\neq2,
\end{array}\right.\]
where $\Ga$ denotes the gamma function. (See e.g. \cite{Ev}, p.~22.) Let $f\in C^\infty_0(B_1)$. We define $\wt f:\R^n\to\R$ to be the extension of $f$ by 0 outside $B_1$. We denote by $*$ convolution in $\R^n$ and define
\[\wt Tf:=(\Phi*\wt f)|_{B_1}.\]
Note that $\Phi$ is locally integrable, hence the convolution is well-defined. Furthermore, $\wt Tf$ is smooth, and $\La\wt Tf=f$.%
\footnote{The first assertion follows from differentiation under the integral, and for the second see e.g.~\cite[Chap.~2, Theorem 1]{Ev}.}%
\begin{claim}\label{claim:C Tf} There exists a constant $C$ such that
\[\Vert \wt Tf\Vert_{W^{2,p}(B_1)}\leq C\Vert f\Vert_{L^p(B_1)},\quad\forall f\in C^\infty_0(B_1).\]
\end{claim}
\begin{proof}[Proof of Claim \ref{claim:C Tf}] Young's inequality states that 
\[\Vert \wt Tf\Vert_{L^p(B_1)}\leq\Vert \Phi\Vert_{L^1(B_2)}\Vert f\Vert_{L^p(B_1)},\quad\forall f\in C^\infty_0(B_1).\] 
Furthermore, the Calder\'on-Zygmund inequality states that there exists a constant $C$ such that for every $f\in C^\infty_0(\R^n)$ we have $\Vert D^2(\Phi*f)\Vert_p\leq C\Vert f\Vert_p$.%
\footnote{This follows e.g.~\cite[Theorem B.2.7]{MS04}, using $(\dd_j\Phi)*f=\dd_j(\Phi*f)$.}
 The statement of Claim \ref{claim:C Tf} follows from this. 
\end{proof}
We fix a constant $C$ as in Claim \ref{claim:C Tf}. By this claim the map $\wt T$ uniquely extends to a bounded linear map
\[T: L^p(B_1)\to W^{2,p}(B_1).\]
Since $\La \wt Tf=f$, for every $f\in C^\infty_0(B_1)$, a density argument shows that $\La Tf=f$, for every $f\in W^{2,p}(B_1)$, i.e., $T$ is a right inverse for $\La:W^{2,p}(B_1)\to L^p(B_1)$. 

Let now $G$ be a compact Lie group, $\lan\cdot,\cdot\ran_\g$ an invariant inner product on $\g=\Lie G$, and $(X,\lan\cdot,\cdot\ran_X)$ a Riemannian manifold as in the hypothesis of Proposition \ref{prop:right d A * d A}. Without loss of generality we may assume that $X$ is a submanifold (with boundary) of $\R^n$, and that $\lan\cdot,\cdot\ran_X$ is the standard metric. Let $\pi:P\to X$ be a $G$-bundle, and $A\in\A(P)$ be a flat connection. We fix a point $p_0\in P$, and denote by $\si:X\to P$ the $A$-horizontal section through $p_0$. This is the unique smooth section of $P$ satisfying $A\,d\si=0$ and $\si(0)=p_0$. (Such a section exists, since $A$ is flat, and $X$ is diffeomorphic to $\BAR B_1$.) For $k\geq0$ we define the map
\[\Psi_k:W^{k,p}(B_1,\g)\to\Ga^{k,p}_{A}(\g_P),\quad\Psi_k\xi:=G\cdot(\si,\xi).\]
This is an isometric isomorphism. We define $R:=-\Psi_2T\Psi_0^{-1}$. It follows that
\[\big\Vert R\big\Vert\leq\Vert\Psi_2\Vert \Vert T\Vert \Vert\Psi_0^{-1}\Vert=\Vert T\Vert\leq C.\]
A straight-forward calculation shows that
\[d_A^*d_A\Psi_2=\Psi_0d^*d=-\Psi_0\La:W^{2,p}(B_1,\g)\to\Ga^p(\g_P).\]
It follows that $R$ is a right inverse for $d_{A}^*d_{A}$. This proves Claim \ref{claim:flat right}.
\end{proof}
We choose constants $C_1:=C$ as in Claim \ref{claim:flat right}, $\de$ as in Lemma \ref{le:A 0 A}, corresponding to $\eps=\min\big\{1/(2C_1),1\big\}$, $(\eps_1,C_2):=(\eps,C)$ as in Proposition \ref{prop:Uhlenbeck}, and $(\eps_2,C_3):=(\eps,C)$ as in Lemma \ref{le:infty 1 p A}. We define 
\[\eps:=\min\left\{\frac\de{C_2},\eps_1,\eps_2\right\}.\] 
Let $P\to X$ be a $G$-bundle of class $W^{2,p}$, and $A\in\A^{1,p}(P)$, such that $\Vert F_A\Vert_p\leq\eps$. The statement of Proposition \ref{prop:right d A * d A} is a consequence of the following claim. 
\begin{claim}\label{claim:R d A * d A} There exists a right inverse $R$ of $d_A^*d_A$, satisfying $\Vert R\Vert\leq4C_1$. 
\end{claim}
\begin{pf}[Proof of Claim \ref{claim:R d A * d A}] We choose a flat connection $\wt A_0$ on $P$ of class $W^{1,p}$.%
\footnote{It follows from an argument involving Theorem \ref{thm:reg bundle} that such a connection exists.}
 Since $\Vert F_A\Vert_p<\eps\leq\eps_1$, by the statement of Proposition \ref{prop:Uhlenbeck} with $P_0=P$ there exists an automorphism $\Phi$ of $P$ of class $W^{2,p}$, such that
\begin{equation}\label{eq:Phi A wt A 0}\Vert\Phi^*A-\wt A_0\Vert_{1,p,\wt A_0}\leq C_2\Vert F_A\Vert_p.
\end{equation}
We define $A_0:=\Phi_*\wt A_0$. By the statement of Claim \ref{claim:flat right} there exists a right inverse $R_0$ of the operator
\[d_{A_0}^*d_{A_0}:\Ga^{2,p}_{A_0}(\g_P)\to\Ga^p(\g_P),\]
satisfying $\Vert R_0\Vert_{A_0}\leq C_1$, where $\Vert\cdot\Vert_{A_0}$ denotes the operator-norm.%
\footnote{Here we use the subscript ``$A_0$'' to indicate that the definition of this norm involves the connection $A_0$.}
 The assumption $\Vert F_A\Vert_p\leq\eps\leq\de/C_2$ and (\ref{eq:Phi A wt A 0}) imply that the condition (\ref{eq:A A 0 de}) is satisfied. Therefore, by (\ref{eq:d A * d A A 0}) with ``$\eps$''$=1/(2C_1)$, we have
\[\big\Vert d_A^*d_A-d_{A_0}^*d_{A_0}\big\Vert_{A_0}\leq\frac1{2C_1}.\]
Therefore, applying Lemma \ref{le:right inverse}, there exists a right inverse $R$ of the operator
\[d_A^*d_A:\Ga^{2,p}_{A_0}(\g_P)\to\Ga^p(\g_P),\]
satisfying $\Vert R\Vert_{A_0}\leq2C_1$. Combining this with inequality (\ref{eq:xi 2 p A}) (with ``$\eps$''$=1$), it follows that $\Vert R\Vert_A\leq4C_1$. This proves Claim \ref{claim:R d A * d A} and completes the proof of Proposition \ref{prop:right d A * d A}.
\end{pf}
\end{proof}
In the proof of Proposition \ref{PROP:RIGHT} we will also use the following lemma.
\begin{lemma}\label{le:n p f} Let $n\in\N$, $p\in[1,\infty]$, and $f\in L^p(B^n_1)$. Then the function
\[B^n_1\ni x\mapsto\int_0^{x_1}f\big(t,x_2,\ldots,x_n\big)dt\]
lies in $L^p$. 
\end{lemma}
\begin{proof}[Proof of Lemma \ref{le:n p f}]\setcounter{claim}{0} Consider first the {\bf case $p=1$ and $n=1$}, and let $f\in L^1((0,1))$. Defining $\chi(t,x):=1$ if $t\leq x$, and $0$, otherwise, Fubini's theorem implies that
\[\int_0^1\int_0^x|f(t)|\,dt\,dx=\int_{[0,1]\x[0,1]}\chi(t,x)|f(t)|\,dt\,dx=\int_0^1(1-t)|f(t)|dt<\infty.\]
The statement of the lemma in the case $p=1$ and $n=1$ follows. For $p=1$ and a general $n$ the proof is similar.

For a general exponent $p$, H\"older's inequality implies that
\[\left|\int_0^{x_1}f\big(t,x_2,\ldots,x_n\big)dt\right|^p\leq|x_1|^{p-1}\int_0^{x_1}\left|f\big(t,x_2,\ldots,x_n\big)\right|^pdt.\]
Hence in general, the statement of the lemma follows from what we have already proved, by considering the function $|f|^p$.
\end{proof}
We are now ready to prove that $d_A^*$ admits a right inverse:
\begin{proof}[Proof of Proposition \ref{PROP:RIGHT} (p.~\pageref{PROP:RIGHT})]\label{proof:prop:right}\setcounter{claim}{0} 
We prove {\bf statement (\ref{prop:right:exists})}.\\
Let $n,G,\lan\cdot,\cdot\ran_\g,p,X,\lan\cdot,\cdot\ran_X,P,A$ be as in the hypothesis. We show that the operator $d_A^*$ admits a bounded right inverse. By Theorem \ref{thm:reg bundle} we may assume w.l.o.g.~that $P$ is smooth. Since by assumption, $X$ is diffeomorphic to $\BAR B_1$, we may assume without loss of generality that $X=\BAR B_1$. The Hodge $*$-operator induces an isomorphism between $\Om^i_{k,p,A,\lan\cdot,\cdot\ran_X}(\g_P)$ and $\Om^{n-i}_{k,p,A,\lan\cdot,\cdot\ran_X}(\g_P)$. Using the equality $d_A^*=-*d_A*$, it follows that the operator
\[d_A^*:\Om^1_{1,p,A,\lan\cdot,\cdot\ran_X}(\g_P)\to\Ga^p_{\lan\cdot,\cdot\ran_X}(\g_P)\]
admits a bounded right inverse, if and only if the operator 
\[d_A:\Om^{n-1}_{1,p,A,\lan\cdot,\cdot\ran_X}(\g_P)\to\Om^n_{0,p,A,\lan\cdot,\cdot\ran_X}(\g_P)\] 
does so. Since $X=\BAR B_1$ is compact, using Remark \ref{rmk:metric}, it follows that the condition that $d_A^*$ admits a bounded right inverse, is independent of the metric $\lan\cdot,\cdot\ran_X$. Hence we may assume without loss of generality that $\lan\cdot,\cdot\ran_X$ is the standard metric on $\BAR B_1\sub\R^n$. 
\begin{claim}\label{claim:L W} There exists a bounded linear map 
\begin{equation}\label{eq:T Ga}T:\Ga^{1,p}_A(\g_P)\to\Om^1_{1,p,A}(\g_P),
\end{equation}
such that $d_A^*T=\id$. 
\end{claim}
\begin{proof}[Proof of Claim \ref{claim:L W}] We define $\Om\sub\R\x P$ to be the subset consisting of all $(t,p)$ such that
\begin{equation}\label{eq:t x 1}|t+x_1|^2+x_2^2+\cdots+x_n^2<1,
\end{equation}
where $x:=\pi(p)\in B_1$. Furthermore, we denote by $\Psi:\Om\to P$ the $A$-parallel transport in $x_1$-direction. Let $(t_0,p_0)\in\Om$. We denote $x_0:=\pi(p_0)\in\bar B_1$. Then $\Psi(t_0,p_0)=p(t_0)$, where 
\[p:\big\{t\in\R\,|\,(t,x_0)\textrm{ satisfies }(\ref{eq:t x 1})\big\}\to P\] 
is the unique path satisfying
\begin{equation}\label{eq:pr p t}\pr\circ p(t)=x_0+(t,0,\ldots0),\quad A\dot p=0,\quad p(0)=p_0.%
\footnote{In some smooth trivialization of the bundle $P$ the conditions (\ref{eq:pr p t}) correspond to the ordinary differential equation $\dot g=-g\xi$ with initial condition $g(0)=\one$, for a path $t\mapsto g(t)\in G$. Here $\xi(t)$ corresponds to the first component of $A$ at $p(t)$. This equation is of the form $\dot g(t)=f(t,g(t))$, where $f$ is continuous in $t$ and Lipschitz continuous in $g$. (Here we use that $A$ is of class $W^{1,p}$ with $p>2$.) Therefore, by the Picard-Lindel\"of theorem, there exists a unique solution of the initial value problem, and thus of (\ref{eq:pr p t}).}
\end{equation}
Let $\xi\in\Ga^{1,p}(\g_P)$. We define $\wt\xi:P\to\g$ by the condition $[p,\wt\xi(p)]=\xi\circ \pi(p)$, for $p\in P$, and 
\begin{equation}\label{eq:wt eta}\wt\eta:P\to\g,\quad\wt\eta(p):=\int_{-x_1}^0\wt\xi\circ\Psi(t,p)dt\in\g,
\end{equation}
where $(x_1,\ldots,x_n):=\pi(p)$. Furthermore, we define the section $\eta:B_1\to\g_P$ by the condition $\eta\circ\pi(p)=[p,\wt\eta(p)]$, for every $p\in P$, and 
\[T\xi:=-\eta\,dx^1.\]
This defines the operator (\ref{eq:T Ga}). Claim \ref{claim:L W} is a consequence of the following.
\begin{claim}\label{claim:T} The section $\eta$ lies in $\Ga^{1,p}_A(\g_P)$ and satisfies
\begin{equation}\label{eq:d A * eta}d_A^*(\eta dx^1)=-\xi.
\end{equation}
\end{claim}
\begin{pf}[Proof of Claim \ref{claim:T}] The section is of class $L^\infty$, since $\xi$ is of this class, by Morrey's embedding theorem. We denote by $e_1,\ldots,e_n$ the standard basis of $\R^n$. We show that
\begin{equation}\label{eq:d A eta}d_A\eta\,e_1=\xi.
\end{equation}
Let $x\in\R^n$. We fix a point $p_0$ in the fiber of $P$ over $(0,x_2,\ldots,x_n)$, and define the path $p(t):=\Psi(t+x_1,p_0)$. Using the equality $\Psi(s,p(t))=\Psi(s+t+x_1,p_0)$, it follows from (\ref{eq:wt eta}) that
\[\wt\eta\circ p(t)=\int_{-x_1-t}^0\wt\xi\circ\Psi(s+t+x_1,p_0)ds=\int_0^{x_1+t}\wt\xi\circ\Psi(s,p_0)ds.\]
It follows that 
\[d_A\wt\eta(p(0))\dot p(0)=\left.\frac d{dt}\right|_{t=0}\wt\eta(p(t))=\wt\xi(p(0)),\]
and therefore, $d_A\eta(x)e_1=\xi(x)$. This proves (\ref{eq:d A eta}). It follows that $d_A\eta\,e_1$ is of class $W^{1,p}$.

Let now $i\in\{2,\ldots,n\}$. We show that $|d_A\eta\,e_i|\in L^p(B^n_1)$. We fix $x\in\R^n$ and choose a smooth path $t\mapsto p(t)$, such that $\pi\circ p(t)=x+te_i$. Using (\ref{eq:wt eta}) with $p=p(t)$ and differentiating under the integral, we obtain
\[d_A\wt\eta(p(0))\dot p(0)=\int_{-x_1}^0d_A\wt\xi\big(\Psi(s,p(0))\big)\left.\frac d{dt}\right|_{t=0}\Psi(s,p(t))ds.%
\footnote{Here we used that the lower limit of the integral expressing $\wt\eta\circ p(t)$ is $-x_1$, for every $t$.}%
\]
Since $\pi_*\dot p(0)=e_i$ and $\pi\circ\Psi(s,p(t))=x+se_1+te_i$, it follows that
\[|d_A\eta(x)e_i|\leq\int_{-x_1}^0\big|d_A\xi(x+se_1)e_i\big|ds.\]
By Lemma \ref{le:n p f} with $f:=|d_A\xi\,e_i|$, using the assumption $\xi\in\Ga^{1,p}_A(\g_P)$, it follows that $|d_A\eta\,e_i|\in L^p(B^n_1)$. Therefore, $\eta$ lies in $\Om^1_{1,p,A}(\g_P)$. To prove equality (\ref{eq:d A * eta}), observe that
\[d_A^*(\eta dx^1)=-*d_A\big(\eta\,dx^2\wedge\ldots\wedge dx^n\big)=-*\big(\xi\,dx^1\wedge\ldots\wedge dx^n\big)=-\xi,\]
where in the second step we used (\ref{eq:d A eta}). This proves Claim \ref{claim:T} and hence Claim \ref{claim:L W}. 
\end{pf}
\end{proof}
We choose a map $T$ as in Claim \ref{claim:L W}, and a flat connection $A_0\in\A(P)$. By Proposition \ref{prop:right d A * d A} there exists a bounded right inverse $R_0$ of
\[d_{A_0}^*d_{A_0}:\Ga^{2,p}_{A_0}(\g_P)\to\Ga^p(\g_P).\]
Statement (\ref{prop:right:exists}) is now a consequence of the following.
\begin{claim}\label{claim:R} The operator
\[R:=d_AR_0+T\big(\id-d_A^*d_AR_0\big):L^p(\g_P)\to\Om^1_{1,p,A}(\g_P)\]
is well-defined and bounded, and $d_A^*R=\id$. 
\end{claim}
\begin{proof}[Proof of Claim \ref{claim:R}] A short calculation shows that $S:=d_A^*d_A-d_{A_0}^*d_{A_0}$ is of first or zeroth order. Hence it is bounded as a map from $\Ga^{2,p}_A(\g_P)$ to $\Ga^{1,p}_A(\g_P)$. Furthermore, the equality
\[\id-d_A^*d_AR_0=-SR_0\]
holds on smooth sections of $\g_P$. This implies that $R$ is well-defined and bounded. A short calculation shows that $d_A^*R=\id$. This proves Claim \ref{claim:R}, and completes the proof of (\ref{prop:right:exists}). 
\end{proof}
To prove {\bf statement (\ref{prop:right:eps C})}, we choose constants $\eps$ and $C$ as in Proposition \ref{prop:right d A * d A}. Let $P\to X$ be a $G$-bundle of class $W^{2,p}$, and $A\in\A^{1,p}(P)$, such that $\Vert F_A\Vert_p\leq\eps$. By the statement of Proposition \ref{prop:right d A * d A}, there exists a right inverse $R$ of the operator $d_A^*d_A:\Ga^{2,p}_A(\g_P)\to\Ga^p(\g_P)$, satisfying $\Vert R\Vert\leq C$. The operator $d_AR$ is a right inverse for $d_A^*$, satisfying $\Vert d_AR\Vert\leq\Vert d_A\Vert\,\Vert R\Vert\leq C$. Here in the last inequality we used the fact $\Vert d_A\Vert\leq1$, where
\[d_A:\Ga^{2,p}_A(\g_P)\to\Om^1_{1,p,A}(\g_P).\]
This proves (\ref{prop:right:eps C}), and completes the proof of Proposition \ref{PROP:RIGHT}.
\end{proof}
\section{Further auxiliary results}\label{sec:add}
The next two results were used in the proofs of Proposition \ref{prop:cpt bdd} (Section \ref{sec:comp}) and Theorem \ref{thm:cpt cpt} (Appendix \ref{sec:vort}).
\begin{thm}[Uhlenbeck compactness]\label{thm:Uhlenbeck compact} Let $n\in\N$, $G$ be a compact Lie group, $X$ a compact smooth Riemannian $n$-manifold (possibly with boundary), $P$ a smooth $G$-bundle over $X$, $p>n/2$ a number, and $A_\nu$ a sequence of connections on $P$ of class $W^{1,p}$. Assume that 
\begin{equation}
\nn\sup_{\nu\in\N}\Vert F_{A_\nu}\Vert_{L^p(X)}<\infty.
\end{equation}
Then passing to some subsequence there exist gauge transformations $g_\nu$ of class $W^{2,p}$, such that $g_\nu^*A_\nu$ converges weakly in $W^{1,p}$.
\end{thm}
\begin{proof}[Proof of Theorem \ref{thm:Uhlenbeck compact}] This is \cite[Theorem A]{We}. See also \cite[Theorem 1.5]{Uhlenbeck}. 
\end{proof}
\begin{prop}[Compactness for $\bar \dd_J$]\label{prop:compactness delbar}\label{prop:comp delbar} Let $M$ be a manifold without boundary, $k\in\N$, $p>2$ numbers, $J$ an almost complex structure on $M$ of class $C^k$, $\Om_1\sub\Om_2\sub\ldots\sub\C$ open subsets, and $u_\nu:\Om_\nu\to M$ a sequence of functions of class $W^{1,p}_\loc$. Assume that $\bar\dd_Ju_\nu$ is of class $W^{k,p}_\loc$, for every $\nu$, and that for every open subset $\Om\sub\bigcup_\nu\Om_\nu$ with compact closure the following holds. If $\nu_0\in\N$ is so large that $\Om\sub \Om_{\nu_0}$ then 
  \begin{eqnarray}
    \label{eq:u nu K}&\exists K\sub M\textrm{ compact: }u_\nu(\Om)\sub K,\quad\forall \nu\geq\nu_0,&\\
  \label{eq:du nu}&\sup_{\nu\geq\nu_0}\Vert du_\nu\Vert_{L^p(\Om)}<\infty,&  \\
\label{eq:sup k p}&\sup_{\nu\geq\nu_0}\Vert\bar\dd_Ju_\nu\Vert_{W^{k,p}(\Om)}<\infty.&
\end{eqnarray}
Then there exists a subsequence of $u_\nu$ that converges weakly in $W^{k+1,p}$ and in $C^k$ on every compact subset of $\bigcup_\nu\Om_\nu$. 
\end{prop}
\begin{proof}[Proof of Proposition \ref{prop:compactness delbar}] The proof goes along the lines of the proof of \cite[Proposition B.4.2]{MS04}. 
\end{proof}
The next lemma was used in the proofs of Propositions \ref{prop:cpt bdd} (Section \ref{sec:comp}) and \ref{prop:reg gauge} (Appendix \ref{sec:vort}), and of Theorem \ref{thm:bubb}. 
\begin{lemma}[Regularity of the gauge transformation]\label{le:g smooth} Let $X$ be a smooth manifold, $G$ a compact Lie group, $P$ a $G$-bundle over $X$, $k\in\N_0$, and $p>\dim X$. Then the following assertions hold.
\begin{enui}
\item \label{le:g C k+1}Let $g$ be a gauge transformation of class $W^{1,p}_\loc$ and $A$ a connection on $P$ of class $C^k$, such that $g^*A$ is of class $C^k$. Then $g$ is of class $C^{k+1}$. 
\item\label{le:g k+1 p} Assume that $X$ is compact (possibly with boundary). Let $\U$ be a subset of the space of $W^{k,p}$-connections on $P$ that is bounded in $W^{k,p}$. Then there exists a $W^{k+1,p}$-bounded subset $\V$ of the set of $W^{k+1,p}$-gauge transformations on $P$, such that the following holds. Let $A\in\U$ and $g$ be a $W^{1,p}$-gauge transformation, such that $g^*A\in\U$. Then $g\in\V$. 
\end{enui}
\end{lemma}
\begin{proof}[Proof of Lemma \ref{le:g smooth}] This follows by induction over $k$, using the equality $dg=g\big(g^*A)-Ag$ and Morrey's inequality (for (\ref{le:g k+1 p})). (For details see \cite[Lemma A.8]{We}.)
\end{proof}
The next proposition was used in the proof of Proposition \ref{prop:quant en loss} (Quantization of energy loss) in Section \ref{sec:comp}. 
\begin{prop}\label{prop:ka 0} Let $n\in\N$, $G$ be a compact Lie group, $P$ a $G$-bundle over $\R^n$, and $A,A'$ smooth flat connections on $P$. Then there exists a smooth gauge transformation $g$ such that $A'=g^*A$. 
\end{prop}
\begin{proof}[Proof of Proposition \ref{prop:ka 0}]\setcounter{claim}{0}%
\footnote{In the case $n=2$, see also \cite[Corollary 3.7]{FrPhD}.}
 In the case $n=1$ such a $g$ exists, since then the condition $A'=g^*A$ can be viewed as an ordinary differential equation for $g$. Let $n\in\N$ and assume by induction that we have already proved the statement for $n$. Let $P$ be a $G$-bundle over $\R^{n+1}$, and $A,A'$ smooth flat connections on $P$. We define $\iota:\R^n\to\R^{n+1}$ by $\iota(x):=(x,0)$. By the induction hypothesis there exists a smooth gauge transformation $g_0$ on $\iota^*P\to\R^n$, such that 
\begin{equation}\label{eq:g 0 iota A}g_0^*\iota^*A=\iota^*A'.
\end{equation} 
Since $P$ is trivializable, there exists a smooth gauge transformation $\wt g_0$ on $P$ such that $\iota^*\wt g_0=g_0$. 

Let $x\in\R^n$. We define $\iota_x:\R\to\R^{n+1}$ by $\iota_x(t):=(x,t)$. There exists a unique smooth gauge transformation $h_x$ on $\iota_x^*P\to\R$, such that
\begin{equation}\label{eq:iota x wt g 0}h_x^*\iota_x^*\wt g_0^*A=\iota_x^*A',\quad h_x(p)=\one,\,\forall p\in\textrm{ fiber of }\iota_x^*P\textrm{ over }0\in\R.
\end{equation}
To see this, note that these conditions can be viewed as an ordinary differential equation for $h_x$ with prescribed initial value. Since this solution depends smoothly on $x$, there exists a unique smooth gauge transformation $h$ on $P$ such that $\iota_x^*h=h_x$, for every $x\in\R^n$. The gauge transformation $g:=\wt g_0h$ on $P$ satisfies the equation $A'=g^*A$. This follows from (\ref{eq:g 0 iota A},\ref{eq:iota x wt g 0}) and flatness of $A$ and $A'$. This proves Proposition \ref{prop:ka 0}.
\end{proof}
The next result was used in the proofs of Proposition \ref{prop:cpt bdd}, Remark \ref{rmk:bar J} (Section \ref{sec:comp}), and Theorem \ref{thm:bubb}. Let $M,\om,G,\g,\lan\cdot,\cdot\ran_\g,\mu,J,\Si,\om_\Si,j$ be as in Chapter \ref{chap:main}. We define the almost complex structure $\bar J$ on $\BAR M$ as in (\ref{eq:bar J}). The energy density of a map $f\in W^{1,p}(\Si,\BAR M)$ is given by
\[e_f(z):=\frac12|df|^2,\]
where the norm is with respect to the metrics $\om_\Si(\cdot,j\cdot)$ on $\Si$ and $\BAR\om(\cdot,\bar J\cdot)$ on $\BAR M$. Let $P$ be a smooth $G$-bundle over $\Si$, $A$ a connection on $P$, and $u:P\to M$ an equivariant map. We define
\[e^\infty_{A,u}:=\frac12|d_Au|^2,\]
where the norm is taken with respect to the metrics $\om_\Si(\cdot,j\cdot)$ on $\Si$ and $\om(\cdot,J\cdot)$ on $M$. Furthermore, we define 
\[\bar u:\Si\to\BAR M,\quad\bar u(z):=Gu(p),\]
where $p\in P$ is an arbitrary point in the fiber over $z$. 
\begin{prop}[Pseudo-holomorphic curves in the symplectic quotient]\label{prop:bar del J} Let $P$ be a smooth $G$-bundle over $\Si$, $p>2$, $A$ a $W^{1,p}_\loc$-connection on $P$, and $u:P\to M$ a $G$-equivariant map of class $W^{1,p}_\loc$, such that $\mu\circ u=0$. Then we have
\[e_{\bar u}=e^\infty_{A,u}.\]
If $(A,u)$ also solves the equation $\bar\dd_{J,A}(u)=0$ then
\begin{eqnarray}\nn\bar\dd_{\bar J}\bar u=0.
\end{eqnarray}
\end{prop}
\begin{proof}[Proof of Proposition \ref{prop:bar del J}] This follows from an elementary argument. For the second part see also \cite[Section 1.5]{Ga}. 
\end{proof}
In the proof of Theorem \ref{thm:bubb} we used the following lemma.
\begin{lemma}[Bound for tree]\label{le:weight} Let $k\in\N_0$ be a number, $(T,E)$ a finite tree, $\al_1,\ldots,\al_k\in T$ vertices, $f:T\to[0,\infty)$ a function, and $E_0>0$ a number. Assume that for every vertex $\al\in T$ we have 
\begin{equation}
\label{eq:f al E}f(\al)\geq E_0\quad\textrm{or}\quad\#\big\{\be\in T\,|\,\al E\be\big\}+\#\big\{i\in\{1,\ldots,k\}\,|\,\al_i=\al\big\}\geq3.
\end{equation}
Then 
\begin{equation}
\nn\#T\leq\frac{2\sum_{\al\in T}f(\al)}{E_0}+k.
\end{equation}
\end{lemma}
\begin{proof}[Proof of Lemma \ref{le:weight}] This follows from an elementary argument.%
\footnote{See e.g.~\cite[Exercise 5.1.2.]{MS04}.}%
\end{proof}
The next result was used in the proof of Proposition \ref{prop:soft} (Section \ref{sec:soft}). Let $(X,d)$ be a metric space%
\footnote{$d$ is allowed to attain the value $\infty$.}%
, $G$ a topological group, and $\rho:G\x X\to X$ a continuous action by isometries. By $\pi:X\to X/G$ we denote the canonical projection. The topology on $X$, determined by $d$, induces a topology on the quotient $X/G$. 

\begin{lemma}[Induced metric on the quotient]\label{le:metr} Assume that $G$ is compact. Then the map $\bar d:X/G\x X/G\to [0,\infty]$ defined by
\[\bar d(\bar x,\bar y):=\min_{x\in\bar x,\,y\in\bar y}d(x,y)\]
is a metric on $X/G$ that induces the quotient topology on $X/G$.
\end{lemma}

\begin{proof}[Proof of Lemma \ref{le:metr}] This follows from an elementary argument. 
\end{proof}
The following two lemmas were used in the proof of Proposition \ref{prop:loc conv S 1 C} (Section \ref{sec:proof:prop:conv S 1 C}).
\begin{lemma}\label{le:subsubseq} Let $X$ be a topological space, $x\in X$, and $x_\nu\in X$ be a sequence. Then $x_\nu$ converges to $x$, as $\nu\to\infty$, if and only if for every subsequence $(\nu_i)_{i\in\N}$ there exists a further subsequence $(i_j)_{j\in\N}$, such that $x_{\nu_{i_j}}$ converges to $x$, as $j\to\infty$.
\end{lemma}
\begin{proof}[Proof of Lemma \ref{le:subsubseq}] This follows from an elementary argument.
\end{proof}
Let $X$ be a topological space and $G$ a group. We fix an action of $G$ on $X$. 
\begin{lemma}[Convergence in the quotient]\label{le:X G} Assume that $X$ is first-countable and that the action of every $g\in G$ is a continuous self-map of $X$. Let $y_\nu\in X/G$, $\nu\in\N$, be a sequence that converges to a point $y\in X/G$, and $x$ be a representative of $y$. Then there exists a representative $x_\nu$ of $y_\nu$, for each $\nu\in\N$, such that $x_\nu$ converges to $x$. 
\end{lemma}
\begin{proof}[Proof of Lemma \ref{le:X G}] This follows from an elementary argument.
\end{proof}
\subsubsection*{The connection $\na^A$}\label{na A}
Next we explain the twisted connection $\na^A$, which appeared in the definition of the space $\X_W^{p,\lam}$, occurring in Theorem \ref{thm:Fredholm}. Let $E\to M$ be a real (smooth) vector bundle. We denote by $\Co(E)$ the affine space of (smooth linear) connections on $E$. Let $\na^E\in\Co(E)$. Let $N$ be a smooth manifold, and $u:N\to M$ be a smooth map. We denote by $u^*E\to N$ the pullback bundle. The pullback connection $u^*\na^E\in\Co(u^*E)$ is uniquely determined by the equality
\[(u^*\na^E)_v(s\circ u)=\na^E_{u_*v}s,\quad\forall v\in TN,\,s\in\Ga(E).\]
Let $G$ be a Lie group, $\pi:P\to X$ a (right-)$G$-bundle, and $E\to P$ a $G$-equivariant vector bundle. Then the quotient $E/G$ has a natural structure of a vector bundle over $X$. Assume that $G$ acts on a manifold $M$, and let $E\to M$ be a $G$-equivariant vector bundle. We denote by $\Co^G(E)$ the space of $G$-invariant connections on $E$. We fix $A\in \A(P)$, $\na^E\in\Co^G(E)$, and $u\in C_G^\infty(P,M)$. We define
\begin{equation}\label{eq:wt na A}\wt\na^A\in\Co^G(u^*E),\quad\wt\na^A_{\wt v}\wt s:=(u^*\na^E)_{\wt v-p(A\wt v)}\wt s,
\end{equation} 
for $\wt s\in\Ga(u^*E)$, $p\in P$, and $\wt v\in T_pP$. We denote by $E^u$ the quotient bundle $(u^*E)/G\to X$. We define the connection $\na^A\in\Co(E^u)$ by
\begin{equation}\label{eq:na A}\na^A_vs:=G\cdot(p_0,\wt\na^A_{\wt v}\wt s),
\end{equation}
for $s\in\Ga(E^u)$ and $v\in TX$, where $(p_0,\wt v)\in TP$ is an arbitrary vector such that $\pi_*\wt v=v$, and $\wt s\in\Ga(u^*E)$ is the $G$-invariant section defined by $s\circ \pi(p)=G\cdot(p,\wt s(p))$, for every $p\in P$. This definition is independent of the choice of $(p_0,\wt v)$, since the connection $\wt\na^A$ is basic (i.e., $G$-invariant and horizontal).\\

The following lemma was mentioned in Chapter \ref{chap:main}. Let $M,\om,G,\g,\lan\cdot,\cdot\ran_\g,\mu,J$ be as in that Chapter. Let $p>2$, $\lam\in\R$, and $P\to\C$ be a $G$-bundle of class $W^{2,p}_\loc$. Recall the definition (\ref{eq:BB p lam P}) of $\BB^p_\lam(P)$. We denote by $\G^{2,p}_\loc(P)$ the group of gauge transformations on $P$ of class $W^{2,p}_\loc$. 
\begin{lemma}\label{le:free} If $\lam>1-2/p$ then the group $\G^{2,p}_\loc(P)$ acts freely on the set $\BB^p_\lam(P)$.
\end{lemma}
\begin{proof}[Proof of Lemma \ref{le:free}]\setcounter{claim}{0} Assume that $\lam>1-2/p$. Let $w:=(A,u)\in\BB^p_\lam(P)$ and $g\in\G^{2,p}_\loc(P)$ be such that $g_*w=w$. Let $p_1\in P$. We show that $g(p_1)=\one$. It follows from hypothesis (H) that there exists $\de>0$ such that $\mu^{-1}(B_\de)\sub M^*$ (defined as in (\ref{eq:M *})). Furthermore, Lemma \ref{le:si} (Appendix \ref{sec:proofs homology}) implies that $u(p_0)\in\mu^{-1}(B_\de)$, for every $p_0\in P$, for which $|\pi(p_0)|$ is large enough. We fix such a point $p_0$. Our hypothesis $p>2$ implies that $P$ is a $C^1$-bundle. Hence we may choose a path $p\in C^1([0,1],P)$ such that $p(i)=p_i$, for $i=0,1$. Consider the map $h:=g\circ p:[0,1]\to G$. By assumption, we have $g_*u=g\circ u=u$. Since $u(p_0)\in M^*$, it follows that $h(0)=\one$. Furthermore, the assumption $g_*A=A$ implies that $h$ solves the ordinary differential equation
\begin{equation}\label{eq:dot h}\dot h=hA\dot p-(A\dot p)h.
\end{equation}
The hypothesis $p>2$ implies that the map $A\dot p:[0,1]\to\g$ is continuous. Hence the equation (\ref{eq:dot h}) is of the form $\dot h(t)=f(t,h(t))$, where $f$ is continuous in $t$ and Lipschitz continuous in $h$. Therefore, by the Picard-Lindel\"of theorem, we have $h\const\one$. In particular, we have $g(p_1)=h(1)=\one$. It follows that $g\const\one$. This proves Lemma \ref{le:free}.
\end{proof}
The next lemma was used in the proof of Theorem \ref{thm:Fredholm} (Section \ref{subsec:reform}). Here for a linear map $D:X\to Y$ we denote $\coker D:=Y/\im D$. 
\begin{lemma}\label{le:X Y Z} Let $X,Y,Z$ be vector spaces and $D':X\to Y$ and $T:X\to Z$ be linear maps. We define $D:=D'|_{\ker T}$. Then the following holds.
  \begin{enui}\item\label{le:X Y Z ker} $\ker D=\ker(D', T)$. 
\item\label{le:X Y Z coker} The map $\Phi:\coker D\to \coker(D',T)$, $\Phi(y+\im D):=(y,0)+\im(D',T)$, is well-defined and injective. If $T:X\to Z$ is surjective then $\Phi$ is also surjective. 
\item\label{le:X Y Z closed} Let $\Vert\cdot\Vert_Y,$ $\Vert\cdot\Vert_Z$ be norms on $Y$ and $Z$ and assume that $\im(D',T)$ is closed in $Y\oplus Z$. Then $\im D$ is closed in $Y$. 
\end{enui}
\end{lemma}
The proof of Lemma \ref{le:X Y Z} is straight-forward and left to the reader. 

The following result was used in the proof of Proposition \ref{prop:triv} in Section \ref{subsec:proofs}. We define the map $f:\C\wo\{0\}\to S^1$ by $f(z):=z/|z|$. For two topological spaces $X$ and $Y$ we denote by $C(X,Y)$ the set of all continuous maps from $X$ to $Y$, and by $[X,Y]$ the set of all (free) homotopy classes of such maps. Let $V$ be a finite dimensional complex vector space. We denote by $\End(V)$ the space of its (complex) endomorphisms of $V$, by $\det:\End(V)\to\C$ the determinant map, and by $\Aut(V)\sub \End(V)$ the group of automorphisms of $V$. 
\begin{lemma}\label{le:Aut} The map $C(S^1,\Aut(V))\to \Z$ given by $\Phi\mapsto \deg\big(f\circ\det\circ\Phi\big)$ descends to a bijection $\big[S^1,\Aut(V)\big]\to \Z$. 
\end{lemma}
\begin{proof}[Proof of Lemma \ref{le:Aut}]\setcounter{claim}{0} We choose a hermitian inner product $V$ and denote by $\U(V)$ the corresponding group of unitary automorphisms of $V$. The map $\det:\U(V)\to S^1$ induces an isomorphism of fundamental groups, see e.g.~\cite[Proposition 2.23]{MS98}. Furthermore, the space $\Aut(V)$ strongly deformation retracts onto $\U(V)$.%
\footnote{This follows from the Gram-Schmidt orthonormalization procedure.}
 Let $\Phi_0\in\Aut(V)$. It follows that the map
\[\big\{\Phi\in C(S^1,\Aut(V))\,\big|\,\Phi(1)=\Phi_0\big\}\to \Z,\quad\Phi\mapsto \deg\big(f\circ\det\circ\Phi\big)\]
descends to an isomorphism between the fundamental group $\pi_1(\Aut(V),\Phi_0)$ and $\Z$. Since this group is abelian, the map
\[\pi_1(\Aut(V),\Phi_0)\to \big[S^1,\Aut(V)\big]\]
that forgets the base point $\Phi_0$, is a bijection. The statement of Lemma \ref{le:Aut} follows from this. 
\end{proof}
The next lemma was used in the proof of Theorem \ref{thm:L w * R} (Section \ref{subsec:proof:thm:L w * R}). Let $X$ and $M$ be manifolds, $G$ a Lie group with Lie algebra $\g$, $\lan\cdot,\cdot\ran_\g$ an invariant inner product on $\g$, $\lan\cdot,\cdot\ran_M$ a $G$-invariant Riemannian metric on $M$, and $\na$ its Levi-Civita connection. For $\xi\in\g$ we denote by $X_\xi$ the vector field on $M$ generated by $\xi$. We define the tensor $\rho:TM\oplus TM\to\g$ by 
\begin{equation}\label{eq:xi rho v v'}\lan\xi,\rho(v,v')\ran_\g:=\lan\na_vX_\xi,v'\ran_M.\end{equation}
A short calculation shows that $\rho$ is skew-symmetric. This two-form was introduced in \cite[p.~181]{Ga}. The next lemma corresponds to \cite[Proposition 7.1.3(a,b)]{Ga}. Let $P\to X$ be a $G$-bundle, $A\in\A(P)$, $u\in C^\infty\big(X,(P\x M)/G\big)$, $v\in\Ga(TM^u)$, and $\xi\in \Ga(\g_P)$. We define the connection $\na^A$ on $TM^u\to X$ as on page \pageref{na A}.
\begin{lemma}\label{le:na A L u} $\na^AL_u\xi-L_ud_A\xi=\na_{d_Au}X_\xi,\quad d_AL_u^*v-L_u^*\na^Av=\rho(d_Au,v).$ 
\end{lemma}
\begin{proof}[Proof of Lemma \ref{le:na A L u}]\setcounter{claim}{0} This follows from short calculations.
\end{proof}
Let $M,\om,G,\g,\lan\cdot,\cdot\ran_\g,\mu$ and $J$ be as in Chapter \ref{chap:main}, and $\lan\cdot,\cdot\ran_M:=\om(\cdot,J\cdot)$. The following remark was used in the proofs of Theorems \ref{thm:Fredholm} (Section \ref{subsec:reform}) and \ref{thm:L w * R} (Section \ref{subsec:proof:thm:L w * R}). Recall the definition (\ref{eq:M *}) of $M^*\sub M$, and that $\PR:TM\to TM$ denotes the orthogonal projection onto $\im L$. 
\begin{rmk}\label{rmk:c}Let $K\sub M^*$ be compact. We define
\[c:=\inf\left\{\frac{|L_x\xi|}{|\xi|}\,\bigg|\,x\in K,\,0\neq\xi\in\g\right\}.\]
Then $c>0$. Let $x\in K$. Then $L_x^*L_x$ is invertible, and
\begin{equation}\label{eq:L x * L x c}|(L_x^*L_x)^{-1}|\leq c^{-2},\quad\big|L_x(L_x^*L_x)^{-1}\big|\leq c^{-1},\quad L_x(L_x^*L_x)^{-1}L_x^*={\PR}_x,\end{equation}
where the $|\cdot|$'s denote operator norms. Furthermore, $|\PR_x v|\leq c^{-1}|L_x^*v|$, for every $v\in T_xM$. These assertions follow from short calculations. $\Box$
\end{rmk}
Assume that hypothesis (H) holds. The following lemma was used in the proof of Proposition \ref{prop:X X w} in Section \ref{subsec:proofs}. For $x\in M$ we denote by $L_x^\C:\g^\C\to T_xM$ the complex linear extension of the infinitesimal action.
\begin{lemma}\label{le:U} There exists a neighborhood $U\sub M$ of $\mu^{-1}(0)$, such that
\begin{equation}\label{eq:c inf}c:=\inf\big\{\big|d\mu(x)L_x^\C\al\big|+|\PR L_x^\C\al|\,\big|\,x\in U,\,\al\in\g^\C:\,|\al|=1\big\}>0.\end{equation}
\end{lemma}
\begin{proof}[Proof of Lemma \ref{le:U}]\setcounter{claim}{0} It follows from hypothesis (H) that there exists $\de_0>0$ such that $\mu^{-1}(\bar B_{\de_0})\sub M^*$. We define
\begin{eqnarray}\nn&C:=\sup\big\{|[\xi,\eta]|\,\big|\,\xi,\eta\in\g:\,|\xi|\leq1,\,|\eta|\leq 1\big\},&\\
\nn&c_0:=\inf\left\{\displaystyle\frac{|L_x\xi|}{|\xi|}\,\bigg|\,x\in \mu^{-1}(\bar B_{\de_0}),\,0\neq\xi\in\g\right\}.&
\end{eqnarray}
Since the action of $G$ on $M^*$ is free, it follows that $L_x:\g\to T_xM$ is injective, for $x\in M^*$. Furthermore, by hypothesis (H) the set $\mu^{-1}(\bar B_{\de_0})$ is compact. It follows that $c_0>0$. We choose a positive number $\de<\min\{\de_0,c_0/C,c_0^3/C\}$, and we define $U:=\mu^{-1}(B_\de)$. 
\begin{Claim}Inequality (\ref{eq:c inf}) holds.
\end{Claim}
\begin{pf}[Proof of the claim] Let $x\in U$ and $\al=\xi+i\eta\in\g^\C$. Then 
\begin{equation}\label{eq:d mu x L x}d\mu(x)L_x^\C\al=[\mu(x),\xi]+L_x^*L_x\eta.\end{equation} 
Using the last assertion in (\ref{eq:L x * L x c}), we have 
\begin{equation}\label{eq:Pr x L x}{\PR}_xL_x^\C\al=L_x\xi-L_x(L_x^*L_x)^{-1}[\mu(x),\eta].\end{equation}
By the first assertion in (\ref{eq:L x * L x c}), we have $|L_x^*L_x\eta|\geq c_0^2|\eta|$. Combining this with (\ref{eq:d mu x L x},\ref{eq:Pr x L x}) and the second assertion in (\ref{eq:L x * L x c}), we obtain
\[\big|d\mu(x)L_x^\C\al\big|+|\PR L_x^\C\al|\geq -C\de|\xi|+c_0^2|\eta|+c_0|\xi|-c_0^{-1}C\de|\eta|.\]
Inequality (\ref{eq:c inf}) follows now from our choice of $\de$. This proves the claim and completes the proof of Lemma \ref{le:U}.
\end{pf}
\end{proof}

\backmatter

\printindex


\begin{thebibliography}{99}

\bibitem[Ad]{Ad} D.~R.~Adams, L.~I.~Hedberg, \emph{Function Spaces and Potential Theory}, Grundlehren der Mathematischen Wissenschaften (Fundamental Principles of Mathematical Sciences), {\bf 314}, Springer-Verlag, Berlin, 1996, xii+366 pp.

\bibitem[Ba]{Ba} R.~Bartnik, \emph{The Mass of an Asymptotically Flat Manifold}, Comm.~Pure Appl.~Math.~{\bf 39} (1986), {\bf no.~5}, 661--693. 

\bibitem[CGS]{CGS} K.~Cieliebak, A.R.~Gaio and D.A.~Salamon, \emph{$J$-holomorphic curves, moment maps, and invariants of Hamiltonian group actions}, Internat.~Math.~Res.~Notices  2000, {\bf no.~16}, 831--882.

\bibitem[CGMS]{CGMS} K.~Cieliebak, A.~R.~Gaio and I.~Mundet i Riera, D.~A.~Salamon, \emph{The symplectic vortex equations and invariants of Hamiltonian group actions}, Journal of Symplectic Geometry {\bf 1} (2002), {\bf no.~3}, 543--645.

\bibitem[CS]{CS} K.~Cieliebak and D.A.~Salamon, \emph{Wall crossing for symplectic vortices and quantum cohomology}, Math.~Ann.~{\bf 335} (2006), {\bf no.~1}, 133--192. 

\bibitem[Ev]{Ev} L.~C.~Evans, \emph{Partial differential equations}, Graduate Studies in Mathematics, {\bf 19}, American Mathematical Society, Providence, RI, 1998, xviii+662 pp.

\bibitem[Fr1]{FrPhD} U.~Frauenfelder, \emph{Floer Homology of Symplectic Quotients and the Arnold-Givental Conjecture}, Ph.D.-thesis, Eidgen\"ossische Technische Hochschule Z\"urich (Switzerland), 2003, 235 pp.

\bibitem[Fr2]{FrArnold} U.~Frauenfelder, \emph{The Arnold-Givental Conjecture and Moment Floer Homology}, Int.~Math.~Res.~Not.~{\bf 42} (2004), 2179--2269. 

\bibitem[Ga]{Ga} A.~R.~Gaio, \emph{J-Holomorphic Curves and Moment Maps}, Ph.D.-thesis, University of Warwick (U.K.), 2000, 219 pp.

\bibitem[GS]{GS} A.R.~Gaio and D.A.~Salamon, \emph{Gromov-Witten invariants of symplectic quotients and adiabatic limits}, J.~Symplectic Geom.~{\bf 3}  (2005), {\bf no.~1}, 55--159. 

\bibitem[GL]{GL} V.~L.~Ginzburg and L.~D.~Landau, \emph{On the theory of superconductors}, Zh.~Eksp.~Teor.~Fiz.~(Sov.~Phys.~JETP) {\bf 20} (1950), 1064--1082. English translation in: L.~D.~Landau, \emph{Collected Papers}, Oxford, Pergamon Press, 1965, p.~546.

\bibitem[Gi]{Gi} A.~Givental, \emph{Equivariant Gromov-Witten Invariants}, Int.~Math.~Res.~Not., 1996, {\bf no.~13}, 613--663.

\bibitem[GiK]{GiK} A.~Givental and B.~Kim, \emph{Quantum cohomology of flag manifolds and Toda lattices}, Comm.~Math.~Phys.~{\bf 168} (1995), {\bf no.~3}, 609--641.

\bibitem[GW1]{GWToric} E.~Gonzalez and C.~Woodward, \emph{Quantum Cohomology of Toric Orbifolds}, arXiv:1207.3253.

\bibitem[GW2]{GWWall} E.~Gonzalez and C.~Woodward, \emph{A wall-crossing formula for Gromov-Witten invariants under variation of git quotient}, arXiv:1208.1727.

\bibitem[Is]{Is} T.~Isobe, \emph{Topological and analytical properties of Sobolev bundles, I: The critical case}, Ann.~Global Anal.~Geom.~{\bf 35} (2009), {\bf no.~3}, 277--337.

\bibitem[IS]{IS} S.~Ivashkovich, V.~Shevchishin, \emph{Gromov compactness theorem for $J$-complex curves with boundary}, Internat.~Math.~Res.~Notices 2000, {\bf no.~22}, 1167--1206.

\bibitem[JT]{JT} A.~Jaffe and C.~Taubes, \emph{Vortices and monopoles, Structure of static gauge theories}, Progress in Physics, {\bf 2}, Birkh\"auser, Boston, Mass., 1980, v+287 pp.

\bibitem[JK]{JK} L.~C.~Jeffrey, F.~Kirwan, \emph{Localization for nonabelian group actions}, Topology {\bf 34} (1995), {\bf no.~2}, 291--327. 

\bibitem[Kav]{Kav} O.~Kavian, \emph{Introduction \`a la th\'eorie des points critiques et applications aux probl\`emes elliptiques} (Introduction to critical point theory and applications to elliptic problems), Math\'ematiques \& Applications (Berlin), {\bf 13}, Springer, 1993, viii+325 pp. 

\bibitem[Kir]{Kir} F.~Kirwan, \emph{Cohomology of Quotients in Symplectic and Algebraic Geometry}, Mathematical Notes, {\bf 31}, Princeton University Press, Princeton, NJ, 1984, i+211 pp. 

\bibitem[Kim]{KimEq} B.~Kim, \emph{On Equivariant Quantum Cohomology}, Internat.~Math.~Res.~Notices, 1996, {\bf no.~17}, 841--851.

\bibitem[Ko]{Kontsevich} M.~Kontsevich, \emph{Enumeration of rational curves via torus actions}, 335--368 in R.~Dijkgraaf, C.~Faber, G.~van der Geer, eds., \emph{The moduli space of curves}, Progress in Mathematics {\bf 129}, Birkh\"auser Boston, 1995. 

\bibitem[Lo1]{Lockhart PhD} R.~B.~Lockhart, \emph{Fredholm properties of a class of elliptic operators on non-compact manifolds}, Ph.D.-thesis, University of Illinois at Urbana-Champaign (U.S.A.), 1979, 91 pp.

\bibitem[Lo2]{Lockhart Fred} R.~B.~Lockhart, \emph{Fredholm properties of a class of elliptic operators on non-compact manifolds}, Duke Math.~J.~{\bf 48} (1981), {\bf no.~1}, 289--312. 

\bibitem[Lo3]{Lockhart Hodge} R.~B.~Lockhart, \emph{Fredholm, Hodge and Liouville Theorems on Noncompact Manifolds}, Trans.~Amer.~Math.~Soc.~{\bf 301} (1987), {\bf no.~1}, 1--35. 

\bibitem[LM1]{LM ell R n} R.~B.~Lockhart and R.~C.~McOwen, \emph{On elliptic systems in $\R^n$}, Acta Math.~150 (1983), 125--135.

\bibitem[LM2]{LM ell mf} R.~B.~Lockhart and R.~C.~McOwen, \emph{Elliptic Differential Operators on Noncompact Manifolds}, Ann.~Sc.~Norm.~Sup.~Pisa {\bf 12} (1985), 409--447.

\bibitem[Lu]{Lu} P.~Lu, \emph{A rigorous definition of fiberwise quantum cohomology and equivariant quantum cohomology},  Comm.~Anal.~Geom.~{\bf 6} (1998), {\bf no.~3}, 511--588.

\bibitem[McO1]{McOwen Lap} R.~C.~McOwen, \emph{The Behavior of the Laplacian on Weighted Sobolev Spaces}, Comm.~Pure Appl.~Math.~{\bf 32} (1979), 783--795.

\bibitem[McO2]{McOwen Fred} R.~C.~McOwen, \emph{Fredholm theory of partial differential equations on complete Riemannian manifolds}, Pacific J.~Math.~{\bf 87}, {\bf no.~1} (1980), 169--185.

\bibitem[McO3]{McOwen ell} R.~C.~McOwen, \emph{On elliptic operators in $\R^n$}, Comm.~Part.~Diff.~Eq.~{\bf 5} (1980), 913--933.

\bibitem[MS1]{MS98} D.~McDuff and D.A.~Salamon, \emph{Introduction to Symplectic Topology}, 2nd edition, Oxford Mathematical Monographs, The Clarendon Press, Oxford University Press, 1998, x+486 pp.

\bibitem[MS2]{MS04} D.~McDuff and D.~A.~Salamon, \emph{J-Holomorphic Curves and Symplectic Topology}, AMS Colloquium Publications, {Vol.~\bf 52}, American Mathematical Society, Providence, RI, 2004, xii+669 pp.

\bibitem[Mu1]{MuPhD} I.~Mundet i Riera, \emph{Yang-Mills-Higgs theory for symplectic fibrations}, Ph.D.-thesis, Universidad Aut\'onoma de Madrid (Spain), 1999, 159 pp.

\bibitem[Mu2]{MuHam} I.~Mundet i Riera, \emph{Hamiltonian Gromov-Witten invariants}, Topology {\bf 42} (2003), 525--553.

\bibitem[MT]{MT} I.~Mundet i Riera and G.~Tian, \emph{A Compactification of the Moduli Space of Twisted Holomorphic Maps}, Adv.~Math.~{\bf 222} (2009), {\bf no.~4}, 1117--1196. 

\bibitem[NWZ]{NWZ} K.~L.~Nguyen, C.~Woodward, F.~Ziltener, \emph{Morphisms of CohFT Algebras and Quantization of the Kirwan Map}, arXiv:0903.4459, to appear in Proceedings of the Hayashibara Forum. 

\bibitem[Ott]{Ott} A.~Ott, \emph{Removal of singularities and Gromov compactness for symplectic vortices}, arXiv:0912.2500, to appear in J.~Symplectic Geom.

\bibitem[ReSi]{ReSi} M.~Reed and B.~Simon, \emph{Methods of Modern Mathematical Physics, I, Functional analysis}, Academic Press, New York, 1980, xv+400 pp.

\bibitem[RoSa]{RoSa} J.~Robbin and D.~A.~Salamon, \emph{The spectral flow and the Maslov index}, Bull.~London Math.~Soc.~{\bf 27} (1995), {\bf no.~1}, 1--33.

\bibitem[Ru]{Ru} Y.~Ruan, \emph{Virtual neighborhoods and pseudo-holomorphic curves}, Proceedings of 6th G\"okova Geometry-Topology Conference.~Turkish J.~Math.~{\bf 23} (1999), {\bf no.~1}, 161--231.

\bibitem[Sa]{Sa} D.~A.~Salamon, \emph{Lectures on Floer Homology}, in \emph{Symplectic Geometry and Topology}, edited by Y.~Eliashberg and L.~Traynor, IAS/Park City Mathematics Series, Vol.~{\bf 7} (1999), 143--230.

\bibitem[SZ]{SZ} D.~A.~Salamon, E.~Zehnder, \emph{Morse theory for periodic solutions of Hamiltonian systems and the Maslov index}, Comm.~Pure Appl.~Math.{\bf~45} (1992), 1303--1360.

\bibitem[Sp1]{SpGW} H.~Spielberg, \emph{The Gromov-Witten invariants of symplectic toric manifolds, and their quantum cohomology ring}, C.~R.~Acad.~Sci.~Paris S\'er.~{\bf I} Math.~{\bf 329} (1999), {\bf no.~8}, 699--704.

\bibitem[Sp2]{SpHirzebruch} H.~Spielberg, \emph{Counting generic genus-0 curves on Hirzebruch surfaces}, Proc.~Amer.~Math.~Soc.~{\bf 130} (2002), {\bf no.~5}, 1257--1264 (electronic). 

\bibitem[VW]{VW} S.~Venugopalan and C.~Woodward, \emph{Hitchin-Kobayashi correspondence for affine vortices}, in preparation.

\bibitem[TW]{TW} S.~Tolman and J.~Weitsman, \emph{The Cohomology Rings of Symplectic Quotients}, Comm.~Anal.~Geom.~{\bf 11} (2003), {\bf no.~4}, 751--773. 

\bibitem[Uh]{Uhlenbeck} K.~Uhlenbeck, \emph{Connections with $L^p$ Bounds on Curvature}, Comm.~Math.~Phys.~{\bf 83} (1982), 31--42.

\bibitem[We]{We} K.~Wehrheim, \emph{Uhlenbeck Compactness}, EMS Series of Lectures in Mathematics, European Mathematical Society, Z\"urich, 2004, viii+212 pp.

\bibitem[Wi]{Wi} E.~Witten, \emph{The Verlinde algebra and the cohomology of the Grassmannian}, Geometry, topology, \& physics, 357--422, Conf.~Proc.~Lect.~Notes Geom.~Top., {\bf IV}, Int.~Press, Cambridge, MA, 1995. 

\bibitem[Wo]{Wo} C.~Woodward, \emph{Quantum Kirwan Morphism and Gromov-Witten Invariants of Quotients}, arXiv:1204.1765.

\bibitem[Zi1]{ZiPhD} F.~Ziltener, \emph{Symplectic Vortices on the Complex Plane and Quantum Cohomology}, Ph.D.-thesis, Eidgen\"ossische Technische Hochschule Z\"urich (Switzerland), 2006, 258 pp.

\bibitem[Zi2]{ZiA} F.~Ziltener, \emph{The Invariant Symplectic Action and Decay for Vortices}, J.~Symplectic Geom.~{\bf 7} (2009), {\bf no.~3}, 357--376.

\bibitem[Zi3]{ZiMaslov} F.~Ziltener, \emph{A Maslov Map for Coisotropic Submanifolds, Leaf-wise Fixed Points and Presymplectic Non-Embeddings}, arXiv:0911.1460.

\bibitem[Zi4]{ZiConsEv} F.~Ziltener, \emph{A Quantum Kirwan Map: Conservation of Homology and an Evaluation Map}, in preparation.

\bibitem[Zi5]{ZiTrans} F.~Ziltener, \emph{A Quantum Kirwan Map: Transversality}, in preparation.

\end{thebibliography}
\end{document}